\documentclass{memo-l}

\usepackage{graphicx}
\usepackage{pinlabel}
\usepackage{tikz-cd}
\usepackage{enumitem}
\usepackage{subfig, caption}
\usepackage{amssymb}
\usepackage{longtable, booktabs}
\usepackage{tabularx}
\usepackage{microtype}
\usepackage[hidelinks, pagebackref]{hyperref}

\makeatletter
\define@key{href}{font}{#1}
\makeatother
\usepackage{xpatch}
\newcommand\hrefdefaultfont{\ttfamily}
\xpatchcmd\href{\setkeys{href}{#1}}{\setkeys{href}{font=\hrefdefaultfont,#1}}{}{\fail}

\renewcommand*{\backref}[1]{}
\renewcommand*{\backrefalt}[4]{
  \ifcase #1 
  [No citations.]
  \or [#2]
  \else [#2]
  \fi }

\numberwithin{subsection}{chapter} 
\numberwithin{figure}{chapter} 
\makeatletter
\let\c@figure\c@subsection
\makeatother

\theoremstyle{plain}
\newtheorem{theorem}[subsection]{Theorem}
\newtheorem{lemma}[subsection]{Lemma}
\newtheorem{claim}[subsection]{Claim}
\newtheorem{corollary}[subsection]{Corollary}
\newtheorem{proposition}[subsection]{Proposition}
\newtheorem*{claim*}{Claim}

\theoremstyle{definition}
\newtheorem{definition/}[subsection]{Definition}
\newtheorem{example/}[subsection]{Example}
\newtheorem{question/}[subsection]{Question}

\theoremstyle{remark}
\newtheorem{remark/}[subsection]{Remark}
\newtheorem{case}{Case}

\newenvironment{remark}
  {%
   \pushQED{\qed}\begin{remark/}}
  {\popQED\end{remark/}}
  
\newenvironment{definition}
  {%
   \pushQED{\qed}\begin{definition/}}
  {\popQED\end{definition/}}

\newenvironment{example}
  {%
   \pushQED{\qed}\begin{example/}}
  {\popQED\end{example/}}

\newenvironment{question}
  {%
   \pushQED{\qed}\begin{question/}}
  {\popQED\end{question/}}

\newcommand{\refcha}[1]{Chapter~\ref{Cha:#1}}
\newcommand{\refsec}[1]{Section~\ref{Sec:#1}}
\newcommand{\reflem}[1]{Lemma~\ref{Lem:#1}}
\newcommand{\refrem}[1]{Remark~\ref{Rem:#1}}
\newcommand{\refdef}[1]{Definition~\ref{Def:#1}}
\newcommand{\refthm}[1]{Theorem~\ref{Thm:#1}}
\newcommand{\refprop}[1]{Proposition~\ref{Prop:#1}}
\newcommand{\reffig}[1]{Figure~\ref{Fig:#1}}
\newcommand{\refexa}[1]{Example~\ref{Exa:#1}}
\newcommand{\refitm}[1]{(\ref{Itm:#1})}
\newcommand{\refcor}[1]{Corollary~\ref{Cor:#1}}
\newcommand{\refapp}[1]{Appendix~\ref{App:#1}}
\newcommand{\refque}[1]{Question~\ref{Que:#1}}
\newcommand{\refclm}[1]{Claim~\ref{Clm:#1}}

\newcommand{\calA}{\mathcal{A}}
\newcommand{\calB}{\mathcal{B}}
\newcommand{\calC}{\mathcal{C}}
\newcommand{\calF}{\mathcal{F}}
\newcommand{\calI}{\mathcal{I}}
\newcommand{\calK}{\mathcal{K}}
\newcommand{\calL}{\mathcal{L}}
\newcommand{\calM}{\mathcal{M}}
\newcommand{\calO}{\mathcal{O}}
\newcommand{\calP}{\mathcal{P}}
\newcommand{\calS}{\mathcal{S}}
\newcommand{\calT}{\mathcal{T}}
\newcommand{\calU}{\mathcal{U}}
\newcommand{\calV}{\mathcal{V}}
\newcommand{\calW}{\mathcal{W}}

\newcommand{\Circle}{S^1(\calV)}
\newcommand{\Disk}{D^2(\calV)}
\newcommand{\Sphere}{S^2(\calV)}
\newcommand{\pair}{\operatorname{\mathsf{P}}}
\newcommand{\Mobius}{\calM(\calV)}
\newcommand{\elec}{\pair(\calV)/{\sim}}

\newcommand{\link}{\operatorname{\mathsf{L}}} 
\newcommand{\Loom}{\operatorname{\mathsf{Loom}}} 
\newcommand{\veer}{\operatorname{\mathsf{V}}} 
\newcommand{\Veer}{\operatorname{\mathsf{Veer}}} 
\newcommand{\leaf}{\operatorname{\mathsf{Le}}} 
\newcommand{\flowbox}{\operatorname{\mathsf{B}}} 
\newcommand{\FlowBox}{\operatorname{\mathsf{Box}^\circ}} 

\newcommand{\axis}{\operatorname{\mathsf{A}}}
\newcommand{\point}{\operatorname{\mathsf{p}}}
\newcommand{\rect}{\operatorname{\mathsf{R}}}
\newcommand{\cell}{\operatorname{\mathsf{c}}}
\newcommand{\astro}{\operatorname{Ast}}
\newcommand{\qui}{\operatorname{Qui}}

\newcommand{\triangleup}{\vartriangle}
\newcommand{\eu}{e^\odot}
\newcommand{\en}{e^\triangleup}
\newcommand{\es}{e^{\negthinspace\triangledown}}
\newcommand{\id}{\operatorname{\mathsf{id}}}

\newcommand{\Aut}{\operatorname{Aut}} 
\newcommand{\Stab}{\operatorname{Stab}}
\newcommand{\PSL}{\operatorname{PSL}} 
\newcommand{\Id}{\operatorname{Id}}
\newcommand{\dev}{\operatorname{dev}}
\newcommand{\homeo}{\mathrel{\cong}} 
\newcommand{\isom}{\cong} 
\newcommand{\bdy}{\partial} 
\newcommand{\subgp}[1]{{\langle #1 \rangle}}

\newcommand{\cover}[1]{{\widetilde{#1}}}
\newcommand{\closure}[1]{{\overline{#1}}}
\newcommand{\interior}{{\operatorname{interior}}}
\newcommand{\cross}{\times}

\newcommand{\CC}{\mathbb{C}}
\newcommand{\HH}{\mathbb{H}}
\newcommand{\NN}{\mathbb{N}}
\newcommand{\RR}{\mathbb{R}}
\newcommand{\ZZ}{\mathbb{Z}}

\newcommand{\CP}{\mathbb{CP}} 
\newcommand{\QP}{\mathbb{QP}} 

\newcommand{\from}{\colon} 
\newcommand{\st}{\mathbin{\mid}} 
\newcommand{\thsup}{{\rm th}}
\newcommand{\stsup}{{\rm st}}

\newcommand{\preacw}{\scalebox{0.7}{$\curvearrowleft$}}
\newcommand{\acw}{\rotatebox[origin=c]{-90}{\preacw}}

\newcommand{\fakeenv}{} 

\newenvironment{restate}[2]  
{ 
 \renewcommand{\fakeenv}{#2} 
 \theoremstyle{plain} 
 \newtheorem*{\fakeenv}{#1~\ref{#2}} 
 \begin{\fakeenv}
}
{
 \end{\fakeenv}
}

\makeindex

\begin{document}

\frontmatter

\title[There and back again]{From veering triangulations to link spaces\\ and back again}


\author{Steven Frankel}
\address{Washington University in St. Louis, St. Louis, MO}
\curraddr{}
\email{steven.frankel@wustl.edu}
\thanks{}

\author{Saul Schleimer}
\address{University of Warwick, Coventry, UK}
\curraddr{}
\email{s.schleimer@warwick.ac.uk}
\thanks{}

\author{Henry Segerman}
\address{Oklahoma State University, Stillwater, OK}
\curraddr{}
\email{henry.segerman@okstate.edu}
\thanks{}

\date{}

\subjclass[2020]{Primary 37C85, 57M50,  57M60}

\keywords{veering triangulations, circular orders, branched surfaces, train tracks, laminations, pseudo-Anosov flows}


\begin{abstract}
This paper is the third in a sequence establishing a dictionary between the combinatorics of veering triangulations equipped with appropriate filling slopes, and the dynamics of pseudo-Anosov flows (without perfect fits) on closed three-manifolds.

Our motivation comes from the work of Agol and Gu\'eritaud.
Agol introduced veering triangulations of mapping tori as a tool for understanding the surgery parents of pseudo-Anosov mapping tori.
Gu\'eritaud gave a new construction of veering triangulations of mapping tori using the orbit spaces of their suspension flows.
Generalising this, Agol and Gu\'eritaud announced a method that, given a closed manifold with a pseudo-Anosov flow (without perfect fits), produces a veering triangulation equipped with filling slopes.

In this paper we build, from a veering triangulation, a canonical circular order on the cusps of the universal cover.
Using this we build the veering circle and the link space.
These are the first entries in the promised dictionary.
The link space and the circle are, respectively, analogous to the orbit space of a flow and to Fenley's boundary at infinity of the orbit space.

In the other direction, and using our previous work, we prove that the veering triangulation is recovered (up to canonical isomorphism) from the dynamics of the fundamental group acting on the link space.
This is the first step in proving that our dictionary gives a bijection between the two theories.
\end{abstract}

\maketitle

\setcounter{tocdepth}{2}
\tableofcontents

\mainmatter

\chapter{Introduction}

This is the third in a series of papers~\cite{SchleimerSegerman20, SchleimerSegerman24} 
providing a dictionary between veering triangulations, equipped with appropriate filling slopes, and (topological) pseudo-Anosov flows (without perfect fits) on closed three-manifolds.
This fulfils part of Smale's program for dynamical systems, outlined in his 1967 Bulletin article~\cite{Smale67}; 
one first finds the appropriate notion of stability for flows (pseudo-Anosov without perfect fits), and then combinatorially classifies such systems (veering triangulations with filling slopes). 
Christy gives an exposition~\cite[page~759]{Christy93} of Smale's program, and proves (following Williams~\cite{Williams74}) that smooth transitive pseudo-Anosov flows determine and are determined by a dynamic branched surface.
Christy further promises~\cite[page~760]{Christy93} an algorithm to determine topological equivalence between such flows.
This work never appeared; such an algorithm (for pseudo-Anosov flows without perfect fits) is one consequence of our dictionary.

\begin{figure}[htbp]
\includegraphics[width=0.7\textwidth]{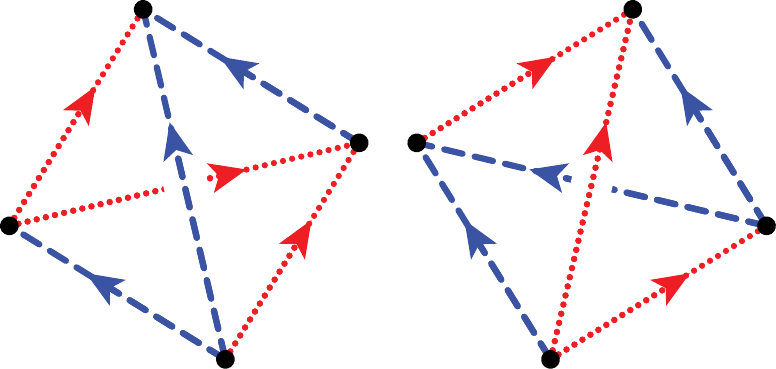}
\caption{The veering structure on the canonical triangulation for the figure-eight knot complement.  }
\label{Fig:VeerFigEightIntro}
\end{figure}

\subsection{The dictionary}
\label{Sec:Dictionary}

We outline our overall programme.

\subsubsection{Link spaces}
\label{Sec:IntroLinkSpaces}
In this third paper, given a \emph{veering triangulation} $\calV$ of a cusped three-manifold $M$, we build the \emph{link space} $\link(\cover{\calV})$ of the universal cover $\cover{\calV}$. 
(The link space is a copy of $\RR^2$ equipped with two foliations and an action of $\pi_1(M)$.)
To do this, we first associate to $\cover{\calV}$ a canonical circular order on its cusps. 
We then construct the \emph{veering circle} and a pair of invariant laminations. 
Following the exposition of Casson and Bleiler~\cite[Section~6]{CassonBleiler88}, we collapse these laminations to produce the link space $\link(\cover{\calV})$.
This done, we prove that the link space is a \emph{loom space}~\cite{SchleimerSegerman24}.
In fact, we show that $\link \from \Veer \to \Loom$ is a functor from the category of veering triangulations to the category of loom spaces.
In~\cite[Section~5]{SchleimerSegerman24} we give a functor $\veer \from \Loom \to \Veer$.
In this paper we show that $(\veer\circ\link)(\cover{\calV})$ is canonically isomorphic to $\cover{\calV}$.
Furthermore, any loom space $\calL$ is canonically isomorphic to $(\link\circ\veer)(\calL)$.
We give a more detailed outline of this paper in \refsec{Outline}. 

\subsubsection{Dynamic pairs}

In our fourth paper~\cite{SchleimerSegerman4} we isotope the \emph{upper and lower branched surfaces} $B^\calV$ and $B_\calV$  (see \refsec{UpperLowerSurfaces}) to form a \emph{dynamic pair} in the sense of Mosher~\cite[Section~2.4]{Mosher96}.
These branched surfaces may coincide in many sectors:
to tease them apart we first construct a semi-local notion of ``horizontal'' using a canonical decomposition of $|\calV|$ into \emph{shearing regions}.
We then perform an intricate sequence of isotopies.

In the appendix we find the \emph{maximal rectangles} in the glued-together horizontal cross-sections;
we show that these correspond to the original veering tetrahedra.
We also show that these rectangles ``overlap'' along face rectangles.

\subsubsection{Flow-box decompositions}

Following Mosher~\cite[Theorem~3.4.1]{Mosher96}, a dynamic pair gives an \emph{ideal flow-box decomposition} of the ambient cusped three-manifold.
Filling along appropriate slopes one obtains a topological pseudo-Anosov flow on the closed manifold.
Thus, in our fifth paper~\cite{SchleimerSegerman5}, we start with a veering triangulation (equipped with appropriate filling slopes), and applying our fourth paper, we produce a transitive pseudo-Anosov flow.
We prove that this flow has no \emph{perfect fits}~\cite[Definition~2.2]{Fenley12}.

\begin{remark}
\label{Rem:AgolTsang}
Agol and Tsang~\cite[Theorem~5.1]{AgolTsang24} have recently given another construction of a pseudo-Anosov flow (without perfect fits) from a veering triangulation (equipped with filling slopes).
Their construction begins with the \emph{flow graph}~\cite[Section~4.3]{LandryMinskyTaylor24}.
They then remove \emph{infinitesimal cycles}~\cite[Section~3]{AgolTsang24}.
Each remaining edge gives a flow box.
The resulting flow-box decomposition gives a pseudo-Anosov flow without perfect fits on the filled manifold.
\end{remark}

We then prove the following.
Suppose that $\calA$ and $\calB$ are (ideal) flow-box decompositions of a three-manifold.
Then there is an equivariant homeomorphism between their \emph{leaf spaces} $\leaf(\calA)$ and $\leaf(\calB)$ if and only if $\calA$ and $\calB$ are \emph{flow equivalent}:
that is, connected by a sequence of isotopies, \emph{slide-isotopies}, splittings, and merges.   
The forward direction is more difficult: 
we first match up the transverse foliations in $\calA$ and $\calB$. 
We then engulf the vertical faces of $\calA$ within those of $\calB$ by splitting $\calB$ and appealing to (local) finiteness.
Dealing with the horizontal faces is by far the hardest step. 
This requires a delicate induction on the topological complexity of their intersection.

Drilling the singular flow lines of a (material) flow-box decomposition $\calA$ gives an ideal flow-box decomposition $\calA^\circ$.
We show that if $\leaf(\calA)$ has no perfect fits then $\leaf(\calA^\circ)$ has the \emph{finite shadow covering property}.  
In general, given an ideal flow-box decomposition with this property, we show that its leaf space is a loom space.

\subsubsection{The correspondence} 

In the remainder of our fifth paper we establish the equivalence of $\Veer$, the category of veering triangulations, and $\FlowBox$, the category of (flow equivalence classes of) ideal flow-box decompositions.
That is, we give functors 
\[
\veer \from \Loom \to \Veer, \qquad \leaf \from \FlowBox \to \Loom, \qquad \flowbox \from \Veer \to \FlowBox
\]
and show that the compositions $(\veer \circ \leaf) \circ \flowbox$ and $\flowbox \circ (\veer \circ \leaf)$ are naturally isomorphic to the identities on the categories $\Veer$ and $\FlowBox$ respectively.
In particular, $\calV$ is combinatorially isomorphic to $(\veer \circ \leaf \circ \flowbox)(\calV)$; 
also, $\calA$ is flow-equivalent to $(\flowbox \circ \veer \circ \leaf)(\calA)$.
The above, together with the equivalences mentioned in \refsec{IntroLinkSpaces}, prove that the following diagram commutes (up to natural transformations).
\begin{equation*}
\label{Dia:CommutingDiagram}
    \begin{tikzcd}[row sep=2.5em]
 \Veer \arrow[dr,"\flowbox", swap]  \arrow[rr,"\link", yshift = +3pt] && \Loom  \arrow[ll,"\veer", yshift = -3pt] \\
 & \FlowBox \arrow[ur,"\leaf", swap]     
 \end{tikzcd}
\end{equation*}

This establishes the desired correspondence between, on the one hand, veering triangulations equipped with appropriate filling slopes, and on the other hand, flow-box decompositions (without perfect fits) of closed three-manifolds (up to a finitely generated equivalence relation).

\begin{remark}
\label{Rem:AgolTsang2}
The flow-box decomposition given by Agol and Tsang~\cite[Theorem~5.1]{AgolTsang24} appears to have fewer boxes than our $\flowbox(\calV)$.
On the other hand, it is not currently known if their construction gives an inverse to the Agol--Gu\'eritaud construction.
\end{remark}

\subsection{Cannon--Thurston maps}
\label{Sec:CT}

As a further application of the work of this paper,  we will prove with Jason Manning~\cite{ManningSchleimerSegerman} that there is a Cannon--Thurston map from the veering circle to the boundary of hyperbolic space.
Very roughly, our proof proceeds as follows.
\begin{enumerate}[label=(\Alph*)]
\item 
We compactify the link space $\link(\calV)$ by the veering circle $\Circle$ to obtain the \emph{veering disk} $\Disk$. 
\item 
We analyse Hausdorff limits in $\Disk$. 
This requires the \emph{astroid lemma}~\cite[Lemma~4.10]{SchleimerSegerman24}.
\item 
We prove that the fundamental group $\pi_1(M)$ acts as a convergence group on the \emph{veering sphere} $\Sphere$: a certain equivariant quotient of $\Circle$. 
\item
All cusps project to bounded parabolic points of $\Sphere$.
This requires \reflem{BranchLoop}.
\item
All other points of $\Sphere$ are conical limit points.
This requires suitably controlled neighbourhood bases, as given by Lemmas~\ref{Lem:Singleton} and~\ref{Lem:Leaf}.
\item
We finally appeal to a result of Yaman~\cite[Theorem~0.1]{Yaman04}
characterising the Bowditch boundary of a relatively hyperbolic group pair.
In particular, $\Sphere$ is equivariantly homeomorphic to the boundary of hyperbolic space.
\end{enumerate}

As a consequence, we find examples of Cannon--Thurston maps for cusped manifolds that do not arise, even up to commensurability, from surface groups.
In the closed case such examples can be deduced from the work of Fenley~\cite[Theorems~B and~D]{Fenley12}.
It is not known if our (cusped) examples \emph{slither over the circle} in the sense of Thurston~\cite{Thurston97}.

\subsection{Veering triangulations}

Inspired by~\cite{FarbLeiningerMargalit11}, Agol introduced veering triangulations in order to better understand the surgery parents of the mapping tori of pseudo-Anosov maps with bounded normalised dilatation~\cite{Agol11}. 
His innovation, following ideas of Hamenst\"adt~\cite{Hamenstadt09}, was to produce a canonical train track splitting sequence for the monodromy.
In particular, he showed that the veering triangulation of the drilled mapping torus (together with the corresponding point of the fibred face of the Thurston norm ball) is a complete conjugacy invariant for the monodromy.
More generally, veering triangulations have wide applications in geometric topology~\cite{Bell15, Gueritaud16, MinskyTaylor17, Strenner23, Landry18, Landry23, Landry22, LandryMinskyTaylor24, Parlak24, Parlak23, Nimershiem23, BaikKim22}, as well as in dynamics~\cite{Frankel18, LandryMinskyTaylor23, LandryTsang23, AgolTsang24, Tsang23, Tsang24}, and their properties have been a subject of study by numerous authors~\cite{HRST11, FuterGueritaud13, Kozai13, HodgsonIssaSegerman16, Sakata16, Worden18, FuterTaylorWorden18}.

Agol asks his readers~\cite[Section~7, third question]{Agol11} to ponder the possibility and meaning of non-layered examples.
The first such was found via computer search by Hodgson, Rubinstein, Tillmann and the third author~\cite{HRST11}.
In other work~\cite{GSS19}, we give the census of all transverse veering triangulations with up to 16 tetrahedra; there are 87,047 of these.
The evidence strongly suggests that non-layered veering triangulations dominate. 


The first hint that Agol's question has a general answer comes from the work of Gu\'eritaud~\cite[Theorem~1.1]{Gueritaud16}.
He gives an alternate construction of the veering triangulation, starting from a singular euclidean structure on the fibre of the given mapping torus.
Agol and Gu\'eritaud extend this: given a manifold equipped with a pseudo-Anosov flows without perfect fits, there is an associated veering triangulation on its surgery parent~\cite{Agol15}.
Given this, it is natural to ask if there is a map in the other direction. 
This was the inspiration for our dictionary, given in \refsec{Dictionary}.

\subsection{Outline}
\label{Sec:Outline}

\makeatletter
\g@addto@macro\@parboxrestore{\lineskiplimit\normallineskip}
\makeatother

\begin{figure}[htbp]
\centering
\labellist
\footnotesize\hair 2pt
\pinlabel {\parbox{3.5cm}{\begin{center} veering triangulation $\calV$\\ \refdef{Veering} \end{center}}} at 51 475
\pinlabel \refthm{Equivalence} [t] at 153 473
\pinlabel {\parbox{4.5cm}{\begin{center} induced veering \\ triangulation $\veer(\calL)$ \end{center}}} at 242 475

\pinlabel \refprop{VeerImpliesExhaust} [r] at 60 432
\pinlabel {\parbox{4.5cm}{\begin{center} continental exhaustion\\ $(C_n)$ of universal cover \\ \refdef{ContinentalExhaustion} \end{center}}} at 60 396

\pinlabel {\parbox{2.2cm}{\begin{center} \cite[Main\\ Construction]{Agol11} \end{center}}} [l] at 135 408 
\pinlabel {\parbox{4.4cm}{\begin{center} train tracks $\tau^f, \tau_f$\\ on a face $f$\\ \refdef{UpperLowerTrackFace}\end{center}}} at 135 364 

\pinlabel \refprop{ExhaustImpliesLayered}  [r] at 60 368
\pinlabel {\parbox{4.5cm}{\begin{center} layering $\calK = (K_i)$\\ of universal cover\\ \refdef{Layered} \end{center}}} at 60 334

\pinlabel \reflem{LayeredImpliesUnique} [r] at 24 288
\pinlabel {\parbox{4.4cm}{\begin{center} unique compatible\\ circular order $\calO_\calV$ \\ Definitions~\ref{Def:CircularOrder} and~\ref{Def:Compatible} \end{center}}} at 23 256

\pinlabel {\parbox{3.4cm}{\begin{center} branched surfaces \\ $B^\calV$, $B_\calV$, \refsec{UpperLowerSurfaces} \end{center}}} at 170 274

\pinlabel {\parbox{4.4cm}{\begin{center} train tracks $\tau^K$, $\tau_K$\\ on a layer $K$\\ \refdef{UpperTrack}\end{center}}} at 100 219 

\pinlabel \refthm{VeeringCircle} [r] at 24 226
\pinlabel {\parbox{4.4cm}{\begin{center} veering circle $\Circle$ \\ \refdef{VeeringCircle} \end{center}}} at 23 197

\pinlabel {\parbox{3.4cm}{\begin{center} branch line \\ $S$, \refsec{BranchLines} \end{center}}} at 170 204

\pinlabel \refthm{Laminations} [r] at 94 153
\pinlabel {\parbox{4.4cm}{\begin{center} laminations $\Lambda^\calV$, $\Lambda_\calV$ in $\Circle$\\ Definitions~\ref{Def:LaminationInS1} and~\ref{Def:UpperLamination}\end{center}}} at 100 125

\pinlabel \refthm{Laminations}(\ref{Itm:LaminationsInM},\ref{Itm:LaminationsUnique}) [l] at 155 81
\pinlabel {\parbox{3.4cm}{\begin{center} essential laminations\\ $\Sigma^\calV$, $\Sigma_\calV$ in $M$ \\ \refsec{SuspendingDescending} \end{center}}} at 154 50

\pinlabel \refthm{LinkSpace} [r] at 58 80
\pinlabel {\parbox{4.4cm}{\begin{center} link space $\link(\calV)$\\ \refdef{LinkSpace} \end{center}}} at 62 55

\pinlabel \refthm{LinkIsLoom} [r] at 58 32
\pinlabel {\parbox{4.4cm}{\begin{center} induced loom \\ space $\link(\calV)$
\end{center}}} at 62 8

\pinlabel \refthm{Equivalence} [t] at 146 5
\pinlabel {\parbox{4.4cm}{\begin{center} loom space $\calL$\\ 
\refdef{Loom} \end{center}}} at 242 8

\pinlabel {\cite[Section 5]{SchleimerSegerman24}} [l] at 243 224
\endlabellist
\includegraphics[width=0.8\textwidth]{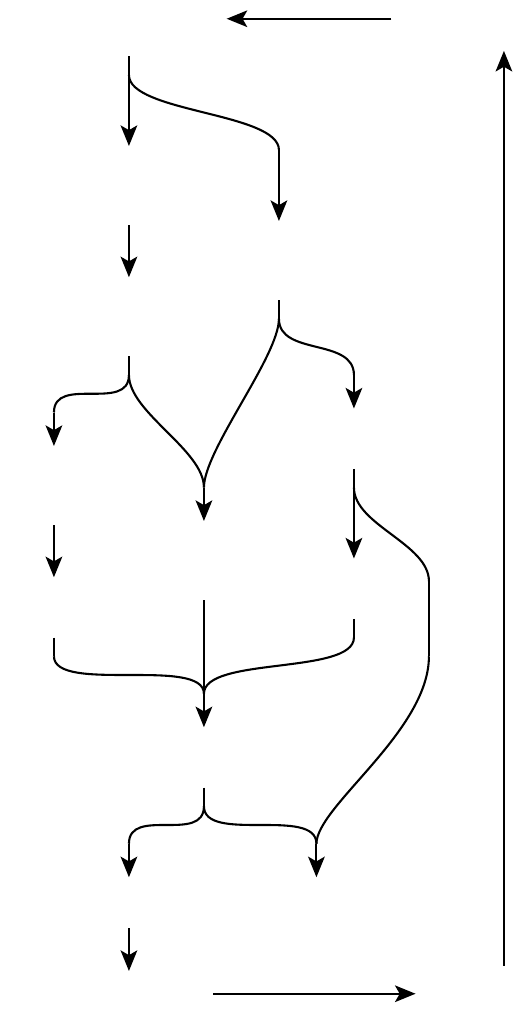}
\caption{The main objects and constructions in this paper.}
\label{Fig:PaperFlowChart}
\end{figure}

The flowchart shown in \reffig{PaperFlowChart} illustrates the interrelations between the main objects as well as their constructions.

In \refcha{OrdersTriangulations} we review the definitions of \emph{circular orders}, \emph{transverse taut ideal} triangulations, and \emph{layered} triangulations.
We then define what it means for a circular order to be \emph{compatible} with a given transverse taut structure.
In \refexa{SurfaceBundlesTwo} we show that layered triangulations are \emph{rigid}: that is, they admit a unique compatible circular order.
In a striking contrast, in \refsec{FlatDSS} we give a non-layered taut ideal triangulation which is far from rigid -- 
by \refthm{Exotic} it admits uncountably many compatible circular orders.
On the other hand, layered or not, veering triangulations are rigid. 

\begin{restate}{Theorem}{Thm:VeerImpliesUnique}
Suppose that $M$ is an oriented three-manifold equipped with a transverse veering triangulation $\calV$.
Then there is a unique compatible circular order $\calO_\calV$ on the cusps of $\cover{\calV}$.
Furthermore, $\calO_\calV$ is dense and $\pi_1(M)$--invariant.
\end{restate}



\noindent
To prove this, in \refcha{Geography} we quickly review the terminology of \emph{train tracks}.
We then introduce the new tools required to obtain \refthm{VeerImpliesUnique}.
The most important of these is the idea of a \emph{continental exhaustion}, introduced in \refdef{ContinentalExhaustion}.
This is a kind of developing map; 
however, instead of constructing a geometric position for each tetrahedron in $\HH^3$, we produce combinatorial positions in a ``thickened version'' of $\HH^2$.
The heart of the proof of \refthm{VeerImpliesUnique} is given by \refprop{VeerImpliesExhaust}; in a delicate induction we promote a transverse veering triangulation to a continental exhaustion.  

In \refcha{SurfacesAndLines} we construct various \emph{branched surfaces}.
Using these and the techniques developed for the proof of \refthm{VeerImpliesUnique}, we promote the circular order $\calO_\calV$ on the cusps $\Delta_\calV$ of $\cover{\calV}$ to the \emph{veering circle} $\Circle$.  

\begin{restate}{Theorem}{Thm:VeeringCircle}
Suppose that $M$ is an oriented three-manifold equipped with a transverse veering triangulation $\calV$. 
Then the order completion of $(\Delta_\calV, \calO_\calV)$ is a circle $\Circle$ with the following properties.
\begin{enumerate}
\item 
The action of $\pi_1(M)$ on $\Delta_\calV$ extends to give a continuous, faithful, orientation-preserving action on $\Circle$.  
\item 
If $\calV$ is finite then all orbits are dense.
\end{enumerate}
\end{restate}

The veering circle $\Circle$ depends only on $\calV$, and not on any other choices.
This is in striking contrast with Thurston's \emph{universal circle} for a given closed manifold equipped with a taut foliation; many choices are required for his construction~\cite[Remark~6.27]{CalegariDunfield03}.
Our situation is instead similar to Fenley's \emph{ideal circle boundary} for a closed three-manifold equipped with a pseudo-Anosov flow without perfect fits~\cite[Theorem~A]{Fenley12}.
There the circle is unique.
We stress however that our results require only a finite amount of combinatorial data, while \cite{CalegariDunfield03, Fenley12} both require substantial topological or dynamical inputs.

\begin{remark}
Since our manifold $M$ comes equipped with a taut ideal triangulation, it necessarily has cusps.
Thus the fundamental group $\pi_1(M)$ surjects $\ZZ$.  
We deduce that $\pi_1(M)$ has \emph{left orders} and so has \emph{left circular orders}. 
See~\cite[Definitions~2.26 and~2.40]{Calegari07}.  
We note that \refthm{VeeringCircle} gives a new left circular order on $\pi_1(M)$, by 
inserting gaps at the cusps and, in these gaps, adding left orders on the peripheral groups.
\end{remark}

We next build the \emph{upper} and \emph{lower laminations}.  

\begin{restate}{Theorem}{Thm:Laminations}
Suppose that $M$ is a three-manifold equipped with a transverse veering triangulation $\calV$. 
Then there is a lamination $\Lambda^\calV$ in $\Circle$ with the following properties.
\begin{enumerate}
\item 
The lamination  $\Lambda^\calV$ is $\pi_1(M)$--invariant.
\item 
The lamination $\Lambda^\calV$ suspends to give a $\pi_1(M)$--invariant lamination $\cover{\Sigma}^\calV$ in $\cover{M}$; this descends to $M$ to give a lamination $\Sigma^\calV$ which 
\begin{enumerate}
\item
is carried by the upper branched surface $B^\calV$, 
\item
has only plane, annulus, and M\"obius band leaves, and
\item
is essential.
\end{enumerate}
\item 
Suppose that $\Sigma$ is a lamination fully carried by $B^\calV$.
Then, after collapsing parallel leaves, $\Sigma$ is tie-isotopic to $\Sigma^\calV$.
\end{enumerate}
There is also a lamination $\Lambda_\calV$ with the corresponding properties with respect to $B_\calV$. 
\end{restate}

\refthm{Laminations}\refitm{LaminationsUnique} is surprising; in \refcor{Surprise} we use this to show that any two non-empty laminations carried by the stable branched surface associated to a pseudo-Anosov homeomorphism are tie-isotopic (after collapsing parallel leaves).  
This answers a question of Danny Calegari~\cite{Calegari21}. 

Using the upper and lower laminations, we build the \emph{link space} with its upper and lower foliations.
After reviewing the axioms of loom spaces~\cite{SchleimerSegerman24}, we prove the following.

\begin{restate}{Theorem}{Thm:LinkIsLoom}
Suppose that $M$ is a three-manifold equipped with a (locally) veering triangulation $\calV$. 
\begin{enumerate}
\item
The link space $\link(\calV)$, equipped with the foliations $F^\calL$ and $F_\calL$ is a loom space. 
\item
Every tetrahedron rectangle is a loom-tetrahedron rectangle; and conversely.
\item
The set of cusp rectangles in $\link(\calV)$, with ideal point a given cusp class of $\elec$, is a loom-cusp; and conversely.
\item
The action of $\pi_1(M)$ is faithful and by loom automorphisms.
\end{enumerate}
\end{restate}

The proof of \refthm{LinkIsLoom} is quite involved.  
To build the link space, we must collapse a pair of transverse laminations in a circle to obtain a pair of foliations in a plane.  
There are various expositions of similar collapsing procedures in the literature;
see \cite[Chapter~6]{CassonBleiler88} and \cite[Section~11.9]{Kapovich09}.  
However, we do not have any of the usual hypotheses: 
we do not have a surface group action,
we do not have (projectively) invariant measures, 
and we do not have compactness (of the original three-manifold). 

In \refthm{Equivalence}, we prove that our construction, of a link space from a veering triangulation, and Gu\'eritaud's construction, of a veering triangulation from a loom space, are essentially inverses. 
We do this by interpreting points and leaves of the link space as combinatorial objects in the veering triangulation.
Here is one application of \refthm{Equivalence}.

\begin{restate}{Corollary}{Cor:Recover}
Suppose that $M$ is a connected three-manifold.
Suppose that $\calV$ is a (locally) veering triangulation on $M$.
Suppose that $\Gamma = \veer \circ \link (\pi_1(M))$ is the induced subgroup of $\Aut(\veer \circ \link (\cover{\calV}))$.  
Then $\veer \circ \link(\cover{\calV}) / \Gamma$ is canonically isomorphic to $\calV$. 
\end{restate}

The material in \refapp{CT} is required for our forthcoming work with Manning described in \refsec{CT}. 
\refapp{CarriedArcs} generalises \cite{SchleimerSegerman20} from carried loops to carried arcs; \refcor{VerticesDistinct} is used here in Chapters \ref{Cha:OrdersTriangulations} and \ref{Cha:Geography}.
\refapp{Notation} is an index of commonly used notation.

\subsection*{Acknowledgements}
We thank Ian Agol for many interesting discussions about veering triangulations, and for posing the question that led to this work. 
These conversations happened at the Institute for Advanced Study during the \emph{Geometric Structures on three-manifolds program}, supported by the National Science Foundation under Grant Number DMS-1128155.

We thank Lee Mosher and S\'ergio Fenley for teaching us about pseudo-Anosov flows;
this directly motivated our work on the link space.

We thank Yair Minsky for helping us understand the connection between the veering hypothesis and Fenley's notion of \emph{perfect fits}.
We thank Marc Culler for suggesting the name \emph{continents}.
These conversations happened at the Mathematical Sciences Research Institute during the Fall 2016 semester program \emph{Geometric Group Theory}, supported by the National Science Foundation under Grant Number DMS-1440140.

We thank Danny Calegari for his excellent talks at the 2016 workshop \emph{Contact structures, laminations and foliations}, at the Mathematical Institute of the LMU in Munich;
these directly and indirectly inspired our investigations into the upper and lower laminations.

We thank Sam Taylor for pointing out the relation between the link space and the Guirardel core.

We thank Jason Manning for many patient discussions about the action of $\pi_1(M)$ on the veering circle.

We thank the referee for their very careful reading and many insightful comments.

We thank the Institute for Computational and Experimental Research in Mathematics in Providence, RI, for their hospitality during the Fall 2019 semester program \emph{Illustrating Mathematics};
we made significant progress on the paper during that time.
The ICERM program was supported by the National Science Foundation under Grant Number DMS-1439786 and the Alfred P.~Sloan Foundation award G-2019-11406.

Frankel was supported in part by National Science Foundation grants DMS-1611768 and DMS-2045323.
Segerman was supported in part by Australian Research Council grant DP1095760 and National Science Foundation grants DMS-1308767, DMS-1708239, and DMS-2203993.

\chapter{Circular orders and taut ideal triangulations}
\label{Cha:OrdersTriangulations}

\subsection{Circular orders}
References for this material include~\cite[Section~2]{Thurston98}, \cite[Section~2.6]{Calegari07}, and~\cite[Chapter~3]{Frankel13}.

\begin{definition}
\label{Def:CircularOrder}
Suppose that $\Delta$ is a set.  
A function $\calO \from \Delta^3 \to \{-1, 0, 1\}$ is a \emph{circular order on $\Delta$} if for all $a, b, c, d \in \Delta$ we have the following.
\begin{itemize}
\item
$\calO(a, b, c) \neq 0$ if and only if all of $a$, $b$, and $c$ are distinct.
\item
$\calO(a, b, c) = \calO(c, a, b) = -\calO(b, a, c)$.
\item
If $\calO(a, b, d) = \calO(b, c, d) = 1$ then $\calO(a, c, d) = 1$. \qedhere
\end{itemize}
\end{definition}

\begin{example}
The canonical example is the unit circle $S^1$ equipped with the anticlockwise circular order.
\end{example}

\begin{definition}
\label{Def:Dense}
We say that a circular order $\calO$ on $\Delta$ is \emph{dense} if, for all $a, b, c \in \Delta$, if $\calO(a, b, c) = 1$, then there is a point $d \in \Delta$ so that $\calO(b, d, c) = 1$. 
\end{definition}

\begin{definition}
Suppose that $\Gamma$ is a group acting on $\Delta$.  
A circular order $\calO$ on $\Delta$ is \emph{$\Gamma$--invariant} if for all $a, b, c \in \Delta$ and for all $\gamma \in \Gamma$ we have $\calO(\gamma a, \gamma b, \gamma c) = \calO(a, b, c)$.
\end{definition}

\begin{example}
A second example comes from taking $F$ to be a compact, connected, oriented surface with non-empty boundary and with $\chi(F) < 0$.
Let $\cover{F}$ be the universal cover of $F$.  
Let $\Delta_F$ be the set of boundary components of $\cover{F}$; 
these are the \emph{cusps} of $\cover{F}$.  
Note that the orientation on $F$ gives a dense $\pi_1(F)$--invariant circular order on $\Delta_F$.  
See \reffig{Farey}. 
\end{example}

\begin{figure}[htbp]
\includegraphics[width=0.5\textwidth]{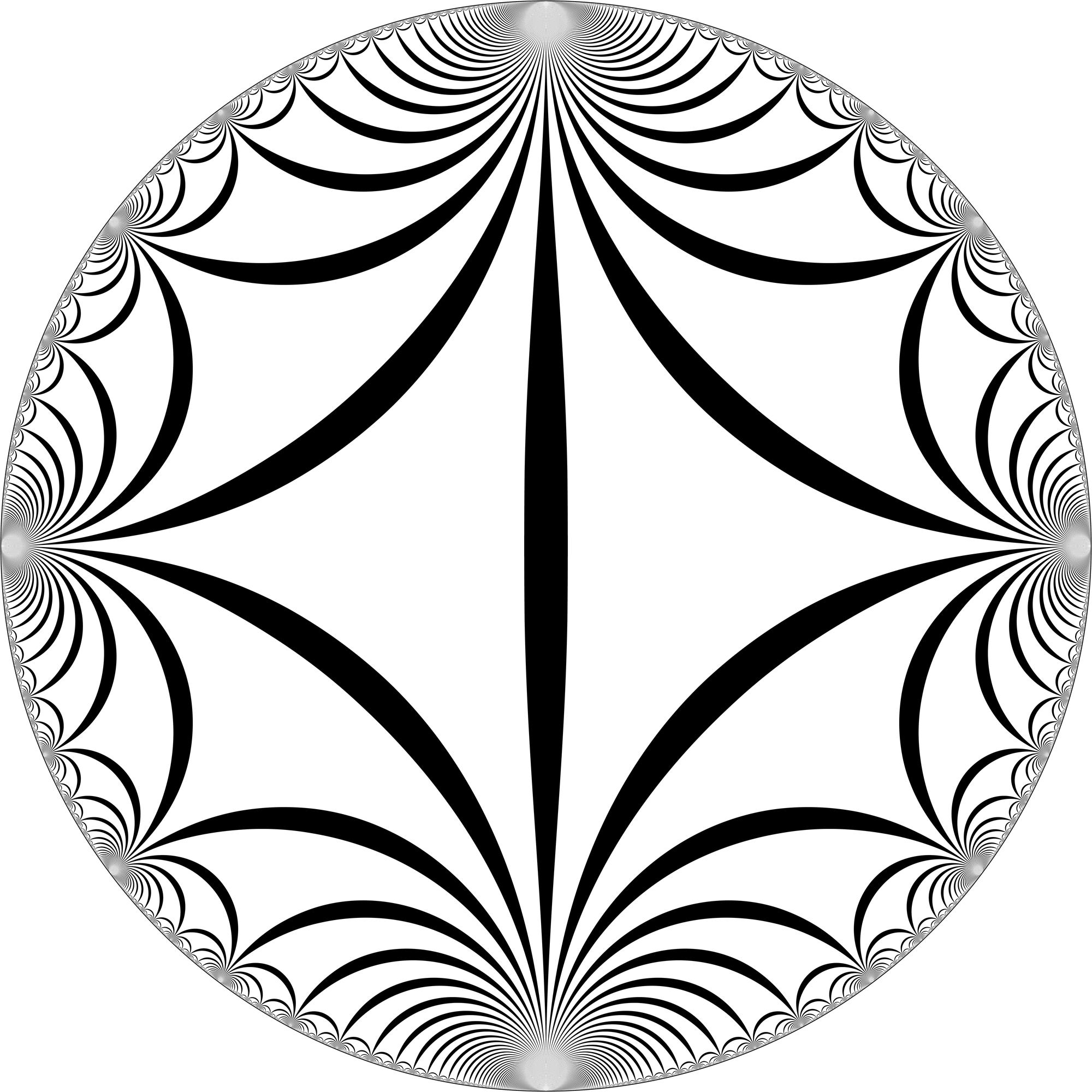}
\caption{Equipping $F$ with an ideal triangulation and lifting to $\cover{F}$ yields the Farey tessellation.  
The cusps of $\cover{F}$ correspond to the rational numbers $\QP^1$.  
Figure by Roice Nelson.}
\label{Fig:Farey}
\end{figure}

For our next example we go up another dimension.  
Suppose that $M$ is a compact, connected, oriented three-manifold, with non-empty boundary, and where all boundary components are tori.  
Let $\cover{M}$ be the universal cover of $M$.  
Let $\Delta_M$ be the set of boundary components of $\cover{M}$;
we call these the \emph{cusps} of $\cover{M}$.
Note that these are naturally in bijection with the peripheral subgroups of $\pi_1(M)$. 
We want to find a dense $\pi_1(M)$--invariant circular order on $\Delta_M$.

\begin{example}
\label{Exa:SurfaceBundlesOne}
Suppose that $F$ is a compact, connected, oriented surface with negative Euler characteristic and with non-empty boundary.  
Suppose that $f \from F \to F$ is an orientation-preserving homeomorphism. 
Define the \emph{surface bundle} $M = M_f$ with \emph{monodromy} $f$ to be the three-manifold obtained by forming $F \cross [0,1]$ and then identifying the point $(x, 1)$ with $(f(x), 0)$, for all $x \in F$.  
The surfaces $F_t = F \cross \{t\}$ are the \emph{fibres} of the bundle $M$.
The intervals $\{x\} \cross [0,1]$ glue together to give an oriented flow $\Phi$ transverse to the fibres.
We choose the orientation on $M$ so that the orientation of $F_0$, followed by the orientation of $\Phi$, makes a right-handed frame.
One simple example is $M_{\Id} \homeo F \cross S^1$;
more interesting ones are given in \refexa{SurfaceBundlesTwo}.

We now identify $F$ with the fibre $F_0$.  
This induces a homeomorphism between $\cover{F} \cross \RR$ and $\cover{M}$ and thus a bijection between $\Delta_F$ and $\Delta_M$.  
The orientation of $\cover{F}$ gives us a dense circular order $\calO_F$ on $\Delta_F$ and thus on $\Delta_M$.  
Finally, note that $\calO_F$ is $\pi_1(M)$--invariant.
\end{example}

\begin{remark}
\label{Rem:BundleRigidity}
The circular order $\calO_F$ of \refexa{SurfaceBundlesOne} possesses a certain kind of rigidity.
Suppose that $F'$ is a properly embedded, connected, oriented surface in $M = M_f$.
Suppose that $F'$ is transverse to the flow $\Phi$, meets every flow line of $\Phi$, and the orientations of $F'$ and $\Phi$, in that order, form a right-handed frame in $M$.
Then $F'$ is a fibre of a surface bundle structure on $M$.  
We deduce that $F'$ gives a circular order $\calO_{F'}$ on the cusps of $\cover{M}$.
Since $\cover{F}$ and $\cover{F}'$ are both transverse to the flow, 
they are both homeomorphic to the leaf-space. 
We deduce that $\calO_F = \calO_{F'}$.
\end{remark}

\subsection{Taut ideal triangulations}
\label{Sec:TautIdealTriangulation}

To discuss more general examples, we will replace bundle structures by ideal triangulations equipped in various ways.  
From now on we will always assume that $M$ is a tame connected three-manifold, possibly with boundary. 

A \emph{triangulation} $\calT$ is a collection $\{ t_i \}$ of \emph{model tetrahedra} (copies of $\{ x \in \RR^4 \st x_i \geq 0, \sum x_i = 1 \}$) glued via \emph{face pairings} $\{\phi_j \}$ (linear maps between faces). 
The faces, edges, and vertices of a model tetrahedron will be called \emph{model cells}. 
Suppose that $\calT$ is a triangulation and that $|\calT|$ is its \emph{realisation space} (the disjoint union of the $t_i$ modulo the maps $\phi_j$). 
We will assume that the interior of every model cell of $\calT$ embeds in $|\calT|$.  
Following~\cite[Section~4.2]{Thurston78}, we say $\calT$ is an \emph{ideal triangulation} of $M$ if $|\calT|$, minus its zero-skeleton, is homeomorphic to the interior of $M$.  
We use $\cover{\calT}$ to denote the induced ideal triangulation of the universal cover, $\cover{M}$. 
The \emph{cusps} of $\calT$, denoted simply as $\Delta_\calT$, is the set of ideal vertices of $\cover{\calT}$. 
When $M$ is (the interior of a) compact three-manifold with Klein bottle and torus boundary components, 
there is an induced bijection between $\Delta_\calT$ and $\Delta_M$ (the peripheral subgroups of $\pi_1(M)$). 

A \emph{taut angle structure}~\cite[Definition~1.1]{HRST11} on an ideal triangulation $\calT$ is an assignment of dihedral angles as follows.
\begin{itemize}
\item
Every model edge of every model tetrahedron has dihedral angle zero or $\pi$.
\item
For every edge, the angle sum of its models is $2\pi$.
\item
For every model vertex, the angle sum of the three adjacent model edges is $\pi$.
\end{itemize}
A model tetrahedron $t$ with such dihedral angles is called a \emph{taut} tetrahedron. 
The four model edges with angle zero give the \emph{equator} of $t$.
See \reffig{TransverseTet}.  

A \emph{transverse} taut angle structure on $\calT$~\cite[Definition~1.2]{HRST11} has, in addition to the dihedral angles, a co-orientation on the faces of $\calT$.
These are arranged so that, for any pair of model faces $f, f'$ of any model tetrahedron $t$ with common model edge $e$, we have the following.
\begin{itemize}
\item 
The dihedral angle of $e$ inside of $t$ is zero if and only if exactly one of the co-orientations on $f, f'$ points into $t$.
\end{itemize}
We extend the co-orientation to the edges of $\calT$ so that, along an edge $e$, it points into the tetrahedron $t$ above $e$.  
We call $e$ the \emph{lower} edge for $t$.
Reversing orientations, we obtain the \emph{upper} edge for $t$.
We define the \emph{upper} and \emph{lower} faces for $t$ similarly.
The co-orientation on the faces of a model taut tetrahedron are shown in \reffig{TransverseTet}.  
Note that our terminology (following~\cite{HRST11}) is a slight variant of the original definition~\cite[page~370]{Lackenby00}.

\begin{figure}[htbp]
\centering
\subfloat[Co-orientations and angles in a transverse taut tetrahedron.]{
\labellist
\small\hair 2pt
\pinlabel 0 at 20 130
\pinlabel 0 at 240 120
\pinlabel 0 at 135 27
\pinlabel 0 at 135 217
\pinlabel $\pi$ at 125 140
\pinlabel $\pi$ at 125 87
\endlabellist
\includegraphics[width=0.35\textwidth]{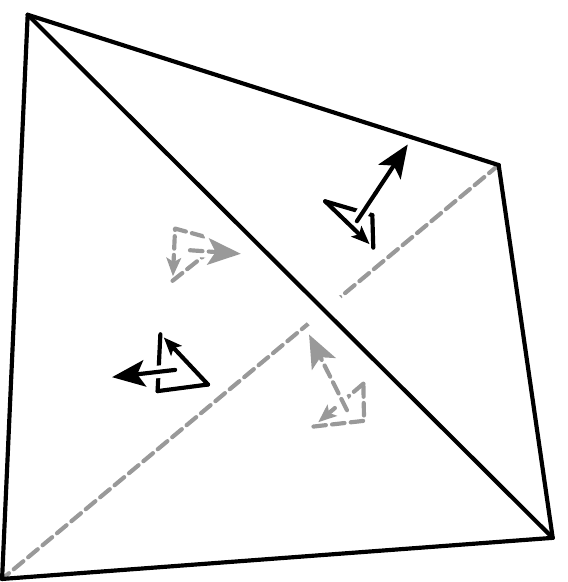}
\label{Fig:TransverseTet}
}
\qquad \qquad
\subfloat[Co-orientations around an edge.]{
\includegraphics[width=0.4\textwidth]{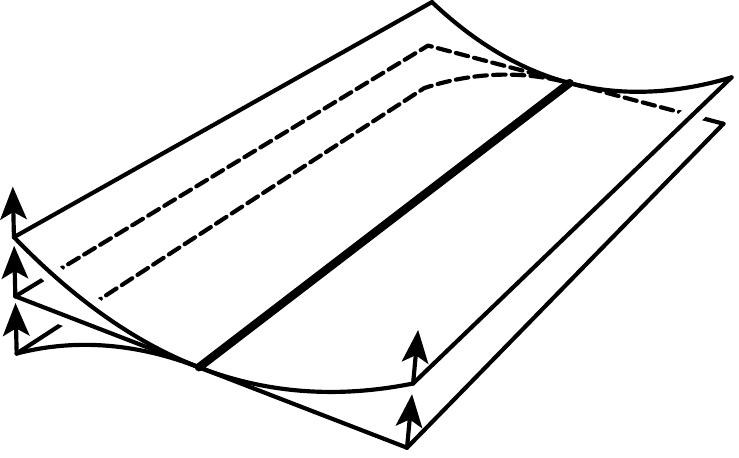}
\label{Fig:TransverseEdge}
}
\caption{}
\label{Fig:Transverse}
\end{figure}

In the following, a transverse taut angle structure on $\calT$ will be called simply a \emph{transverse taut structure}.  

\subsection{Horizontal branched surface}
\label{Sec:HorizontalBranchedSurface}

References for this material include~\cite[Section~1.2]{Mosher96} and~\cite[Section~6.3]{Calegari07}.
Suppose that $B$ is a branched surface.  
The one-skeleton $B^{(1)}$ is here called the \emph{branch locus}.  

\begin{remark}
\label{Rem:Locus}
This terminology is fairly common in the literature; 
for example, see~\cite[page~26]{Mosher96} and~\cite[page~371]{Lackenby00} and~\cite[page~157]{Li02}.  
It is called the singular locus in~\cite[Definition 6.16]{Calegari07} and is also closely related to the \emph{sutures} of \emph{sutured manifolds}~\cite[Definition~2.6]{Gabai83}. 
\end{remark}

The branch locus decomposes as a union of \emph{branch components}: 
these are connected one-manifolds immersed into $M$ meeting (and possibly self-intersecting) only at the vertices of the zero-skeleton $B^{(0)}$.  
The components of $B - B^{(1)}$ are called the \emph{sectors} of $B$.  

Suppose now that $M$ is a three-manifold equipped with a transverse taut ideal triangulation $\calT$.
We now isotope the two-skeleton of $\calT$ near each edge, as shown in \reffig{TransverseEdge}, to obtain $B(\calT)$: 
a co-oriented branched surface without vertices~\cite[page~371]{Lackenby00}.
The branched surface $B(\calT)$ is \emph{taut}~\cite[page~374]{Lackenby00};
this gives taut ideal triangulations their name.  
Note that the sectors of $B(\calT)$ are exactly the original triangular faces of $\calT$. 

Let $\cover{B}$ be the preimage of $B(\calT)$ in the universal cover $\cover{M}$.
Note that $\cover{B}$ is again a co-oriented branched surface. 
We now have a basic result.

\begin{theorem}
\label{Thm:TetEmbeds}
Suppose that $M$ is a three-manifold equipped with a taut ideal triangulation $\calT$.
Suppose that $t$ is a model tetrahedron of $\cover{\calT}$.  
Then $t$ embeds in $\cover{M}$; 
furthermore, the model vertices of $t$ lie in four distinct cusp of $\Delta_\calT$.
\end{theorem}

\begin{proof}
Let $\cover{B}$ be the preimage, in $\cover{M}$, of $B(\calT)$. 
Let $f'$ and $g'$ be the lower model faces of $t$.
Let $f$ and $g$ be the upper model faces of $t$.
Note that $f \cup g$ is a connected surface carried by $\cover{B}$, as is $f' \cup g'$. 
From~\cite[Corollary~1.1]{SchleimerSegerman20} we deduce that each of $f \cup g$ and $f' \cup g'$ is a disk embedded in $\cover{M}$. 
From \refcor{VerticesDistinct} we deduce that each of $f \cup g$ and $f' \cup g'$ meets four distinct cusps of $\Delta_\calT$.  
(In~\cite{SchleimerSegerman20} we assumed that $M$ is compact, but that hypothesis is never used.)
It follows that no pair of model faces (or model edges) of $t$ are identified in $\cover{M}$. 
\end{proof}


\subsection{Layered triangulations}

Suppose that $B$ is a general branched surface. 
Recall that a \emph{tie-neighbourhood} $N(B)$ of $B$ is a regular neighbourhood of $B$ equipped with a foliation by intervals, called \emph{ties}. 
See~\cite[Figure~1.2]{Mosher96}. 
A surface $F$, properly embedded in $M$, is \emph{carried} by $B$ if it is properly isotopic into $N(B)$ and is there transverse to the ties.  
A carried surface $F$ is \emph{fully carried} if, after the isotopy, it intersects all ties.
Properties of $B$ are generally inherited by its carried surfaces: decomposition into sectors, co-orientations, and so on. 

\begin{definition}
\label{Def:Layered}
Suppose that $M$ is a three-manifold equipped with a transverse taut ideal triangulation $\calT$.
A \emph{layering} $\calK = (K_i)$ of $\calT$ is a collection of pairwise disjoint surfaces, all carried by $B(\calT)$.
Here the index set is either 
\begin{itemize}
\item 
$\ZZ / n\ZZ$ if $\calT$ has $n$ tetrahedra or 
\item 
$\ZZ$ if $\calT$ is infinite.
\end{itemize}
The surfaces $K_i$ inherit a transverse orientation from $B(\calT)$.  
Furthermore they have the following property.
\begin{itemize}
\item
For every tetrahedron $t$ of $\calT$ there is a unique $j$ so that $t - N(B(\calT))$ lies in the component of $M - (K_j \sqcup K_{j+1})$ immediately above $K_j$ and immediately below $K_{j+1}$.  
\item
For every $K_i$ 
there is a tetrahedron $t_i$ immediately above $K_i$.
\end{itemize}
We call the surfaces $K_i$ the \emph{layers} of $\calK$. 
\end{definition}

\begin{lemma}
\label{Lem:LayeredImpliesLayered}
Suppose that $M$ is a three-manifold equipped with a transverse taut ideal triangulation $\calT$.
If $\calT$ is layered then so is $\cover{\calT}$. 
\end{lemma}

\begin{proof}
If $M$ is simply connected there is nothing to prove.
Suppose that $\calF = (F_i)$ is the given layering of $\calT$ and $t_i$ is the unique tetrahedron between $F_{i}$ and $F_{i+1}$.  
Fixing $i$, we define an \emph{elevation} of $F_i$ to be an induced map from $\cover{F}_i$ to $\cover{M}$.  

The elevations of the $F_i$ do not give a layering of $\cover{\calT}$;  
adjacent elevations of $F_i$ and $F_{i+1}$ cobound infinitely many tetrahedra (all of which are lifts of $t_i$).
Thus we must rearrange the order of attachment.

Fix $G$, an elevation of $F_0$.  
Since $G$ is properly embedded, 
it separates $\cover{M}$.
Choose a linear order $(t^i)_{i \in \NN}$ on the tetrahedra of $\cover{\calT}$ lying \emph{above} $G$.  
We now build a layering onto $G$.  
By induction suppose that we have attached all of the tetrahedra $(t^i)_{i = 0}^{n-1}$, plus perhaps finitely more, to $G$.  
Now consider the tetrahedron $t^n$.  
If it is already attached, there is nothing to prove.  
If it is not we proceed as follows.  
Consider all possible paths $\gamma$ in $\cover{M}$ so that
\begin{itemize}
\item 
$\gamma$ starts on $G$ and ends in $t^n$,
\item 
$\gamma$ is transverse to the two-skeleton, and
\item 
$\gamma$ crosses each face $f$ in the direction of its co-orientation.
\end{itemize} 
Note that each such $\gamma$ crosses the elevations of the layers $F_i$ the same number of times; we call this number, $h(t^n)$, the \emph{height} of $t^n$.  
Thus each $\gamma$ meets at most $h(t^n)$ tetrahedra.  
Also, since a tetrahedron has only two lower faces, there are at most $2^{h(t^n)}$ such paths.  
So, we attach all of the finitely many tetrahedra met by any of these paths, in height order.  
The last one attached is $t^n$, completing the inductive step. 

Doing the same below $G$ completes the proof.
\end{proof}

\subsection{Compatibility}
\label{Sec:Compatibility}

Suppose that $M$ is an oriented three-manifold equipped with a transverse taut ideal triangulation $\calT$.
Recall that $\Delta_\calT$ is the set of cusps of $\cover{\calT}$.
Suppose that $f$ is a face of $\cover{\calT}$.  
Let $\Delta_f \subset \Delta_\calT$ be the cusps that $f$ meets. 
By \refthm{TetEmbeds} there are exactly three of these.
The co-orientation of $f$ and the induced orientation of $\cover{M}$ picks out a unique circular order $\calO_f$ on $\Delta_f$ (say, anticlockwise as viewed from above).  

\begin{definition}
\label{Def:Compatible}
Suppose that $\calO$ is a circular order on $\Delta \subset \Delta_\calT$.  
We say that $\calO$ is \emph{compatible} with $\calT$ if it has the following property. 
For any face $f \in \cover{\calT}$ if $\Delta_f \subset \Delta$ 
then $\calO | \Delta_f = \calO_f$.  
\end{definition}

\begin{lemma}
\label{Lem:LayeredImpliesUnique}
Suppose that $M$ is an oriented three-manifold equipped with a transverse taut ideal triangulation $\calT$.
If $\cover{\calT}$ is layered then there is a unique compatible circular order $\calO$ on $\Delta_\calT$. 
Furthermore, $\calO$ is dense and $\pi_1(M)$--invariant. 
\end{lemma}
 
\begin{proof}
We are given a layering $(F_i)_{i \in \ZZ}$ of $\cover{M}$.  
Applying \cite[Corollary~1.1]{SchleimerSegerman20}, each $F_i$ is a copy of the Farey tessellation.  
Each gives the next via a single flip across a tetrahedron.  
So $F_i$ and $F_{i+1}$ meet the same subset of cusps of $\Delta_\calT$.  
But every cusp meets some $F_j$.  
Thus every $F_i$ meets all cusps of $\Delta_\calT$.

By \refcor{VerticesDistinct}, the layer $F_i$ meets every cusp at most once.  
Since $F_i$ is disk, it gives a circular order $\calO_i$ on $\Delta_\calT$.  
Picking an oriented edge of $F_i$ determines an order isomorphism between $(\Delta_\calT, \calO_i)$ and $\QP^1$, the rational points of $S^1$.  
Thus the circular order $\calO_i$ is dense.  
Since $F_i$ differs from $F_j$ by a finite number of flips, we deduce that $\calO_i = \calO_j$. 
Also, every face of $\cover{\calT}$ lies in some $F_i$.  
So, setting $\calO = \calO_0$, we deduce that $\calO$ is compatible with $\calT$.  

Suppose that $\calO'$ is another compatible circular order on $\Delta_\calT$.  
Then $\calO'$ and $\calO$ agree on all faces of $F_0$; 
we deduce that $\calO' = \calO$.   
Thus $\calO$ is the unique circular order on $\Delta_\calT$ compatible with $\calT$.

Finally, fix any $\gamma \in \pi_1(M)$ and define $\calO_\gamma$ via 
\[
\calO_\gamma(a, b, c) = \calO(\gamma a, \gamma b, \gamma c).
\]
The action of $\gamma$ on $\Delta_\calT$ is bijective, so $\calO_\gamma$ is a circular order on $\Delta_\calT$.  
Now, if $f$ is a face of $\cover{\calT}$ with vertices $x$, $y$, and $z$ then $\calO_\gamma(x, y, z) = \calO(\gamma x, \gamma y, \gamma z)$.  
Note that $\gamma x$, $\gamma y$, and $\gamma z$ are the vertices of the face $\gamma f$.  
Also, $\gamma$ sends the co-orientation of $f$ to that of $\gamma f$; 
also $\gamma$ preserves the orientation of $M$.  
It follows that $\calO_\gamma$ is compatible with $\calT$.  
Thus, by uniqueness, we have $\calO_\gamma = \calO$. 
\end{proof}

\chapter{Examples}

Here we give examples of circular orders coming from transverse taut structures.

\subsection{Layered surface bundles}
\label{Sec:SurfaceBundlesTwo}

We follow the notation of \refexa{SurfaceBundlesOne}.

\begin{example}
\label{Exa:SurfaceBundlesTwo}
Suppose $M = M_f$ is an oriented surface bundle with fibre $F$ and monodromy $f$.  
Let $\calT$ be a layered ideal triangulation of $M$.  
The best-known example is shown in \reffig{VeerFigEight}; 
it is the canonical triangulation of the figure-eight knot complement.
In general, since $\calT$ is layered, \reflem{LayeredImpliesLayered} implies that $\cover{\calT}$ is layered.  
Thus, by \reflem{LayeredImpliesUnique} the cusps of $\Delta_\calT$ admit a unique circular order $\calO_F$ which is compatible with $\calT$.  
Furthermore, it is $\pi_1(M)$--invariant.  
Since the layering of $\cover{\calT}$ has $\cover{F}$ as one of its layers, the circular order $\calO_F$ is the same as the one constructed in \refexa{SurfaceBundlesOne}.  
\end{example}

\begin{figure}[htbp]
\includegraphics[width=0.6\textwidth]{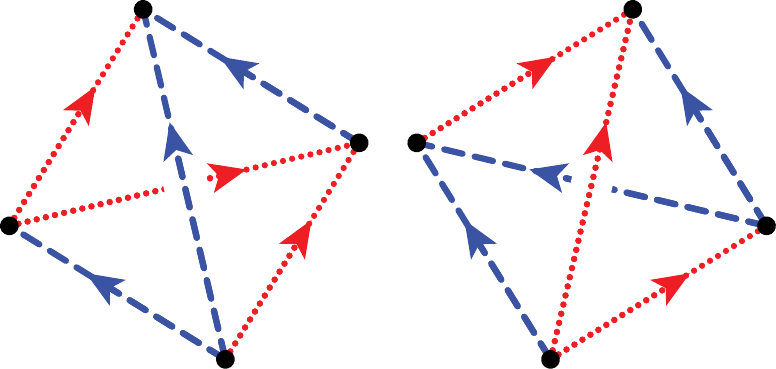}
\caption{The veering structure on the canonical triangulation for the figure-eight knot complement.  
This manifold is \texttt{m004} in the SnapPy census~\cite{snappy}.  
This veering structure is \texttt{cPcbbbiht\_12} in the veering census~\cite{GSS19}.  
See \refcha{Veering} for the definition of veering; in this section we only use the transverse taut structure.  
The $\pi$--edges of the transverse taut structure are the diagonals of the squares.  
The model edges on the sides of the squares all have dihedral angle zero.}
\label{Fig:VeerFigEight}
\end{figure}

\subsection{Geometric structures}

Before we give our next family of examples, there is a bit of necessary background.  
Useful references include~\cite[Chapter~4]{Thurston78} and~\cite[Section~2]{Tillmann12}.  
Take $\calP = \CC - \{0, 1\}$.  
Suppose that $M$ is a compact, connected, oriented three-manifold with boundary being a single torus.  
Suppose that $\calT = (t_i)_{i = 0}^{n-1}$ is an ideal triangulation of $M$.  
We say that $\calT$ admits a \emph{geometric structure} if there is a tuple of \emph{shapes} $\rho_\infty = (z_i) \in \calP^n$ as follows.
\begin{itemize}
\item
The tuple $\rho_\infty$ solves Thurston's gluing and holonomy equations (called the completeness equations in ~\cite[page~800]{Tillmann12}). 
\item
Each $z_i \in \rho_\infty$ has positive imaginary part. 
\end{itemize}
It follows that $\rho_\infty$ gives the interior of $M$ a complete, finite volume, hyperbolic metric.  
The tuples of $\calP^n$ that solve the gluing equations (ignoring the holonomy and positivity) make up the \emph{shape variety} $\calS(\calT)$.  
Let $\calS_\infty$ be the irreducible component of $\calS(\calT)$ which contains $\rho_\infty$.  
Our assumptions on $M$ and the existence of $\rho_\infty$ imply that $\calS_\infty$ is a complex curve~\cite[page~314]{NeumannZagier85}.  

Fix a triangle $f$ of $\cover{\calT}$ and a circular order on ideal vertices of $f$.  
This choice allows us to define, for every $\rho \in \calS_\infty$, a \emph{developing map} 
\[
\dev_{\rho} \, \from \,\, \cover{M} \,\, \to \,\, \HH^3
\]
We extend $\dev_\rho$ to give a function from $\Delta_\calT$, the set of cusps, to $\CP^1 = \bdy_\infty \HH^3$.  
Our choices ensure that the vertices of $f$ are sent to $0$, $1$, and $\infty$, respectively. 

\begin{definition}
\label{Def:Collide}
If $u$ and $v$ are distinct cusps then we say that $u$ and $v$ \emph{collide} at $\rho \in \calS_\infty$ if $\dev_\rho(u) = \dev_\rho(v)$. 
\end{definition}

\begin{lemma}
\label{Lem:FiniteCollisions}
Distinct cusps $u$ and $v$ collide only finitely many times. 
\end{lemma}

\begin{proof}
For any cusp $w \in \Delta_\calT$ we define a function $\dev(w) \from \calS_\infty \to \CP^1$ by taking $\rho \mapsto \dev_\rho(w)$.  Note that the coordinate functions, restricted to $\calS_\infty$, are meromorphic.  Also, $\dev(w)$ can be written as a rational function in terms of the coordinates.  Thus $\dev(w)$ is also meromorphic. 

Since $u$ and $v$ are distinct, the functions $\dev(u)$ and $\dev(v)$ disagree at $\rho_\infty$.  Thus their difference is not identically zero. 
\end{proof}

\begin{figure}[htbp]
\labellist
\small\hair 2pt
\pinlabel $z$ at 53 53
\pinlabel $z$ at 144.5 53
\pinlabel $z$ at 53 106
\pinlabel $z$ at 144.5 106
\pinlabel $\frac{z-1}{z}$ at 85 15
\pinlabel $\frac{z-1}{z}$ at 176.5 15
\pinlabel $\frac{z-1}{z}$ at 85 68
\pinlabel $\frac{z-1}{z}$ at 176.5 68
\pinlabel $\frac{z-1}{z}$ at 85 121
\pinlabel $\frac{z-1}{z}$ at 176.5 121
\pinlabel $\frac{1}{1-z}$ at 85 39
\pinlabel $\frac{1}{1-z}$ at 176.5 39
\pinlabel $\frac{1}{1-z}$ at 85 92
\pinlabel $\frac{1}{1-z}$ at 176.5 92

\pinlabel $z$ at 42 46
\pinlabel $z$ at 133.5 46
\pinlabel $z$ at 225 46
\pinlabel $z$ at 42 99
\pinlabel $z$ at 133.5 99
\pinlabel $z$ at 225 99
\pinlabel $\frac{z-1}{z}$ at 39 11
\pinlabel $\frac{z-1}{z}$ at 130.5 11
\pinlabel $\frac{z-1}{z}$ at 222 11
\pinlabel $\frac{z-1}{z}$ at 39 64
\pinlabel $\frac{z-1}{z}$ at 130.5 64
\pinlabel $\frac{z-1}{z}$ at 222 64
\pinlabel $\frac{z-1}{z}$ at 39 117
\pinlabel $\frac{z-1}{z}$ at 130.5 117
\pinlabel $\frac{z-1}{z}$ at 222 117
\pinlabel $\frac{1}{1-z}$ at 13.5 27
\pinlabel $\frac{1}{1-z}$ at 105 27
\pinlabel $\frac{1}{1-z}$ at 196.5 27
\pinlabel $\frac{1}{1-z}$ at 13.5 80
\pinlabel $\frac{1}{1-z}$ at 105 80
\pinlabel $\frac{1}{1-z}$ at 196.5 80

\pinlabel $w$ at 37 52.5
\pinlabel $w$ at 128.5 52.5
\pinlabel $w$ at 220 52.5
\pinlabel $w$ at 37 105.5
\pinlabel $w$ at 128.5 105.5
\pinlabel $w$ at 220 105.5
\pinlabel $\frac{w-1}{w}$ at 53 40
\pinlabel $\frac{w-1}{w}$ at 144.5 40
\pinlabel $\frac{w-1}{w}$ at 53 93
\pinlabel $\frac{w-1}{w}$ at 144.5 93
\pinlabel $\frac{1}{1-w}$ at 77 27
\pinlabel $\frac{1}{1-w}$ at 168.5 27
\pinlabel $\frac{1}{1-w}$ at 77 80
\pinlabel $\frac{1}{1-w}$ at 168.5 80

\pinlabel $w$ at 50.5 8
\pinlabel $w$ at 142 8
\pinlabel $w$ at 50.5 61
\pinlabel $w$ at 142 61
\pinlabel $w$ at 50.5 114
\pinlabel $w$ at 142 114
\pinlabel $\frac{w-1}{w}$ at 7.5 38
\pinlabel $\frac{w-1}{w}$ at 99 38
\pinlabel $\frac{w-1}{w}$ at 190.5 38
\pinlabel $\frac{w-1}{w}$ at 7.5 91
\pinlabel $\frac{w-1}{w}$ at 99 91
\pinlabel $\frac{w-1}{w}$ at 190.5 91
\pinlabel $\frac{1}{1-w}$ at 7.5 15
\pinlabel $\frac{1}{1-w}$ at 99 15
\pinlabel $\frac{1}{1-w}$ at 190.5 15
\pinlabel $\frac{1}{1-w}$ at 7.5 68
\pinlabel $\frac{1}{1-w}$ at 99 68
\pinlabel $\frac{1}{1-w}$ at 190.5 68
\pinlabel $\frac{1}{1-w}$ at 7.5 121
\pinlabel $\frac{1}{1-w}$ at 99 121
\pinlabel $\frac{1}{1-w}$ at 190.5 121
\endlabellist
\includegraphics[width=\textwidth]{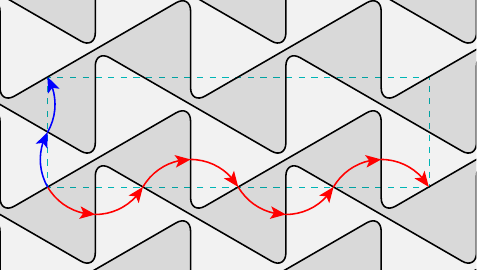}
\caption{The triangulation of the universal cover of the cusp torus of the figure-eight knot complement, as induced by the canonical triangulation (shown in \reffig{VeerFigEight}).  The tetrahedron shapes are all regular ($z = w = e^{\pi i/3}$).  The dashed rectangle is a fundamental domain for the tiling, with the meridian $m$ drawn vertically and the longitude $l$ drawn horizontally.  For more general shapes $z$ and $w$ we obtain the complex dihedral angles shown, following the usual conventions~\cite[page~47]{Thurston78}.}
\label{Fig:CuspTriangulationComplete}
\end{figure}

\subsection{Exotic taut structures}
\label{Sec:FlatDSS}

We are now ready for our next example.  
Let $M$ be the figure-eight knot complement.  
Let $\calT$ be the canonical triangulation of $M$, shown in \reffig{VeerFigEight}, there equipped with a taut structure.  
The triangulation $\calT$ admits two more taut structures, each of which is again transverse.  
These \emph{exotic} structures are generated by \emph{leading-trailing deformations} of the veering triangulation~\cite[Proposition 6.8]{FuterGueritaud13}.
We have the following. 

\begin{theorem}
\label{Thm:Exotic}
Each of the exotic taut structures on $\calT$ admits uncountably many compatible circular orders.
\end{theorem}


\noindent
This result, and the techniques needed for its proof, are not used later in the paper. 
Thus the proof can be skipped upon a first reading.  
However, see \refque{RigidVersusNotSo}.

\begin{proof}[Proof of \refthm{Exotic}]
The triangulation $\calT$ induces a triangulation of the boundary torus $\bdy M$.  
Taking preimages gives a triangulation of the universal cover of the cusp torus; 
this is shown in \reffig{CuspTriangulationComplete} along with other details.  
This is a view of $\cover{\calT}$ as seen from a fixed cusp $c_\infty$.  
The cusps connected to $c_\infty$ by an edge of $\cover{\calT}$ form a combinatorial copy of the lattice $\ZZ \oplus \ZZ e^{\pi i/3}$.  
Let $c_{0,0}$ be any fixed cusp in this lattice.  
Let $m$ and $l$ be the usual meridian and longitude of the figure-eight knot.  
We define a sublattice by taking $c_{p,q} = m^p l^q (c_{0,0})$.  

As shown in \reffig{CuspTriangulationComplete}, we label the corners of the cusp triangles with the corresponding complex dihedral angles of the ideal tetrahedra.  We can now derive Thurston's gluing equations for this triangulation.  The equations for the two ideal edges are identical: namely 
\[
z(z-1)w(w-1) = 1.
\]
As discussed above, the solutions $\rho = (z, w) \in \calP^2$ form the shape variety $\calS(\calT)$.  Again consulting \reffig{CuspTriangulationComplete}, we compute the holonomies for $m$ and $l$ and obtain
\[
H(m) = 1/(w(1-z)) \quad \mbox{and} \quad H(l) = z^2(1-z)^2.
\] 
Recall that Thurston's hyperbolic Dehn surgery equation~\cite[page~57]{Thurston78}, with real \emph{surgery coordinates} $\mu$ and $\lambda$, is
\[
\mu \log H(m) + \lambda \log H(l) = 2 \pi i.
\]
Here analytic continuation must be used to obtain the correct branch of the logarithm. 
Suppose we have found shapes $z$ and $w$ with positive imaginary part solving the gluing equations, as well as $\mu$ and $\lambda$ solving the surgery equation.
Then the tetrahedra with those shapes fit together to give an incomplete hyperbolic structure on $M$.  

We now concentrate on the case where $-4 < \mu < 4$ and $\lambda = 1$.
This is the top edge of the hourglass shown in \reffig{m004_canonical_dss}.
These ``surgeries'' are on the ``boundary'' of Dehn surgery space: the tetrahedron shapes $z$ and $w$ are real and their volume has decreased to zero.
The shapes along the top edge give the triangulation a fixed transverse taut structure: one of the two exotic structures from~\cite[Proposition 6.8]{FuterGueritaud13}.
The other exotic structure corresponds to the bottom edge of the hourglass.
Again, see \reffig{m004_canonical_dss}.

\begin{figure}[htbp]
\includegraphics[width=\textwidth]{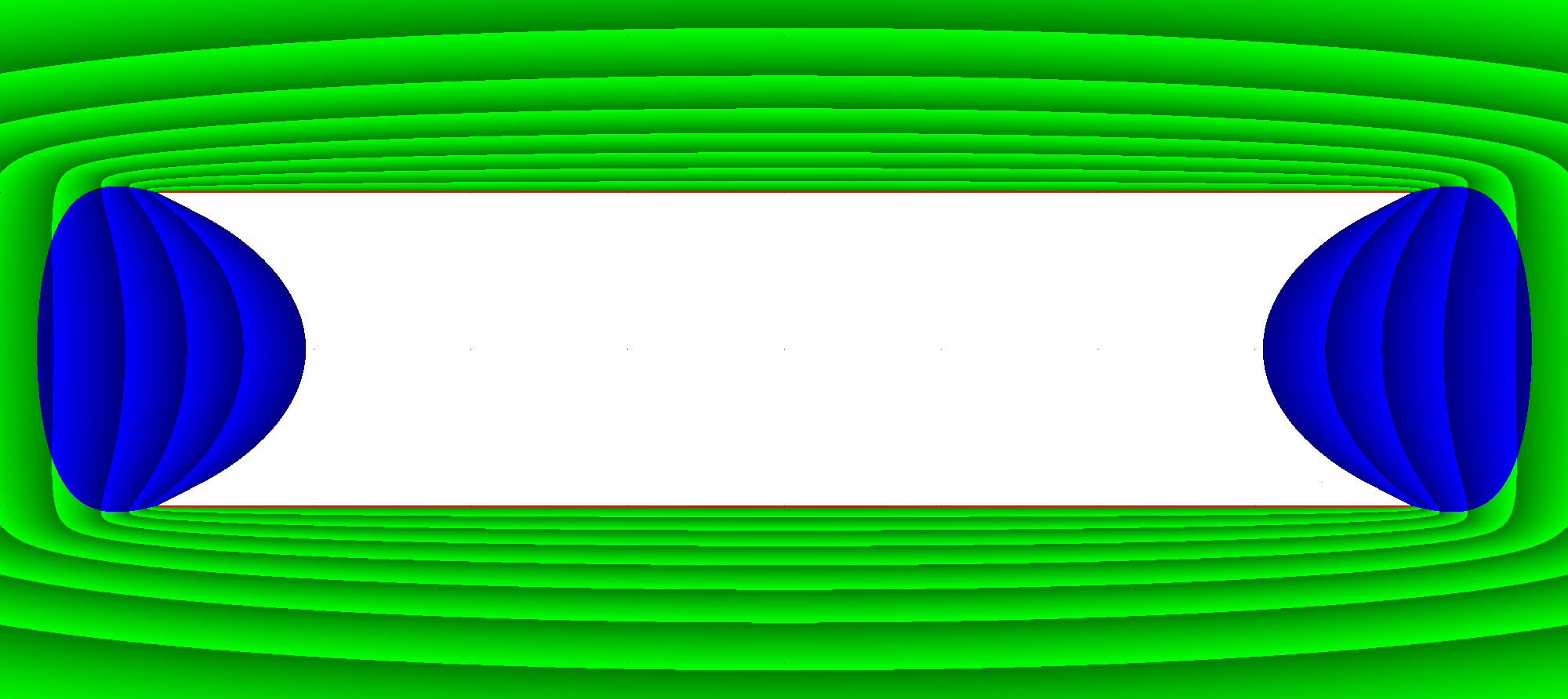}
\caption{The Dehn surgery space of the figure-eight knot complement.  
We call the white region the ``hourglass''.  
The complement of the hourglass is the component of the positive volume locus which contains the shapes giving the complete, finite volume, hyperbolic structure.  
The origin $(0, 0)$ is at the centre, the point $(0, 1)$  is the midpoint of the top of the hourglass, and the point $(4,1)$ is the top right corner of the hourglass.  
For a discussion of the colour scheme and more see~\cite{HSS09}. 
Also see~\cite[page 4.21]{Thurston78} and~\cite[page 70]{Weeks85}. }
\label{Fig:m004_canonical_dss}
\end{figure}

Let $\calI \subset \calS_\infty$ be the set of (pairs of) shapes corresponding to the top edge of the hourglass.
So, for each $\mu \in (-4, 4)$ we have shapes $\rho_\mu = (z_\mu, w_\mu) \in \calI$ and a map $\dev_\mu \from \Delta_\calT \to \bdy \HH^2$.
We adapt the conventions that $\bdy \HH^2 = \RR \cup \infty$ and that $\dev_\mu(c_\infty) = \infty$. 

\reffig{fig8_(-1_1)_shapes2} is a cartoon of the image of the developing map for the shapes $(\mu, \lambda) = (-1,1)$.
The vertical positions of features are purely combinatorial; all of the cusps of the original lattice $\ZZ \oplus \ZZ e^{\pi i/3}$ map into a single horizontal line.
For some $\mu$ (for example, those at the rational points) $\dev_\mu$ is not one-to-one and we do not get a circular order.
However, by \reflem{FiniteCollisions}, a pair of cusps collide at only a finite number of points of $\calI$.
The collection of pairs of distinct cusps is countable;
thus the set of $\rho_\mu \in \calI$ having some collision is only countable.
The uncountable remainder give circular orders;
these are denoted by $\calO_\mu$.  
The definition of $\dev_\mu$ implies that $\calO_\mu$ is compatible with $\calT$.  
The deck group $\pi_1(M)$ acts on cusps $\Delta_\calT$ directly and acts on $\bdy \HH^2 = \RR \cup \{\infty\}$ via the holonomy representation to $\PSL(2,\RR)$.  
The circular orders $\calO_\mu$ are obtained by pulling back from $\bdy \HH^2$ and so are invariant.

Finally, we show that all of the circular orders $\calO_\mu$ are distinct.  
We begin with an example.  
\reffig{fig8_(-1p1_1)_shapes} shows the tetrahedron shapes at Dehn surgery coordinates $(\mu, \lambda) = (-1.1,1)$.  
Since these are not integral, unlike \reffig{CuspTriangulationComplete} the associated developing map now does not factor through a tiling.  
Instead, we have a cut along the negative real axis where the ends of the longitude and meridian do not line up.  
The cusps associated to the end of the longitude have slid to the left relative to the cusps associated to the end of the meridian.  
However, if we go eleven times backwards along the meridian and ten times forwards along the longitude, the corresponding holonomy is again the identity, and again the cusps line up.  

In general, if $\mu = p/q$ is rational (and in lowest terms) then the holonomy of $m^p l^q$ is trivial.  
We deduce that $\dev_\mu(c_{p,q}) = \dev_\mu(c_{0,0})$.  
If $\mu < p/q$ then $\dev_\mu(c_{p,q})$ is to the left of $\dev_\mu(c_{0,0})$ in our picture.  
If $\mu > p/q$ then $\dev_\mu(c_{p,q})$ is to the right of $\dev_\mu(c_{0,0})$ in our picture.

Therefore, if $\rho_\mu \in \calI$ has no collisions, then the circular order $\calO_\mu$ satisfies 
\[
\calO_\mu(\dev_\mu(c_\infty), \dev_\mu(c_{0,0}), \dev_\mu(c_{p,q})) = \pm 1
\]
as $\mu < p/q$ or $p/q < \mu$, respectively.  
Now, for any two circular orders $\calO_\mu$ and $\calO_{\mu'}$, there is a rational number $p/q$ between $\mu$ and $\mu'$; 
therefore $\calO_\mu$ and $\calO_{\mu'}$ are distinct.  
This completes the proof of \refthm{Exotic}.
\end{proof}

\begin{figure}[htbp]
\labellist
\small\hair 2pt
\pinlabel $z$ at 303 31
\pinlabel $z$ at 303 131
\pinlabel $w$ at 260 89
\pinlabel $w$ at 102 148
\pinlabel $z$ at 109 248
\pinlabel $z$ at 189 302
\pinlabel $w$ at 310 172
\pinlabel $w$ at 213 237

\tiny
\pinlabel $z$ at 256 112
\pinlabel $z$ at 256 149
\pinlabel $w$ at 242 128.5
\pinlabel $w$ at 182.5 151

\pinlabel $z$ at 187 192
\pinlabel $z$ at 218 213
\pinlabel $w$ at 261 170
\pinlabel $w$ at 224.5 194.5

\pinlabel $0$ at 102 344
\pinlabel $1$ at 182 344

\endlabellist
\includegraphics[width=0.8\textwidth]{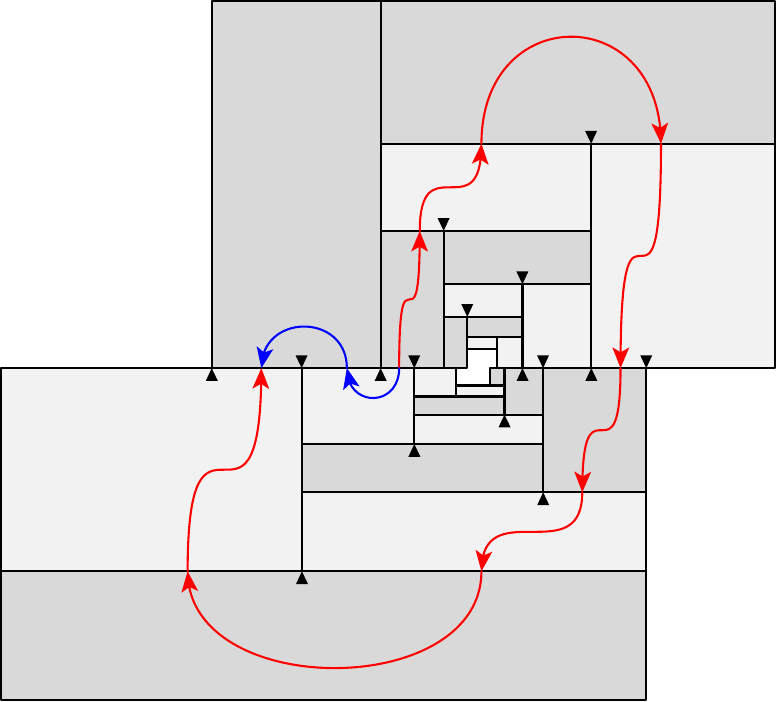}
\caption{A cartoon of the image of the developing map, as applied to the triangulation of the universal cover of the boundary torus of the figure-eight knot complement.  Here the tetrahedron shapes are $z \approx 1.8731617275018602$ and $w \approx -0.428119859615$ (calculated by SnapPy~\cite{snappy}).  These correspond to the surgery coefficients $(\mu, \lambda) = (-1, 1)$.  Since the shapes are real, the tetrahedra are flat.  Thus the cusp triangles degenerate to intervals.  The vertical scale in this figure has no geometric significance; 
we draw the triangles as rectangles instead of as intervals in order to see their gluings.  The vertical sides of each rectangle give the vertices (of the cusp triangle) with dihedral angle zero;
the small black arrow in each rectangle points at the vertex with dihedral angle $\pi$.}
\label{Fig:fig8_(-1_1)_shapes2}
\end{figure}

\begin{figure}[htbp]
\labellist
\small\hair 2pt
\pinlabel $z$ at 335 31
\pinlabel $z$ at 335 131
\pinlabel $w$ at 290 89
\pinlabel $w$ at 115 148
\pinlabel $z$ at 137 248
\pinlabel $z$ at 218 302
\pinlabel $w$ at 342 172
\pinlabel $w$ at 241 237

\tiny
\pinlabel $w$ at 210 150.5
\pinlabel $z$ at 216 192

\pinlabel $0$ at 129.5 344
\pinlabel $1$ at 210.5 344

\endlabellist
\includegraphics[width=0.8\textwidth]{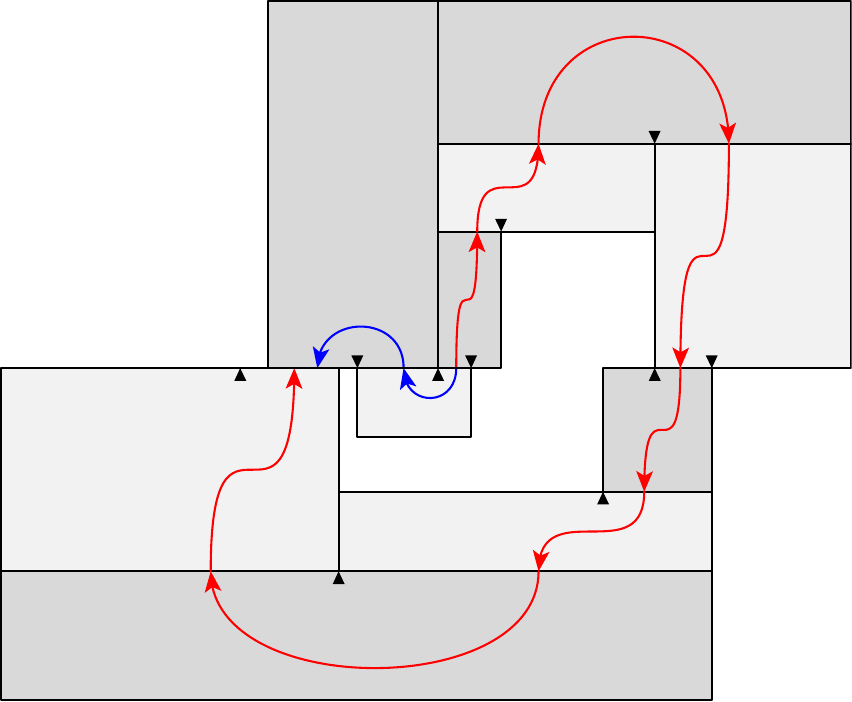}
\caption{Another cartoon, as in \reffig{fig8_(-1_1)_shapes2}.
Here the tetrahedron shapes are $z \approx 1.90447711038$ and $w \approx -0.411335649452$ (calculated by SnapPy~\cite{snappy}).
These correspond to the surgery coefficients $(\mu, \lambda) = (-1.1, 1)$.}
\label{Fig:fig8_(-1p1_1)_shapes}
\end{figure}

\begin{question}
\label{Que:RigidVersusNotSo}
As discussed in \refexa{SurfaceBundlesTwo} the transverse taut structures coming from layered triangulations of surface bundles have unique compatible circular orders.  
As we shall show in \refthm{VeerImpliesUnique}, veering triangulations also have unique compatible circular orders.  

On the one hand, most layered triangulations have no veering structure.  
On the other hand, Hodgson, Rubinstein, Tillmann, and the third author~\cite[Section~4]{HRST11} give a veering triangulation for the SnapPy manifold \texttt{s227}; 
they also show that \texttt{s227} is not fibred.  
Thus \refthm{VeerImpliesUnique} gives examples of circular orders that do not come from fibrations.  
We give many more non-fibred manifolds with veering triangulations in~\cite[\href{https://math.okstate.edu/people/segerman/veering/veering_census_with_data.txt}{\texttt{veering\_census\_with\_data.txt}}]{GSS19}.  
We also use \emph{veering Dehn surgery} to give an infinite family of such manifolds~\cite{veering_dehn_surgery}. 

However, from \refthm{Exotic} we have that the exotic transverse taut structures (of~\cite[Proposition~6.8]{FuterGueritaud13}) each have uncountably many compatible circular orders.  
This raises several questions:  
Suppose that we are given a three-manifold equipped with a transverse taut ideal triangulation.
What are necessary and sufficient conditions on the taut structure to ensure there is a unique compatible circular order?
Is this property decidable?
Is there a transverse taut ideal triangulation that admits \emph{no} compatible circular orders?
\end{question}

\chapter{Geography of taut triangulations}
\label{Cha:Geography}

In this section, after reviewing material on \emph{train tracks}, we introduce the new concepts of \emph{landscapes}, \emph{continents}, and \emph{continental exhaustions}. 
 
\subsection{Train tracks}
\label{Sec:TrainTracks}

Here we will closely follow the drawing style, and thus the imposed definitions, introduced in~\cite[Figure~11]{Agol11}.  
A \emph{pre-track} $\tau$ in a surface $F$ is a locally finite, properly embedded, smooth graph.  
The vertices of $\tau$ are called \emph{switches} while the edges of $\tau$ are called \emph{branches}.  
Every switch of $\tau$ is equipped with a tangent line.  
We call $\tau$ a \emph{train track} if 
\begin{itemize}
\item 
every switch in the interior of $F$ has at least one branch entering on each of its two sides and
\item 
every switch in $\bdy F$ has its tangent perpendicular to $\bdy F$ and has valence at least one. 
\end{itemize}
This is a variant of the usual definition of train tracks (see, for example,~\cite[Definition~8.9.1]{Thurston78}); 
we are omitting the Euler characteristic condition on the components of $F - \tau$. 

Suppose that $s \in \tau$ is a switch in the interior of $F$.
We say $s$ is \emph{large} if $s$ has exactly two branches entering on each side.
We say $s$ is \emph{small} if $s$ has exactly one branch entering on each side.
We say $s$ is \emph{mixed} if one branch enters on one side of $s$ and two enter on the other.
See \reffig{Tracks}.

\begin{figure}[htbp]
\subfloat[A large switch.]{
\includegraphics[width=0.28\textwidth]{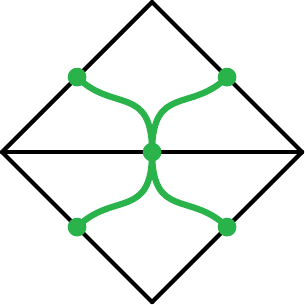}
\label{Fig:SwitchLarge}
}
\quad
\subfloat[A small switch.]{
\includegraphics[width=0.28\textwidth]{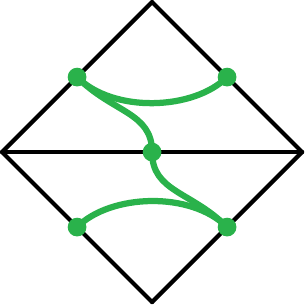}
\label{Fig:SwitchSmall}
}
\quad
\subfloat[A mixed switch.]{
\includegraphics[width=0.28\textwidth]{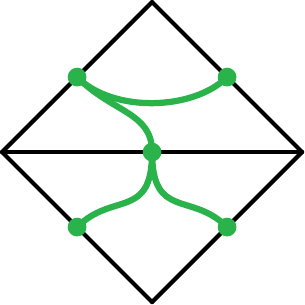}
\label{Fig:SwitchMixed}
}
\caption{Switches.}
\label{Fig:Tracks}
\end{figure}

We may \emph{split} (either to the right or to the left) the track $\tau$ along a large switch to obtain a new track $\tau' \subset F$.  See \reffig{SwitchSplit}.  After a split the components of $F - \tau'$ are homeomorphic to those of $F - \tau$.  The reverse of a split is called a \emph{fold}.

\begin{figure}[htbp]
\includegraphics[width=.935\textwidth]{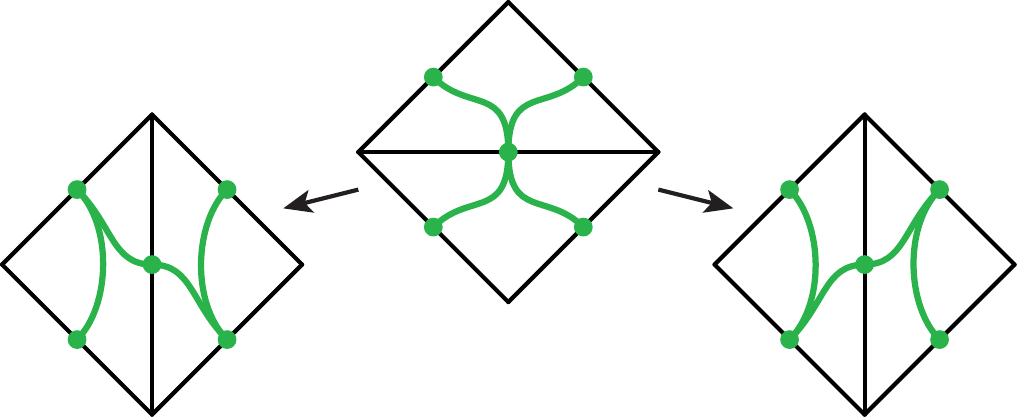}
\caption{A large switch can split either to the left or to the right.}
\label{Fig:SwitchSplit}
\end{figure}

A \emph{train route} in a train track $\tau$ is a smooth embedding of $[0,1]$, of $S^1$, of $\RR_{\geq 0}$, or of $\RR$ into $\tau$.  
Such a route is called, respectively, a \emph{train interval}, a \emph{train loop}, a \emph{train ray}, or a \emph{train line}.  

\subsection{Landscapes and rivers}

For the next several sections, we fix $M$ a three-manifold equipped with a transverse taut ideal triangulation $\calT$.
Let $B = B(\calT)$ be the associated taut branched surface, made from the faces of $\calT$.   
Let $\cover{B}$ be the preimage of $B$ in the universal cover $\cover{M}$.

\begin{definition}
A \emph{landscape} $L \subset \cover{B}$ is a connected embedded carried surface, which is a union of triangles of $\cover{B}$.  
\end{definition}

\noindent
A landscape $L$ inherits an ideal triangulation and a co-orientation from $B$.  
Boundary components of $L$ (if any) are necessarily carried by edges of $\cover{\calT}$. 
We call these \emph{coastal edges} of $L$.   
We now restate \cite[Corollary~1.1]{SchleimerSegerman20} and \refcor{VerticesDistinct} in terms of landscapes.

\begin{lemma}
\label{Lem:Disk}
The interior of a landscape $L$ is an open disk. 
Any pair of vertices of a landscape are distinct cusps of $\Delta_\calT$.
\qed
\end{lemma}

Thus any landscape $L$ induces a circular order $\calO_L$ on the cusps of $\Delta_L \subset \Delta_\calT$; 
furthermore this circular order is compatible with $\calT$ in the sense of \refdef{Compatible}.

Simple examples of landscapes include: a single face of $\cover{\calT}$, the upper or lower boundary of a single tetrahedron of $\cover{\calT}$, or any elevation of a surface carried by $B$.  

\begin{definition}
\label{Def:UpperLowerTrackFace}
Suppose that $f$ is a face of $\cover{\calT}$.
Suppose that $t$ is the tetrahedron attached to $f$, above $f$.  
Let $e$ be the lower $\pi$--edge of $t$; 
let $e'$ and $e''$ be the remaining edges of $f$.  
We define $\tau^f$, the \emph{upper track} for $f$ as follows.  
The track $\tau^f$ has three switches; 
we place these at the midpoints of the three edges.  
The track $\tau^f$ has two branches, running from $e'$ to $e$ and from $e''$ to $e$.  
See \reffig{UpperLowerTracks}. 

The \emph{lower track} $\tau_f$ is defined similarly;
we use the tetrahedron attached to, and immediately below, $f$.
\end{definition}

In this section, it sometimes is convenient to think of water flowing along the branches of $\tau^f$;
in the notation of the previous definition, the water flows away from $e'$ and $e''$ and flows towards $e$.  

\begin{definition}
\label{Def:TrackCusp}
We define the \emph{track-cusp} of $\tau^f$ by taking the switch of $\tau^f$ and adding a small neighbourhood of it in the region of $f - \tau^f$ between the two branches.  
See \reffig{UpperLowerTracks}.  
We define the track-cusp of $\tau_f$ similarly. 
\end{definition}

\begin{definition}
\label{Def:UpperTrack}
Suppose that $L$ is a landscape carried by $\cover{B}$.  
We define the \emph{upper track} $\tau^L$ to be the union of $\tau^f$ as $f$ runs over the ideal triangles of $L$.  
We define $\tau_L$ similarly.  
See \reffig{Tracks} for simple examples.  
\end{definition}

\begin{figure}[htbp]
\subfloat[Two taut tetrahedra $t^f$ and $t_f$ above and below a face $f$.]{
\labellist
\small
\hair 2pt
\pinlabel $f$ at 88 71
\endlabellist
\includegraphics[height=1in]{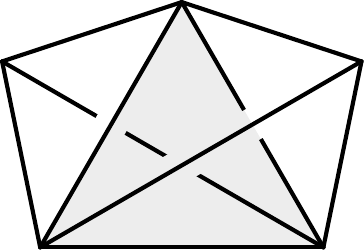}
\label{Fig:TwoTautTetrahedra}
}
\quad
\subfloat[The upper train track $\tau^f$ in $f$.]{
\labellist
\small
\hair 2pt
\pinlabel $\tau^f$ [tr] at 65 42
\endlabellist
\includegraphics[height=1in]{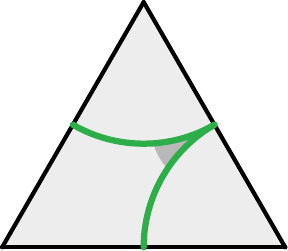}
\label{Fig:UpperTrack}
}
\quad
\subfloat[The lower train track $\tau_f$ in $f$.]{
\labellist
\small
\hair 2pt
\pinlabel $\tau_f$ [tl] at 77 37
\endlabellist
\includegraphics[height=1in]{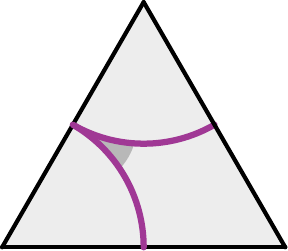}
\label{Fig:LowerTrack}
}
\caption{  
The upper track $\tau^f$ points at the bottom $\pi$--edge of $t^f$ while the lower track $\tau_f$ points at the top $\pi$--edge of $t_f$. 
The track-cusps in $t^f$ and $t_f$ are shaded darker grey.}
\label{Fig:UpperLowerTracks}
\end{figure}

Suppose $e$ is an interior edge for a landscape $L$.   
Drawing on the water analogy above, we call $e$ a \emph{sink} for $\tau^L$ if $e \cap \tau^L$ is a large switch (that is, the flows in the two adjacent triangles flow into $e$).  
We call $e$ a \emph{fall} for $\tau^L$ if $e \cap \tau^L$ is a mixed switch (that is, the flows cross $e$).  
We call $e$ a \emph{watershed} for $\tau^L$ if $e \cap \tau^L$ is a small switch (that is, the flows both flow out of $e$).  
Again, see \reffig{Tracks}.

Suppose $e$ is a coastal edge for a landscape $L$.  
We call $e$ a \emph{coastal sink} for $\tau^L$ if $e$ meets a track cusp of $\tau^L$. 

\begin{remark}
\label{Rem:Sink}
Note that if $e$ is a (coastal) sink for $\tau^L$ then there is a tetrahedron, immediately above $L$, which shares (one) two faces with $L$ and whose lower $\pi$--edge is equal to $e$.  
There is a similar statement for~$\tau_L$. 
\end{remark}

We end this section with a special kind of landscape.

\begin{definition}
\label{Def:River}
Suppose that $R$ is a landscape with finitely many triangles.
We call $R$ an \emph{upper river} if 
\begin{itemize}
\item
every triangle of $R$ meets at most two others and 
\item
every interior edge of $R$ is a fall for $\tau^R$. 
\end{itemize}
For an example, see \reffig{River}(0).
Note that the flows equip $\tau^R$ with a consistent orientation.
The first triangle of $R$ is the \emph{source} of the river.  
The final edge of $R$ is the \emph{mouth} of the river.  
We use $\ell(R)$ to denote the number of triangles in $R$.
This is the \emph{length} of $R$.

Lower rivers are defined using $\tau_L$.
\end{definition}

\begin{lemma}
\label{Lem:River}
Let $R$ be a maximal upper river in a landscape $L$. 
Then the mouth of $R$ is a (possibly coastal) sink for $\tau^L$. 
There is a similar statement for lower rivers. \qed
\end{lemma}

\subsection{Continents}

Let $M$ be an connected oriented three-manifold equipped with a transverse taut ideal triangulation $\calT$.
Recall that $\cover{B}$ is the preimage, in $\cover{M}$, of the co-oriented branched surface $B(\calT)$.

Suppose that $C \subset \cover{\calT}$ is a collection of model tetrahedra.
We say that $C$ is \emph{face-connected} if any pair of its tetrahedra are connected by a path in $C$ (transverse to the two-skeleton of $C$). 

We extend the definition of a layering to a finite face-connected collection of tetrahedra as follows.

\begin{definition}
\label{Def:Continent}
Suppose that $C \subset \cover{\calT}$ is a face-connected collection of $n$ model tetrahedra. 
A \emph{layering} $(K_i)_{i = 0}^n$ of $C$ is exactly as in \refdef{Layered} except that the index set is $\{0, 1, \ldots, n\}$ and the second condition holds for all $K_i$ other than $K_n$.

We say that $C$ is a \emph{continent} if it admits a layering.
\end{definition}

As a simple example, any single tetrahedron of $\cover{\calT}$ is a continent. 
See \reffig{Continent} for a cartoon of a more complicated continent;
we justify this cartoon in \reflem{ContinentEmbeds}.

\begin{figure}[htb]
\centering
\includegraphics[width=0.5\textwidth]{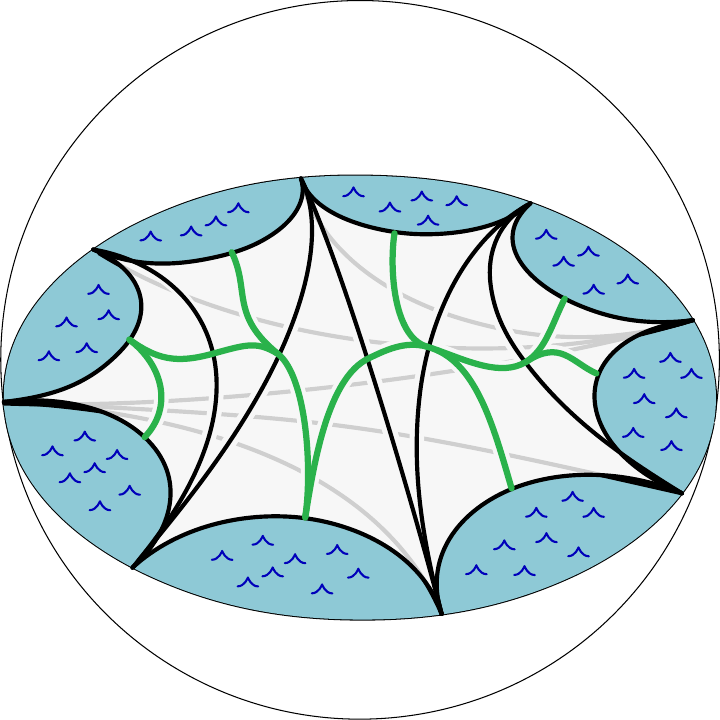}
\caption{A continent.}
\label{Fig:Continent}
\end{figure}

Suppose that $C$ is a continent. 
We say that a face $f$ of $\cover{\calT}$ is a \emph{lower face} of $C$ if 
\begin{itemize}
\item
$f$ lies in $\bdy C$ and
\item
the co-orientation on $f$ points into $C$. 
\end{itemize}
We define \emph{upper faces} of $C$ similarly. 

\begin{lemma}
\label{Lem:Boundary}
Suppose that $C \subset \cover{\calT}$ is a continent.
Suppose that $(L_i)_{i = 0}^n$ is a layering of $C$.
Then the layer $L_0$ is carried by the union of the lower faces of $C$; 
the layer $L_n$ is carried by the union of the upper faces. 
\end{lemma}

\begin{proof}
We first prove that all upper faces lie in $L_n$ and all lower faces lie in $L_0$.
We prove the contrapositive.
Suppose that $f$ is a face of $C$ not in $L_n$.
Let $j$ be the largest index so that $L_j$ contains $f$.
So $f$ is not contained in $L_{j+1}$.
That is, $f$ is a lower face of $t_j$.
Thus $f$ is not an upper face of $C$.
We deal with lower faces similarly.

The other directions require the following.

The layers $L_i$ are carried by $\calB(\cover{\calT})$ so they are landscapes.
By \reflem{Disk}, each layer $L_i$ is an embedded triangulated disk with distinct vertices meeting distinct cusps of $\Delta_\calT$.
Let $\calO_i$ be the induced cyclic ordering on the cusps $\Delta_i$ of $L_i$.
Since $(L_i)$ is a layering, we deduce that $\Delta_i = \Delta_{i+1}$ and $\calO_i = \calO_{i+1}$ for all $i < n$.
To fix ideas, we choose an order preserving map from $\Delta_0$ to $\bdy \HH^2$, the boundary of the hyperbolic plane.
The convex hull of $\Delta_0$ is then an ideal polygon, and the edges of $L_i$ give this polygon an ideal triangulation.

Suppose that $R$ is a finite sided convex polygon in $\HH^2$, perhaps with ideal vertices.
The \emph{index} of $R$ is one, minus a quarter for each material corner of $R$, minus a half for each ideal corner of $R$. By convention, the index of the empty polygon is zero; 
the same holds for line segments.
For any two ideal triangles in $\HH^2$, the index of their intersection is non-positive; 
this is proven by examining the number of sides of the intersection.
We say that two ideal triangles in $\HH^2$ \emph{link} if their intersection has negative index. 
See \reffig{LinkingTriangles}.

\begin{figure}[htb]
\centering
\subfloat[Index $-1/2$.]{
\includegraphics[width=0.30\textwidth]{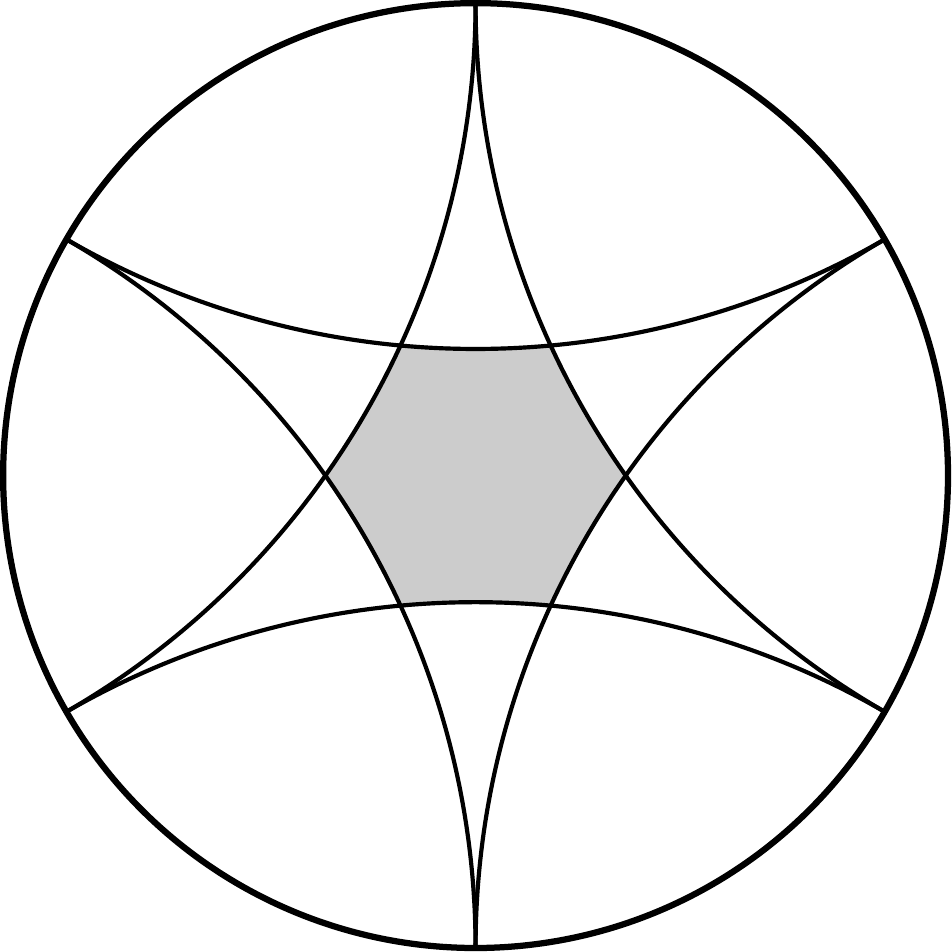}
\label{Fig:LinkingTrianglesMMMMMM}
}
\subfloat[Index $-1/4$.]{
\includegraphics[width=0.30\textwidth]{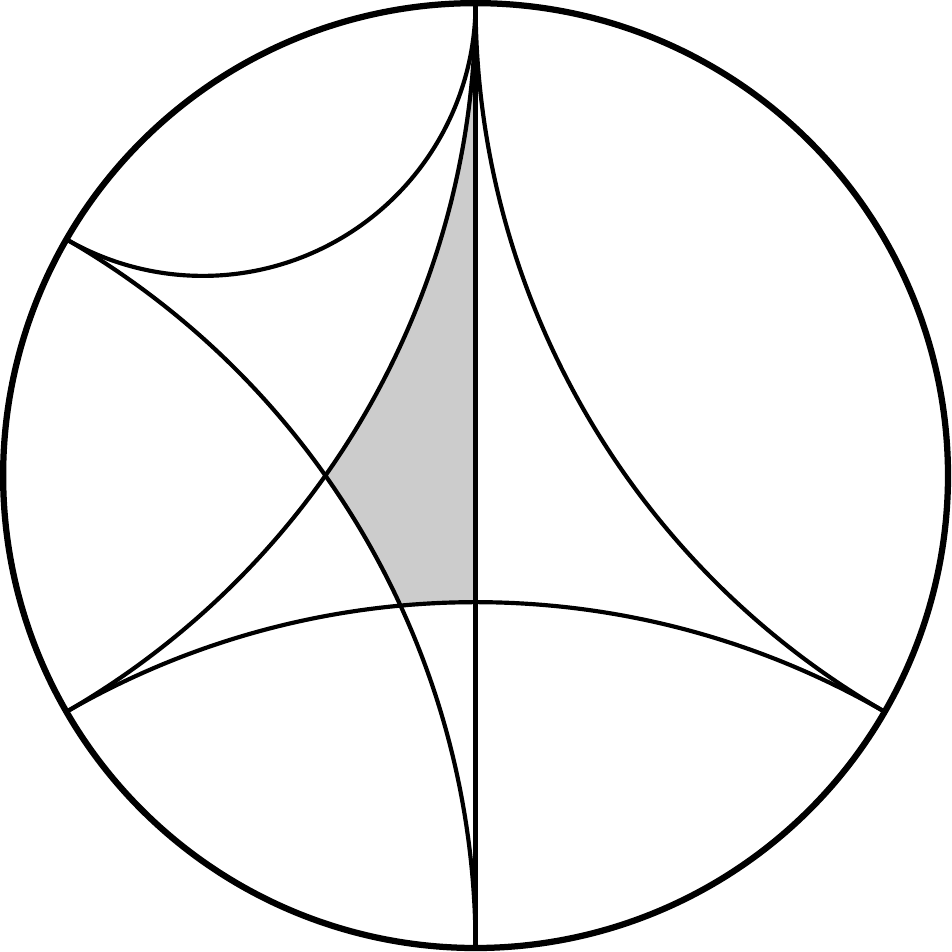}
\label{Fig:LinkingTrianglesIMMM}
}
\subfloat[Index $-1/4$.]{
\includegraphics[width=0.30\textwidth]{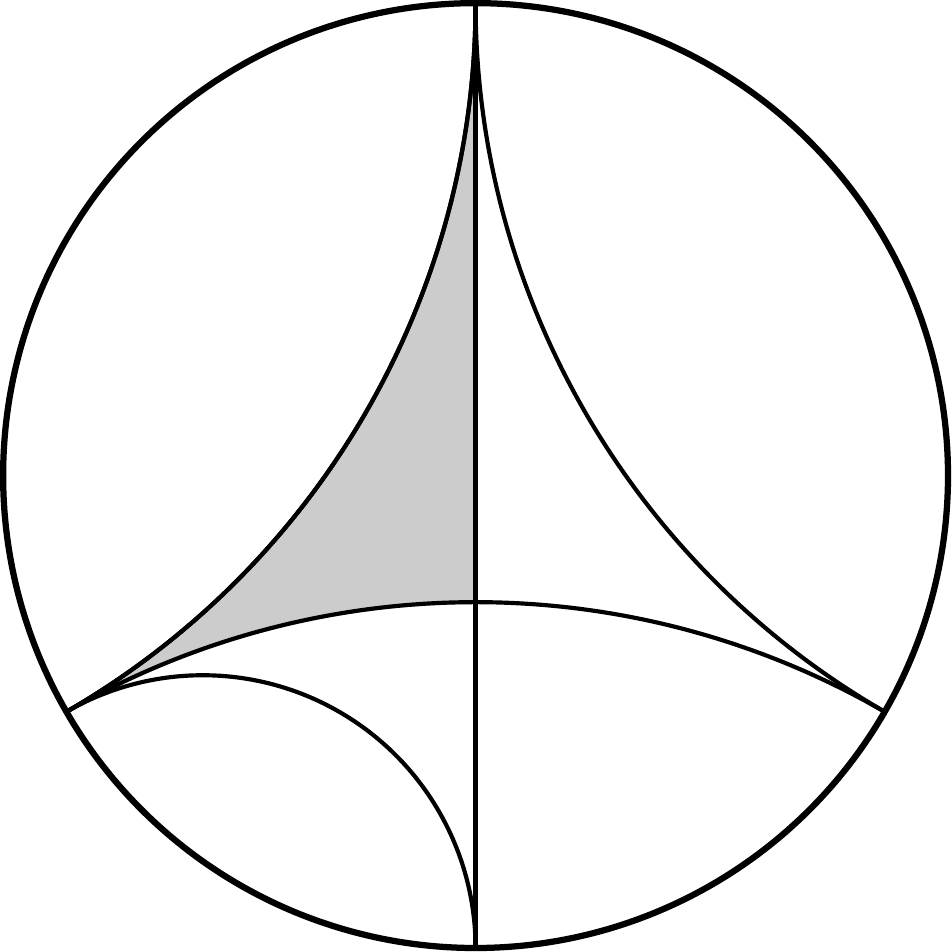}
\label{Fig:LinkingTrianglesIIM}
}
\caption{Two ideal triangles can link in four ways; 
the three shown in this figure, and where the triangles coincide (index $-1/2$).
}
\label{Fig:LinkingTriangles}
\end{figure}

From the above, we deduce the following.

\begin{claim}
\label{Clm:LinkingTriangles}
Two triangles that lie in the same layer, and link, are equal.
\qed
\end{claim}

We now return to proving the lemma.
That is, we show that faces in $L_n$ are upper faces and faces in $L_0$ are lower faces.
Again we prove the contrapositive.
Suppose that $f$ is a face of $C$ which is not an upper face.
Thus there is a tetrahedron $t = t_k$ of $C$ immediately above $f$.
Thus $f$ is a lower face of $t_k$.
We now build a sequence of faces $(f_i)_{i=k}^n$, starting with $f_k = f$, so that $f_i$ carries a face of $L_i$ and so that $f_i$ links $f$ for all $i$.
For the induction step at $j \geq k$, suppose that we have $f_j$.
If $f_j$ carries a face of $L_{j+1}$ then set $f_{j+1} = f_j$.
Otherwise there is a tetrahedron $t$ between $L_j$ and $L_{j+1}$ having $f_j$ as a lower face.
Let $g_j$ be the other lower face of $t$.
Let $h$ and $h'$ be the upper faces of $t$.
Note that $f \cap (f_j \cup g_j) = f \cap (h \cup h')$.
By the additivity of index~\cite[page 57]{Mosher03}, this polygon has negative index.
Thus one of the upper faces, say $h$, links $f$. 
Set $f_{j+1} = h$.

By the definition of $f$ and the construction of $f_{k+1}$ we have that $f_k \neq f_{k+1}$.
Thus there is a vertical arc starting at the centre of $f_k$, transverse to the two-skeleton, and ending at the centre of $f_n$.
If $f=f_n$ then the sequence $(f_i)$ gives a vertical loop (transverse to $B(C)$) from $f$ to itself. 
Since $\cover{M}$ is simply connected, this contradicts~\cite[Theorem~3.2]{SchleimerSegerman20}.
If not, then by \refclm{LinkingTriangles}, $f$ does not lie in $L_n$, as desired.
Similarly, all faces in $L_0$ are lower faces.
\end{proof}

A single continent may admit many different layerings;
these always differ by ``commuting'' flips.
However, by \reflem{Boundary}, the union of the lower faces $L'$ is always the first layer, and the union of the upper faces $L$ is always the last layer. 
Thus we call $L'$ and $L$ the \emph{lower} and \emph{upper} landscapes of $C$ respectively.
The coastal edges of $L'$ and $L$ are the same;
we call these the \emph{coastal} edges of $C$. 
We use $\Delta_C$ to denote the set of cusps of $\Delta_\calT$ meeting $C$.

\begin{corollary}
\label{Cor:ContinentOrder}
Suppose that $C$ is a continent in $\cover{\calT}$. 
Then there is a unique circular order $\calO_C$ on $\Delta_C$ that is compatible with $\calT$. \qed
\end{corollary}

\begin{lemma}
\label{Lem:ContinentEmbeds}
Suppose that $C$ is a continent in $\cover{\calT}$.
Let $L'$ and $L$ be the lower and upper landscapes of $C$ respectively. 
Then $L' \cap L = \bdy L' = \bdy L$. 
\end{lemma}

\begin{proof}
\reflem{Boundary} implies that $L'$ and $L$ cannot meet along a common face.

Suppose, for a contradiction, that $e$ is an edge of $\cover{\calT}$ that lies in the interior of $L$ and also lies in $L'$.
By \refcor{ContinentOrder}, we deduce that $e$ lies in the interior of $L'$.

Let $a$ and $b$ be the cusps of $\Delta_C$ meeting $e$.
Let $f$ and $g$ be the faces of $L$ adjacent to $e$.
Let $f'$ and $g'$ be the faces of $L'$ adjacent to $e$.
Let $c$ and $c'$ be the vertices of $f$ and $f'$ (respectively) other than $a$ and $b$. 
We arrange matters so that $\calO_C(a, b, c) = \calO_C(a, b, c')$. 

Suppose that $(L' = L_0, L_1, \ldots, L_n = L)$ is a layering of $C$.
We now find a sequence of faces $(f_i)_{i = 0}^n$ with $f_i$ carrying a face of $L_i$. 
We will require of each face $f_k$ that either
\begin{itemize}
\item
$a$ and $b$ are cusps of $f_k$ and its third cusp $c_k$ has
$\calO_C(a, b, c_k) = \calO_C(a, b, c)$ or
\item
$a$ is a cusp of $f_k$ and the remaining cusps of $f_k$ link $(a, b)$ in the circular order $\calO_C$.
\end{itemize}

\begin{figure}[htb]
\centering
\subfloat{
\labellist
\small\hair 2pt
\pinlabel $a$ [t] at 229 0
\pinlabel $b$ [b] at 229 468
\endlabellist
\includegraphics[width=0.30\textwidth]{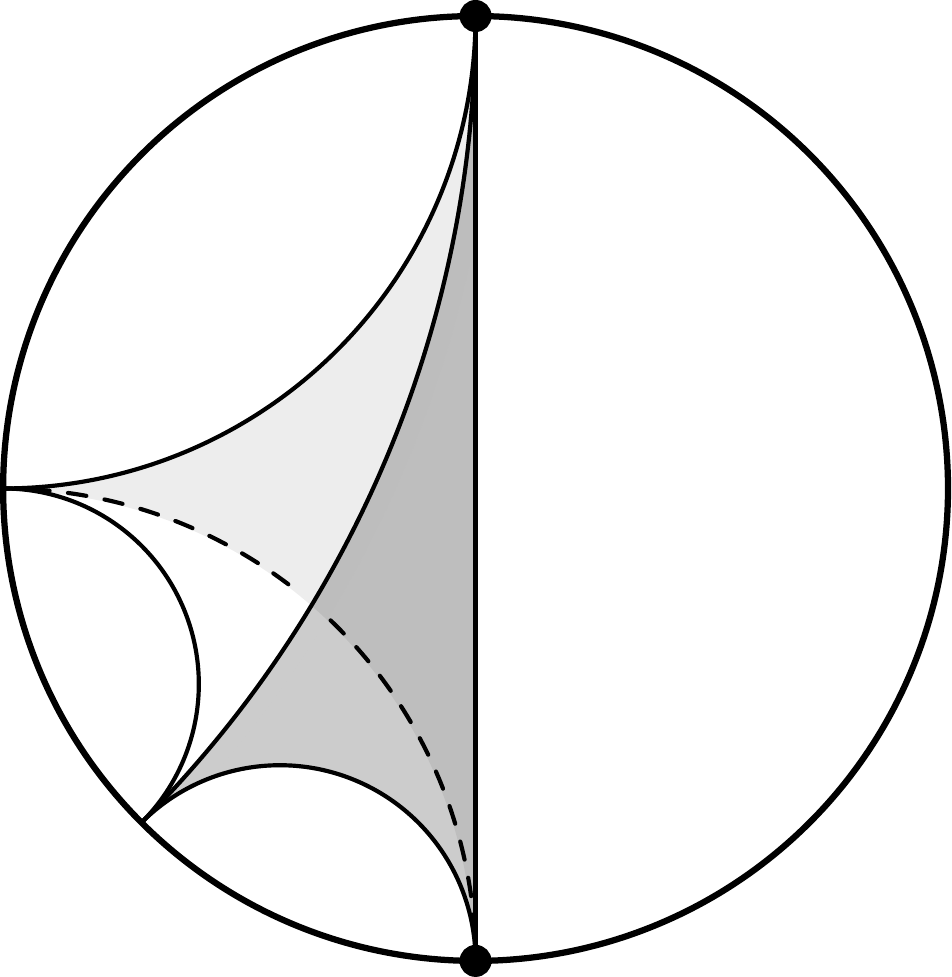}
}
\subfloat{
\labellist
\small\hair 2pt
\pinlabel $a$ [t] at 229 0
\pinlabel $b$ [b] at 229 468
\endlabellist
\includegraphics[width=0.30\textwidth]{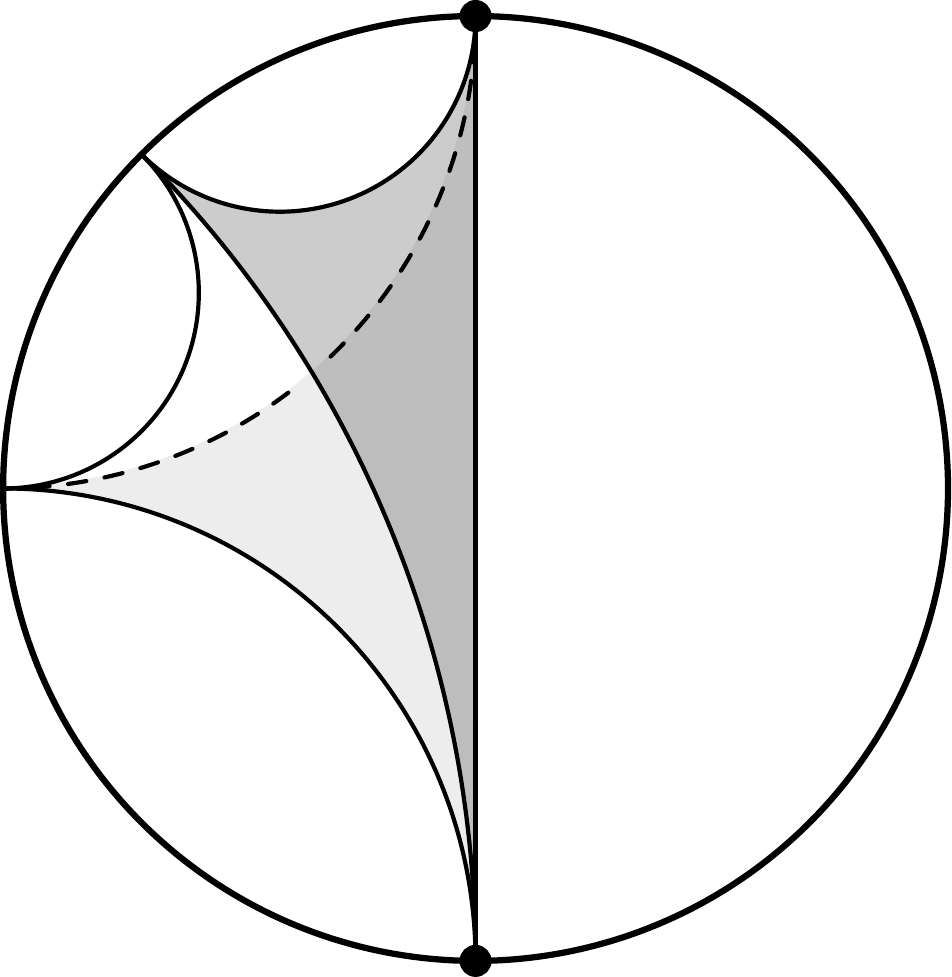}
}
\subfloat{
\labellist
\small\hair 2pt
\pinlabel $a$ [t] at 229 0
\pinlabel $b$ [b] at 229 468
\endlabellist
\includegraphics[width=0.30\textwidth]{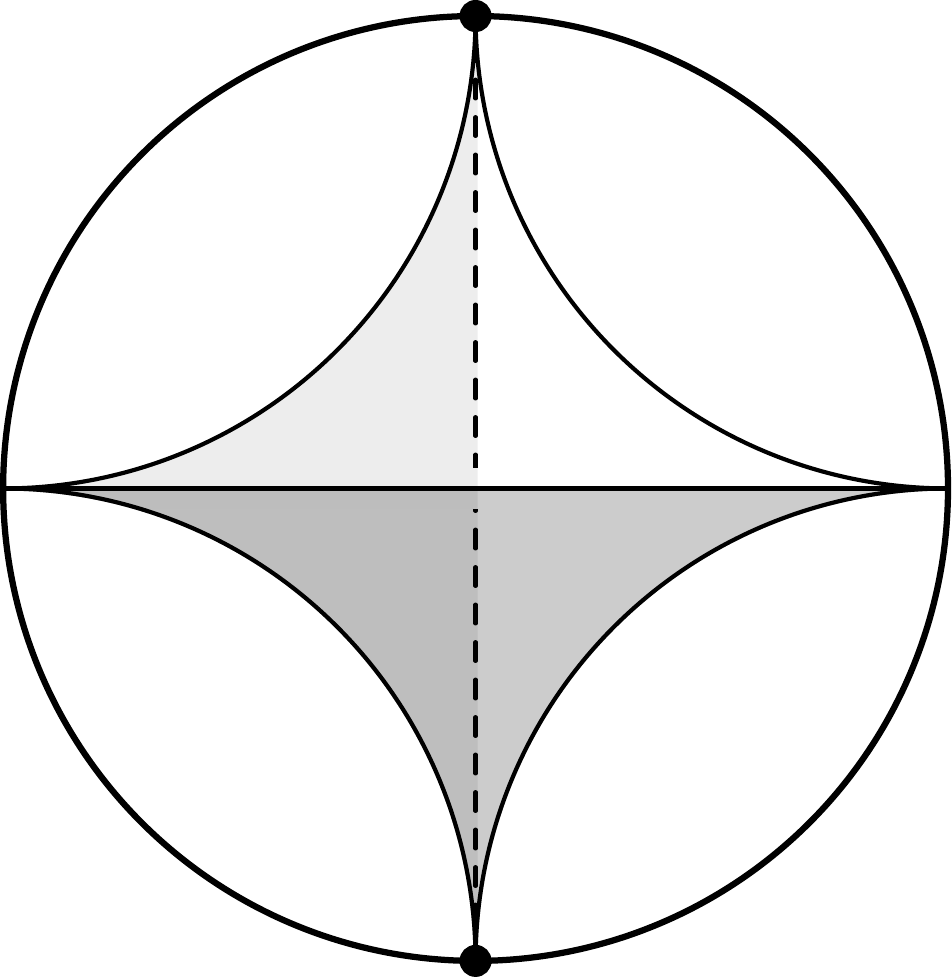}
}

\vspace{10pt}
\subfloat{
\labellist
\small\hair 2pt
\pinlabel $a$ [t] at 229 0
\pinlabel $b$ [b] at 229 468
\endlabellist
\includegraphics[width=0.30\textwidth]{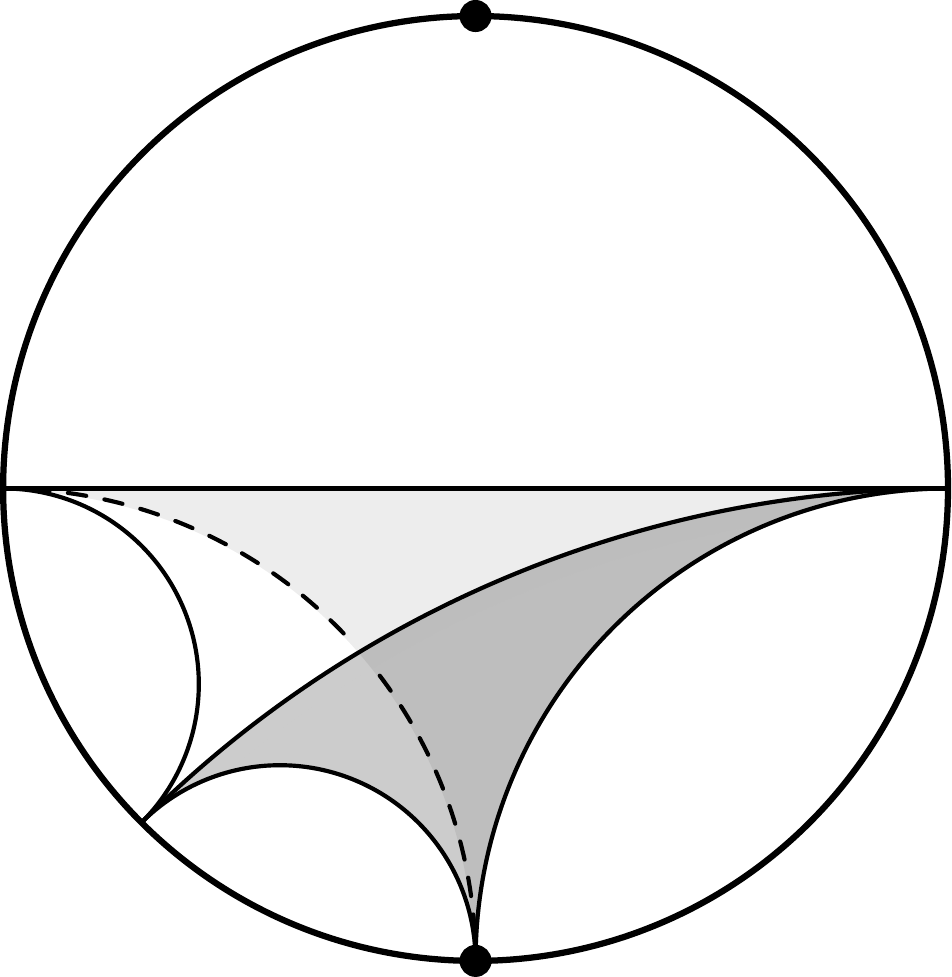}
}
\subfloat{
\labellist
\small\hair 2pt
\pinlabel $a$ [t] at 229 0
\pinlabel $b$ [b] at 229 468
\endlabellist
\includegraphics[width=0.30\textwidth]{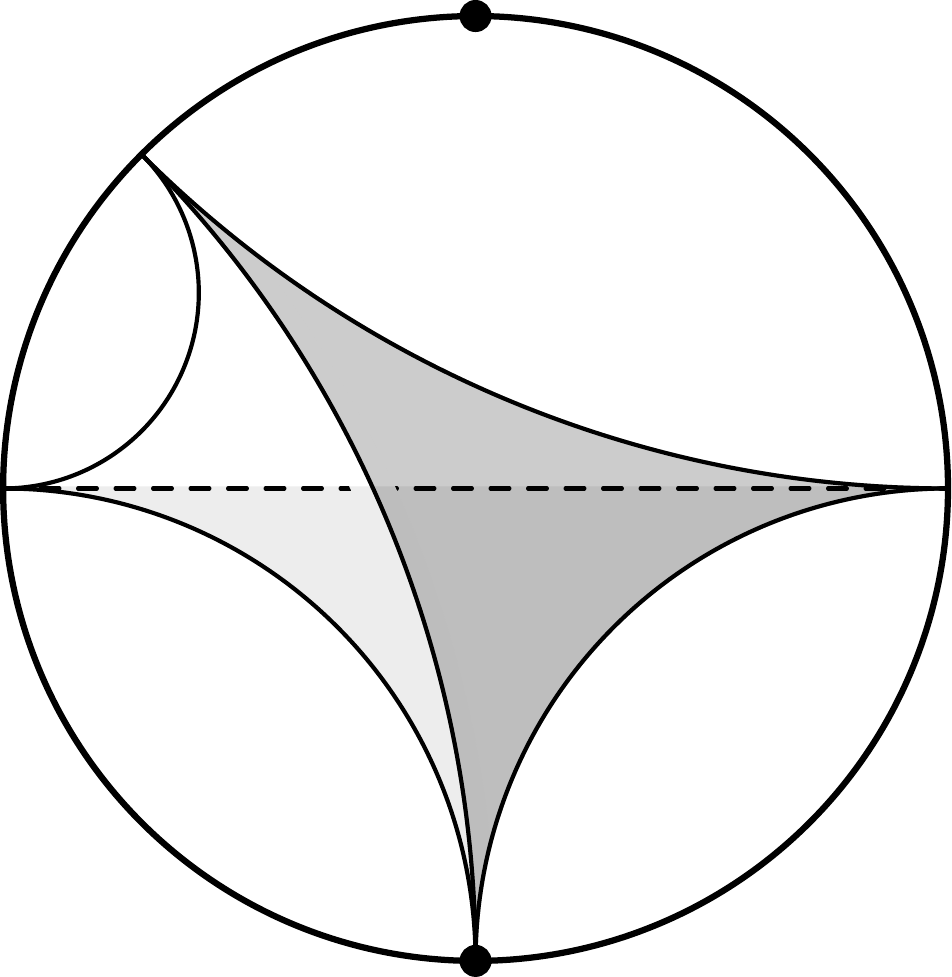}
}
\subfloat{
\labellist
\small\hair 2pt
\pinlabel $a$ [t] at 229 0
\pinlabel $b$ [b] at 229 468
\endlabellist
\includegraphics[width=0.30\textwidth]{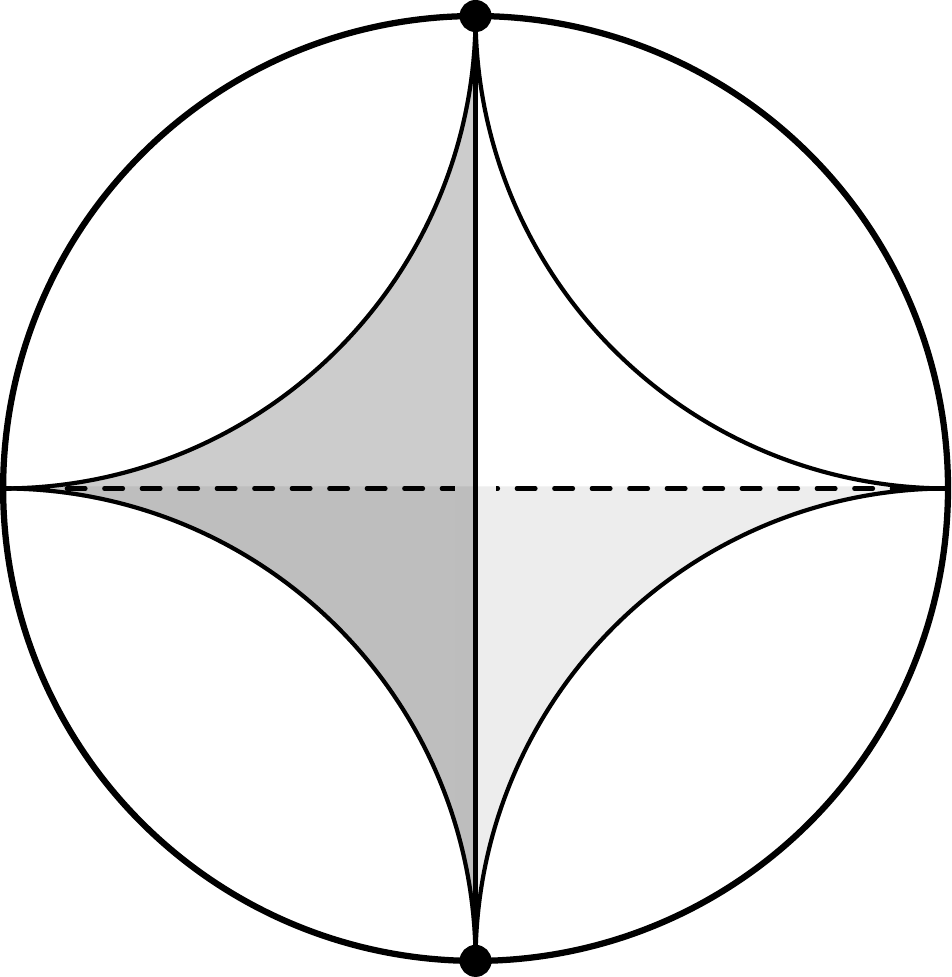}
}
\caption{The six ways, up to symmetry, to pass from $f_k$ (below) to $f_{k+1}$ (above). 
}
\label{Fig:Upward}
\end{figure}

For the base case we take $f_0 = f'$.
For the induction step, given $f_k$, if $f_k$ carries a face of $L_{k+1}$ then we set $f_{k + 1} = f_k$. 
Otherwise, we have that $f_k$ carries a face of $L_{k}$ but no face of $L_{k+1}$.  
Thus there is a tetrahedron $t_k$ above $L_k$, below $L_{k+1}$, and containing $f_k$ as a lower face. 
We deduce that exactly one of the two upper faces of $t_k$ has the desired property; 
we take this face to be $f_{k + 1}$.  
See \reffig{Upward}.


It follows from the inductive hypothesis that $f_n = f$. 
We claim that there are two indices $p < q$ 
so that 
\begin{itemize}
\item
$f_i = f'$ for all $i \leq p$,
\item
$f_i = f$ for all $q \leq i$, and
\item
$f_i$ is an interior face of $C$ for all $p < i < q$. 
\end{itemize}
To prove this, we argue by contradiction.
Suppose that all of the faces $f_k$ meet $e$.
Thus all of the layers $L_k$ meet $e$.  
Since $e$ is an interior edge of $L$ and $L'$, and since no face supports both $L$ and $L'$, we deduce that $e$ separates $C$.
Thus $C$ is not face-connected, a contradiction.

Thus there is some largest $k$ so that $f_k$ does not meet the cusp $b$.
Thus $f_{k+1}$ is an upper face of the tetrahedron below $e$. 
We now consider the collection of faces meeting $e$ and on the same side of $e$ as $f_{k+1}$.  
The transverse orientation of $\cover{B}$ gives these faces an order with $f_{k+1}$ being first.  
All of the faces $f_i$, with $k + 1 \leq i < q$, are interior faces of $C$ so not equal to $f'$.  
Thus $f$ is below $f'$ in the ordering of faces meeting $e$ (on that side).

The sequence of faces $(f_i)$ gives a vertical path $\gamma$ (transverse to $B(C)$) from $f'$ to $f$. 
We combine $\gamma$ with the path of faces around $e$ from $f$ up to $f'$ to obtain a vertical loop.
Since $\cover{M}$ is simply connected, this contradicts \cite[Theorem~3.2]{SchleimerSegerman20}.
\end{proof}

\begin{corollary}
\label{Cor:ThreeBall}
Suppose that $C$ is a continent in $\cover{\calT}$.
Then $C$ is a closed three-ball, minus finitely many points on its boundary corresponding to the cusps of $\Delta_C$. 
Furthermore, the dihedral angle for any edge $e$ in $\bdy C$ is either zero or $\pi$ as $e$ is or is not coastal.\qed
\end{corollary}

We now discuss how to obtain new continents from old. 

\begin{lemma}
\label{Lem:ExtendByTet}
Suppose that $C$ is a continent in $\cover{\calT}$. 
Let $L$ and $L'$ be the upper and lower landscapes of $C$, respectively.
Suppose that $e'$ is a (coastal) sink of $\tau^L$.
Suppose that $t$ is a tetrahedron of $\cover{\calT}$ having $e'$ as its lower edge.
Then $t$ is not in $C$, and $C \cup t$ is a continent.
A similar statement holds, replacing lower with upper and $\tau^L$ with $\tau_{L'}$.
\end{lemma}

\begin{proof}
Let $(L_i)_{i=0}^n$ be the given layering of $C$.
Note that $L = L_n$ and $L' = L_0$.
Let $f$ and $g$ be the upper faces of $t$;
let $f'$ and $g'$ be the lower faces of $t$.
If $e'$ is a sink of $L$ then $L$ meets both of $f'$ and $g'$.
If $e'$ is a coastal sink of $L$, then $L$ meets only one, say $f'$.
Since $f' \subset \bdy C$, we deduce that $t$ is not contained in $C$.

Let $D = C \cup t$; 
note that $D$ is face-connected. 

There are now two cases, as $e'$ is a sink or a coastal sink.
\begin{case}
Suppose that $e'$ is a coastal sink.
Then $e'$ is a coastal edge of each layer $L_i$.
Note that $f'$ and $g'$ are on opposite sides of $e'$.
However, $f'$ carries $L_n$ and $g'$ does not.
Thus $K_i = L_i \cup g'$ is a landscape for $i \leq n$.
We also define $K_{n+1} = (L_n - f') \cup f \cup g$.
Note that $K_{n+1}$ is obtained from $K_n$ by flipping across $t$.
Thus $K_{n+1}$ is a landscape and $(K_i)$ is a layering of $D$.
\end{case}

\begin{case}
Suppose that $e'$ is a sink.
Define $K_i = L_i$ for $i \leq n$.
We also define $K_{n+1} = (L_n - (f' \cup g')) \cup f \cup g$.
Then $(K_i)$ is a layering of $D$. \qedhere
\end{case}
\end{proof}

\subsection{Continental exhaustions}

\begin{definition}
\label{Def:ContinentalExhaustion}
A \emph{continental exhaustion} of $\cover{\calT}$ is a sequence of continents $(C_n)_{n \in \NN}$ so that 
\begin{itemize}
\item
$C_n \subset C_{n+1}$ and
\item
$\cover{\calT} = \bigcup C_n$. \qedhere
\end{itemize}
\end{definition}

\begin{lemma}
\label{Lem:OneAtATime}
Any continental exhaustion can be refined so that $C_{n+1}$ is obtained from $C_n$ by adding exactly one tetrahedron.
\end{lemma}

\begin{proof}
Suppose that $C \subsetneq C'$ are continents.
Let $L$ be the upper landscape of $C$.
Suppose (as the other case is similar) that there is a tetrahedron $t$ of $C' - C$ attached to $C$ along a face $f$ of $L$.
Let $R$ be the maximal upper river in $L$ starting at $f$.
Let $(f_i)_{i = 1}^n$ be the faces of $R$, with $f_i$ flowing into $f_{i+1}$.
Let $(e_i)_{i = 1}^n$ be the edges of $R$ where $f_i$ flows into $e_i$.
By \reflem{River}, the final edge $e_n$ is a (coastal) sink for $\tau^L$.

We now induct on $n$, the length of $R$.
In the base case $n = 1$ and $e_1$ is a sink.
We set $t_1 = t$ and note that $e_1$ is the lower edge of $t_1$. 
So, by \reflem{ExtendByTet}, the union $C \cup t$ is a continent.

In the induction step we have a tetrahedron $t_k$ of $C'$ which is attached to $C$ along $f_k$.
Also, $e_k$ is the lower edge of $t_k$. 
If $e_k$ is a sink then we proceed as in the base case.
Suppose instead that $e_k$ is a fall.
Since $C'$ is a continent, and applying \refcor{ThreeBall}, the dihedral angle at $e_k$ is not $2\pi$.
Thus gives a tetrahedron $t_{k+1}$ of $C'$, attached to $C$ along $f_{k+1}$, with lower edge $e_{k+1}$.
\end{proof}

Applying a result of Brown~\cite{Brown61} we have the following.

\begin{lemma}
\label{Lem:ExhaustImpliesThreeSpace}
Suppose that $\calT$ is a taut ideal triangulation and that $\cover{\calT}$ admits a continental exhaustion. 
Then the realisation $|\cover{\calT}|$ is homeomorphic to $\RR^3$. \qed
\end{lemma}

\begin{proposition}
\label{Prop:ExhaustImpliesLayered}
Suppose that $\calT$ is a transverse taut ideal triangulation.
Suppose that $\cover{\calT}$ admits a continental exhaustion. 
Then $\cover{\calT}$ is layered. 
\end{proposition}

\begin{proof}
Suppose that $\cover{\calT}$ has a continental exhaustion $(C_n)$.  
By \reflem{OneAtATime} we may assume that $C_n$ is a union of exactly $n$ tetrahedra.  

The continents $C_n$ are layered by definition.
However, these layerings need not be compatible with each other.
We re-choose the layerings as follows.

Our induction hypothesis is that $C_n$ is layered by $\calK_n = (K_i^n)_{i = p(n)}^{q(n)}$, 
a sequence of $n+1$ landscapes.
That is,
\begin{itemize}
\item
$K^n_{p(n)}$ is the lower landscape for $C_n$,
\item
$K^n_{q(n)}$ is the upper landscape for $C_n$, and
\item
$K^n_i$ and $K^n_{i+1}$ cobound a single tetrahedron of $C_n$.
\end{itemize}
Thus $q(n) - p(n) = n$.
We further assume that for any $m < n$ and for any $i$ between $p(m)$ and $q(m)$ we have
\begin{itemize}
\item
$K^m_i \subset K^n_i$.
\end{itemize}

Suppose now that $C_{n+1}$ is obtained from $C_n$ by attaching the tetrahedron $t$.
As the other case is similar we assume that $t$ is attached to the upper landscape of $C_n$.
There are now two cases: 
either $C_{n+1}$ has the same cusps as $C_n$ or $C_{n+1}$ has one more cusp than $C_n$. 

Suppose that $C_{n+1}$ and $C_n$ have the same cusps.  
Recalling that $t$ is attached above $C_n$,
we set $p(n+1) = p(n)$ and $q(n+1) = q(n) + 1$, we set $K^{n+1}_{q(n+1)}$ to be the upper landscape of $C_{n+1}$, and we set $K^{n+1}_i = K^n_i$ for all $i$ between $p(n)$ and $q(n)$.  

Suppose that $C_{n+1}$ has one more cusp than $C_n$.
Then there is a unique lower face $f$, of $t$, which meets the new cusp.  
Note that $f$ is not contained in $C_n$ and meets exactly one coastal edge of $C_n$.  
For $i$ between $p(n)$ and $q(n)$ we take $K^{n+1}_i = K^n_i \cup f$ and note that $K^{n+1}_i$ lies in, and contains the coastal edges of, $C_{n+1}$.  
Note that $t$ is layered onto $K^{n+1}_{q(n)}$.  
So set $p(n+1) = p(n)$, set $q(n+1) = q(n) + 1$, and define $K^{n+1}_{q(n+1)}$ to be the upper landscape for $C_{n+1}$.  

Thus $\calK_{n+1} = (K^{n+1}_i)_{i = p(n+1)}^{q(n+1)}$ is the desired layering of $C_{n+1}$.  
Note that, in both cases, we have $K^n_i \subset K^{n+1}_i$.

Finally, for any $i \in \ZZ$ we define $K_i = \bigcup_{n \in \NN} K^n_i$.
Note that $K_i$ has no boundary edges and that $K_{i+1}$ is obtained from $K_i$ by a single in-fill.  
Thus $(K_i)_{i \in \ZZ}$ is the desired layering of $\cover{\calT}$. 
\end{proof}

\begin{remark}
\label{Rem:LayeredImpliesExhaust}
The converse of \refprop{ExhaustImpliesLayered} also holds;
the proof is similar to, but more difficult than, that of \reflem{LayeredImpliesLayered}.  
The converse is not needed for \refthm{VeerImpliesUnique}, so we omit it. 
\end{remark}





\begin{remark}
\label{Rem:BadExample}
Note that \refthm{Exotic} and \reflem{LayeredImpliesUnique} imply that the (universal covers of the) exotic triangulations of~\cite[Proposition 6.8]{FuterGueritaud13} do not have layerings.
Thus by \refprop{ExhaustImpliesLayered} they do not have continental exhaustions.
\end{remark}

\chapter{Veering triangulations}
\label{Cha:Veering}

We now give the definition of \emph{veering}, following~\cite{Agol11, HRST11}.
Recall that the three-manifold $M$ is equipped with an orientation. 

\begin{definition}
\label{Def:Veering}
A \emph{veering triangulation} $\calV$ is a taut ideal triangulation $\calT$ equipped with an edge colouring.  
Each edge is coloured red or blue as follows. 
Suppose that $t$ is a model taut tetrahedron and $f \subset t$ is a face.  
Suppose $e$, $e'$, and $e''$ are the edges of $f$, ordered anticlockwise as viewed from the outside of $t$, 
with $e$ having dihedral angle $\pi$ inside of $t$.  
Then $e'$ is red and $e''$ is blue.  
See \reffig{VeeringTet} for a model veering tetrahedron. 
\end{definition}


If $\calV$ is a veering triangulation whose taut angle structure is transverse then we will say that $\calV$ is a \emph{transverse veering structure}.

\begin{figure}[htbp]
\labellist
\small\hair 2pt
\pinlabel 0 [tr] at 29 29
\pinlabel 0 [tl] at 85 29
\pinlabel 0 [br] at 29 85
\pinlabel 0 [bl] at 85 85
\pinlabel $\pi$ [t] at 80 55
\pinlabel $\pi$ [r] at 55 80
\endlabellist
\includegraphics[width=0.3\textwidth]{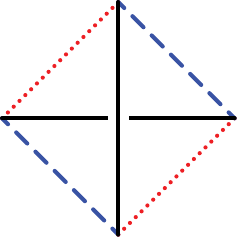}
\caption{A model veering tetrahedron.
The dotted edges are red;
the dashed edges are blue.
The two edges with angle $\pi$ may have either colour.
See \reffig{UpperGluingAutomaton} for the four possible veering tetrahedra.}
\label{Fig:VeeringTet}
\end{figure}

We can now state the main result of this section.  

\begin{theorem}
\label{Thm:VeerImpliesUnique}
Suppose that $M$ is an oriented three-manifold equipped with a transverse veering triangulation $\calV$. 
Then there is a unique compatible circular order $\calO_\calV$ on the cusps of $\cover{M}$.  
Furthermore, $\calO_\calV$ is dense and $\pi_1(M)$--invariant.
\end{theorem}

Given the work we have already done, to prove \refthm{VeerImpliesUnique} it suffices to prove the following.

\begin{proposition}
\label{Prop:VeerImpliesExhaust}
Suppose that $M$ is a three-manifold equipped with a transverse veering triangulation $\calV$. 
Then $\cover{\calV}$ admits a continental exhaustion.
\end{proposition}

\refprop{VeerImpliesExhaust} and \refprop{ExhaustImpliesLayered} give the following. 

\begin{corollary}
\label{Cor:VeerImpliesLayered}
Suppose that $M$ is a three-manifold equipped with a transverse veering triangulation $\calV$. 
Then $\cover{\calV}$ admits a layering.  \qed
\end{corollary}

\begin{proof}[Proof of \refthm{VeerImpliesUnique}]
\refcor{VeerImpliesLayered} tells us that $\cover{\calV}$ admits a layering.
\reflem{LayeredImpliesUnique} now gives the desired unique, compatible, dense, $\pi_1(M)$--invariant circular order. 
\end{proof}

\subsection{Combinatorics of veering triangulations}

We begin by setting down some of the combinatorics of veering triangulations.
Suppose that $M$ is a three-manifold equipped with a transverse veering triangulation $\calV$.
Suppose that $f$ is a face of $\cover{\calV}$.
So $f$ has at least one blue edge and at least one red edge.
Ordering the cusps of $f$ as in \refsec{Compatibility} -- that is, looking from above, in anticlockwise order -- there is a unique cusp $u(f)$ of $f$ where the colours switch from blue to red.
Consulting \reffig{VeeringTet} we see that the branches of the upper track $\tau^f$ flow away from $u(f)$.
(Swapping colours gives a corresponding statement for $\tau_f$.)
This implies that, in the presence of a veering structure, our upper track is the same as the train track used by Agol~\cite[Main~Construction]{Agol11} to define veering triangulations of mapping tori.
See \reffig{VeeringTriangles}.

\begin{figure}[htb]
\centering
\subfloat[Upper tracks.]{
\includegraphics[width=0.44\textwidth]{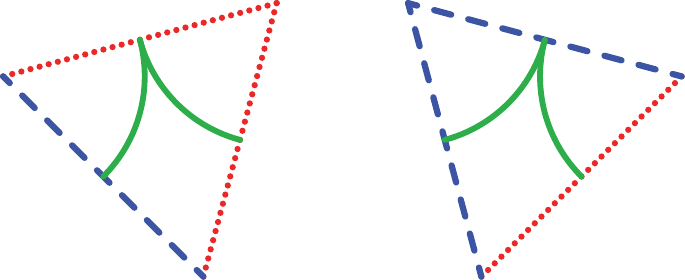}
\label{Fig:VeeringTrianglesUpperTrack}
}
\qquad
\subfloat[Lower tracks.]{
\includegraphics[width=0.44\textwidth]{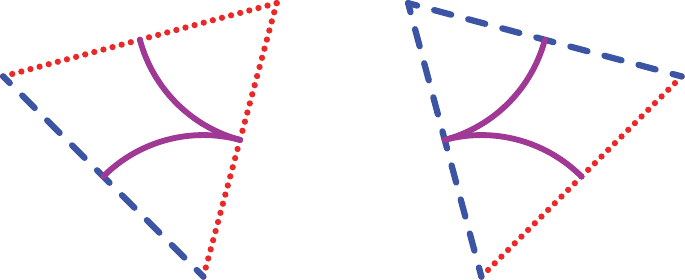}
\label{Fig:VeeringTrianglesLowerTrack}
}
\caption{The two possible faces in a veering triangulation, as viewed from above.  We draw the upper track in green, and the lower track in purple.
Our colour scheme is chosen so that the following holds.  
Suppose that $L$ is a landscape.  
In $L$ we walk anticlockwise about a cusp (as viewed from above).  
As we do so we pass, cyclically, a sequence of red edges, then an upper track cusp (green), then a sequence of blue edges, and then a lower track cusp (purple). 
This follows the standard colour wheel ordering. 
}
\label{Fig:VeeringTriangles}
\end{figure}

\begin{definition}
\label{Def:FanToggle}
Suppose that $t$ is a model veering tetrahedron.
\begin{itemize}
\item
If $t$ has three red and three blue edges then we call $t$ a \emph{toggle tetrahedron}.  
\item
If $t$ has more red than blue edges (more blue than red edges) then we call $t$ a red (blue) \emph{fan tetrahedron}.   \qedhere
\end{itemize}
\end{definition}

The four possible model veering tetrahedra are shown in \reffig{GluingAutomaton}; 
we also show all face gluings that respect the veering structure.
In \reffig{EdgeNeighbourhood} we show one possibility for the tetrahedra on the two sides of an edge. 

\begin{figure}[htb]
\centering
\subfloat[Upper tracks.]{
\includegraphics[width=0.47\textwidth]{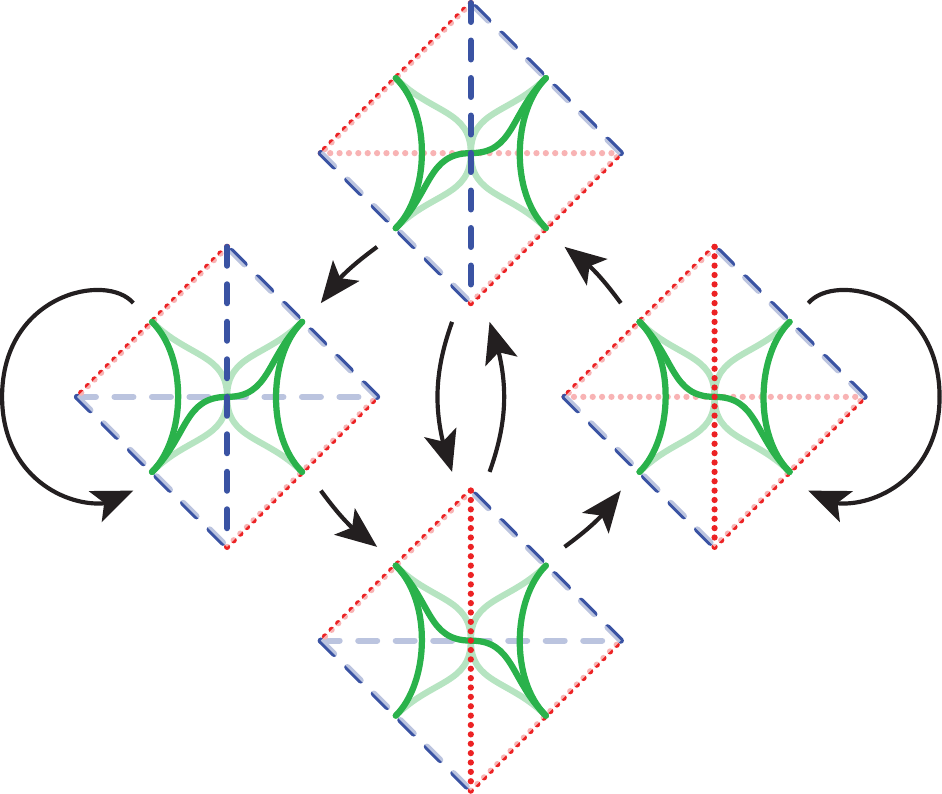}
\label{Fig:UpperGluingAutomaton}
}
\thinspace
\subfloat[Lower tracks.]{
\includegraphics[width=0.47\textwidth]{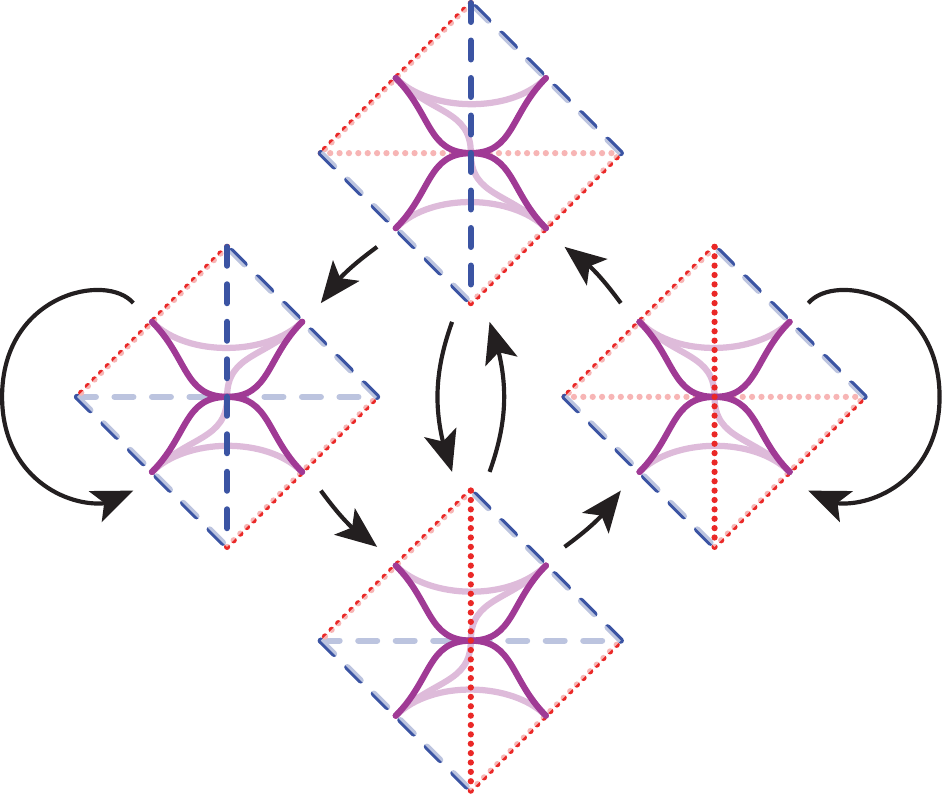}
\label{Fig:LowerGluingAutomaton}
}
\caption{In each diagram: moving anticlockwise from the top we have a toggle tetrahedron (blue on top), a blue fan tetrahedron, the other toggle tetrahedron (red on top), and a red fan tetrahedron.  An arrow points from one tetrahedron to another if the second tetrahedron can be glued on top of the first.  In \reffig{UpperGluingAutomaton} (\reffig{LowerGluingAutomaton}) we draw in each face the corresponding upper (lower) track; the tracks on the lower faces are more faintly drawn.}
\label{Fig:GluingAutomaton}
\end{figure}

\begin{figure}[htbp]
\labellist
\small\hair 2pt
\pinlabel $e$ [tr] at 165 165
\endlabellist
\includegraphics[width=0.35\textwidth]{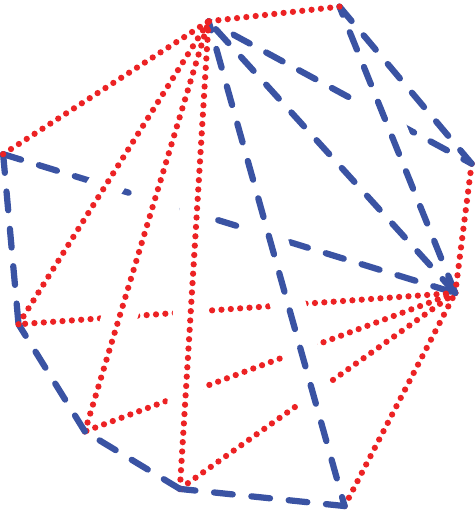}
\caption{One possible neighbourhood of a blue edge, $e$.
To the upper right of $e$ there is a single blue fan tetrahedron.
To its lower left there are two toggle tetrahedra, one of each type, and a stack of two red fan tetrahedra.
We have not drawn the tetrahedron above $e$ (in front of the page) or the one below $e$ (behind the page).
See also~\cite[Figure~12]{Agol11}.}
\label{Fig:EdgeNeighbourhood}
\end{figure}

\begin{remark}
In other literature the toggle tetrahedra are called \emph{hinge tetrahedra}; 
see~\cite[page~1247]{Gueritaud06} or~\cite[page~211]{FuterGueritaud13}. 
\end{remark}

We record, in \reffig{PossibleTwoTriangles}, the upper tracks in all possible two-triangle veering landscapes.
Note that watersheds and sinks fit together to give the boundaries of veering tetrahedra.  

\begin{figure}[htbp]
\subfloat[A left watershed.]{
\includegraphics[width=0.28\textwidth]{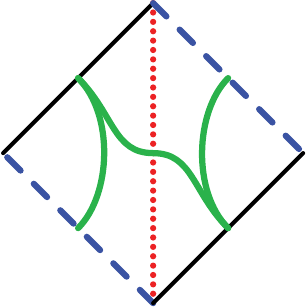}
\label{Fig:LeftWatershed}
}
\quad
\subfloat[A right watershed.]{
\includegraphics[width=0.28\textwidth]{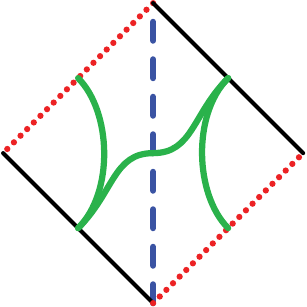}
\label{Fig:RightWatershed}
}

\subfloat[A left fall.]{
\includegraphics[width=0.28\textwidth]{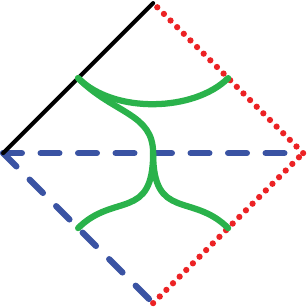}
\label{Fig:LeftFall}
}
\quad
\subfloat[A right fall.]{
\includegraphics[width=0.28\textwidth]{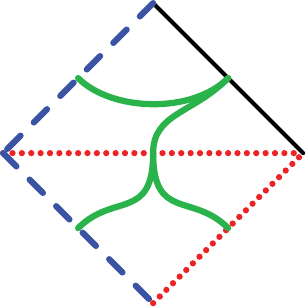}
\label{Fig:RightFall}
}
\quad
\subfloat[A sink.]{
\includegraphics[width=0.28\textwidth]{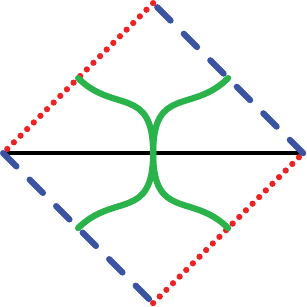}
\label{Fig:Sink}
}
\caption{Upper tracks on all possible two-triangle veering landscapes.
Black edges can be either red or blue.}
\label{Fig:PossibleTwoTriangles}
\end{figure}

\begin{definition}
\label{Def:Riverbanks}
Suppose that $R$ is an upper river in $\cover{\calV}$.  
The \emph{riverbank edges} of $R$ are the boundary edges of $R$, other than the mouth.
\end{definition}

From Figures~\ref{Fig:LeftFall} and~~\ref{Fig:RightFall} 
we deduce that the edges on the left riverbank of $R$ are blue, while those on the right are red. 

Veering triangulations are far more rigid than taut ideal triangulations.
In particular, the veering hypothesis strongly constrains the combinatorics about an edge, as follows.

\begin{lemma}
\label{Lem:EdgeNeighbourhood}
Suppose that $M$ is a three-manifold equipped with a transverse veering triangulation $\calV$. 
Suppose that $e$ is an oriented blue edge in $\cover{\calV}$.  Then the following hold.  
\begin{enumerate}
\item
There is exactly one tetrahedron above, and exactly one below,~$e$.  
\item
There is at least one tetrahedron to the right of $e$, and similarly to its left.  
\label{Itm:HalfEdgeDegreeAtLeastOne}
\item
If there is exactly one tetrahedron to the right of $e$, then it is a blue fan tetrahedron.  
The same holds to the left of $e$. 
\item
If there are at least two tetrahedra to the right of $e$, then the resulting stack of tetrahedra contains, from bottom to top, a red-on-top toggle, some number of red fans, and finally a blue-on-top toggle.  
The same holds to the left of $e$. 
\item
There are exactly four faces meeting $e$ that have more blue edges than red.
\end{enumerate}
Similar statements hold when $e$ is red instead of blue. \qed
\end{lemma}

\subsection{Convex continents}

Here we give two simple ways to enlarge a given continent. 

\begin{lemma}
\label{Lem:NoNewSinks}
Suppose that $M$ is a three-manifold equipped with a transverse veering triangulation $\calV$.
Suppose that $C \subset \cover{\calV}$ is a continent.
Suppose that the lower track on the bottom of $C$ has no sinks.
Then this again holds after a single coastal landfill on top of $C$.
The corresponding statement holds after switching upper and lower as well as top and bottom.
\end{lemma}

\begin{proof}
Suppose that $t$ is the given coastal landfill tetrahedron above $C$.  
Suppose $e$ is the lower $\pi$--edge of $t$.  
Let $f$ be the free lower face of $t$.  
Suppose that $t'$ is the tetrahedron having $e$ as its upper $\pi$--edge.  
From \reflem{EdgeNeighbourhood}\refitm{HalfEdgeDegreeAtLeastOne} we deduce that $t'$ cannot be glued to $t$ along $f$. 
Thus the lower track in the lower boundary of $C \cup t$ again has no sinks.
\end{proof}

\begin{definition}
\label{Def:Convex}
Suppose that $C \subset \cover{\calV}$ is a continent with upper and lower landscapes $L$ and $L'$.  
We say that $C$ is \emph{convex} if neither $\tau^L$ nor $\tau_{L'}$ have sinks. 
\end{definition}

For example, a continent containing only one tetrahedron is convex. 

\begin{lemma}
\label{Lem:FiniteInFill}
Suppose that $M$ is a three-manifold equipped with a transverse veering triangulation $\calV$. 
Suppose that $C \subset \cover{\calV}$ is a continent.
Then there exists a convex continent $C'$ which contains $C$ and has the same set of cusps as $C$.
\end{lemma}

In fact, $C'$ is unique; 
this follows from \refthm{LinkIsLoom} and \cite[Proposition~6.46]{SchleimerSegerman24}. 

\begin{proof}[Proof of \reflem{FiniteInFill}]
Let $L$ and $L'$ be the upper and lower landscapes of $C$.  
Note that an in-fill above $C$ leaves $L'$ unchanged.  
Thus we can argue separately for the top and bottom of $C$. 

Set $C_0 = C$.  
Now suppose that we have built from $C_0$ a continent $C_k$ via a finite sequence of in-fills. 
Let $L_k$ be the upper landscape of $C_k$; 
let $\tau^k$ be the upper track for $L_k$.  
If $\tau^k$ has no sinks we are done; 
otherwise we build $C_{k+1}$ from $C_k$ by a single in-fill above.

Suppose that $t_k = t$ is the given in-fill tetrahedron above $L_k$.
By \reflem{Disk}, the landscape $L_k$ is a triangulated disk.  
Let $\tau^k$ be the upper track for $L_k$; 
thus $\tau^k$ is a tree, with all leaves on coastal edges.  
The track $\tau^{k+1}$ is obtained from $\tau^k$ by doing a single split.  
See \reffig{UpperGluingAutomaton}.  
This gives a one-to-one correspondence between the track-cusps of $\tau^k$ and those of $\tau^{k+1}$. 

A track-cusp $s$ of $\tau^k$ meets some edge $e$ of $L_k$;
the flows behind $s$ give $e$ a co-orientation in $L_k$.
The orientation divides the coastal edges into two collections: 
those in front of $e$ and those behind.  
When we attach $t_k$, the track-cusp $s$ moves forward to a new track-cusp $s'$ and a new edge $e'$, both in $L_{k+1}$. 
Again, see \reffig{UpperGluingAutomaton}.  
Recall that $L_{k+1}$ and $L_k$ have the same coastal edges.
Now, at least one coastal edge goes from being in front of $e$ to being behind $e'$, and none move forward.  
Any such sequence $s, s', s'', \ldots$ of track-cusps eventually reaches the coast or otherwise halts.

Let $n = |L_0|$ be the number of triangles in $L_0$, and thus in $L_k$.
Thus any sequence of track-cusps $s, s', s'', \ldots$ as above has length at most $n - 1$.
Also, the number of track-cusps of $\tau^0$ equals $n$.
Thus after adding at most $n(n-1)/2$ in-fill tetrahedra we arrive at an upper track $\tau^k$ without sinks.
\end{proof}

\subsection{Complexity and channelisation}

Our aim is to build a continental exhaustion for the universal cover of a veering triangulation. 
Given a continent $C$ and a tetrahedron $t$, we must enlarge $C$ to include $t$. 
For this we wish to use \reflem{ExtendByTet}. 
However this only permits us to add tetrahedra in certain locations; 
by \reflem{River} these are found by following rivers downstream to (coastal) sinks. 
Note that \refthm{Exotic} shows that in the absence of the veering hypothesis, 
this procedure may never arrive at $t$.

To show that the veering hypothesis suffices, 
we require two more subtle definitions to control the rivers that appear in the argument. 

Suppose that $R$ is an upper river. 
We orient $R$ in the direction of the flows.  
Let $e_i$ be the $i^\thsup$ fall, counting from the source. 
We define $h_i$ to be the \emph{height} of the fall $e_i$; 
that is, $h_i$ is the degree of $e_i$ in $\cover{\calV}$ minus the degree of $e_i$ below the river $R$.
Recall that $\ell(R)$, the length of $R$, is the number of triangles in $R$. 

\begin{definition}
\label{Def:Complexity}
The \emph{complexity} of a river $R$ is the list 
\[
c(R) = (\ell(R), h_1, h_2, \ldots, h_{\ell-1})
\]
We order complexities lexicographically.
\end{definition}

\begin{figure}[htb]
\centering
\labellist
\small\hair 2pt
\pinlabel $f$ [r] at 367 1197
\pinlabel $(0)$ at 310 1145
\pinlabel $(1)$ at 15 940
\pinlabel $(1')$ at 1190 940
\pinlabel $(2')$ at 15 760
\pinlabel $(2)$ at 1190 760
\pinlabel $(3')$ at 15 577
\pinlabel $(3)$ at 1190 577
\pinlabel $(8)$ at 15 207
\pinlabel $(8')$ at 1190 207
\pinlabel (A) at 15 27
\pinlabel (B) at 1190 27
\endlabellist
\includegraphics[width=\textwidth]{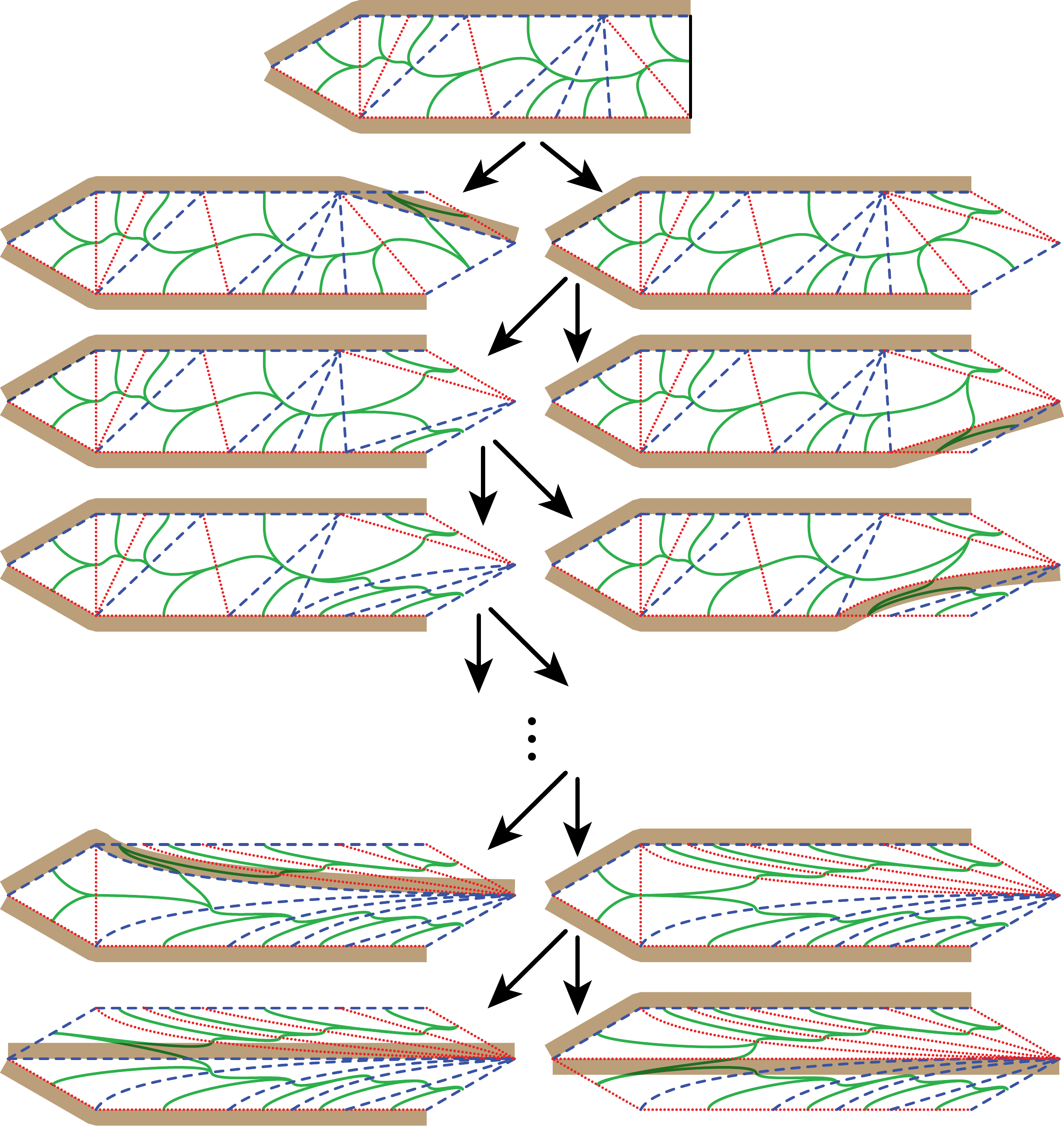}
\caption{Some possible steps in a channelisation process. 
The figures $(0)$, $(1)$, $\ldots$, $(8)$, (A), and (B) all show rivers: 
$(0)$ is the initial river and the others arise from landfilling.  
The figures $(1')$, $(2')$, $\ldots$ , $(8')$ are forked rivers.  
The primed indices count the total number of triangles in the two distributaries after the sink.  
The riverbanks are shaded brown.  
We place a subfigure on the left or right according to the colour of the upper edge of the most recently added tetrahedron.}
\label{Fig:River}
\end{figure}

\begin{definition}
\label{Def:Channelisation}
Suppose $C$ is a convex continent; thus all rivers in $\bdy C$ flow to the coast.  
Suppose that $f$ is a face of the upper landscape of $C$. 
Let $R$ be the maximal river with $f$ as its source.  
Let $f'$ be the last triangle of $R$ and let $e'$ be the mouth of $R$. 
See \reffig{River}(0) for one possible example. 
Let $t'$ be the tetrahedron above $f'$. 
Note that $e'$ is the lower edge of $t'$. 

We now landfill $t'$ onto $C$.
By \reflem{NoNewSinks}, the union $C \cup t'$ is again a continent, with no new sinks in its lower boundary.
We now appeal to \reflem{FiniteInFill}: after in-filling the top of $C \cup t'$ a finite number of times we obtain a convex continent $C'$.
We say that $C'$ is a \emph{channelisation} of $(C, f)$.
We make similar definitions for the lower landscape of $C$. 
\end{definition}


In general, a single channelisation process will not cover the face $f$. 
The heart of the proof of \refprop{VeerImpliesExhaust} is in showing that a finite number of channelisations suffices to cover $f$.

\begin{proof}[Proof of \refprop{VeerImpliesExhaust}]
We must build a continental exhaustion.  We begin by choosing an enumeration $(t_m)_{m \in \NN}$ of the tetrahedra of $\cover{\calV}$.  
Our first continent is $C_0 = t_0$; note that $C_0$ is convex.  Now suppose that we have obtained a convex continent $C_n$.  Let $t_{m(n)}$ be the smallest tetrahedron in our enumeration which is not contained in $C_n$ and which meets $C_n$ along a face $f$.  It will suffice to find a convex continent $C_{n+1}$ containing $C_n$ and $t_{m(n)}$. (For suppose that $t'$ is any tetrahedron.  Let $(t_{p(k)})_{k = 0}^K$ be a sequence of tetrahedra so that $t_{p(0)} = t_0$, so that $t_{p(K)} = t'$, and the tetrahedra $t_{p(k)}$ and $t_{p(k+1)}$ share a face.  Induction proves that $t_{p(k)}$ belongs to $C_{p(k)}$.)

Breaking the symmetry of the situation, we assume that $t_{m(n)}$ is above $C_n$.  
Thus $f$ is a face of the upper landscape of $C_n$.  
Set $C_n^0 = C_n$.  
Suppose that we have constructed $C_n^k$, convex.  
If $t_{m(n)}$ is contained in $C_n^k$ then we set $C_{n+1} = C_n^k$ and we are done.  
If not, then $f$ is still a face of the upper landscape of $C_n^k$.  
Then we define $C_n^{k+1}$ to be the result of channelising the pair $(C_n^k, f)$.  
Note that $C_n^{k+1}$ is again convex.  
Let $R^k$ be the maximal river in the upper landscape of $C_n^k$ with $f$ as its source.

\begin{lemma}
\label{Lem:RiversSimplify}
Suppose that $f$ is a face of the upper landscapes of both $C_n^k$ and $C_n^{k+1}$.  
Then $c(R^{k+1})$ is strictly smaller than $c(R^k)$.  
\end{lemma}

\noindent
Note that \reflem{RiversSimplify} completes the proof of \refprop{VeerImpliesExhaust}; 
since complexities are well-ordered, eventually $f$ is not a face of some upper landscape.  
At that stage $t_{m(n)}$ is contained in the continent, as desired. 

Before giving the proof of \reflem{RiversSimplify} we require an expanded notion of a river. We give this in the next definition.  We also give an extended example to illustrate how the proof works. 

\begin{definition}
\label{Def:Forked}
Suppose that $C$ is a continent with upper landscape $L$.  
Let $f$ be a face of $L$.  
We define the \emph{forked river} $R(C, f)$ as follows. 

Let $R^f$ be the maximal upper river in $L$ with source $f$.  
If $R^f$ ends at the coast then we take $R(C, f) = R^f$; 
in this case the forked river is simply a river.
Suppose instead that $R^f$ ends at a sink $e$ for $\tau^L$.
Let $f_0$ be the triangle in $L - R^f$ adjacent to $e$.
Let $e_0$ and $e_1$ be the edges of $f_0 - e$.  
Let $R^{e_0}$ and $R^{e_1}$ be the maximal rivers starting at $e_0$ and $e_1$, respectively, and flowing away from $f_0$.  
These are the two \emph{distributaries} flowing from $f_0$; they may contain zero triangles.  
Finally, the desired forked river is the landscape 
\[
R(C, f) = R^f \cup f_0 \cup R^{e_0} \cup R^{e_1}
\]
If $R^{e_i}$ contains no triangles then we take the mouth of $R^{e_i}$ to be $e_i$.   

Suppose that $R$ is a forked river.
The \emph{length} of $R$, denoted $\ell(R)$, is the number of triangles it contains.
If $R$ contains a sink, then we say that the \emph{mouths} of $R$ are the mouths of its distributaries.
The \emph{riverbank edges} of $R$ are the boundary edges of $R$ other than the mouth(s).
\end{definition}

For examples of forked rivers, with distributaries of various sizes (including zero), see subfigures $(1')$, $(2')$, $(3')$, and $(8')$ of \reffig{River}.  

\begin{example}
\label{Exa:RiversSimplify}
We now discuss channelisation, illustrated in \reffig{River}, in more detail.  
Subfigure~\ref{Fig:River}(0) shows a possible river $R^k$, with source face $f$.  
Suppose that $t'$ is the coastal landfill tetrahedron.  
Note that attaching $t'$ fills the mouth of $R^k$ and adds one new cusp and two new coastal edges.  
After the landfill there are two possibilities according to the colour, blue or red, of the top edge of $t'$.  
These are shown immediately below $(0)$ on the left and right in subfigures $(1)$ and $(1')$, respectively.

In $(1)$ we see a river $R^{k+1}$.  
This is isomorphic (as a landscape) to $R^k$; however the last fall of $R^{k+1}$ is less high.  
Thus, following \refdef{Complexity}, we have $c(R^{k+1}) < c(R^k)$.  
In the remainder of the $(k+1)^\stsup$ channelisation procedure we in-fill all sinks on the top of $C_n^k \cup t'$.
There are none in $R^{k+1}$.  
There are none outside of the subfigure (1) because $C_n^k$ was convex.  
Thus the only possible in-fill is at the track-cusp pointing out of the top of the subfigure (1).  
This may lead to further sinks, which we in-fill in their turn.  
However, these in-fills do not alter $R^{k+1}$ because the only place to fill $R^{k+1}$ is at the coast.  

In $(1')$ we have a forked river $R^k_1$ as in \refdef{Forked}.  
Note that $R^k_1$ has one more coastal edge than the original river $R^k$ but has the \emph{same} riverbanks.  
Thus the only in-fill on the top (or bottom) of the current continent is the unique sink in $R^k_1$.  
In $(1')$ the distributary on the left riverbank has length one while the distributary to the right has length zero.  

When we in-fill the sink of $R^k_1$, there are two possibilities, shown in $(2')$ and $(2)$, as the top edge of the landfill tetrahedron is blue or red.  
In example $(2)$ we obtain a river $R^{k+1}$ which has been routed along the distributary (to the left) in $R^k_1$.  
Again we claim that $c(R^{k+1}) < c(R^k)$.  
Firstly, in-fills disjoint from $R^{k+1}$ do not alter it.  
Also, while $R^k$ and $R^{k+1}$ have the same length, the height of the second-to-last fall of $R^{k+1}$ is smaller than that of $R^k$.  

In example $(2')$ we obtain a forked river $R^k_2$.  
Note that the sink is further upstream while the total length of the distributaries has increased by one.  
Again the only in-fill on the boundary of the continent is the sink of $R^k_2$.  
At that point we will either obtain $(3)$, a river with lower complexity, or $(3')$, a forked river where the sink has moved upstream and one of the distributaries is one triangle larger.

Continuing in this way, if we do not make a river, then eventually the two distributaries contain all of the riverbank edges.  This is shown in the forked river of subfigure~$(8')$.  Here the sink has reached the source and again there is a unique in-fill.  Depending on the colour of the upper edge of this in-fill, we arrive at one of the two rivers shown in subfigures (A) or (B).  In this case our continent now contains $t$. The channelisation process may continue beyond (A) or (B), but by \reflem{FiniteInFill} it eventually terminates and gives a convex continent.
\end{example}

We now consider the general case.  

\begin{proof}[Proof of \reflem{RiversSimplify}]
Recall that $f$ is the upper face of $C_n^0 = C_n$ which we must layer onto.  
Recall that $C^{k+1}_n$ is the result of channelisation of the pair $(C^k_n, f)$.  
By hypothesis, the face $f$ is contained in the upper landscapes of both $C^k_n$ and $C^{k+1}_n$.  
Thus $R^k$ and $R^{k+1}$ are both defined. 

Set $C^k_{n,0} = C^k_{n}$ and let $C^k_{n,i}$ be the continent obtained after attaching the first $i$ landfills of this $(k+1)^\stsup$ channelisation.  
Only the first landfill is coastal; the rest are in-fill. 
We define $R^k_i$ to be the forked river for the pair $(C^k_{n,i}, f)$.  

\begin{claim}
\label{Clm:ChannelisationInduction}
\mbox{}
\begin{enumerate}
\item
\label{Itm:CoastalEdges}
The mouth(s) of $R^k_i$ are coastal edges. 
\item
\label{Itm:ForkedRiverInduction}
Suppose that $R^k_i$ contains a sink.  Then $\ell(R^k_i) = \ell(R^k) + 1$.  Also, the riverbanks of $R^k_i$ are the same as those of $R^k$.  Also, $R^k \cap R^k_i$ is the sub-river of $R^k$ with source $f$ that ends at the sink of $R^k_i$.  
\item
\label{Itm:RiverInduction}
Suppose $i > 0$ and $R^k_i$ does not contain a sink.  Then $\ell(R^k_i) \leq \ell(R^k)$ and $R^k_i = R^{k+1}$.  
\end{enumerate}
\end{claim}

\refclm{ChannelisationInduction} proves \reflem{RiversSimplify}, as follows.  By \refitm{RiverInduction} the length of $R^{k+1}$ is no greater than that of $R^k$.  If it is shorter we are done; if it is the same length then by \refitm{ForkedRiverInduction} the in-fill giving $R^{k+1}$ has reduced the height of one fall without altering the heights further upstream.

\begin{proof}[Proof of \refclm{ChannelisationInduction}]
If $i = 0$ then we only need check \refitm{CoastalEdges} for $R^k_0 = R^k$.  This holds because $C^k_n$ is convex.  We now induct on $i$, starting with the base case of $i = 1$.  

So suppose that $i = 1$.  
Let $f'$ be the coastal triangle of $R^k_0$.  
Let $t'$ be the tetrahedron immediately above $f'$.  
Thus $C^k_{n,1} = C^k_{n,0} \cup t'$.  
If $f' = f$ then $f$ is not a face of the upper landscape of $C^{k+1}_n$, contrary to hypothesis.  
Thus $f' \neq f$.  
Thus there is an edge $e'$ of $f'$ that is neither on the riverbank nor on the coast.
There are four cases to check as the riverbank edge of $f'$ is either blue or red and 
as the upper edge of $t'$ is either blue or red.  
These are very similar to the general case discussed immediately below, so we omit them.  
(See \reffig{Moves}, but delete the triangle attached to $f'$ in the first row of figures.)

Suppose now that $i > 1$.  
If $R^k_i$ contains no sink then no future infill in the $k^\thsup$ channelisation will alter $R^k_i$.
That is, $R^{k+1} = R^k_{i+1} = R^k_i$ and there is nothing to prove.  
Suppose instead that $R^k_i$ contains a sink.  
We define $f'$ to be the triangle of $R^k_i$ immediately before the sink.
Let $t'$ be the tetrahedron immediately above $f'$.  
As before, if $f' = f$ then $f$ does not appear in the upper landscape of the next continent, contrary to hypothesis.  
So $f' \neq f$.   
Let $e'$ be the edge of $f'$ that is neither on the riverbank nor equal to the sink.  
There are again four possibilities according to the colours of the riverbank edge of $f'$ and the upper edge of $t'$. 
See \reffig{Moves}.

\begin{figure}[htb]
\centering
\subfloat[Landfilling along the blue riverbank.]{
\labellist
\small\hair 2pt
\pinlabel (1) [r] at 115 175
\pinlabel $(1')$ at 250 180
\pinlabel $f'$ at 132 310
\pinlabel $e'$ [r] at 106 280
\endlabellist
\includegraphics[width=0.46\textwidth]{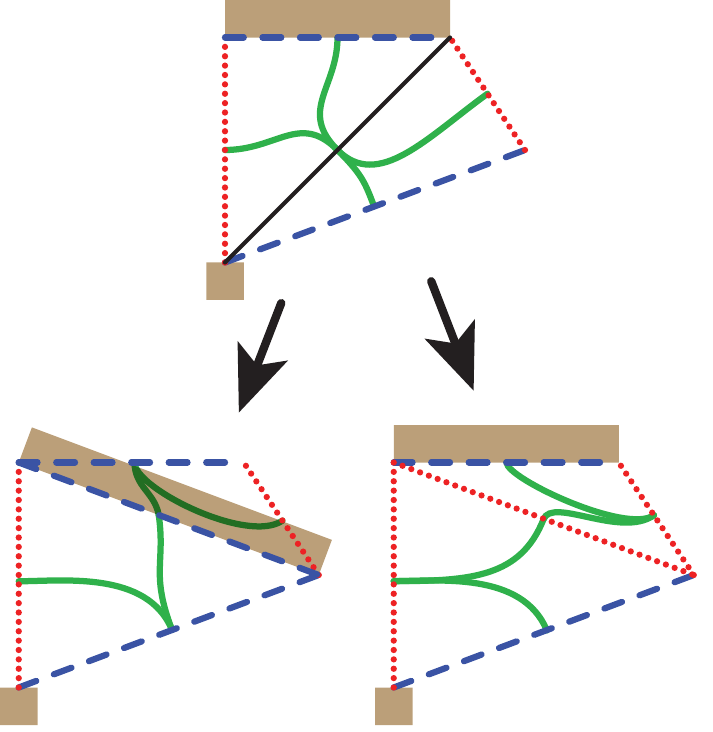}
\label{Fig:landfill_moves_blue}
}
\quad
\subfloat[Landfilling along the red riverbank.]{
\labellist
\small\hair 2pt
\pinlabel $(2')$ [r] at 115 170
\pinlabel (2) at 255 170
\pinlabel $f'$ at 132 250
\pinlabel $e'$ [r] at 106 280
\endlabellist
\includegraphics[width=0.46\textwidth]{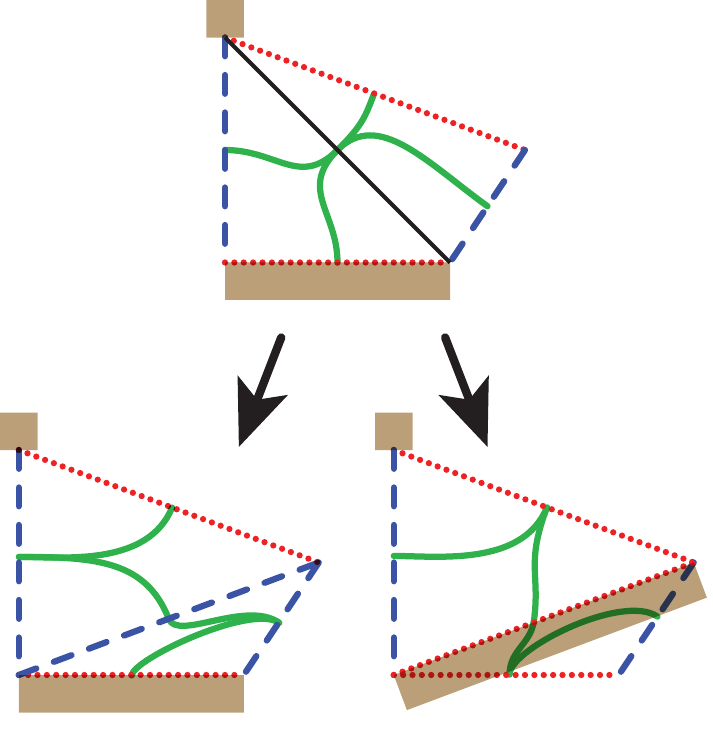}
\label{Fig:landfill_moves_red}
}
\caption{The two possible results of attaching $t'$ to $f'$, depending on the colour of its uppermost edge.  
The riverbank is shown shaded in brown.  
The sink, covered by the landfill tetrahedron, can be of either colour.  
Since the results are the same in each case, we here colour that edge black.}
\label{Fig:Moves}
\end{figure}

Suppose that $f'$ has an edge on the blue riverbank and the upper edge of $t'$ is blue.  
Then, as shown in \reffig{landfill_moves_blue}(1), after attaching $t'$ we have that $R^k_{i+1}$ is a river.  
We deduce that the mouth of $R^k_{i+1}$ is the mouth of the distributary meeting the red riverbank of $R^k_i$, which is coastal by induction.  
This proves \refitm{CoastalEdges} in this case.  
Since $R^k_{i+1}$ is a river, \refitm{ForkedRiverInduction} holds vacuously.  
Note that exactly one of the two upper faces of $t'$ are contained in $R^k_{i+1}$.  
Thus $R^k_{i+1}$ is at least one triangle smaller than $R^k_i$. 
Applying \refitm{ForkedRiverInduction} to $R^k_i$ we deduce that $\ell(R^k_{i+1}) \leq \ell(R^k)$.  
This proves \refitm{RiverInduction} in this case. 

Suppose that $f'$ has an edge on the blue riverbank and the upper edge of $t'$ is red.  
Then, as shown in \reffig{landfill_moves_blue}(1), after attaching $t'$ we have that $R^k_{i+1}$ is again a forked river.  
Both lower faces of $t'$ belonged to $R^k_i$ and both upper faces of $t'$ belong to $R^k_{i+1}$.  
Thus these forked rivers have the same number of triangles.  
Likewise they have the same boundaries and so have the same riverbanks and mouths.  
The distributary meeting the blue riverbank has grown by one triangle and the sink of $R^k_{i+1}$ is one triangle further upstream.  
This proves the induction hypotheses in this case. 

The remaining two cases (where $f'$ has a red riverbank edge) are the same, switching sides and colours correctly. 

This completes the proof of \refclm{ChannelisationInduction}.
\end{proof}
This completes the proof of \reflem{RiversSimplify}.
\end{proof}
This completes the proof of \refprop{VeerImpliesExhaust}.
\end{proof}

\subsection{No parallel edges}

With \refcor{VeerImpliesLayered} in hand we may now give a few global constraints on the combinatorics of veering triangulations. 

\begin{lemma}
\label{Lem:NoParallelEdges}
Suppose that $M$ is a three-manifold equipped with a veering triangulation $\calV$.
Then any pair of cusps of $\cover{\calV}$ are connected by at most one edge.
\end{lemma}

\noindent
Unlike \refthm{TetEmbeds}, the assumption of tautness alone does not suffice to prove \reflem{NoParallelEdges}. 

\begin{proof}[Proof of \reflem{NoParallelEdges}]
Fix cusps $c$ and $d$ of $\cover{\calV}$.
Suppose that $e$ is an edge of $\cover{\calV}$ connecting them.
\refcor{VeerImpliesLayered} gives a layering $\calK = (K_i)$ of $\cover{\calV}$.
Let $\tau^i$ and $\tau_i$ be the upper and lower tracks, respectively, of $K_i$.
For all $i$ we define $P_i \subset K_i$ as follows.  
If there is an edge $e_i$ in $K_i$ connecting $c$ to $d$ then $P_i = e_i$. 
If there is no such edge then we apply \reflem{Disk} and let $P_i$ be the strip of triangles connecting $c$ and $d$.

Let $t$ be the tetrahedron immediately below $e$.
Let $K_p$ be the lowest layer of $\calK$ containing the upper faces of $t$.
We deduce that $\tau^p$ has a small switch at $e$.
Recall that, as shown in \reffig{UpperGluingAutomaton}, the track $\tau^{i-1}$ is obtained from $\tau^i$ by folding at exactly one small switch.
Inducting downwards, we have that for all $i < p$
\begin{itemize}
\item
$P_i$ is a strip of at least two triangles,
\item
the track-cusps of $\tau^{P_i}$, in the first and last triangles of $P_i$, point away from $c$ and $d$, respectively, and
\item
these track-cusps are connected by a train interval in $\tau^{P_i}$.
\end{itemize}
Thus there is no edge connecting $c$ and $d$ below $K_p$. 

Now suppose that $t'$ is the tetrahedron immediately above $e$.  
Let $K_q$ be the highest layer of $\calK$ containing the lower faces of $t'$. 
We now repeat the previous argument replacing $K_p$ by $K_q$ and upper train tracks with lower.

We deduce that, for $i < p$ and for $i > q$, the layer $K_i$ does not have an edge connecting $c$ and $d$.  
Thus $p \leq q$.
From \reflem{EdgeNeighbourhood} we deduce that $p$ is in fact strictly less than $q$.
Finally note that $e$ lies in $K_j$ if $p \leq j \leq q$.  
This is because $e$ can only be removed from $K_j$ by flipping across $t$ or $t'$. 
Since $K_j$ is a copy of the Farey tessellation (\reflem{Disk}) it contains no bigons.  
That is, $e$ is the only edge in $K_j$ connecting $c$ to $d$.  
\end{proof}

\begin{remark}
When the triangulation $\calV$ is finite there is an alternative argument using strict angle structures~\cite{HRST11, FuterGueritaud13} and combinatorial area~\cite[Theorem~6.1]{Hodgson15}.
\end{remark}

\subsection{Train routes wiggle}

We next control the train rays of the upper and lower tracks, in the presence of a veering triangulation. 

\begin{definition}
Suppose that $K$ is a layer of a layering.  
Fix an orientation on $\cover{M}$ and a transverse orientation on $\cover{B}$.
This induces an orientation on $K$.
A train ray in $\tau^K$ travels through the triangles of $K$, entering through one edge of each triangle and exiting through the edge either to the left or the right.  
We refer to the former as a \emph{left turn} and the latter as a \emph{right turn}. 
\end{definition}

\begin{lemma}
\label{Lem:CannotTurnLeftForever}
Suppose that $K$ is a layer of a layering and $\rho$ is a train ray in the upper track $\tau^K$ on $K$.
Then $\rho$ turns both left and right infinitely often.
The same is true for a train ray in the lower track $\tau_K$.
\end{lemma}

\begin{figure}[htbp]
\subfloat[An upper train track that carries a left-turning ray that crosses only blue edges.]{
\label{Fig:TrainRouteTurningLeftBlue}
\includegraphics[height=1.7in]{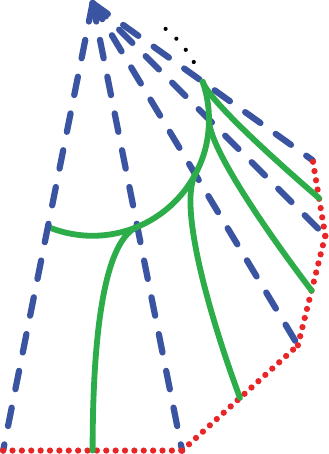}
} \quad
\subfloat[An upper train track that carries a left-turning ray that crosses only red edges (left). The corresponding lower track (right).]{
\label{Fig:TrainRouteTurningLeftRed}
\includegraphics[height=1.7in]{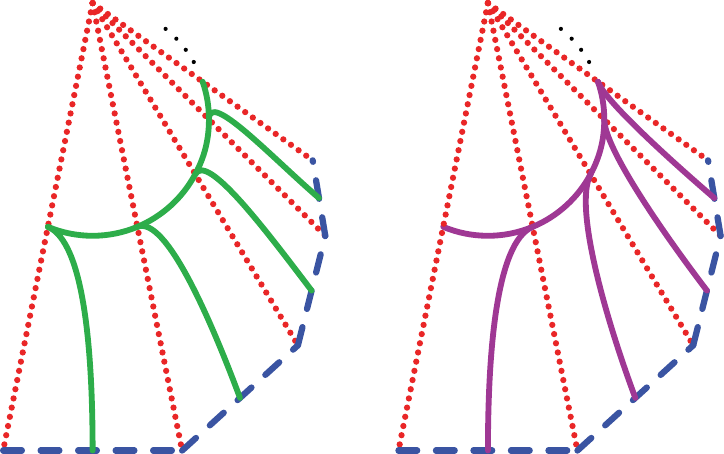}
} 
\caption{}
\label{Fig:TrainRouteTurningLeft}
\end{figure}

\begin{proof}
Suppose, for a contradiction, that $\rho$ is a train ray carried by $\tau^K$ that turns left forever.  
Consulting \reffig{VeeringTriangles}, we see that $\rho$ can cross a face of $K$ and travel from a blue edge to a blue edge, from a blue edge to a red edge, or from a red edge to a red edge, but that it can never travel from a red edge to a blue edge. 
Replacing $\rho$ by a sub-ray, we may assume that either $\rho$ crosses only blue edges or $\rho$ crosses only red edges.  
Let $P \subset K$ be the union of triangles that $\rho$ passes through.  
The two possibilities are shown in \reffig{TrainRouteTurningLeft}.

Suppose that we are in the first case, shown in \reffig{TrainRouteTurningLeftBlue}.  
There are no sinks of $\tau^P$ in $K$ meeting any edge (interior or boundary) of $P$.  
Thus it is impossible to layer a tetrahedron onto any triangle of $P$.  
Therefore $K$ is not a layer of a layering, a contradiction.

Suppose that we are in the second case, shown in \reffig{TrainRouteTurningLeftRed}.  
Here we instead consider the lower train track $\tau_P$.  
We conclude that it is impossible to layer a tetrahedron onto the bottom of $P$, and again we reach a contradiction.

Similar arguments apply to train rays in $\tau^K$ that turn right forever, and train rays in $\tau_K$ that turn in either direction forever.
\end{proof}

\chapter{Branched surfaces and branch lines}
\label{Cha:SurfacesAndLines}

Here, following ideas of Agol, given $M$, an oriented three-manifold equipped with a transverse veering triangulation $\calV$, we construct the upper and lower branched surfaces $B^\calV$ and $B_\calV$.  
Their branch loci play an important role in \refcha{LaminationsAlone}.  

\subsection{Branched surfaces}
\label{Sec:UpperLowerSurfaces}

Again, see~\cite[Section~6.3]{Calegari07} for a reference on branched surfaces.  

We construct the \emph{upper branched surface} $B^\calV$ in two stages.  
First, fix a tetrahedron $t$ in $\cover{\calV}$.  
Let $L$ be the two-triangle landscape above $t$ and $L'$ the two-triangle landscape below.  
Note that $\tau^L$, the upper track for $L$, is obtained from the upper track for $L'$ by doing a single split.  
If we perform this split continuously then it sweeps out a two-complex $B^t$, shown in \reffig{NormalUpperBranchedSurface}.  

The two-complex $B^t$ is not quite a branched surface; 
at the unique vertex the two one-cells are tangent rather than transverse.
We abuse terminology and nonetheless call $B^t$ the \emph{upper} branched surface, in $t$, in \emph{normal position}.
We use this terminology because the sectors of $B^t$ are three \emph{normal disks}:
two triangles and a quadrilateral.  
The track-cusps sweep along the two branch components of $B^t$; 
each branch interval is a normal arc in a lower face of $t$.  
The three sectors and two branch intervals meet at the midpoint of the bottom edge of $t$; 
this is the unique vertex of $B^t$.  
Again, see \reffig{NormalUpperBranchedSurface}.

We now remedy the failure of $B^t$ to be a branched surface.  
We fold $B^t$ in a small regular neighbourhood of its branch intervals.
We call the result the \emph{dual position} of $B^t$ because, after folding, $B^t$ is isotopic to the dual two-skeleton for $t$.  
See \reffig{DualUpperBranchedSurface}; 
note how the branch locus is no longer contained in the lower faces of $t$.   

\begin{figure}[htb]
\centering
\subfloat[Normal position.]{
\label{Fig:NormalUpperBranchedSurface}
\includegraphics[width=0.47\textwidth]{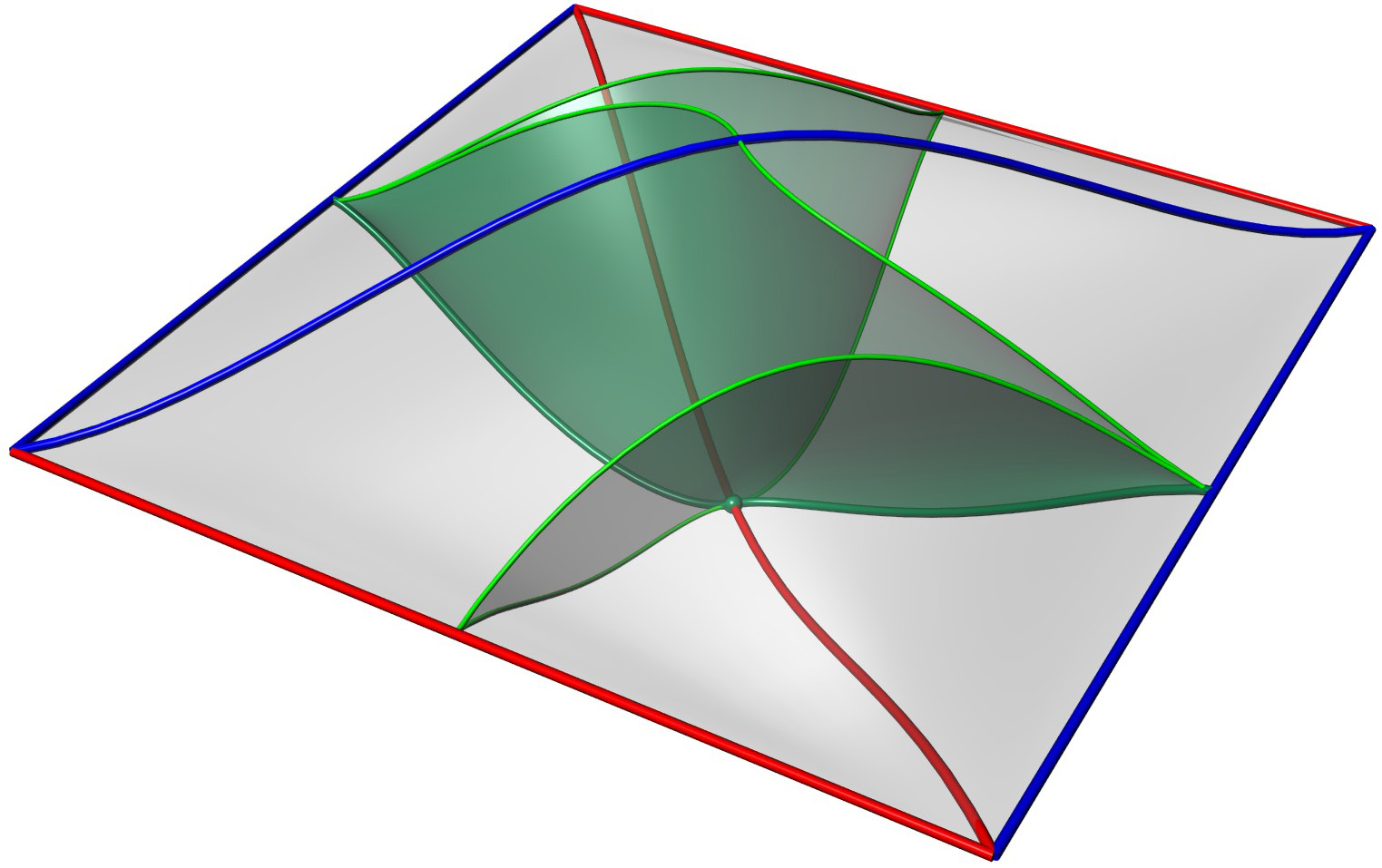}
}
\subfloat[Dual position.]{
\label{Fig:DualUpperBranchedSurface}
\includegraphics[width=0.47\textwidth]{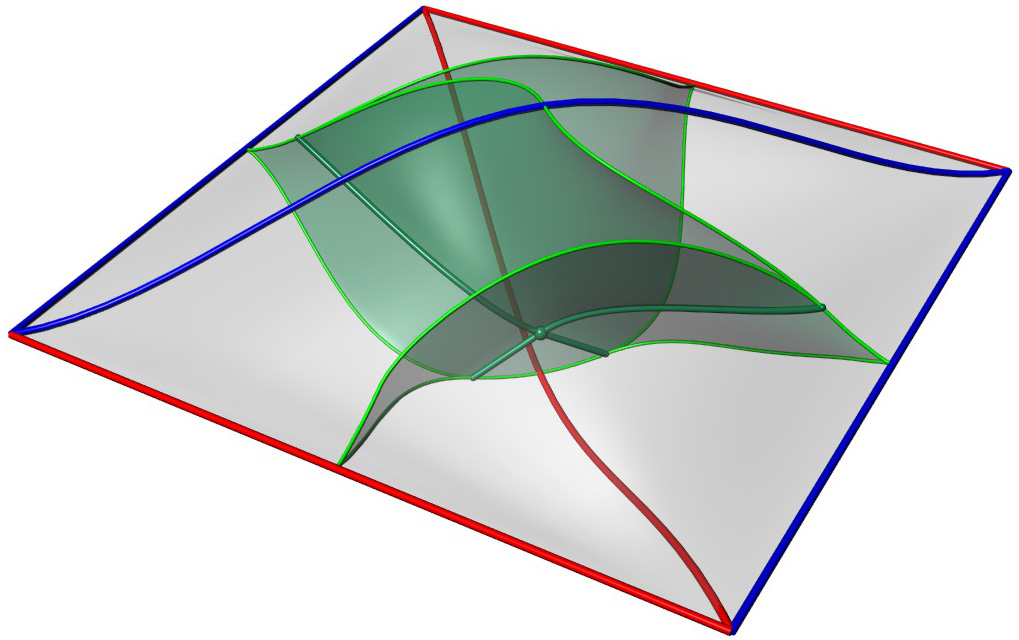}
}
\caption{Two positions of the upper branched surface in a tetrahedron.}
\label{Fig:UpperBranchedSurface}
\end{figure}

We now push $B^t$ down to $M$.
We define $B^\calV \subset M$, the \emph{upper branched surface} in \emph{normal position} or in \emph{dual position}, to be the union $\bigcup_{t \in \calV} B^t$ where the latter are also in normal or dual position.  

\begin{remark}
\label{Rem:Shift}
We justify giving a single name, $B^\calV$, to the two positions with the following claim: 
the normal and dual positions of $B^\calV$ are isotopic in $M$.  
To see this, start with $B^\calV$ in normal position.
Now isotope $B^\calV$ slightly upwards.
This makes the branch locus transverse to $B(\calV)$.
We may then isotope a bit further, to make the branch locus self-transverse.
The result is $B^\calV$ in dual position.
Again, see \reffig{DualUpperBranchedSurface}.
\end{remark}

We will almost always draw $B^\calV$ in normal position.
We define $B_\calV$ similarly, in both its normal and dual position, 
by using the lower tracks $\tau_L$ and $\tau_{L'}$ to build $B_t$, and so on. 

\begin{remark}
\label{Rem:Dual}
The branched surfaces $B^\calV$ and $B_\calV$ (in dual position) are both isotopic to the dual two-skeleton of $\calV$.
It follows that $B^\calV$ and $B_\calV$ are isotopic two-complexes (that is, ignoring their branched surface structures).  
Again, see \reffig{DualUpperBranchedSurface}.
\end{remark}

\subsection{Branch lines}
\label{Sec:BranchLines}

The branched surface $B^\calV \subset M$ cuts $M$ into a disjoint union of \emph{upper cusp neighbourhoods}.  
Each of these is homeomorphic to a torus (or annulus, or plane) crossed with a ray.  
Suppose that $\check{c}$ is a cusp of $\calV$ and $\check{N}$ is the corresponding upper cusp neighbourhood.  
We will abuse notation slightly and identify $\check{N}$ with its closure in the induced path metric; this makes $\check{N}$ homeomorphic to a torus (or annulus, or plane) cross an interval.  
The inner boundary of $\check{N}$ receives a piecewise smooth structure from $B^\calV$.  
The locus of non-smooth points in $\bdy \check{N}$ is the union of \emph{upper branch loops (or lines)} for $\check{N}$. 
We orient the upper branch loops (or lines) so that, with $B^\calV$ in dual position, the orientation on an upper branch loop (or line) agrees with the transverse orientation on $B(\calV)$.  

Suppose that $\check{S}$ is an upper branch loop (or line) in $\check{N}$.  
Let $c$ be an elevation of $\check{c}$ to $\cover{\calV}$. 
Let $N^c$ be the corresponding elevation of $\check{N}$.  
Let $S \subset N^c$ be an elevation of $\check{S}$.  
We call $S$ an \emph{upper branch line}.  

\begin{lemma}
\label{Lem:BranchLinesAreLines}
Branch lines are in fact lines -- not intervals, rays, or loops.
\end{lemma}

\begin{proof}
With notation as immediately above: $\check{S}$ cannot be an interval or a ray as it has no boundary.
Thus $\check{S}$ is either a loop or a line.  
Consulting \reffig{NormalUpperBranchedSurface} we find that we may orient $\check{S}$ to everywhere agree with the given co-orientation.  

There are now two cases as $\check{S}$ is a loop or a line.
If it is a line then $S$ is as well and we are done.
If it is a loop then it is \emph{vertical} in sense of~\cite[Definition~2.1]{SchleimerSegerman20}.
Applying~\cite[Theorem~3.2]{SchleimerSegerman20} we find that $\check{S}$ is essential in $\pi_1(M)$.  
\end{proof}

\begin{definition}
\label{Def:Adjacent}
We say two upper branch lines $S$ and $T$ in $N^c$ are \emph{adjacent} if they are not separated by a third branch line in $N^c$.  
\end{definition}

\begin{figure}[htb]
\centering
\includegraphics[width=0.5\textwidth]{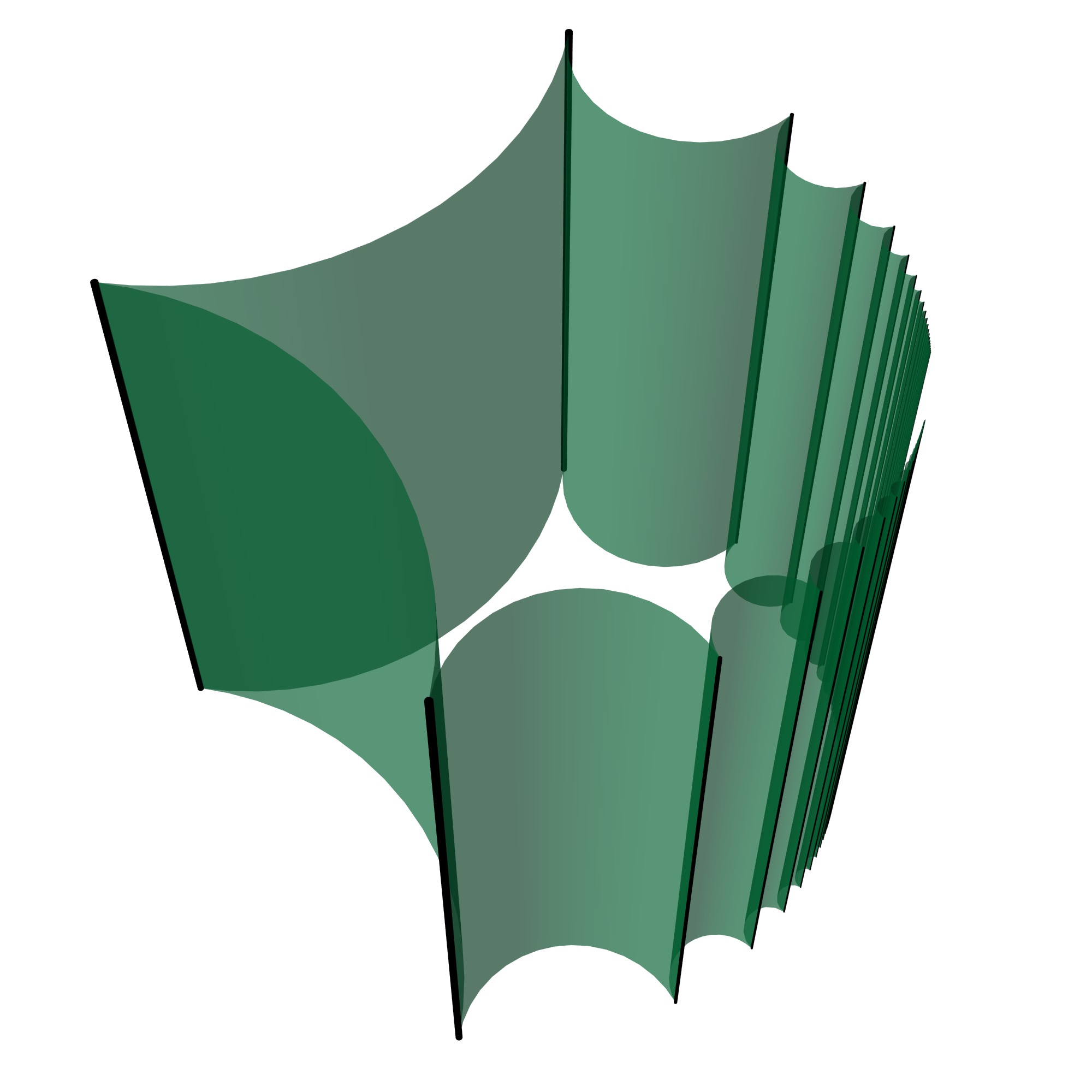}
\caption{The boundary of an upper cusp neighbourhood $N^c$: 
the component of $\cover{M} - \cover{B}^\calV$ containing the cusp $c$.
The black lines represent the branch lines in $N^c$.}
\label{Fig:CuspNeighbourhood}
\end{figure}

We now show that Figures~\ref{Fig:CuspNeighbourhood} and~\ref{Fig:NeighbourhoodLayer} give an accurate depiction of an upper cusp neighbourhood. 

\begin{figure}[htb]
\centering
\labellist
\small\hair 2pt
\pinlabel $c$ [t] at 137 90
\endlabellist
\includegraphics[width=0.5\textwidth]{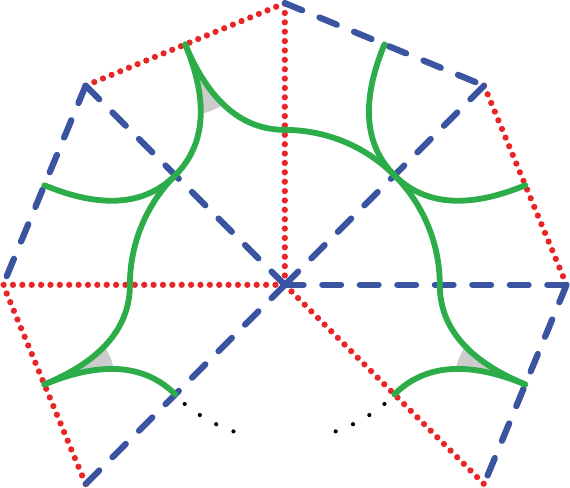}
\caption{The strip of triangles $K(c)$ in the layer $K$ meeting the cusp $c$.
The complementary component of the train track $\tau^{K(c)}$ meeting $c$ is the intersection of $N^c$ and $K$.
The track-cusps in the interior of $N^c$ are shaded grey.}
\label{Fig:NeighbourhoodLayer}
\end{figure}

\begin{lemma}
\label{Lem:BranchLines}
Suppose that $c$ is a cusp of $\cover{\calV}$.  
Let $(S_i)$ be the branch lines on the boundary of $N = N^c$.  
Fix $K$, a layer of some layering $\calK$ of $\cover{\calV}$. 
\begin{enumerate}
\item
\label{Itm:BranchLinesLayer}
For each $i$, there is a unique track-cusp $s_i$ of $S_i$ meeting $K$.
\item
\label{Itm:Cyclic}
The upper branch lines in $N$, equipped with the adjacency relation, are indexed by $\ZZ$.  
\item
\label{Itm:NeighbourhoodLayer}
The intersection $N \cap K$ is contained in the faces of $K$ incident to $c$.
The track-cusp $s_i$ points away from $c$, in the face that contains $s_i$ and meets $c$. 
\item
\label{Itm:BranchStrip}
The strip in $\bdy N$ cobounded by $S_i$ and $S_{i+1}$ meets $K$ in a train interval running from $s_i$ to $s_{i+1}$.
\item
\label{Itm:CuspNeighbourhood}
The cusp neighbourhood $N$ is homeomorphic to $(N \cap K) \cross \RR$.
\end{enumerate}
\end{lemma}

\begin{proof}
Fix a branch line $S_i$ and fix a layer $K$. 
Fix also a track-cusp $s'$ of $S_i$. 
The track-cusp $s'$ lies in some face $f'$, which lies in some layer $K'$.
We now induct on the number of layers between $K$ and $K'$ to find some track-cusp $s$ of $S_i$ lying in $K$.
This proves the existence claimed in \refitm{BranchLinesLayer}.

We now prove uniqueness.
Let $s$ and $s'$ be adjacent track-cusps of $S_i$, with $s'$ below $s$.  
Let $f$ and $f'$ be the faces of $\cover{B}$ that contain $s$ and $s'$ respectively. 
Consulting \reffig{NormalUpperBranchedSurface}, we see that $f$ and $f'$ are upper and lower faces of a tetrahedron $t$.
Let $e$ be their shared edge.
The dihedral angle between $f$ and $f'$ at $e$ is zero.
Since the circular order $\calO_\calV$ is compatible (\refdef{Compatible}), the faces $f$ and $f'$ cannot lie in a single layer.

Let $L$ be the lowest layer of $\calK$ containing $f$.
By the above paragraph, $f'$ is below $L$.
Recall that $\cover{\calV}$ is homeomorphic to $\RR^3$ by \refprop{VeerImpliesExhaust} and \reflem{ExhaustImpliesThreeSpace}.
Thus $L$ separates $f'$ from all layers above $L$.
Thus no layer can contain $s'$ and a track cusp of $S_i$ above $s'$.
This proves uniqueness and thus we have \refitm{BranchLinesLayer}.

By \refitm{BranchLinesLayer}, every branch line $S_i$ in the boundary of $N$ meets $K$ in a track-cusp $s_i$.
Let $K(c)$ be the strip of triangles in the layer $K$ meeting the cusp $c$.
Thus $N \cap K$ is contained in $K(c)$.
Thus all of the $s_i$ lie in $K(c)$.
By \reflem{CannotTurnLeftForever}, the track-cusps $(s_i)$ exit both ends of $K(c)$, giving \refitm{Cyclic} and \refitm{NeighbourhoodLayer}.  

The strip in $\bdy N$ cobounded by $S_i$ and $S_{i+1}$ meets $K$ transversely within $\tau^K$.
The endpoints are the switches contained in $s_i$ and $s_{i+1}$, giving \refitm{BranchStrip}.
See \reffig{NeighbourhoodLayer}.

Suppose that $K'$ is the layer of $\calK$ immediately above $K$.
Move $K$ and $K'$ slightly apart, keeping both carried by the branched surface $\cover{B}$.
Let $t$ be the tetrahedron between $K$ and $K'$.
If $t \cap N$ is empty then $K$ and $K'$ cobound a product region $P \homeo (N \cap K) \cross [0,1]$.
The product $P$ is in turn carried by $\cover{B}$.
Suppose instead that $t \cap N$ is non-empty.
Away from $t$ the layers again cobound a product region $P_0$.
Let $Q_t$ be a small closed neighbourhood of $t$.
Inside of $Q_t \cap N$ the layers $K$ and $K'$ cobound a small neighbourhood $P_t$ of $t \cap N$.
Consulting \reffig{NormalUpperBranchedSurface} we see that $P_t$ is again a product region.
We form $P = P_0 \cup P_t$.
Again $P \homeo (N \cap K) \cross [0,1]$ and we have proven \refitm{CuspNeighbourhood}.
See \reffig{CuspNeighbourhood}.
\end{proof}

\begin{remark}
\label{Rem:BranchLines}
\reflem{BranchLines}\refitm{Cyclic} is a version of \cite[Observation~2.3]{FuterGueritaud13} in our context.  
When $M$ is compact, the branch loops in $M$ equip every component of $\bdy M$ with a \emph{branch slope}. 
Similar slopes are called the \emph{ladderpole slopes} in~\cite[Observation~2.4]{FuterGueritaud13}.  
We give more details in \refsec{Ladders}.

More generally, the branch loops are similar to the 
\emph{parabolic locus} of \emph{pared manifolds}~\cite[Definition~4.8]{Morgan84}, 
the \emph{degeneracy slopes} of~\cite[page~62]{GabaiOertel89}, 
and the \emph{maw loops} of~\cite[page~27]{Mosher96}.
\end{remark}

\begin{corollary}
\label{Cor:BranchLinesToggle}
Suppose that $S$ is an upper branch line of $\cover{B}^\calV$ in $\cover{\calV}$.
Then every subray of $S$ meets the lower faces of some toggle tetrahedron.
In particular every subray of $S$ meets both red and blue edges. 
A similar statement holds for the lower branch lines of $\cover{B}_\calV$.
\end{corollary}

\begin{proof}
Let $R \subset S$ be an ascending subray. 
Thus $R$ passes through an infinite sequence $(s_i)_{i \in \NN}$ of track-cusps.
Let $f_i$ be the face containing $s_i$.
Let $t_i$ be the tetrahedron immediately above $f_i$.
Suppose, for a contradiction, that $t_i$ is always a fan tetrahedron.
It follows that all of the $t_i$ are same colour and attached to a single edge of the opposite colour.
See \reffig{GluingAutomaton}.
But this contradicts the finiteness of edge degrees. 

A similar argument applies when $R$ is a descending ray or is contained in a lower branch line. 
\end{proof}

\subsection{Sectors}
\label{Sec:Sectors}

We now consider the various ways the normal quadrilaterals and triangles of $B^\calV$ can meet. 
See \reffig{UpperGluingAutomaton}.  
Recall that $\calV$ endows the faces and edges of $B(\calV)$ with a co-orientation. 

\begin{lemma}
\label{Lem:Sector}
Suppose that $D$ is a sector of $B^\calV$ (in normal position).  Then $D$ is a disk which 
\begin{itemize}
\item
contains exactly one normal quadrilateral, 
\item
contains exactly two (possibly empty) collections of normal triangles above the normal quadrilateral (to its left and right), 
\item
has exactly one \emph{upper} (\emph{lower}) \emph{vertex} where the co-orientation points out of (into) $D$, at the top (bottom) of the normal quadrilateral, 
\item
has exactly two \emph{cusp vertices} where $\bdy D$ has a cusp with respect to the co-orientation, and 
\item
has an \emph{under-side vertex}, where the co-orientation points into $D$, at the bottom of each normal triangle. 
\end{itemize}
\end{lemma}

\begin{proof}
Considering \reffig{NormalUpperBranchedSurface}, we see that in each tetrahedron, the branched surface consists of one normal quadrilateral and two normal triangles.
Both bottom edges of the quadrilateral, and precisely one bottom edge of each triangle, lies in the branch locus. Thus these bottom edges form part of the boundary of the containing sector.
The upper edges of the normal disks are either in the boundary of the sector or are glued to the free bottom edge of some triangle.
The gluings between the normal disks inside of a sector are given in \reffig{UpperGluingAutomaton}.
The desired properties follow.
\end{proof}

See \reffig{Sector} for a fairly generic example of a sector of $B^\calV$.  

\begin{figure}[htbp]
\subfloat[]{
  \label{Fig:SectorTaut}
  \includegraphics[height = 3.0 cm]{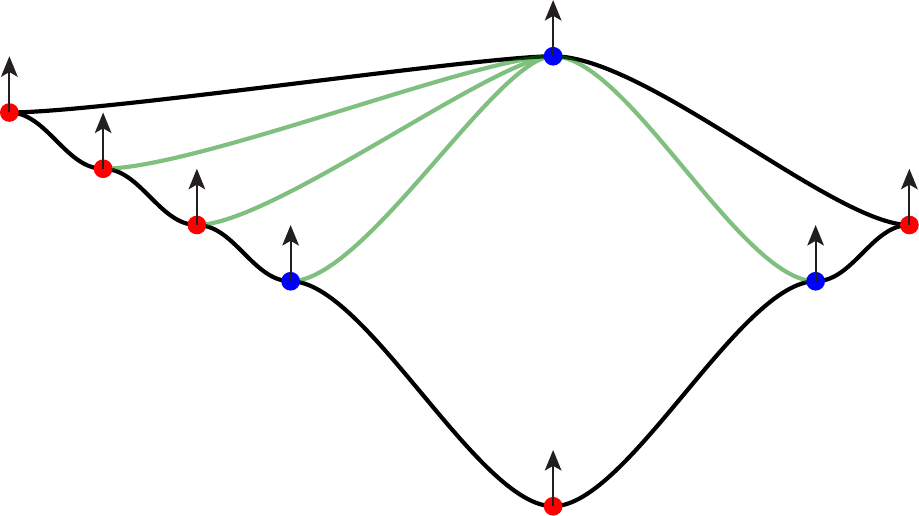}
  }
\quad
\subfloat[]{
  \label{Fig:SectorDiamond}
  \includegraphics[height = 4.0 cm]{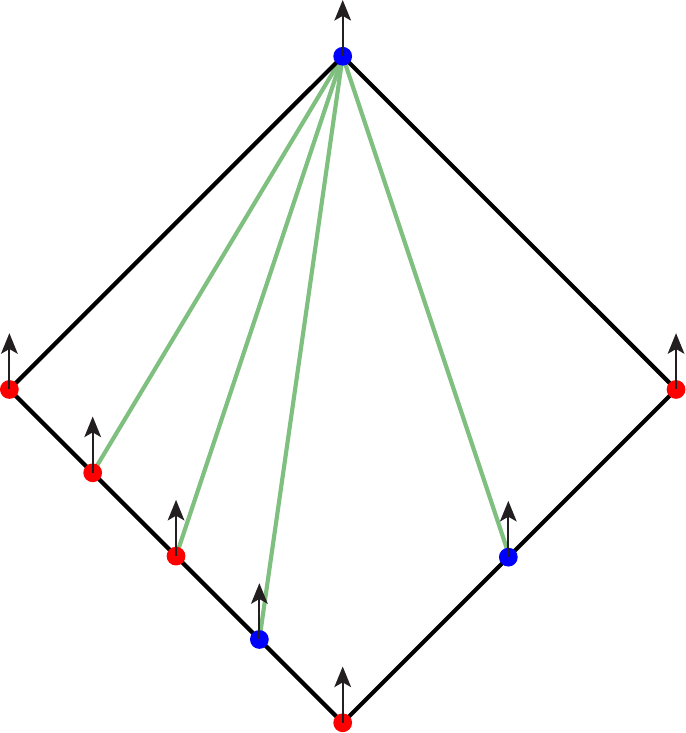}
}
  \caption{Any sector of $B^\calV$ (in normal position) contains exactly one normal quadrilateral and some (perhaps zero) normal triangles.  
The boundary of the sector is drawn in black, while internal boundaries between normal disks are drawn in green.
In \reffig{SectorTaut} we draw a sector in a style that emphasises the smoothing coming from $B(\calV)$.
In \reffig{SectorDiamond} we draw the same sector in a polygonal style, without smoothing. }
\label{Fig:Sector}
\end{figure}

\subsection{Leaves}
\label{Sec:Leaves}

We refer to~\cite[Definitions~6.9 and~6.14, page~215]{Calegari07} for the definitions of \emph{surface laminations} in three-manifolds, their \emph{leaves}, \emph{essential laminations}, and how branched surfaces \emph{carry} laminations.

Suppose that $\calV$ is a transverse veering triangulation. 
Let $B^\calV$ be the resulting upper branched surface (the discussion is similar for $B_\calV$).

\begin{remark}
\label{Rem:Laminar}
The upper and lower branched surfaces $B^\calV$ and $B_\calV$ are \emph{laminar} in the sense of Tao~Li~\cite[Definition~1.4]{Li02}.
Thus they carry essential laminations.
However more is true.
Since $\calV$ is veering, there is a canonical pair of laminations $\Sigma^\calV$ and $\Sigma_\calV$ fully carried by $B^\calV$ and $B_\calV$ respectively. 
Furthermore, $\Sigma^\calV$ and $\Sigma_\calV$ are the \emph{only} laminations fully carried by $B^\calV$ and $B_\calV$, up to Denjoy operations.
See Sections~\ref{Sec:SuspendingDescending} and~\ref{Sec:Uniqueness}.

When $\calV$ is not finite, we still have canonical laminations $\Sigma^\calV$ and $\Sigma_\calV$.
These are needed in order to complete our programme of classifying pseudo-Anosov flows (without perfect fits).
\end{remark}

\begin{figure}[htbp]
\includegraphics[width = 0.8\textwidth]{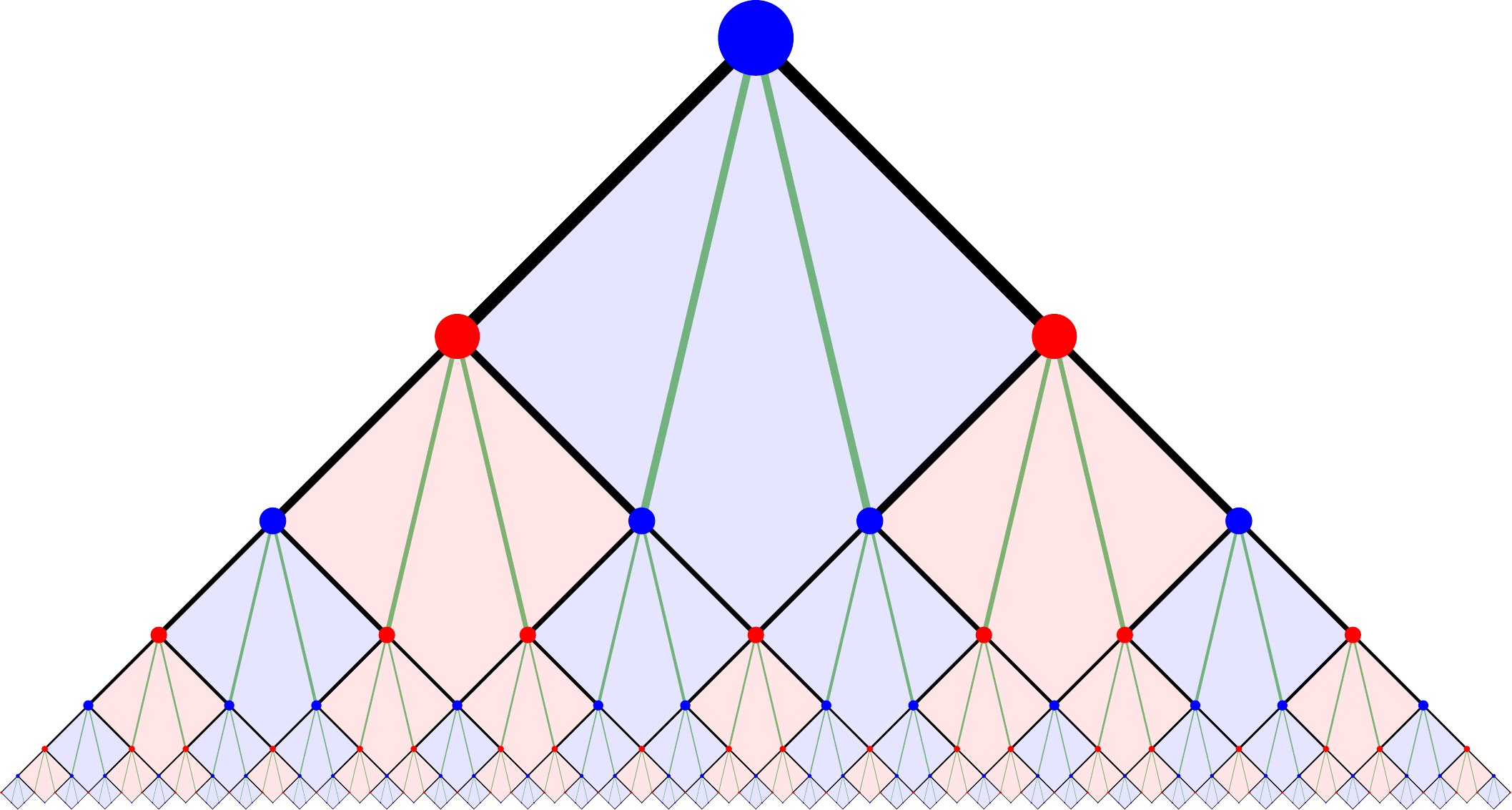}
\caption{The decomposition of a leaf into sectors, and then into normal disks.  
We draw the sectors in the style of \reffig{SectorDiamond}.
This leaf $\sigma$ is fully carried by the upper branched surface for the veering triangulation of the figure-eight knot complement given in \reffig{VeerFigEight}.  
There are two kinds of sector, here coloured light blue and light red according to the colour of the dual edge, as in \refrem{Dual}.}
\label{Fig:SectorsForFig8KnotComplement}
\end{figure}

Any leaf $\sigma$ of any lamination $\Sigma$ carried by $B^\calV$ (or $B_\calV$) inherits a decomposition into sectors.
For an example, see \reffig{SectorsForFig8KnotComplement}.
Here $\sigma$ is fully carried by the upper branched surface for the veering triangulation of the figure-eight knot complement, given in \reffig{VeerFigEight}.

\begin{figure}[htb]
\centering
\subfloat[]{
\includegraphics[width=0.26\textwidth]{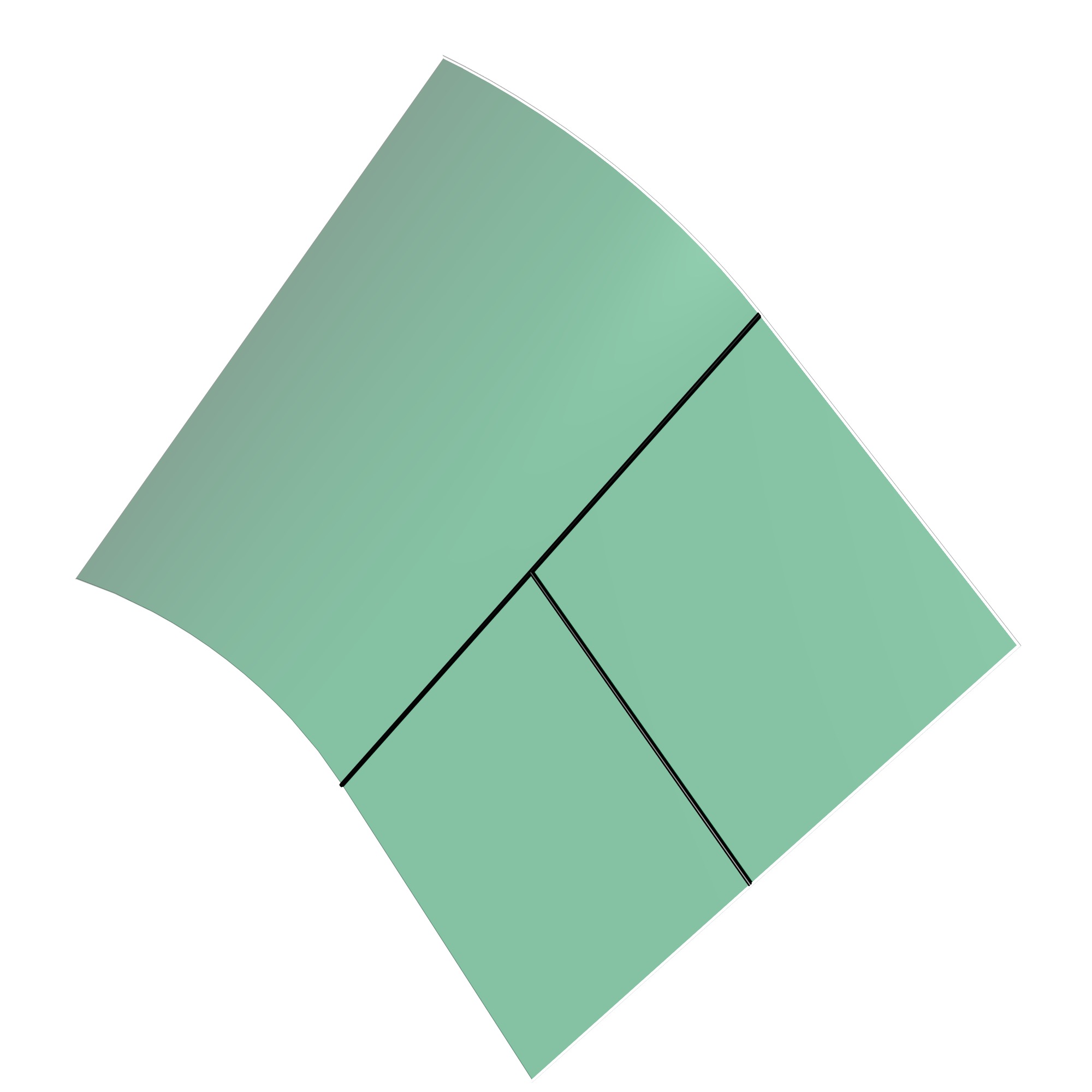}
\label{Fig:branched_surface_front}
}
\subfloat[]{
\includegraphics[width=0.26\textwidth]{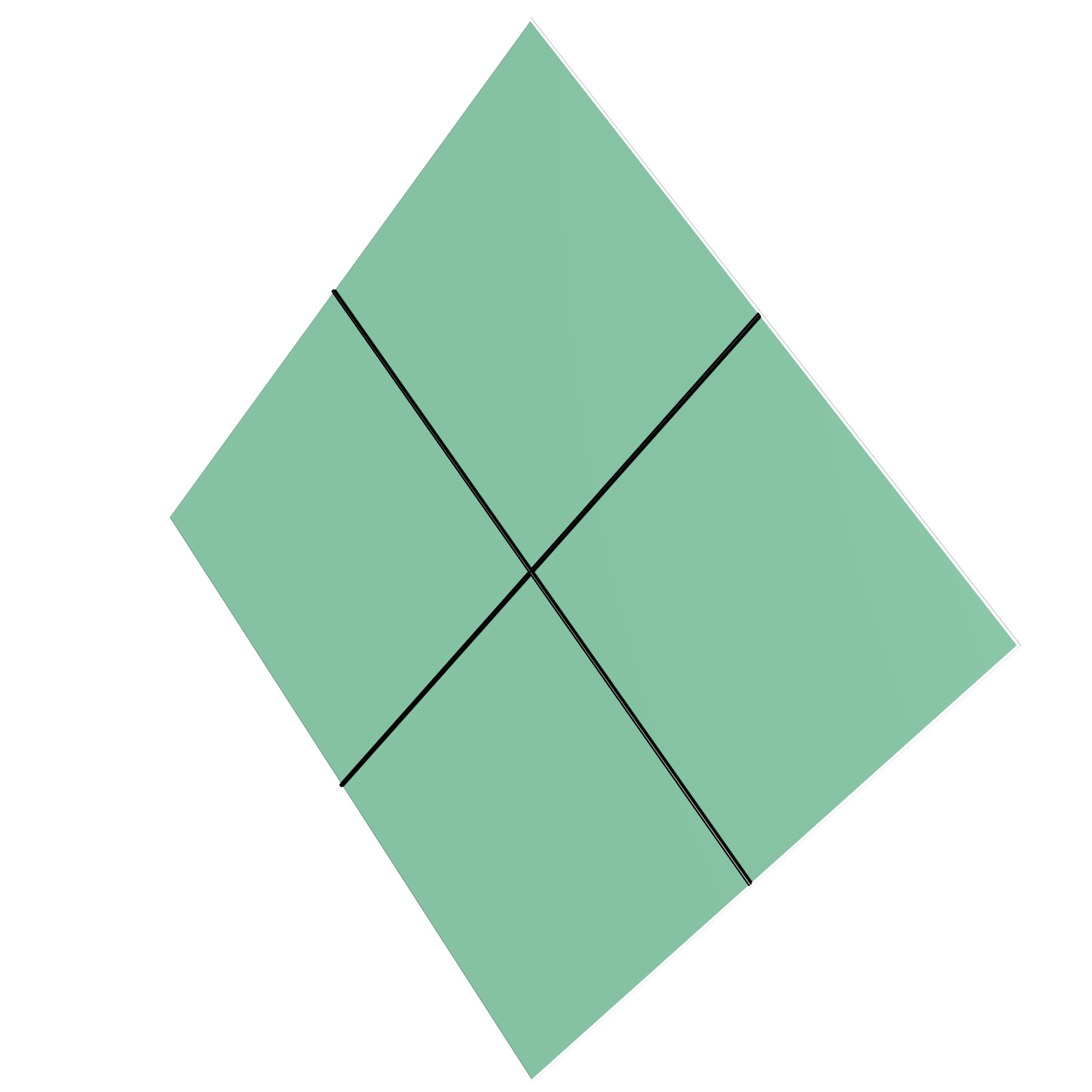}
\label{Fig:branched_surface_middle}
}
\subfloat[]{
\includegraphics[width=0.26\textwidth]{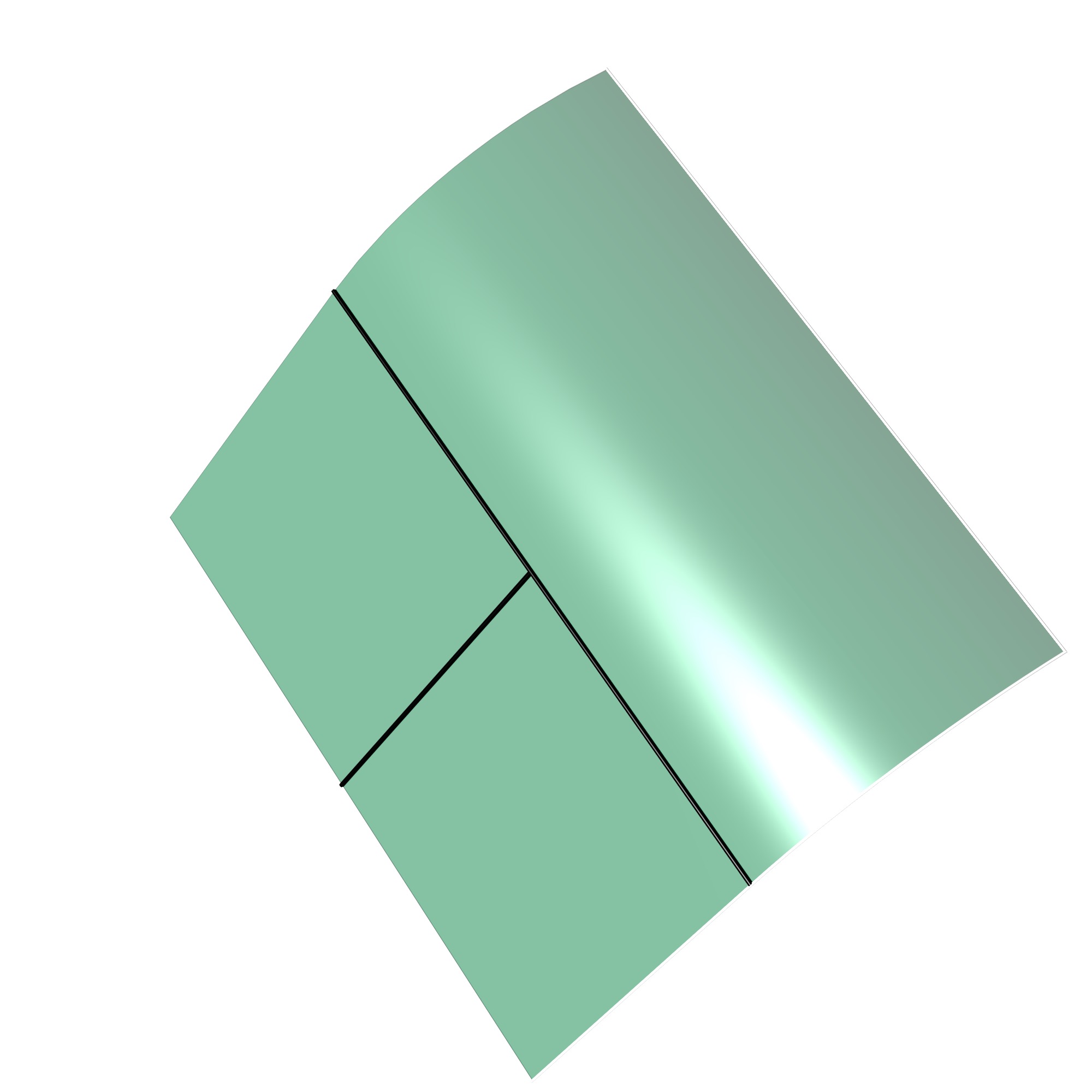}
\label{Fig:branched_surface_back}
}
\caption{A leaf carried by the upper branched surface is patterned by the branch lines of the branched surface in one of these three ways in the vicinity of a vertex of the branched surface (compare with \reffig{DualUpperBranchedSurface}).}
\label{Fig:TAndX}
\end{figure}

Suppose that $\sigma$ is a leaf of any lamination $\Sigma$ carried by $B^\calV$ (say).
The decomposition of $B^\calV$ into sectors induces a cellulation of $\sigma$.
In this cellulation, drawn in the style of \reffig{SectorDiamond}, the two-cells are diamonds (possibly with many vertices on their lower sides) while the vertices all have degree three or four.
We call these vertices, respectively, \emph{T-junctions} or \emph{X-junctions}.
See \reffig{TAndX}.

\chapter{The veering circle}

In this section we build the \emph{veering circle} $\Circle$.
Here is the desired statement. 

\begin{theorem}
\label{Thm:VeeringCircle}
Suppose that $M$ is an oriented three-manifold equipped with a transverse veering triangulation $\calV$. 
Then the order completion of $(\Delta_\calV, \calO_\calV)$ is a circle $\Circle$ with the following properties.
\begin{enumerate}
\item 
\label{Itm:Acts}
The action of $\pi_1(M)$ on $\Delta_\calV$ extends to give a continuous, faithful, orientation-preserving action on  $\Circle$.  
\item 
\label{Itm:Dense}
If $\calV$ is finite then all orbits are dense.
\end{enumerate}
\end{theorem}

\subsection{Arcs at infinity}

We introduce several pieces of notation.  

\begin{definition}
\label{Def:Arcs}
Suppose that $M$ is a three-manifold equipped with a transverse veering triangulation $\calV$. 
Suppose that $a$ and $b$ are cusps in $\Delta_\calV$.  
Define
\[
[a, b]_\Delta^{\acw} = \{ c \in \Delta_\calV \st \calO_\calV(a, c, b) \geq 0 \}
\]
This is the \emph{arc} in $\Delta_\calV$ that is anticlockwise of $a$ and clockwise of $b$. 

Now fix an edge $e$ of $\cover{\calV}$.
Orient $e$ by ordering the cusps $c$ and $c'$ at the ends of $e$.
This, together with the transverse taut structure on $\cover{\calV}$ determines a co-orientation of $e$ using the right-hand rule.
We set $\Delta(e) = [c', c]_\Delta^{\acw}$.
Note that $e$ and $\Delta(e)$ have the same endpoints.
The co-orientation on $e$ points towards the arc $\Delta(e)$.
See \reffig{CoorientedEdge}.  
\end{definition}

\begin{figure}[htbp]
\labellist
\small\hair 2pt
\pinlabel {$c$} at 13 450 
\pinlabel {$c'$} at 565 450 
\pinlabel {$e$} at 220 340 
\pinlabel {$\Delta(e)$} at 220 610
\endlabellist
\[
\begin{array}{c}
\includegraphics[height = 3.5 cm]{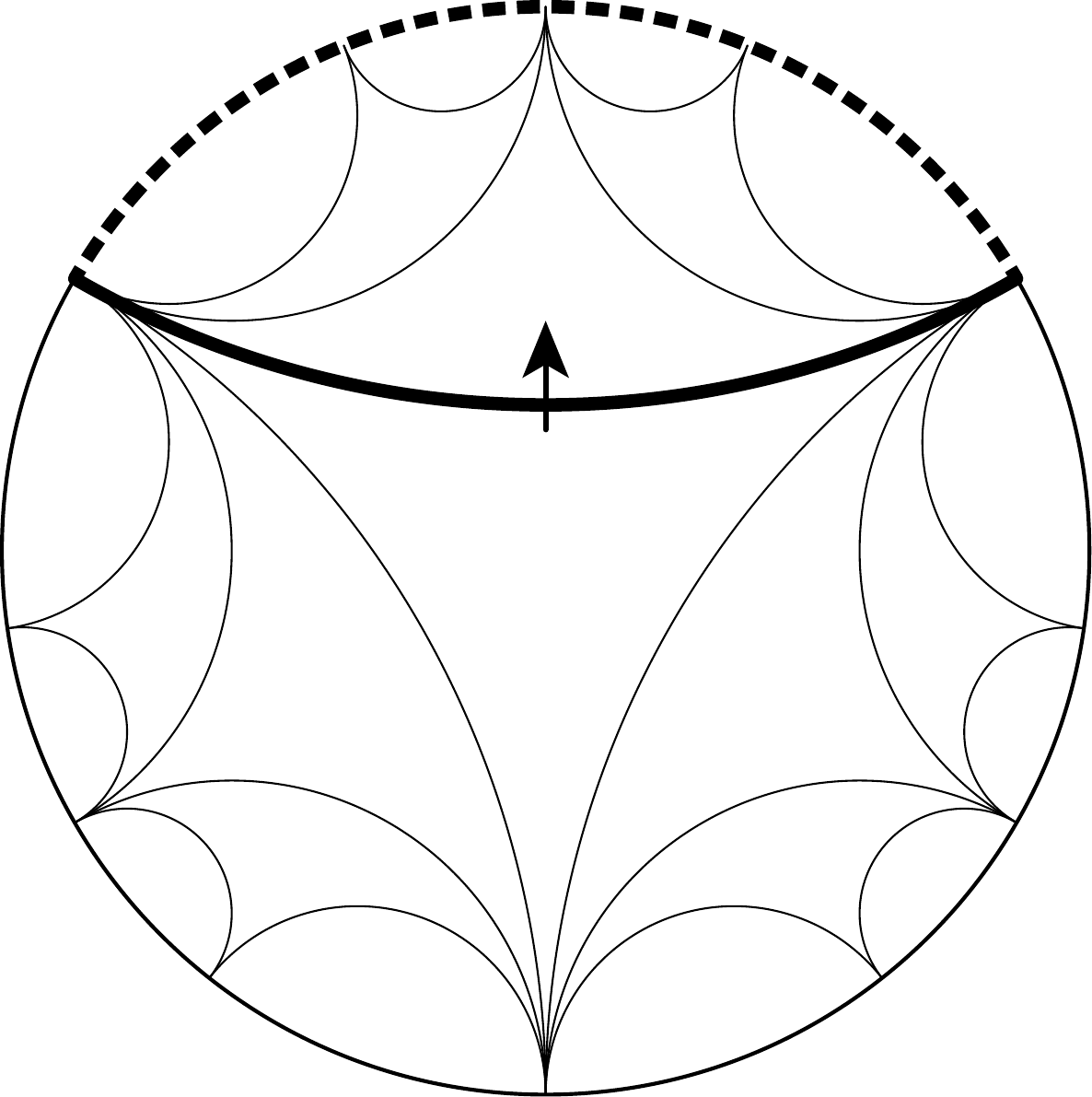}
\end{array}
\]
\caption{The arc $\Delta(e)$ (drawn dashed) in $\Delta_\calV$ corresponding to a co-oriented edge $e$.
In \refsec{BuildingVeeringCircle}, we take the order completion of $\Delta_\calV$ to obtain $\Circle$.
The closure of $\Delta(e)$ is the arc $A(e) \subset \Circle$. }
\label{Fig:CoorientedEdge}
\end{figure}

Suppose that $K$ is a layer of a layering of $\cover{\calV}$.  
The edges in $K$ make it into a copy of the Farey tessellation (\reflem{Disk}).  
The dual graph is an infinite trivalent tree $T$.  
We record the following fact.

\begin{lemma}
\label{Lem:AlternatingPathsShrink}
Suppose that $\rho$ is a ray in $T$.
Let $(e_n)$ be the edges crossed by $\rho$.  
Then $\rho$ turns left and right infinitely often if and only if the nested intersection $\bigcap \Delta(e_n)$ is empty.
\end{lemma}

\begin{proof}
Suppose that $\rho$ turns left and right infinitely often.  
Suppose that $c$ is a cusp of $\cover{\calV}$.  
Let $H$ be the line in $T$ adjacent to $c$.  
Thus no sub-ray of $\rho$ is contained in $H$.  

Let $e_m \in (e_n)$ be the edge which has, among all edges of $(e_n)$ closest to $H$, the largest index.
Thus, either $e_m$ is the last edge (that $\rho$ crosses) that meets $c$, or $e_m$ is the second edge (that $\rho$ crosses) of a Farey triangle, the third edge of which separates $\rho$ from $c$.  
In either case, $c$ is not in the arc $\Delta(e_{m+1})$.  

On the other hand, suppose that $\rho$ turns right only finitely many times. 
Thus there is a cusp $c$ so that infinitely many of the $e_n$ meet $c$.  
Note that $\rho$ turns left through all such $e_n$.  
Appealing to the combinatorics of the Farey tessellation, we deduce that 
$\bigcap \Delta(e_n) = \{c\}$. 
\end{proof}

Suppose that $\rho$ is a train ray in the upper track $\tau^K$.  
Let $(e_n)_{n \in \NN}$ be the sequence of co-oriented edges crossed by $\rho$; 
here the co-orientation of $e_n$ agrees with the orientation of $\rho$.  
Note that the arc $\Delta(e_{n+1})$ is contained in, and shares exactly one endpoint with, the arc $\Delta(e_n)$.  

\begin{corollary}
\label{Cor:RoutesShrink}
Let $(e_n)$ be the edges crossed by a train ray $\rho$ (in some layer).  
Then the nested intersection $\bigcap \Delta(e_n)$ is empty. 
\end{corollary}

\begin{proof}
This follows from Lemmas~\ref{Lem:CannotTurnLeftForever} and~\ref{Lem:AlternatingPathsShrink}.
\end{proof}

Suppose that $S$ is an upper branch line. 
Let $(s_n)_{n \in \ZZ}$ be the track-cusps of $S$, where $s_{n+1}$ is the track-cusp immediately above $s_n$. 
Let $e_n$ be the edge meeting $s_n$. 
We co-orient $e_n$ away from $s_n$. 
Note that  $\Delta(e_{n+1})$ is contained in, and shares exactly one endpoint with, $\Delta(e_n)$.  
See \reffig{UpperGluingAutomaton} for examples.  

\begin{lemma}
\label{Lem:BranchLinesShrink}
Let $(e_n)$ be the edges crossed by the branch line $S$.  Then the nested intersection $\bigcap \Delta(e_n)$ is empty. 
\end{lemma}

The proof of \reflem{BranchLinesShrink} is more difficult than that of \refcor{RoutesShrink}.  This is because a branch line, unlike a train ray, is never contained in a single layer of a layering, by \reflem{BranchLines}\refitm{BranchLinesLayer}.

\begin{proof}[Proof of \reflem{BranchLinesShrink}]
Fix a cusp $c$ of $\cover{\calV}$.  
By \refprop{VeerImpliesExhaust} there is a continent $C$ which meets both $S$ and $c$.  
By \reflem{FiniteInFill} we can assume that $C$ is convex.  
We define a sequence of convex continents $(C_n)$ as follows. 

Let $C_0 = C$.
Let $L_n$ be the upper landscape of $C_n$, let $s_n$ be the track-cusp of $L_n$ meeting $S$, and let $f_n \subset L_n$ be the face containing $s_n$.
Note that the edge of $f_n$ meeting $s_n$ is one of the edges associated to $S$.
Choosing indices correctly, we arrange that $e_n$ (the edge crossed by $S$, defined above) lies in $f_n$.
Note that $e_n$ is co-oriented away from $f_n$.

Let $\tau^n$ be the upper track for $L_n$ and let $R_n \subset L_n$ be the maximal river with $f_n$ as its source.
Since $C_n$ is convex, the mouth, $e'_n$ say, of $R_n$ is coastal.
We co-orient $e'_n$ away from $C_n$.
Let $c'_n$ and $d'_n$ be the cusps at the endpoints of $e'_n$ meeting the left and right riverbanks of $R_n$, respectively.
To obtain $C_{n+1}$ we channelise repeatedly until the face $f_n$ is covered;
by \reflem{RiversSimplify} this process terminates.

The proof of the following claim is omitted;
it is similar to the analysis of forked rivers in \reflem{RiversSimplify}.
See in particular \reffig{River}(A) and (B).  

\begin{claim*}
For all $n > 0$, the interval $\Delta(e_{n})$ shares an endpoint with $\Delta(e'_{n})$ which is not a cusp of $C_{n-1}$.
This new shared endpoint is either $c'_{n}$ or $d'_{n}$ as the edge $e_{n}$ is blue or red. \qed
\end{claim*}

By \refcor{BranchLinesToggle}, both colours appear amongst the edges $(e_n)$. 
Thus, for sufficiently large $m$, neither $c'_0$ nor $d'_0$ lies in the interval $\Delta(e_{m})$.
We deduce that no cusp of $C_0$ lies in $\Delta(e_{m})$.
In particular, $c$ does not lie in $\bigcap \Delta(e_n)$.  
\end{proof} 

\subsection{Completing the cusps}
\label{Sec:BuildingVeeringCircle}

Suppose that $\Delta$ is a countable set and $\calO$ is a dense circular order in the sense of \refdef{Dense}.  
The \emph{order completion} of the pair $(\Delta, \calO)$ is homeomorphic to $S^1$.
See~\cite[Proposition 2.1.7]{Thurston98} or~\cite[Theorem~2.47]{Calegari07}.  

\begin{definition}
\label{Def:VeeringCircle}
Suppose that $M$ is a three-manifold equipped with a transverse veering triangulation $\calV$. 
Let $\Delta_\calV$ be the set of cusps of $\cover{\calV}$.
Let $\calO_\calV$ be the dense circular order given by \refthm{VeerImpliesUnique}.  
Thus the order completion of $(\Delta_\calV, \calO_\calV)$ is a circle $\Circle$; 
we call this the \emph{veering circle}.  
\end{definition}

\begin{definition}
\label{Def:CircleArc}
For $x, y \in \Circle$ we define $[x, y]^{\acw}$ to be the arc of $\Circle$ anticlockwise of $x$ and clockwise of $y$.  
\end{definition}

Note that if $x$ and $y$ are cusps then this arc is the closure of $[x, y]_\Delta^{\acw}$.  
Similarly, we define $A(e) \subset \Circle$ to be the closure of $\Delta(e)$.  
See \reffig{CoorientedEdge}.

We have the following consequence of \refcor{RoutesShrink} and \reflem{BranchLinesShrink}.

\begin{corollary}
\label{Cor:Irrational}
Suppose that $(e_n)$ is the sequence of edges of $\cover{\calV}$ crossed by either 
\begin{itemize}
\item
a train ray $\rho$ (in some layer) or 
\item
a branch line $S$.  
\end{itemize}
Then the intersection $\bigcap A(e_n)$ is a single non-cusp point of $\Circle$. \qed
\end{corollary}

We call this point the \emph{endpoint at infinity} and it is denoted by $\bdy \rho$ or $\bdy S$, respectively.  
An oriented train line $\sigma$ (in some layer) has two endpoints at infinity, denoted $\bdy_\pm \sigma$.

\begin{remark}
\label{Rem:NeighbourhoodBasis}
Suppose that $K$ is a layer of a layering of $\cover{\calV}$.  
Applying \reflem{Disk}, the triangulation of $K$ is a copy of the Farey tessellation.  
Fix $x \in \Circle$.  
We now find a closed neighbourhood basis, about $x$, in $S^1(\calV)$. 
\begin{itemize}
\item
Suppose that $x$ is not a cusp.  
Transversely orient all edges $e$, of $K$, towards $x$.  
By \reflem{AlternatingPathsShrink}, the resulting arcs $\{ A(e) \}$ give the desired basis.
\item
Suppose that $x$ is a cusp.  
Suppose that $e$ and $e'$ are edges of $K$ meeting $x$ that are co-oriented away from each other.  
Now the unions $\{ A(e) \cup A(e') \}$ give the desired basis.
\qedhere
\end{itemize}
\end{remark}

\begin{proof}[Proof of \refthm{VeeringCircle}\refitm{Acts}]
Fix $\gamma \in \pi_1(M)$.  
Suppose that $a, b \in \Delta_\calV$ are distinct cusps.  
Since $\gamma$ preserves the circular order on $\Delta_\calV$ (by \refthm{VeerImpliesUnique}) the image of $[a, b]^{\acw} \cap \Delta_\calV$ is $[\gamma a, \gamma b]^{\acw} \cap \Delta_\calV$.  
This implies that $\gamma$ is well-defined on the non-cusp points of $\Circle$. 
and thus sends arcs to arcs.  Thus $\gamma$ is continuous 
and orientation preserving, and so is $\gamma^{-1}$.  
Thus $\gamma$ is a homeomorphism.  
Finally, the action is faithful because $\Circle$ contains $\Delta_\calV$.
\end{proof}

\subsection{Parabolics}
\label{Sec:CrossingParabolics}

In this section, we assume that $\calV$ is finite.
For the proof of \refthm{VeeringCircle}\refitm{Dense}, we must analyse the action of cusp stabilisers in $\pi_1(M)$. 
Suppose that $\check{c}$ is a cusp of $\calT$.
Let $\check{N} \subset M$ be the upper cusp neighbourhood of $\check{c}$.
Fix a basepoint $\check{p} \in \check{N}$.
We lift to obtain a cusp $c \in \Delta_\calV$ and a basepoint $p \in N = N^c$.
This gives an isomorphism between $\pi_1(M)$ and the deck group of $\cover{M}$, and thus an isomorphism between the $\Stab(c)$ and $\pi_1(\check{N}) \isom \ZZ^2$.
We suppress this from our notation.

\begin{definition}
\label{Def:BranchSlope}
Let $\beta \in \Stab(c)$ be the based homotopy class whose free homotopy class contains any (thus all) oriented branch loops $\check{S}$ in $\bdy \check{N}$.
In an abuse of notation we again call $\beta$ the \emph{branch slope}.
\end{definition}

By \reflem{BranchLines}\refitm{Cyclic} we may index the branch lines of $N = N^c$, say $(S_i)_{i \in \ZZ}$, so that $S_i$ and $S_{i+1}$ are adjacent on $\bdy N$.
We choose the indexing so that the $S_i$ march anticlockwise around $N$ as viewed from above.
We have the following. 

\begin{lemma}
\label{Lem:CrossingParabolics}
\mbox{}
\begin{enumerate}
\item
\label{Itm:TipsOfCrown}
$\Circle - \{ c \} = \bigcup_i \, [\bdy S_i, \bdy S_{i+1}]^{\acw}$  
\item
\label{Itm:NonBranchSlope}
Suppose that $\gamma \in \Stab(c) - \subgp{\beta}$, where $\beta$ is the branch slope.
Then there is a non-zero integer $k$ so that, for all $i$, we have $\gamma(\bdy S_i) = \bdy S_{i + k}$.
Also, if $x$ is any point of $\Circle$ then $\gamma^n(x)$ converges to $c$ as $n$ tends to infinity.
\end{enumerate}
\end{lemma}

\begin{figure}[htbp]
\labellist
\small\hair 2pt
\pinlabel {$c$} [t] at 205 0 
\pinlabel {$\bdy S_{-2}$} [tl] at 315 34
\pinlabel {$\bdy S_{-1}$} [l] at 405 162
\pinlabel {$s_{-1}$} [r] at 312 150
\pinlabel {$\bdy S_0$} [br] at 130 395
\pinlabel {$s_0$} [t] at 178 218  
\pinlabel {$\bdy S_1$} [r] at 5 172
\pinlabel {$s_1$} [l] at 98 150
\pinlabel {$\bdy S_2$} [tr] at 94 37
\pinlabel {$\bdy S_3$} [t] at 115 19
\endlabellist
\includegraphics[width = 0.6\textwidth]{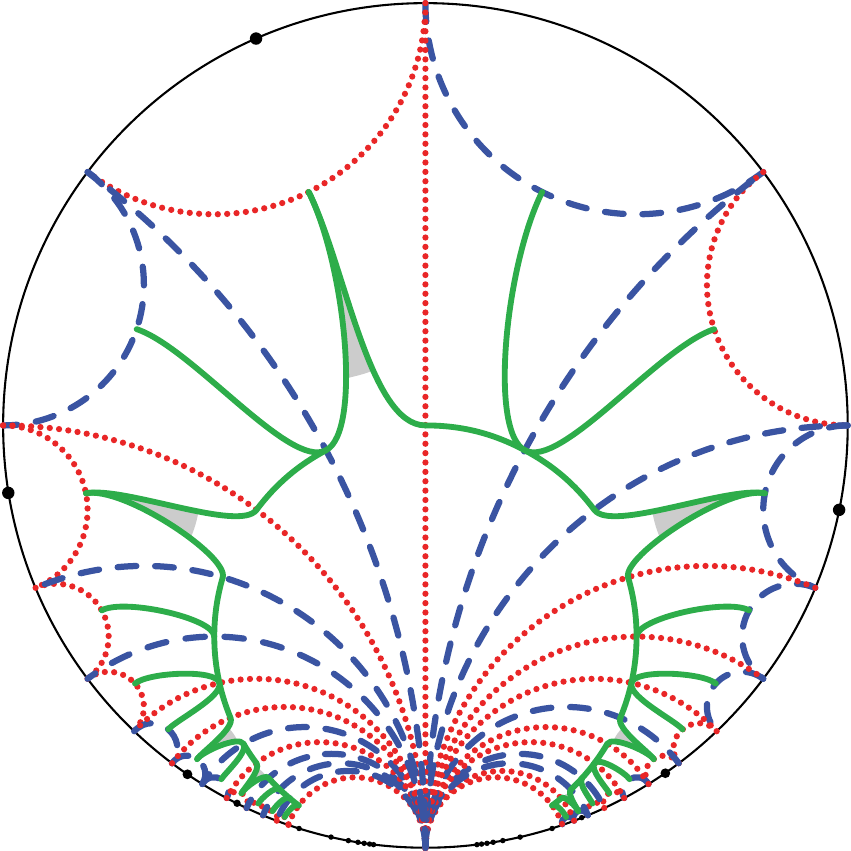}
\caption{The points $(\bdy S_i)$ converge to $c$ as $i$ tends to infinity. Compare with \reffig{NeighbourhoodLayer}.}
\label{Fig:TipsOfCrown}
\end{figure}

\begin{proof}
Suppose that $K$ is a layer of a layering of $\cover{\calV}$.
By \reflem{BranchLines}\refitm{BranchLinesLayer} there is a unique track-cusp $s_i$ of $S_i$ contained in $K$. 
By \reflem{BranchLines}\refitm{NeighbourhoodLayer}, the track-cusp $s_i$ meets an edge $e_i \subset K$ which we co-orient away from $c$. 
See \reffig{NeighbourhoodLayer}.
By \refcor{Irrational} we deduce that $\bdy S_i \in A(e_i)$.
From \refrem{NeighbourhoodBasis} and \reflem{BranchLines}\refitm{BranchStrip} we deduce that the points $\bdy S_i$ tend to $c$ as $i$ tends to infinity.
See \reffig{TipsOfCrown}.
This proves \refitm{TipsOfCrown}. 

Since $\gamma$ is not a power of $\beta$, it crosses $\beta$. 
Suppose that there are $n$ branch loops in $\bdy \check{N}$. 
Let $m$ be the algebraic intersection number between $\gamma$ and $\beta$. 
Set $k = mn$. 
We now deduce \refitm{NonBranchSlope} from \refitm{TipsOfCrown}.
\end{proof}

\begin{proof}[Proof of \refthm{VeeringCircle}\refitm{Dense}]
Fix any point $x \in \Circle$.  
Also, fix a cusp $c \in \Delta_\calV$.  
Suppose that $\gamma \in \Stab(c)$ is not a power of the branch slope.  
By \reflem{CrossingParabolics}\refitm{NonBranchSlope} the sequence $\gamma^n(x)$  converges to $c$.  
Thus the orbit of $x$ accumulates on $\Delta_\calV$; 
the latter is dense in $\Circle$ by construction.  
\end{proof}

\chapter{Laminations}
\label{Cha:LaminationsAlone}

In this section we build the \emph{upper lamination} $\Lambda^\calV$ in $\Circle$.  
The \emph{lower lamination} $\Lambda_\calV$ is constructed similarly. 

\begin{theorem}
\label{Thm:Laminations}
Suppose that $M$ is a three-manifold equipped with a transverse veering triangulation $\calV$. 
Then there is a lamination $\Lambda^\calV$ in $\Circle$ with the following properties.
\begin{enumerate}
\item 
\label{Itm:LaminationsInS1}
The lamination  $\Lambda^\calV$ is $\pi_1(M)$--invariant.
\item 
\label{Itm:LaminationsInM}
The lamination $\Lambda^\calV$ suspends to give a $\pi_1(M)$--invariant lamination $\cover{\Sigma}^\calV$ in $\cover{M}$; this descends to $M$ to give a lamination $\Sigma^\calV$ which 
\begin{enumerate}
\item
\label{Itm:Carried}
is fully carried by the upper branched surface $B^\calV$,
\item
\label{Itm:Types}
has only plane, annulus, and M\"obius band leaves, and
\item
\label{Itm:Essential}
is essential.
\end{enumerate}
\item 
\label{Itm:LaminationsUnique}
Suppose that $\Sigma$ is a lamination fully carried by $B^\calV$.
Then, after collapsing parallel leaves, $\Sigma$ is tie-isotopic to $\Sigma^\calV$.
\end{enumerate}
There is also a lamination $\Lambda_\calV$ with the corresponding properties with respect to $B_\calV$. 
\end{theorem}

\noindent
We prove this in parts, giving the necessary definitions as we proceed.

\subsection{Laminations in the circle}
\label{Sec:Mobius}

Following Thurston~\cite[page~187]{Thurston78}, we define $\calM$ to be the \emph{M\"obius band past infinity}: 
that is, $S^1 \cross S^1$, minus the diagonal, and quotiented by the symmetry that interchanges the two factors.  
A \emph{leaf at infinity} $\lambda = \lambda(a, b)$ is the point of $\calM$ with \emph{endpoints} $a$ and $b$.  

Suppose that $\mu = \lambda(c, d)$ is another leaf.  
We say that $\lambda$ and $\mu$ are \emph{asymptotic} if they share an endpoint.  
We say that $\lambda$ and $\mu$ are \emph{linked} if 
\begin{itemize}
\item
the points $a$, $b$, $c$, and $d$ are all distinct and 
\item
the points $c$ and $d$ lie in different components of $S^1 - \{a, b\}$.  
\end{itemize}
We say that $\lambda$ and $\mu$ are \emph{unlinked} if they are not linked.

\begin{definition}
\label{Def:LaminationInS1}
A \emph{lamination at infinity} $\Lambda$ in $S^1$ is a closed subset of $\calM$ 
where no pair of leaves $\lambda, \mu \in \Lambda$ are linked.
\end{definition}

\noindent 
Recall that $S^1$ is the Gromov boundary of the hyperbolic plane $\HH^2$ and that $\calM$ parametrises the space of unoriented bi-infinite geodesics in $\HH^2$. 
This justifies the terminology of \refdef{LaminationInS1}.

\subsection{From cusps to train lines}
\label{Sec:ConnectingArc}

Fix $c$ a cusp, set $N = N^c$, and let $S$ be an upper branch line of $N$.  
Suppose that $s$ is a track-cusp of $S$.
Suppose that $f$ is the face of $\cover{\calV}$ containing $s$.  
Note that $c$ is a vertex of $f$.  
We define a \emph{connecting arc} $\ell^f(c, s) \subset f$ to be 
\begin{itemize}
\item 
a smooth arc in $f$, 
\item 
connecting $c$ to the switch in $s$, and 
\item 
meeting $\tau^f$ only at that switch.  
\end{itemize}
We take $\ell^f(s,c) = \ell^f(c,s)$.  
If $K$ is a layer containing $f$, then we set $\ell^K(c,s) = \ell^f(c,s)$.  
See \reffig{ConnectingArc}.

\begin{figure}[htbp]
\labellist
\small\hair 2pt
\pinlabel {$c$} at -5 0 
\pinlabel {$s$} at 73 37
\endlabellist
\includegraphics[height = 3.5 cm]{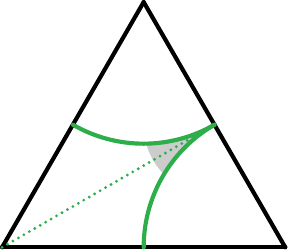}
\caption{The connecting arc (dotted) from the cusp $c$ to the track-cusp~$s$.}
\label{Fig:ConnectingArc}
\end{figure}

\begin{definition}
\label{Def:CuspLine}
Suppose that $K$ is a layer of a layering.  
\reflem{BranchLines}\refitm{BranchLinesLayer} implies that the upper branch line $S$ meets $K$ in a track-cusp, say $s$.  
We define $\ell^K(s, S) \subset K$, the \emph{cusp train ray} from $s$ to $\bdy S$, to be a train ray carried by $\tau^K$ starting at $s$ and with endpoint at infinity $\bdy S$.  

We define $\ell^K(c, S) \subset K$, the \emph{cusp line} from $c$ to $\bdy S$, to be the union $\ell^K(c, S) = \ell^K(c, s) \cup \ell^K(s, S)$.
\end{definition}

\begin{lemma}
\label{Lem:CuspLines}
Suppose that $c \in \Delta_\calV$ is a cusp. 
Suppose that $K$ is a layer of a layering $\calK$.  
Suppose that $S$ is an upper branch line in the boundary of $N^c$.
Then a cusp line $\ell^K(c, S)$ exists.  
It is properly embedded in $K$ and is uniquely determined (up to isotopy of the connecting arc) by $c$ and $S$. 
\end{lemma}

\begin{proof}
Choose indices in $\calK = (K_i)$ so that $K_0 = K$.
We define $\tau^i$ to be the upper track for $K_i$.
Also, applying \reflem{BranchLines}\refitm{BranchLinesLayer}, 
we choose the indexing of the track-cusps of $S$ so that $s_i$ is the track-cusp of $S$ in $K_i$.
Note that here we allow $s_{i+1} = s_i$.

Let $\Phi^i \from \tau^{i+1} \to \tau^i$ be the \emph{carrying map}; 
it is the identity everywhere except at one watershed which is folded by $\Phi^i$ to produce a sink.
See \reffig{UpperBranchedSurface}.
Note that $\Phi^i$ sends switches to switches.
We extend $\Phi^i$ to obtain a map from $K_{i+1}$ to $K_i$.
It is a homeomorphism on all components of $K_{i+1} - \tau^{i+1}$ except for two.
In those two it crushes two track-cusps which are then carried by $\tau^i$.
In particular, if $s_{i+1} \neq s_i$ then $\Phi^i(s_{i+1}) \neq s_i$.
Again, see \reffig{UpperBranchedSurface}.

Fix $i < j$ and define $\Phi^{i,j} \from \tau^j \to \tau^i$ by composing carrying maps.
When $i = j$ we take $\Phi^{i,i}$ to be the identity map. 

For each $i$, fix a connecting arc $\ell^i = \ell^i(c, s_i) \subset K_i$.
We arrange matters so that $\Phi^{i, i+1}(\ell^{i+1})$ contains $\ell^i$.
Set $\ell^{i,j} = \Phi^{i,j}(\ell^j)$.
Thus, for any $i \leq j \leq k$ we have $\ell^{i,j} \subset \ell^{i,k}$.
Furthermore, the containment is proper whenever $s_k \neq s_j$.

Set
\[
\ell^i(c, S) = \bigcup_{i \leq j} \ell^{i,j} \quad \text{and} \quad \ell^i(s_i,S) = \ell^i(c, S) - \ell^i(c, s_i)
\]

\begin{claim*}
For all $i$, the union $\ell^i(c, S)$ is a cusp line. 
Moreover, $\ell^i(s_i,S)$ is a cusp train ray.
\end{claim*}

\noindent
This, and setting $\ell^K(c, S) = \ell^0(c, S)$, proves the lemma. 

We now prove the claim. 
Since $\ell^i(s_i, S)$ is carried by $\tau^i$ it is a one-manifold.
The non-compact end of $\ell^i(s_i, S)$ is properly embedded ray in $K_i$; 
this is because it is an ascending union that does not stabilise and because $\tau^i$ is properly embedded in $K_i$. 
Since $\ell^i(c, S)$ contains $\ell^i(c, s_i)$ the former is a properly embedded line in $K_i$. 

All that is left is to show that the endpoint at infinity for $\ell^i(s_i, S)$ is $\bdy S$. 
Note that $\Phi^{i,j}(\ell^j(c, S)) = \ell^i(c, S)$.
Thus $\ell^i(c, S)$ and $\ell^j(c, S)$ are identical in $\cover{\calV}$ 
away from at most $j - i$ tetrahedra (those between $K_j$ and $K_i$).
We deduce that $\bdy \ell^i(c, S) = \bdy \ell^j(c, S)$ for all $i$ and $j$.

Take $e_j \subset K_j$ to be the edge pointed at by the track-cusp $s_j$. 
Recall that by \refcor{Irrational}, the unique point of $\bigcap_{i \leq j} A(e_j)$ is $\bdy S$. 
Since the cusp train ray $\ell^j(s_j, S)$ crosses $e_j$, we deduce that $\bdy \ell^j(s_j, S)$ lies in $A(e_j)$.  
We deduce that $\bdy \ell^i(s_i, S) = \bdy S$.  
This proves the claim and thus the lemma.
\end{proof}

With notation as above, we define $\lambda(c, \bdy S)$ to be an \emph{(upper) cusp leaf} associated to $S$.  
We will abuse notation and write $\lambda(c, S)$ for $\lambda(c, \bdy S)$.


\begin{lemma}
\label{Lem:Unlinked}
Any pair of upper cusp leaves $\lambda(c, S)$ and $\lambda(d, T)$ are unlinked.  
\end{lemma}

\begin{proof}
If $c = d$ then $\lambda(c, S)$ and $\lambda(c, T)$ are asymptotic and thus unlinked.

Now assume that $c \neq d$.  
Fix a layer $K$ of a layering.  
Let $\ell^K(c, s)$ and $\ell^K(d, t)$ be the associated connecting arcs.  
These are disjoint.  
This also holds for all layers above $K$; 
applying the carrying maps we deduce that $\ell^K(c,S)$ does not cross $\ell^K(d,T)$. 
\end{proof}


\begin{lemma}
\label{Lem:BranchLinesNotAsymptotic}
Suppose that $S$ and $T$ are upper branch lines.  
If $\bdy S = \bdy T$ then $S = T$.
\end{lemma}

\begin{proof}
Let $(K_i)$ be a layering of $\cover{\calV}$.  Let $\tau^i$ be the upper track for $K_i$.  Again, applying \reflem{BranchLines}\refitm{BranchLinesLayer}, we index the track-cusps of $S$ so that $s_i$ lies in $K_i$; we define $t_i$ similarly.  

Let $\ell^0(s_0,S)$ and $\ell^0(t_0,T)$ be the cusp train rays given by \reflem{CuspLines}. 
Since $\bdy S = \bdy T$, the cusp train rays eventually cross the same collection of edges in $K_0$.  

If neither of $\ell^0(s_0, S)$ and $\ell^0(t_0, T)$ is contained in the other, then their union is a sub-tree $Y \subset \tau^0$.  
Note that $Y$ has a single trivalent vertex, smoothed according to $\tau^0$.  
There is a track-cusp, say $u_0$, at this vertex. 
Let $U$ be the branch line running through $u_0$.  
Let $a$, $b$ and $c$ be the cusps associated to $S$, $T$, and $U$.  
Appealing to \reflem{Unlinked}, we find that the cusp leaf $\lambda(c, U)$ does not link either $\lambda(a, S)$ or $\lambda(b, T)$.  
Thus $\bdy S = \bdy U = \bdy T$ and so $\ell^0(u_0, U) \subset \ell^0(t_0, T)$. 

Relabelling if necessary, we thus restrict to the case where $\ell^0(s_0, S) \subset \ell^0(t_0, T)$.  
If they are equal, we are done.  
For a contradiction, suppose that they are not equal.  
Consulting \reffig{UpperGluingAutomaton}, we deduce that $\ell^i(s_i, S) \subset \ell^i(t_i, T)$ for all $i$.  
Thus, the track-cusps $t_i$ follow behind the track-cusps $s_i$ forever.  
Let  $\ell^i(t_i, s_i) = \ell^i(t_i, T) - \ell^i(s_i, S)$ be the train interval in $\tau^i$ connecting $t_i$ to $s_i$. 

For each train track $\tau^i$, we measure the \emph{trailing distance} $d_i$, as follows.  
This is the number of track-cusps incident to $\ell^i(t_i, s_i)$ (including $s_i$ but excluding $t_i$) that point in the same direction as $s_i$.  
\reffig{ChasingTrackCusps} shows an example where $d_i$ equals four. 

\begin{figure}[htbp]
\labellist
\small\hair 2pt
\pinlabel {$t_i$} at -5 30
\pinlabel {$1$} at 150 20
\pinlabel {$2$} at 53 40
\pinlabel {$3$} at 29 20
\pinlabel {$s_i$} at 184 20
\endlabellist
\includegraphics[width=0.5\textwidth]{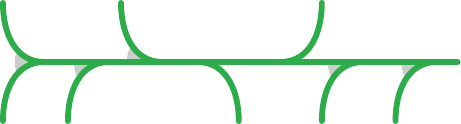}
\caption{Counting the distance between track-cusps $t_i$ and $s_i$.}
\label{Fig:ChasingTrackCusps}
\end{figure}

We claim that $d_{i+1} \leq d_i$. 
To see this, note that as we move up through the layers, track-cusps move by splitting past each other.  
Thus, track-cusps pointing in the same direction as $s_i$ may leave $\ell^i(t_i, s_i)$ but no such track-cusps may enter. 

Thus, above some layer, say $K_j$, the trailing distance $d_i$ becomes constant.  
So for $i > j$ all of the track-cusps in $\ell^i(t_i, s_i)$ follow $s_i$ forever.  
We now consider the branch line $V$ with track-cusps $v_i$ that follow immediately behind $s_i$ in $\ell^i(t_i, s_i)$; that is, at trailing distance one.

\begin{figure}[htbp]
\centering
\subfloat[A strip of triangles containing $\ell^i(v_i, s_i)$.]{
\labellist
\small\hair 2pt
\pinlabel {$v_i$} [r] at 90 70
\pinlabel {$s_i$} [r] at 612 70
\endlabellist
\includegraphics[width=0.98\textwidth]{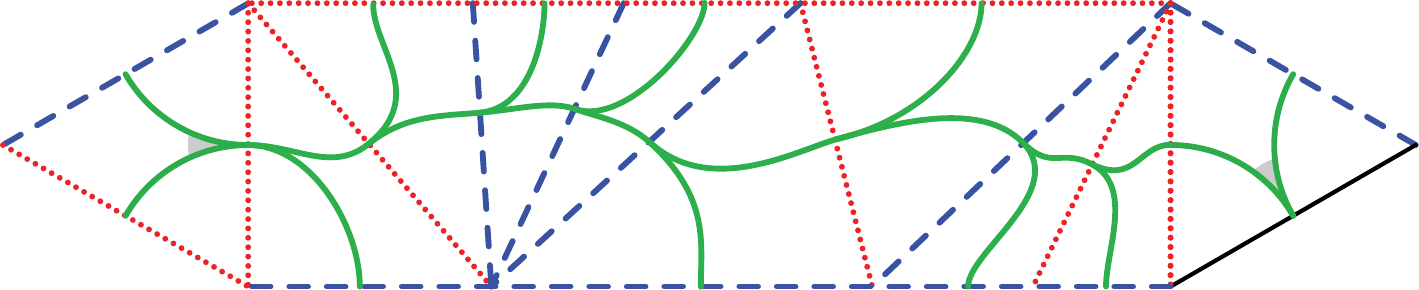}
\label{Fig:ImmediateFollower}
}

\subfloat[Before landfilling a tetrahedron $t$ above $s_i$.]{
\labellist
\small\hair 2pt
\pinlabel {$s_{i}$} [br] at 55 85
\endlabellist
\includegraphics[width=0.22\textwidth]{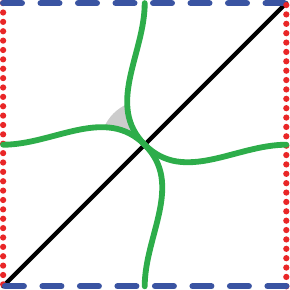}
\label{Fig:ImmediateFollowerSplit}
}
\qquad
\subfloat[If the upper edge of $t$ is blue.]{
\labellist
\small\hair 2pt
\pinlabel {$s_{i+1}$} [br] at 65 52
\endlabellist
\includegraphics[width=0.22\textwidth]{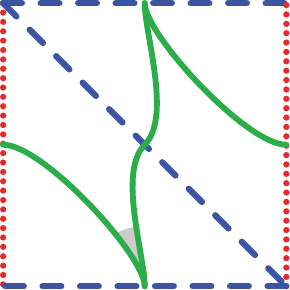}
\label{Fig:ImmediateFollowerSplitOpp}
}
\qquad
\subfloat[If the upper edge of $t$ is red.]{
\labellist
\small\hair 2pt
\pinlabel {$s_{i+1}$} [br] at 105 77
\endlabellist
\includegraphics[width=0.22\textwidth]{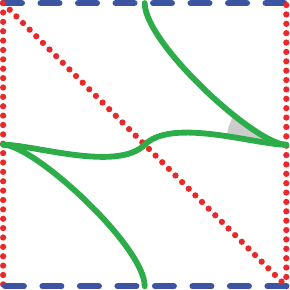}
\label{Fig:ImmediateFollowerSplitSame}
}
\caption{}
\label{Fig:ImmediateFollowerChanges}
\end{figure}

Let $P_i$ be the minimal strip of triangles containing $\ell^i(v_i, s_i)$.  
Breaking symmetry, we assume that $\ell^j(v_j, s_j)$ enters the face containing $s_j$ through a red edge.  
\reffig{ImmediateFollower} shows an example of what $P_j$ could look like.   

For $i \geq j$, as we move from layer $K_i$ to $K_{i+1}$ there are three possibilities; the tetrahedron $t$ between them either:
\begin{itemize}
\item misses $P_i$, 
\item is layered on the edge that $s_i$ points at, or 
\item is layered on the edge that $v_i$ points at. 
\end{itemize}
The tetrahedron $t$ cannot meet any other edges of $P_i$ as they are not sinks. 

If we layer onto the edge meeting $s_i$, as shown in \reffig{ImmediateFollowerSplit}, then there are two possibilities, depending on the colour of the upper edge of $t$. 
These are shown in Figures~\ref{Fig:ImmediateFollowerSplitOpp} and~\ref{Fig:ImmediateFollowerSplitSame}.  
Either the length of $P_{i+1}$ equals that of $P_i$ or goes up by one, with the new triangle being a majority red triangle.  
In this case, $P_{i+1}$ has one more internal red edge than $P_i$. 

If we layer onto the edge meeting $v_i$, then the track-cusp $v_i$ splits past the first track-cusp pointing backwards along $P_i$.  
Thus $P_{i+1}$ is one triangle shorter than $P_i$.  
Note then, that there is no way for $P_{i+1}$ to have more internal blue edges than $P_i$ has.  
Thus the branch line $V$ visits only a finite number of blue edges.  
This contradicts \refcor{BranchLinesToggle}.
\end{proof}

\subsection{Building the laminations}

Suppose that $S$ and $T$ are adjacent upper branch lines lying in $N = N^c$.  
Fix a layer $K$ of a layering; let $s$ and $t$ be the upper track-cusps in $K$ meeting $S$ and $T$ respectively (\reflem{BranchLines}\refitm{BranchLinesLayer}).  
Since $S$ and $T$ are adjacent, there is a train interval $\ell^K(s, t) \subset \tau^K \cap \bdy N$ connecting $s$ to $t$ (\reflem{BranchLines}\refitm{BranchStrip}).  
See \reffig{NeighbourhoodLayer}.

\begin{definition}
\label{Def:BoundaryLeaf}
The \emph{boundary train line} in $K$ between $S$ and $T$ is
\[
\ell^K(S, T) = \ell^K(S, s) \cup \ell^K(s, t) \cup \ell^K(t, T)
\]
The \emph{(upper) boundary leaf} is $\lambda(S, T) = \lambda(\bdy S, \bdy T)$. 
Suppose now that $R$ is the other branch line in $N$ adjacent to $S$. 
Then $\lambda(R,S)$ is \emph{adjacent} to $\lambda(S, T)$.
\end{definition}

\begin{definition}
\label{Def:UpperLamination}
Let $\Mobius$ be the M\"obius band past infinity for $\Circle$, as defined in \refsec{Mobius}.
We define $\Lambda^\calV \subset \Mobius$, the \emph{upper lamination} for $\calV$, to be the closure of the union of upper boundary leaves:
\[
\Lambda^\calV = \closure{\cup \, \lambda(S, T)}
\]
Here the pairs $(S, T)$ range over all adjacent upper branch lines. 
We equip $\Lambda^\calV \subset \Mobius$ with the subspace topology. 
Non-boundary leaves of $\Lambda^\calV$ are called \emph{interior leaves}.  
Similarly, we define $\Lambda_\calV$ using the lower boundary leaves. 
\end{definition}


\begin{remark}
We justify the names \emph{boundary} and \emph{interior}, for leaves of the laminations $\Lambda^\calV$ and $\Lambda_\calV$, in \reflem{Approach}.
\end{remark}

\begin{definition}
Any leaf $\lambda \in \Lambda^\calV$ separates $\Circle$ into two components, which we call the two \emph{sides} of $\lambda$. 
\end{definition}

Here we gather together the basic properties of $\Lambda^\calV$.

\begin{lemma}
\label{Lem:Laminations}
The upper lamination $\Lambda^\calV$ has the following properties. 
\begin{enumerate}
\item
\label{Itm:NoLinking}
It is a lamination.
\item
\label{Itm:PiOneInvariant}
It is $\pi_1(M)$--invariant.  
\item 
\label{Itm:LeavesAreCarried}
For any leaf $\lambda \in \Lambda^\calV$ and for any layer $K$ of any layering, 
there is a train line $\ell$ carried by $\tau^K$ with the same endpoints as $\lambda$.  
\item 
\label{Itm:FullyCarried}
For any layer $K$ of any layering and for any branch $b$ of $\tau^K$ there is a leaf $\lambda \in \Lambda^\calV$ so that the corresponding train line runs along $b$.
\item
\label{Itm:Irrational}
No endpoint of any leaf is a cusp.
\item
\label{Itm:Asymptotic}
A pair of distinct leaves $\lambda, \lambda' \in \Lambda^\calV$ share an endpoint if and only if they are adjacent boundary leaves.
\end{enumerate}
Replacing upper by lower, we obtain the same properties for $\Lambda_\calV$.
\end{lemma}

\begin{proof}
\noindent
\begin{enumerate}
\item 
Being linked is an open property for pairs of leaves in $\Mobius$.  
So it suffices to prove that any pair of upper boundary leaves $\lambda(S, T)$ and $\lambda(U, V)$ are unlinked.
We fix a layer $K$ and argue that the associated boundary train lines $\ell^K(S, T)$ and $\ell^K(U, V)$ are unlinked.  Suppose that $S, T \subset N^c$ and $U, V \subset N^d$.
We consider the unions $\ell^K(S, c) \cup \ell^K(c, T)$ and $\ell^K(U, d) \cup \ell^K(d, V)$.
If $c = d$ then we are done by the definition of adjacency and by \reflem{BranchLines}\refitm{Cyclic}.
If $c \neq d$ then, by \reflem{Unlinked}, the set $\{ \bdy U, d, \bdy V\}$ is contained in exactly one of the components of $\Circle - \{ \bdy S, c, \bdy T\}$.

\item 
Suppose that $\gamma \in \pi_1(M)$ is a deck transformation.
Suppose that $S$ and $T$ are upper branch lines.
Then $S$ and $T$ are adjacent if and only if $\gamma(S)$ and $\gamma(T)$ are adjacent.
Thus the collection of upper boundary leaves is $\pi_1(M)$--invariant and hence so is its closure.

\item
If $\lambda = \lambda(S, T)$ is a boundary leaf, 
then we take $\ell = \ell^K(S, T)$ to be the associated boundary train line.

Suppose instead that $\lambda$ is an interior leaf.
Orient $\lambda$ and let $\bdy_- \lambda$ and $\bdy_+ \lambda$ be its endpoints.
Fix a sequence $( \lambda_n )$ of oriented upper boundary leaves converging (respecting orientation) to $\lambda$.
Pass to a subsequence to ensure that $\bdy_+ \lambda_n$ converges to $\bdy_+ \lambda$ monotonically from one side.
Applying \refitm{NoLinking}, we deduce that $\bdy_- \lambda_n$ converges to $\bdy_- \lambda$ monotonically and from the same side.
Let $\ell_n \subset \tau^K$ be the boundary train line associated to $\lambda_n$.  

We now apply \refrem{NeighbourhoodBasis}:
pick nested neighbourhood bases of $\bdy_- \lambda$ and $\bdy_+ \lambda$ respectively.
If either is a cusp then its bases require a pair of edges;
we only keep those edges on the same side as the $\lambda_n$.
This gives us two collections of ``nested'' edges $( e^-_i )$ and $( e^+_j )$.
We pass to subsequences repeatedly (both of the $\lambda_n$ and of the $e^\pm_i$) to arrange that $\lambda_j$ links both $e^-_j$ and $e^+_j$.
Recall that $\tau^K$ is a tree.
Thus the boundary train line $\ell_j$ crosses $e^-_j$ and $e^+_j$; 
let $L_j$ be the resulting unique train interval $L_j$ running from $e^-_j$ to $e^+_j$.
Note that $L_j$ is a subinterval of all $\ell_k$ for $k \geq j$.
Also, since the two collections of edges are each nested, we have $L_j \subset L_{j+1}$.
Thus $\ell = \bigcup_j L_j$ is the desired train line. 

\item
By \reflem{BranchLines} there are upper branch lines $S$ and $T$ so that $b$ lies on the boundary train line $\ell^K(S, T)$.
The desired leaf of $\Lambda^K$ is then $\lambda(S, T)$.

\item
This follows from property \refitm{LeavesAreCarried} and \refcor{Irrational}. 

\item
The backwards direction follows from the definition of adjacency given in \refdef{BoundaryLeaf}.

For the forward direction, suppose that $\lambda$ and $\lambda'$ are distinct leaves of $\Lambda^\calV$ that share a single common endpoint, $x$.
Let $y$ and $y'$ be the other endpoints of $\lambda$ and $\lambda'$, respectively. 

Suppose that both $\lambda$ and $\lambda'$ are boundary leaves.
\reflem{BranchLinesNotAsymptotic} implies that there is a branch line $S$ so that $x=\bdy S$.
Thus $\lambda$ and $\lambda'$ are adjacent.

Suppose instead that $\lambda$ is an interior leaf.
Let $K$ be a layer of a layering of $\cover{\calV}$.
By \refitm{LeavesAreCarried}, there are train lines $\ell$ and $\ell'$ carried by $\tau^K$ with the same endpoints as $\lambda$ and $\lambda'$ respectively.
Thus the union of $\ell$ and $\ell'$ forms a sub-tree $Y \subset \tau^K$ with a single trivalent vertex, smoothed according to $\tau^K$.
See \reffig{NonAsymptotic}.

\begin{figure}[htbp]
\labellist
\small\hair 2pt
\pinlabel {$y$} at -5 50
\pinlabel {$\bdy S'$} at 15 15
\pinlabel {$t$} at 12 43
\pinlabel {$s$} at 118 43
\pinlabel {$c$} at 109 -5
\pinlabel {$y'$} at 126 -5
\pinlabel {$x$} [l] at 175 50
\endlabellist
\includegraphics[width=0.5\textwidth]{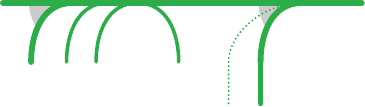}
\caption{One possible picture for an asymptotic pair of train lines.}
\label{Fig:NonAsymptotic}
\end{figure}

There is a track-cusp $s$ at the vertex of $Y$.
Let $S$ be the branch line containing $s$ and let $c$ be its associated cusp.
The points $y$ and $y'$ are on opposite sides of $\ell^K(c,S)$.
Let $S'$ be the branch line adjacent to $S$ on the side containing $y$.
Then $\ell^K(S, S')$ is a boundary train line.
By \refitm{NoLinking}, it does not link either $\ell$ or $\ell'$.
Thus $\ell^K(S, S')$ has one endpoint at $x = \bdy S$ and the other at $\bdy S'$.
Since $\lambda$ is not a boundary leaf, we have that $\ell \neq \ell^K(S, S')$.
Thus, the $\bdy S'$ is not equal to $y$.
We deduce that in $\Circle$, the point $\bdy S'$ is separated from $x$ by $\{y, y'\}$. 

As in the above argument, the union of $\ell$ and $\ell^K(S, S')$ forms a sub-tree giving a new track-cusp, $t$.
Let $T$ be the branch line passing through $t$.
Again, $x = \bdy T$.
This implies that $\bdy S = \bdy T$.
Since $s \neq t$, \reflem{BranchLines}\refitm{BranchLinesLayer} implies that $S \neq T$.
This contradicts \reflem{BranchLinesNotAsymptotic}.
\qedhere
\end{enumerate}
\end{proof}

\begin{proof}[Proof of \refthm{Laminations}\refitm{LaminationsInS1}]
Parts~\refitm{NoLinking} and~\refitm{PiOneInvariant} of \reflem{Laminations} give the desired statement.
\end{proof} 

We now justify the names of boundary and interior leaves of $\Lambda^\calV$.  
The following lemma proves that $\Lambda^\calV$ is \emph{perfect}: closed and without isolated leaves. 

\begin{lemma}
\label{Lem:Approach}
For either side of an interior leaf $\lambda \in \Lambda^\calV$, there is a sequence of boundary leaves $\lambda_k$ that converge to $\lambda$ from that side.
A boundary leaf $\lambda \in \Lambda^\calV$ is also a limit of boundary leaves $\lambda_k$, but only on the side not containing its cusp.
Similar properties hold for $\Lambda_\calV$.
\end{lemma}

\begin{proof}
Let $K$ be a layer of a layering.
Let $\ell$ be the oriented train line carried by $\tau^K$ with the same endpoints as $\lambda$, as given by \reflem{Laminations}\refitm{LeavesAreCarried}.
Suppose that $s_0$ is a track-cusp on the side of $\ell$ from which we wish to approach.
This exists by \reflem{CannotTurnLeftForever}.
See \reffig{Approach}.

\begin{figure}[htbp]
\labellist
\small\hair 2pt
\pinlabel {$s_2$} at 14 31
\pinlabel {$\bdy S_1$} at 16 2
\pinlabel {$\bdy S'_0$} at 47 2
\pinlabel {$s_0$} at 93 31
\pinlabel {$\bdy S_0$} at 160 2
\pinlabel {$\bdy S'_1$} at 177 2
\pinlabel {$s_1$} at 160 31
\pinlabel {$\bdy S_2$} at 226 2
\pinlabel {$s_3$} at 226 31
\endlabellist
\includegraphics[width=0.9\textwidth]{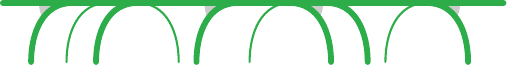}
\caption{Approaching a leaf by a sequence of boundary leaves.}
\label{Fig:Approach}
\end{figure}

The track-cusp $s_0$ is part of a branch line $S_0$. 
As in the proof of \reflem{Laminations}\refitm{Asymptotic}, the track-cusp $s_0$ gives us a pair of asymptotic boundary train lines both ending at $\bdy S_0$.
Let $\ell_0$ be the one of these that separates $s_0$ from $\ell$.
Note that $\ell_0$ is not equal to $\ell$, either because $\lambda$ is not a boundary leaf or because $\lambda$ is a boundary leaf whose cusp is on the other side from $s_0$.
Since the train lines $\ell$ and $\ell_0$ both run past $s_0$, their intersection is non-empty.  
By \reflem{Laminations}\refitm{Asymptotic} the intersection is a train interval.
Let $s_1$ be the track-cusp at the end $\ell \cap \ell_0$ pointing in the opposite direction to $s_0$.  

Suppose that $k>0$.
By induction, we have the following:
\begin{itemize}
\item a track-cusp $s_k$ (in a branch line $S_k$) meeting $\ell$ on the desired side and
\item a boundary train line $\ell_{k-1}$ that passes through the switch of $s_k$
\end{itemize}
such that
\begin{itemize}
\item if $k$ is even, $s_k$ points in the same direction as $s_0$;
if $k$ is odd, $s_k$ points in the opposite direction,
\item the boundary train line $\ell_{k-1}$ does not separate $s_k$ from $\ell$, and
\item the boundary train line $\ell_{k-1}$ is not asymptotic to $\bdy S_k$.
\end{itemize}

For the induction step, let $\ell_k$ be the boundary train line that ends at $\bdy S_k$ and that separates $s_k$ from $\ell$.
By \reflem{Laminations}\refitm{NoLinking}, the new boundary train line separates $\ell_{k-1}$ from $\ell$.
Thus we have $\ell \cap \ell_{k-1} \subset \ell \cap \ell_k$, and the latter is at least one branch of $\tau^K$ longer than the former, near $s_k$.
Let $s_{k+1}$ be the track-cusp at the end of the train interval $\ell \cap \ell_k$ pointing in the opposite direction to $s_k$.
This completes the construction of $s_{k+1}$ and $\ell_k$. 

Since $s_{k+1}$ points in the opposite direction to $s_k$, the first inductive property holds.
The second property holds by construction.
The third property holds by \reflem{Laminations}\refitm{Asymptotic}.

Let $\lambda_k$ be the leaf with the same endpoints as $\ell_k$.
Since the train intervals $\ell \cap \ell_k$ grow in both directions, and end on the midpoints of edges of the Farey triangulation of $K$, by \refrem{NeighbourhoodBasis} the leaves $\lambda_k$ converge to $\lambda$. 

We finally claim that a boundary leaf $\lambda(S, T)$ cannot be approached from the side containing its cusp $c$.
This is because, by \reflem{Unlinked}, any other boundary leaf $\lambda'$ is contained in one of the three components of $\Circle - \{\bdy S, c, \bdy T\}$.
\end{proof}


\subsection{Cantor cross-section}
\label{Sec:Cantor}

We now discuss the transverse structure of the upper lamination; 
similar properties hold for the lower.

\begin{definition}
\label{Def:LeavesThatLink}
Fix distinct cusps $c \neq d$.
Let $\Lambda^{(c, d)}$ be the leaves of $\Lambda^\calV$ that separate $c$ from $d$.
We equip $\Lambda^{(c, d)}$ with the subspace topology.
We define a total order on $\Lambda^{(c, d)}$ by taking $\lambda <^{(c,d)} \lambda'$ if the endpoints of $\lambda$ separate $c$ from some endpoint of $\lambda'$.
See \reffig{TotalOrder}. 
Thus $\lambda'$ separates some endpoint of $\lambda$ from $d$. 
\end{definition}

Note that by \reflem{Laminations}\refitm{NoLinking}, the subspace topology on $\Lambda^{(c, d)}$ agrees with the order topology coming from $<^{(c,d)}$.

\begin{figure}[htb]
\centering
\labellist
\small\hair 2pt
\pinlabel $c$ [t] at 199 0
\pinlabel $\bdy S'$ [tl] at 286 24
\pinlabel $\bdy T$ [bl] at 299 372
\pinlabel $d$ [b] at 199 400
\pinlabel $\bdy T'$ [br] at 132 384
\pinlabel $\bdy S$ [tr] at 59 61
\endlabellist
\includegraphics[width=0.6\textwidth]{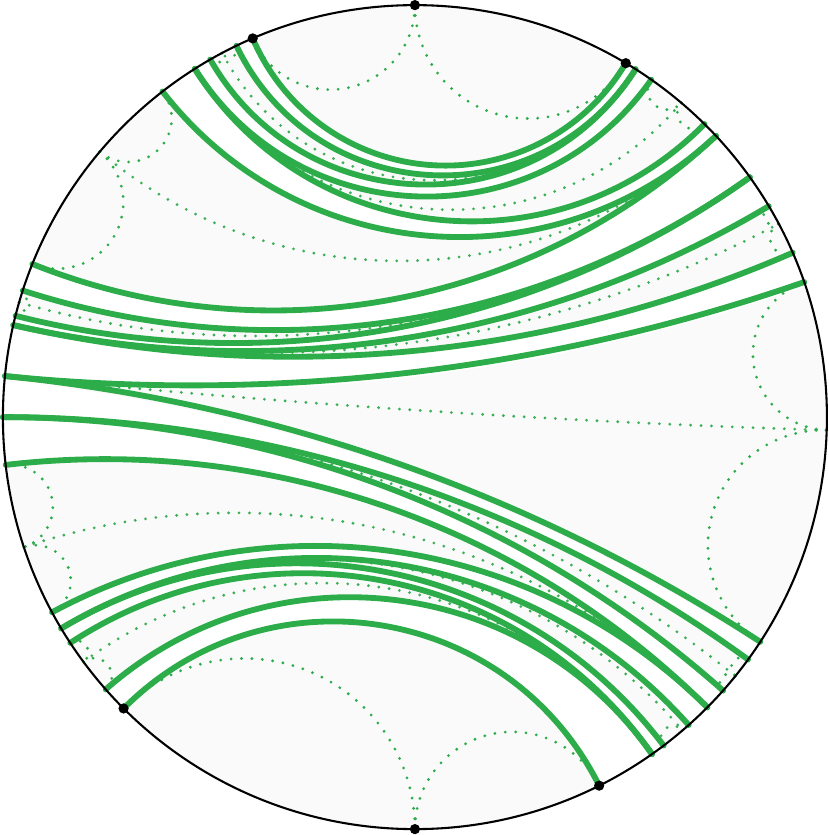}
\caption{Some of the boundary leaves (solid) in $\Lambda^{(c, d)}$.  The corresponding cusp leaves are dotted.}
\label{Fig:TotalOrder}
\end{figure}

Recall that $[x, y]^{\acw}$ is the closed arc in $\Circle$ anticlockwise of $x$ and clockwise of $y$.  
Let $\calC \subset [0,1]$ be the middle-thirds Cantor set equipped with its usual topology and total order. 

\begin{lemma}
\label{Lem:Cantor}
The subspace $\Lambda^{(c, d)}$ is order isomorphic (and thus homeomorphic) to the Cantor set $\calC$.  
Thus interior leaves are dense in $\Lambda^\calV$. 
\end{lemma}

\begin{proof}
By \reflem{CrossingParabolics}\refitm{TipsOfCrown} there is exactly one boundary leaf $\lambda(S, S')$ associated to $c$ that lies in $\Lambda^{(c, d)}$.
Similarly, there is exactly one boundary leaf $\lambda(T, T')$ associated to $d$ that lies in $\Lambda^{(c, d)}$. 
These then are the minimum and maximum points of the order $<^{(c, d)}$.  
Breaking the symmetry of the situation we assume that the order of these points in $\Circle$ is 
\[
c, \bdy S', \bdy T, d, \bdy T', \bdy S
\]
See \reffig{TotalOrder}.

We deduce that any leaf of $\Lambda^{(c, d)}$ has one endpoint in $[\bdy S', \bdy T]^{\acw}$ and the other in $[\bdy T', \bdy S]^{\acw}$. 
Thus $\Lambda^{(c, d)}$ is closed in the M\"obius band past infinity $\Mobius$. 
Also, $\Lambda^{(c, d)}$ is contained in the quotient of $[\bdy S', \bdy T]^{\acw} \cross [\bdy T', \bdy S]^{\acw}$ lying in $\Mobius$.  
Thus $\Lambda^{(c, d)}$ is compact.  

We now recursively construct ternary codes for the leaves of $\Lambda^{(c, d)}$.
By \reflem{Approach} there is a pair of asymptotic upper boundary leaves $\lambda$ and $\lambda'$ in $\Lambda^{(c, d)}$. 
In fact there are countably many such (because there are countably many boundary leaves).
Choose an ordering of these pairs $((\lambda_n, \lambda'_n)\mid n\in\NN)$.
Break symmetry and assume that $\lambda_n <^{(c,d)} \lambda'_n$.  

We now recursively build various clopen subintervals $I_\omega$ of $\Lambda^{(c,d)}$ where $\omega$ ranges over finite ternary words.
Set $I_\epsilon = \Lambda^{(c,d)}$;
here $\epsilon$ is the empty word.
At stage $n$, we have that $(\lambda_n, \lambda'_n)$ lies in the interior of some $I_\omega$ previously produced.

We now define 
\[
I_{\omega 0} = \{ \mu \in \Lambda^{(c, d)} \st \mu \leq^{(c,d)} \lambda_n \} \\
\quad
\mbox{and}
\quad
I_{\omega 2} = \{ \mu \in \Lambda^{(c, d)} \st \lambda'_n \leq^{(c,d)} \mu \} 
\]
Each is clopen (and compact).

Now, for any leaf $\lambda \in \Lambda^{(c, d)}$ we send it to the real number whose ternary expansion agrees with the subscripts of the nested sets $I_\omega$ containing $\lambda$.
Note that $\lambda(S, S')$ and $\lambda(T', T)$ are sent to $\bar{0}$ and $\bar{2}$ (zero repeating and two repeating) respectively.
Similarly, the asymptotic pair of boundary leaves splitting $I_\omega$ receive the codes $\omega \bar{0}$ and $\omega \bar{2}$.
Conversely, for any ternary expansion $\omega$ let $\omega[{:}n]$ be the prefix of length $n$.
Then the nested intersection $\bigcap_n I_{\omega[{:}n]}$ is a singleton by \reflem{Approach} and the fact that we eventually split along each pair of boundary leaves $(\lambda, \lambda')$.
Thus the coding is a continuous order preserving injection, homeomorphic onto its image. 

Since interior points are dense in the Cantor set $\calC$ we deduce that interior leaves are dense in $\Lambda^{(c, d)}$.  
We deduce interior leaves are dense in $\Lambda^\calV$ from this and \reflem{Approach}. 
\end{proof}

\subsection{Suspending and descending}
\label{Sec:SuspendingDescending}

Suppose that $e$ is an edge of $\cover{\calV}$.  
Suppose that $c$ and $d$ are the cusps meeting $e$.  
Following \refdef{LeavesThatLink}, we define $\Lambda^e = \Lambda^{(c, d)}$ to be the leaves of $\Lambda^\calV$ that link $e$. 
By \reflem{Cantor}, the subspace $\Lambda^e$ is homeomorphic to a Cantor set. 
 
Now, for every edge $e$ of $\cover{\calV}$, we place a copy $\calC^e$ of $\Lambda^e$ along $e$.
We arrange matters so that if $\gamma \in \pi_1(M)$ then $\gamma(\calC^e) = \calC^{\gamma(e)}$.
Suppose that $f$ is a face of $\cover{\calV}$ with edges $e_0$, $e_1$, and $e_2$. 
Breaking symmetry, suppose that $\Lambda^{e_0} = \Lambda^{e_1} \cup \Lambda^{e_2}$. 
We connect corresponding points of $\calC^{e_0}$ to corresponding points of $\calC^{e_1}$ or $\calC^{e_2}$ using disjoint normal arcs $\calC^f$ in $f$.
We again arrange matters so that the arcs of $\calC^f$ are sent to the arcs of $\calC^{\gamma(f)}$ by an element $\gamma$ of $\pi_1(M)$.
Finally, inside of every tetrahedron $t \in \cover{\calV}$, we choose a collection of normal disks $\calC^t$ spanning the normal curves on the boundary of $t$.
To prove that only normal disks are needed to span the curves in $\bdy t$, we fix a layering, consider the layers $K$ and $K'$ immediately below and above $t$, and apply \reflem{Laminations}\refitm{LeavesAreCarried}. 
See \reffig{NormalUpperBranchedSurface}.
Again we do this equivariantly. 
Let $\cover{\Sigma}^\calV$ be the union $\bigcup_t \calC^t$.  
 
\begin{proof}[Proof of \refthm{Laminations}\refitm{LaminationsInM}]
The lamination $\cover{\Sigma}^\calV$ consists of planes.  
Also, $\cover{\Sigma}^\calV$ is $\pi_1(M)$--invariant.
Thus it descends to give a $\pi_1$--injective lamination $\Sigma^\calV$ in $M$, carried by $B^\calV$.
By \reflem{Laminations}\refitm{FullyCarried}, this lamination is fully carried.
This proves \refthm{Laminations}\refitm{Carried}.

Suppose that $\sigma$ is leaf of $\Sigma^\calV$.
Thus $\sigma$ has a cellulation as a union of sectors;
see Sections~\ref{Sec:Sectors} and~\ref{Sec:Leaves} and the figures therein. 
Suppose that $\sigma$ is closed.
By the above $\sigma$ has Euler characteristic zero, so it is a torus or Klein bottle. 
We must now derive a contradiction.

We construct a smooth \emph{upwards loop} in $\sigma$; 
that is, a loop $\gamma$ which exits each sector $s$ it meets through the upper vertex of $s$.
By~\cite[Theorem~3.2]{SchleimerSegerman20}, the upwards loop $\gamma$ is essential.
Thus $\sigma - \gamma$ is an annulus or M\"obius band.
Let $(P_i)_{i\in\ZZ}$ be the set of elevations of $\sigma - \gamma$ to $\cover{\sigma}$.
Each of these is homeomorphic to $(0,1) \times \RR$; the two boundary components are distinct elevations of $\gamma$. 
Let $\calK = (K_j)$ be a layering of $\cover\calV$.
For all $i$ and $j$, define $\alpha(i,j) = P_i \cap K_j$.
Note that $\alpha(i,j)$ is a properly embedded arc in $P_i$.
Also $\alpha(i,j)$ and $\alpha(i,j+1)$ cobound a rectangle in $P_i$.
Each $\alpha(i,j)$ is carried by a smooth subarc of the one-skeleton of $\sigma$.
Thus each $\alpha(i,j)$ receives a \emph{normal length} $\ell(i,j)$.
From \reffig{Sector} we observe that $\ell(i,j+1) \leq \ell(i,j)$, with equality if and only if there are no normal triangles between $\alpha(i,j)$ and $\alpha(i,j+1)$.
Since the fundamental group of $\sigma$ acts cocompactly on $\cover\sigma$, we deduce that $\ell(i,j+1) = \ell(i,j)$ for all $i$ and $j$.
Thus there are no normal triangles in $\cover\sigma$. 
We now reach a contradiction as in \refcor{BranchLinesToggle};
let $Q$ be a diagonal strip of quadrilaterals in $\cover\sigma$.
All quadrilaterals of $Q$ are contained in tetrahedra meeting a single edge of $\cover\calV$.
This contradicts the finiteness of edge degrees.
We deduce that $\Sigma^\calV$ has no closed leaves.

\begin{remark}
When $\calV$ is finite, it admits a strict angle structure by~\cite[Theorem~1.5]{HRST11} or by \cite[Theorem~1.4]{FuterGueritaud13}.
So, in the finite case, this gives an alternate proof that there are no closed surfaces carried by $B^\calV$.
\end{remark}

We now show that $\sigma$ is either a plane, annulus, or M\"obius band.
Again we use the decomposition of $\sigma$ into sectors (\refsec{Leaves}).
Suppose that $\gamma = (\gamma_i)_{i = 0}^{K-1}$ is an essential simple closed multicurve in the leaf $\sigma$.
Isotope $\gamma$ inside of $\sigma$ to be transverse to $\sigma^{(1)}$, the one-skeleton of $\sigma$.
We define the \emph{length} of $\gamma$ to be the number of points of intersection between $\gamma$ and $\sigma^{(1)}$.
So $\sigma^{(1)}$ cuts $\gamma$ into a collection of arcs in each sector:
these are \emph{maxima}, \emph{minima}, \emph{diagonals}, and \emph{verticals}.
The maxima and minima are of two types: they either separate the vertices of the sector, or they do not.
We call the former \emph{separating} and the latter \emph{non-separating}.
See \reffig{SectorArcs}.

\begin{figure}[htb]
\subfloat[]{
\label{Fig:SectorArcsBigon}
\centering
\labellist
\scriptsize\hair 2pt
\pinlabel {maxima} [r] at 115 46
\pinlabel {diagonals} [t] at 405 49
\pinlabel {verticals} [b] at 115 272
\pinlabel {{\parbox{2.2cm}{\begin{center}separating \\minimum\end{center}}}}  [b] at 240 270
\pinlabel {{\parbox{2.2cm}{\begin{center}non-separating \\minimum\end{center}}}} [b] at 400 216
\endlabellist
\includegraphics[width = 0.52\textwidth]{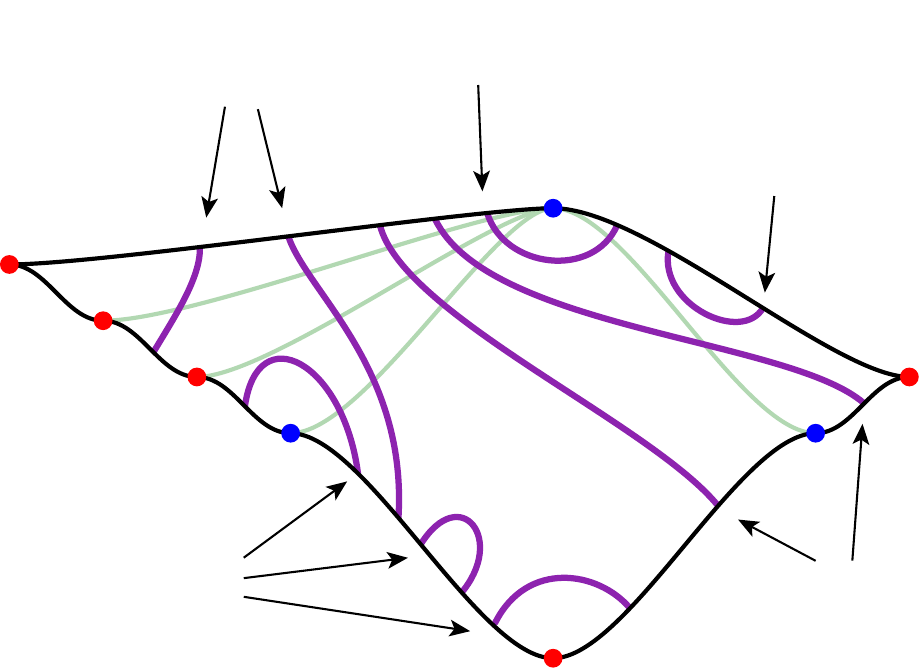}
}
\quad
\subfloat[]{
\label{Fig:SectorArcsDiamond}
\centering
\includegraphics[width = 0.4\textwidth]{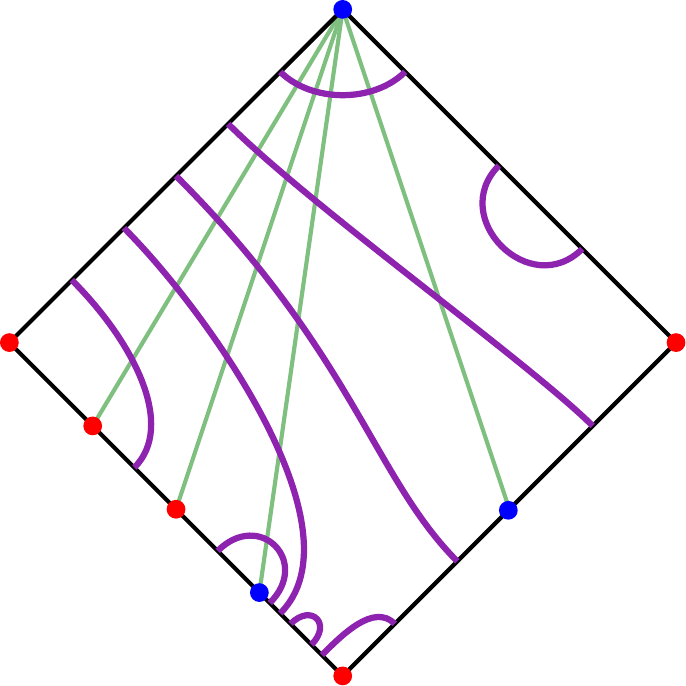}
}
\caption{The various possible subarcs of $\gamma$ meeting a fixed sector, drawn in the two styles of \reffig{Sector}.  
Verticals separate one cusp vertex from the upper and lower vertices of the normal quadrilateral.
Diagonals separate the upper vertex of the normal quadrilateral and one of the cusp vertices from the lower vertex of the normal quadrilateral and the other cusp vertex.
Maxima and minima do not separate the cusp vertices.}
\label{Fig:SectorArcs}
\end{figure}

\begin{figure}[htb]
\subfloat[]{
\label{Fig:BigonMoveNonSeparating}
\centering
\includegraphics[height = 4cm]{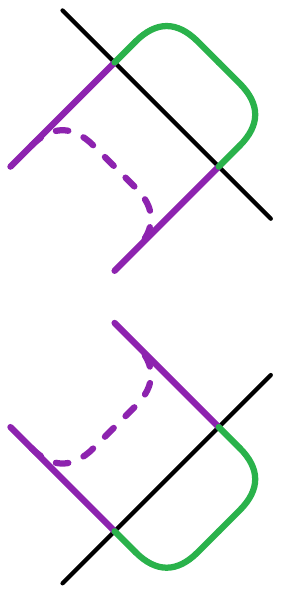}
}
\quad
\subfloat[]{
\label{Fig:BigonMoveT}
\centering
\includegraphics[height = 4cm]{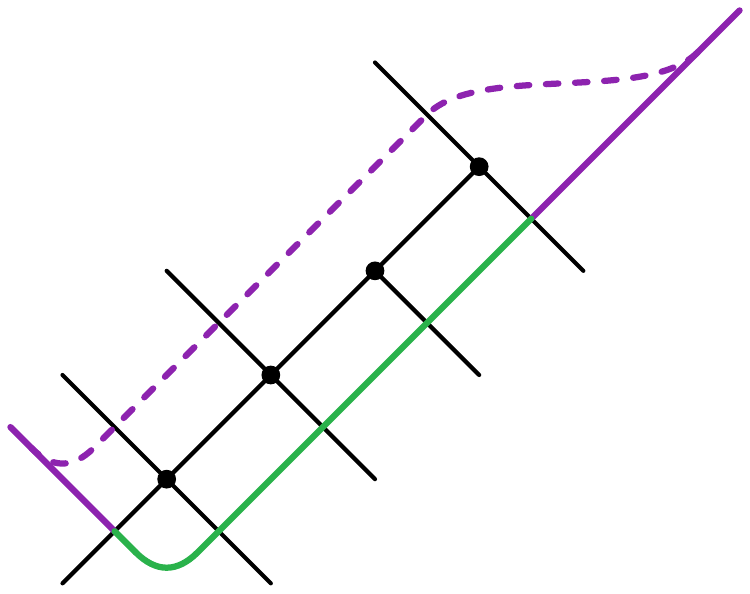}
}

\subfloat[]{
\label{Fig:BigonMoveVertical}
\centering
\includegraphics[height = 4cm]{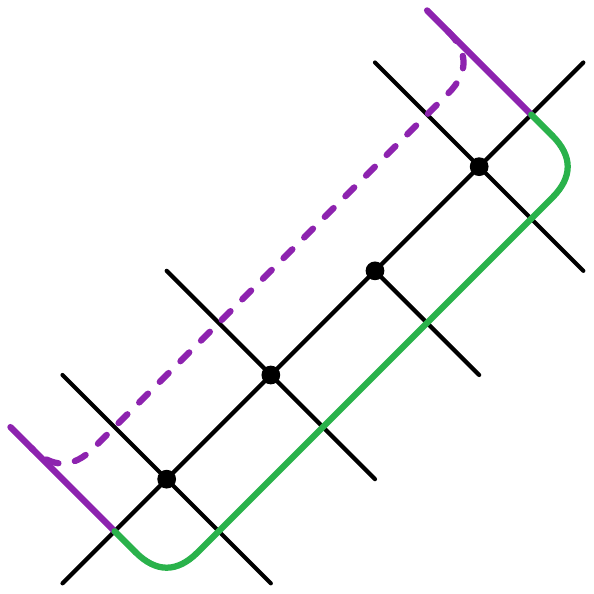}
}
\quad
\subfloat[]{
\label{Fig:BigonMoveMax}
\centering
\includegraphics[height = 4cm]{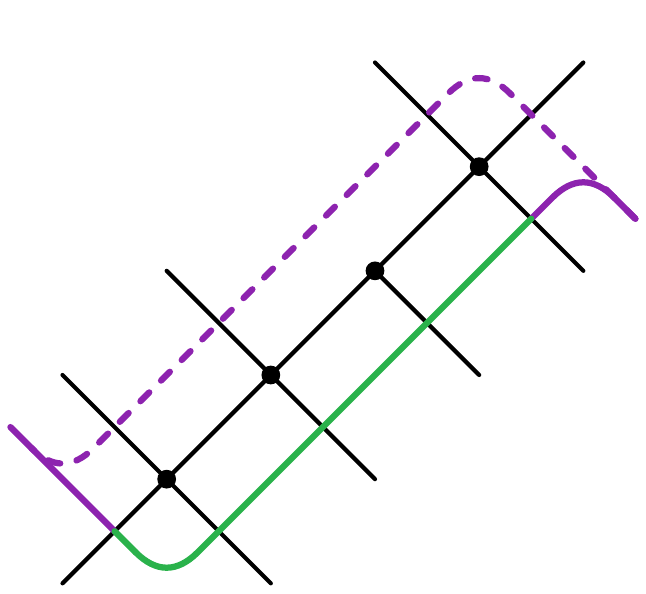}
}
\caption{Examples of arcs $\delta$ and the corresponding ambient isotopies.
In each case the arc $\delta \subset \gamma_k$ is drawn in green.
The result of the ambient isotopy is drawn with a dashed line.}
\label{Fig:BigonMoves}
\end{figure}

\begin{claim*}
The multicurve $\gamma = (\gamma_i)_{i = 0}^{K-1}$ can be isotoped to have no minima.
\end{claim*}

\begin{proof}
Fix $k \geq 0$.
We assume that, for $i < k$, the curves $\gamma_i$ have no minima.
We now perform a sequence of ambient isotopies of $\gamma$ that fix the $\gamma_i$ (for $i < k$) pointwise while reducing the following lexicographic complexity of $\gamma_k$:
\[
( \mbox{number of minima of $\gamma_k$, length of $\gamma_k$} )
\]
If $\gamma_k$ has no minima then we are done with $\gamma_k$. 

The ambient isotopies of $\gamma$ give \emph{bigon moves} of $\gamma_k$ that reduce complexity. 
A bigon move is applied to a (perhaps long) subarc $\delta$ of $\gamma_k$. 

The simplest case is when $\delta$ is a non-separating maximum or minimum. 
We push $\delta$ across a subarc of $\sigma^{(1)}$;
this does not increase the number of minima and reduces the length by two.
See \reffig{BigonMoveNonSeparating}.

When we cannot apply the previous move, we take $\delta$ to be an arc
\begin{itemize}
\item
starting with a separating minimum, 
\item
followed by a (possibly empty) sequence of diagonals fellow-travelling a smooth arc $\epsilon$ in $\sigma^{(1)}$, and 
\item
ending either when $\epsilon$ ends in a T-junction (\reffig{BigonMoveT}) or when $\delta$ ends  
    \begin{itemize}
    \item
    with a vertical (\reffig{BigonMoveVertical}) or
    \item
    immediately before a separating maximum (\reffig{BigonMoveMax}).
    \end{itemize}
\end{itemize}   

For the subcases shown in Figures~\ref{Fig:BigonMoveT} and~\ref{Fig:BigonMoveVertical} we push $\delta$ across $\epsilon$;
this does not increase the number of minima and reduces the length.
After applying all moves of the above types, we either have no minima in $\gamma_k$ (and are done) or the only subarcs $\delta$ are of the type shown in \reffig{BigonMoveMax} (and thus $\gamma_k$ has no verticals).
In particular, since $\epsilon$ ends at an X-junction, the separating maximum does not separate the end of $\epsilon$ from the cusp vertices of its sector.
(Because we are not in the case where $\epsilon$ ends in a T-junction.)

All remaining cases lead to contradiction, as follows.
Pick one minimum and consider its two subarcs (heading in opposite directions), $\delta$ and $\delta'$.
First, assume that $\delta$ and $\delta'$ do not end in the same maximum.
In this case, we repeatedly applying the bigon move shown in \reffig{BigonMoveMax} while keeping $\gamma_k$ fixed outside of a neighbourhood of $\delta \cup \delta'$.
This either reduces the length or reduces the number of minima.
Finally, suppose that $\delta$ and $\delta'$ end in the same maximum.
In this case, by repeatedly applying the bigon move shown in \reffig{BigonMoveMax}, we either reduce the length, or find an infinite or periodic strip of quadrilaterals.
This gives a contradiction, again as in \refcor{BranchLinesToggle}.
\end{proof} 

Recall that $\sigma$ is not closed.
If $\sigma$ is not a plane, annulus, or M\"obius band then we may choose a multicurve $\gamma$ which cuts out of $\sigma$ a compact subsurface $\sigma'$ of negative Euler characteristic.
Applying the claim immediately above, we may assume that $\gamma$ has no minima and thus no maxima.
Thus $\gamma$ is a union of diagonals and verticals.  
Consider a component $s'$ of $\sigma' - \sigma^{(1)}$, with $s'$ contained in the sector $s$.
Thus $s'$ is a component of $s - \gamma$.
Consulting \reffig{SectorArcsBigon}, we deduce that $s'$ is either 
\begin{itemize} 
\item
a \emph{cusped bigon} (all of $s$), 
\item 
a \emph{boundary trigon} (contains a cusp vertex), or
\item
a \emph{rectangle} (contains no cusp vertices).
\end{itemize}
See~\cite[Table~1]{SchleimerSegerman20}.
Thus $\sigma'$ has Euler characteristic zero.
This is a contradiction.
This proves \refthm{Laminations}\refitm{Types}.

Finally, \reflem{BranchLines}\refitm{CuspNeighbourhood} implies all complementary regions to $\Sigma^\calV$ are torus (or annulus, or plane) crossed with a ray, with paring locus on the inner boundary.
Thus $\Sigma^\calV$ is an \emph{essential lamination}. 
(See~\cite[Definition~6.14]{Calegari07}.)
This completes the proof of \refthm{Laminations}\refitm{Essential}.
\end{proof}

\subsection{Uniqueness}
\label{Sec:Uniqueness}

To simplify the combinatorics, here we relax the transversality assumption on our paths.
That is, suppose that $\gamma \subset B^\calV$ is an oriented path. 
In this section and the next we say that $\gamma$ is an \emph{ascending path} if 
\begin{itemize}
\item
$\gamma$ enters every sector through a lower edge or a lower vertex and 
\item
$\gamma$ exits every sector through an upper edge or its upper vertex.
\end{itemize}
An ascending path $\gamma$ is an \emph{upwards path} if $\gamma$ exits every  sector through its upper vertex.

\begin{remark}
\label{Rem:Ascend}
Suppose that $\sigma$ is a leaf carried by the upper branched surface $B^\calV$. 
Suppose that $\gamma$ is an ascending path that ends in $\sigma$. 
Then after a small isotopy, $\gamma$ is completely contained in $\sigma$.
\end{remark}

\begin{definition}
\label{Def:Parallel}
Fix $M$ a three-manifold and $B \subset M$ a branched surface.  
Let $N = N(B)$ be a tie-neighbourhood of $B$.  
A \emph{tie-isotopy} is an ambient isotopy of $M$ supported in $N$ and which preserves the ties of $N$. 

Suppose that $\Sigma$ and $\Sigma'$ are laminations carried by $B$.  
We say that leaves $\sigma$ and $\sigma'$ of $\Sigma$ and $\Sigma'$ are \emph{parallel} if they are tie-isotopic.
\end{definition}

Note that, by construction, $\Sigma^\calV$ has no (distinct) parallel leaves.

\begin{lemma}
\label{Lem:Squeeze} 
Suppose that $\Sigma$ and $\Sigma'$ are laminations fully carried by $B^\calV$.  
Then for any leaf $\sigma$ of $\Sigma$ there is a leaf $\sigma'$ of $\Sigma'$ so that $\sigma$ and $\sigma'$ are parallel. 
\end{lemma}

\begin{proof}
We first assume that $M \homeo \RR^3$. 
Fix $K$ a layer of a layering of $\calV$. 
Let $\ell = \sigma \cap K$. 
Fix a sector $E$ of $B^\calV$ which meets $\ell$. 
 
We have the following claim.
\begin{claim*}
For every $n > 0$ there is a leaf $\sigma'_n$ of $\Sigma'$ so that 
\begin{itemize}
\item
$\sigma'_n$ crosses $E$ and
\item
the train lines $\ell$ and $\ell'_n = \sigma'_n \cap K$ fellow travel for at least $n$ triangles of $K$, to either side of $E \cap K$. 
\end{itemize}
\end{claim*}

\begin{proof}[Proof of Claim]
Let $\gamma_n \subset \sigma$ be the unique (up to isotopy) upwards path which begins in $E$ and which crosses $n$ sectors.  
See \reffig{TAndX}.  
Let $E_n$ be the final sector that $\gamma_n$ meets.  

Since $\Sigma'$ is fully carried, there is a leaf $\sigma'_n$ of $\Sigma'$ crossing $E_n$.
By \refrem{Ascend} applied to $\gamma_n$, the leaf $\sigma'_n$ also crosses $E$. 

By \reflem{BranchLines}\refitm{BranchLinesLayer}, all branch lines meet $K$.  
This includes the branch lines $(S_i)$ that meet $\gamma_n$. 
Let $s_i$ be the track-cusp of $S_i$ contained in $K$.  
Since $S_i$ is a branch line, the lower component of $S_i - \gamma_n$ is (after a small isotopy) contained in $\sigma$.  
Equally well (after a different small isotopy) $S_i - \gamma_n$ is contained in $\sigma'_n$.
Since the lower component of $S_i - \gamma_n$ contains $s_i$ we deduce that $\ell = \sigma \cap K$ and $\ell'_n = \sigma'_n \cap K$ both cross the edge of $K$ meeting $s_i$.
Since all of the $s_i$ are distinct, and since there are $n$ of them to each side of $E$, the claim follows.
\end{proof}

Since $\Sigma'$ is a lamination, and since all of the $\sigma'_n$ meet $E$, we pass to a convergent subsequence of the leaves $\sigma'_n$ (in the Gromov-Hausdorff topology).  
Thus $\Sigma'$ contains the limiting leaf $\sigma'$.  
Thus $\ell' = \sigma' \cap K$ has the same endpoints in $\Circle$ as $\ell$. 
This holds for all layers of the layering. 
We deduce that $\sigma$ and $\sigma'$ are tie-isotopic in the tie-neighbourhood $N(B^\calV)$, as desired.

If $M$ is not homeomorphic to $\RR^3$ then we lift to $\cover{M}$ and $\cover{\calV}$. 
By \reflem{ExhaustImpliesThreeSpace} we have that $\cover{M}$ is homeomorphic to $\RR^3$.
The preimages of the laminations $\Sigma$ and $\Sigma'$ are again fully carried by the preimage of $B^\calV$.
So we apply the above argument and note that the tie-isotopy may be chosen to be equivariant, by inducting on the skeleta of $B^\calV$.
\end{proof}

\begin{definition}
\label{Def:Collapse}
Fix $M$ a three-manifold and $B \subset M$ a branched surface.
Suppose that $\sigma$ and $\sigma'$ are distinct parallel leaves of a lamination $\Sigma$ carried by $B$.
Then we may delete all leaves between $\sigma$ and $\sigma'$ and tie-isotope $\sigma'$ to $\sigma$.
This tie-isotopy extends to give a new lamination $\Sigma'$.
We say that $\Sigma'$ is obtained by \emph{collapsing} the given parallel leaves.
\end{definition}

\begin{proof}[Proof of \refthm{Laminations}\refitm{LaminationsUnique}] 
Suppose that $\Sigma$ is a lamination fully carried by $B^\calV$.  
From \reflem{Squeeze} we deduce that every leaf of $\Sigma$ is tie-isotopic to some leaf of $\Sigma^\calV$.  
However distinct leaves of $\Sigma$ may be parallel; 
if so, this correspondence will not be one-to-one. 
So we collapse maximal collections of parallel leaves in $\Sigma$.
Now the leaves of $\Sigma$ are naturally in bijection with the leaves of $\Sigma^\calV$.
Pick any order on the countably many boundary leaves of $\Sigma$.
Recalling that tie-isotopies are ambient, we tie-isotope boundary leaves of $\Sigma$, in order, to the corresponding boundary leaves of $\Sigma^\calV$.  
This completes the proof of \refthm{Laminations}\refitm{LaminationsUnique}.
\end{proof}

\subsection{An application}

In the context of pseudo-Anosov homeomorphisms, \refthm{Laminations}\refitm{LaminationsUnique} gives a strong uniqueness result. 
Before proving this, we require the following.

\begin{lemma}
\label{Lem:Transitive}
Suppose that $M$ is a three-manifold equipped with a finite transverse veering triangulation $\calV$. 
Suppose that $D$ and $D'$ are sectors of $B^\calV$.  
Then there is an ascending path $\gamma \subset B^\calV$ that starts in $D$ and ends in $D'$.  
\end{lemma}

\begin{proof}
The branch locus of $B^\calV$, in dual position, is a four-valent graph with a cyclically ordered edge-ends at the vertices.  
Since $\calV$ is finite, the branch locus of $B^\calV$ decomposes as a union of smooth circles; 
each circle is an image of some branch line under the universal covering map.

By \refrem{Dual}, there is some oriented path $\gamma$ in $B^\calV$ that starts and ends in the interiors of $D$ and $D'$ respectively.  
Homotope an initial segment of $\gamma$, staying in $B^\calV$, so that it initially exits $D$ through its uppermost vertex $v$. 
Similarly, we may make $\gamma$ enter $D'$ through its lowest vertex $v'$.  
Let $\beta$ be the subpath of $\gamma$ between $v$ and $v'$.  
Since the branch locus is isotopic to the dual one-skeleton (again by \refrem{Dual}) we may homotope $\beta$, relative to its endpoints, into the branch locus.  
Tighten $\beta$ so that it is locally an embedding.  

Consider a maximal arc $\beta'$ of $\beta$ which is contained in a smooth circle $C$ of the branch locus, and whose orientation disagrees with the given co-orientation.  
We call $\beta'$ a \emph{descending} arc.  
We replace $\beta'$ by its (ascending) complement in $C$. 
We now alternate between tightening and replacing descending arcs with ascending arcs, until all arcs are ascending.  
This produces the desired ascending path. 
\end{proof}

\begin{corollary}
\label{Cor:NonEmptyImpliesFull}
Suppose that $M$ is a three-manifold equipped with a finite transverse veering triangulation $\calV$. 
Suppose that $\Sigma$ is a non-empty lamination carried by $B^\calV$.
Then $\Sigma$ is fully carried.
\end{corollary}


\begin{proof}
Let $\sigma$ be any leaf of $\Sigma$. 
Let $D'$ be any sector crossed by $\sigma$. 
Let $D$ be any other sector of $B^\calV$. 
By \reflem{Transitive} there is an ascending path $\gamma$ in $B^\calV$ connecting $D$ to $D'$. 
From \refrem{Ascend} we deduce that $\sigma$ crosses $D$.  
Thus $\Sigma$ is fully carried.
\end{proof}

\begin{corollary}
\label{Cor:Surprise}
Suppose that $S$ is a compact, connected, oriented surface.  
Suppose that $f \from S \to S$ is a pseudo-Anosov homeomorphism.  
Let $B^f$ be the stable branched surface in the mapping torus $M_f$. 
Let $\Sigma^f$ be the suspension of the stable lamination for $f$. 
Then any non-empty lamination $\Sigma$ carried by $B^f$ is tie-isotopic to $\Sigma^f$, after collapsing parallel leaves.  
\end{corollary}

\begin{proof}
We remove the boundary of $M_f$ and drill out the singular orbits of the suspension flow $\Phi_f$.  
By \cite[Main~Construction]{Agol11} the resulting manifold $M$ admits a veering triangulation $\calV$; 
furthermore, $B^f$ is isotopic to $B^\calV$ and $\Sigma^f$ is isotopic to $\Sigma^\calV$.  

Suppose that $\Sigma$ is a non-empty lamination carried by $B^f$.  
By \refcor{NonEmptyImpliesFull} we have that $\Sigma$ is fully carried. 
By \refthm{Laminations}\refitm{LaminationsUnique} we have that, after collapsing parallel leaves, $\Sigma$ is tie-isotopic to $\Sigma^f$.
\end{proof}

\chapter{Interaction between the upper and lower laminations}

With the upper and lower laminations in hand we may explore their interactions.  

\subsection{No shared endpoints}


As usual we fix $M$ a three-manifold with a transverse veering triangulation $\calV$.  
We take $\cover{M}$ to be the universal cover; 
we take $\cover{B}$ to be the preimage, in $\cover{M}$, of the horizontal branched surface $B(\calV)$. 
If $L \subset \cover{M}$ is a landscape carried by $\cover{B}$ 
then we take $\tau^L$ and $\tau_L$ to be the induced upper and lower tracks. 
Finally we have $\Circle$ the veering circle and $\Lambda^\calV$ and $\Lambda_\calV$ the upper and lower laminations. 

\begin{lemma}
\label{Lem:NoMixedTypeAsymptotics}
No leaf of $\Lambda^\calV$ is asymptotic to a leaf of $\Lambda_\calV$.
\end{lemma}

\begin{proof}
For a contradiction, suppose that $\lambda = \lambda(x, y)$ and $\mu = \lambda(y, z)$ are leaves of $\Lambda^\calV$ and $\Lambda_\calV$ respectively, which share a common endpoint $y \in \Circle$.  

Fix $K$ a layer of some layering.  Let $\tau^K$ and $\tau_K$ be the upper and lower train tracks in $K$.  \reflem{Laminations}\refitm{LeavesAreCarried} gives us train lines $\ell$ in $\tau^K$ and $m$ in $\tau_K$ that have the same endpoints as $\lambda$ and $\mu$, respectively.  We orient $\ell$ and $m$ towards $y$.   Since $\bdy_+ \ell = y = \bdy_+ m$, and applying \reflem{Laminations}\refitm{Irrational}, there are subrays of $\ell$ and $m$ that cross an identical collection of edges in the triangulation of $K$.  This collection of edges determines an infinite strip $P \subset K$ of triangles.  There are only two kinds of veering triangle, both shown in \reffig{VeeringTriangles}; from this we deduce that the internal edges (and the initial edge) of $P$ are all the same colour.  The boundary edges (other than the initial edge) are the opposite colour.  Breaking symmetry, we assume that the interior edges of $P$ are red.  For an example, see \reffig{NoMixedTypeAsymptotics}.  

\begin{figure}[htb]
\centering
\subfloat[A strip of triangles $P_0$ with blue boundary edges and red interior edges.  This locally contains the fellow travelling upper and lower train lines.]{
\includegraphics[width=0.8\textwidth]{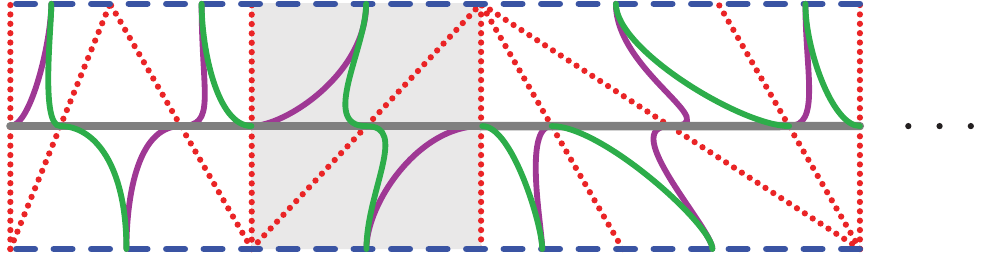}
\label{Fig:NoMixedTypeAsymptotics}
}

\subfloat[Layering a fan tetrahedron on top of $P_i$.]{
\includegraphics[width=0.8\textwidth]{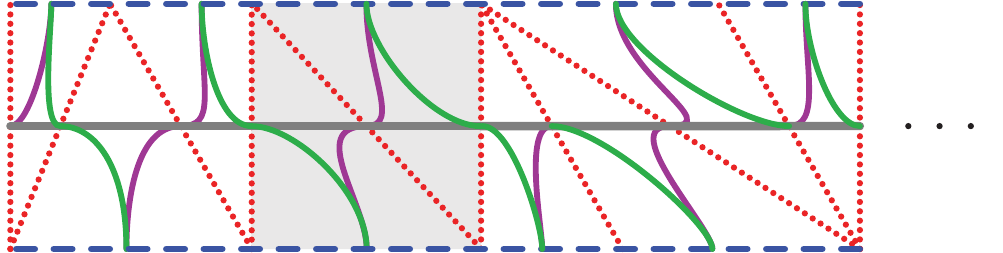}
\label{Fig:NoMixedTypeAsymptoticsRemains}
}

\subfloat[Layering a toggle tetrahedron on top of $P_i$ sends upper (green) train routes out of the strip.]{
\includegraphics[width=0.8\textwidth]{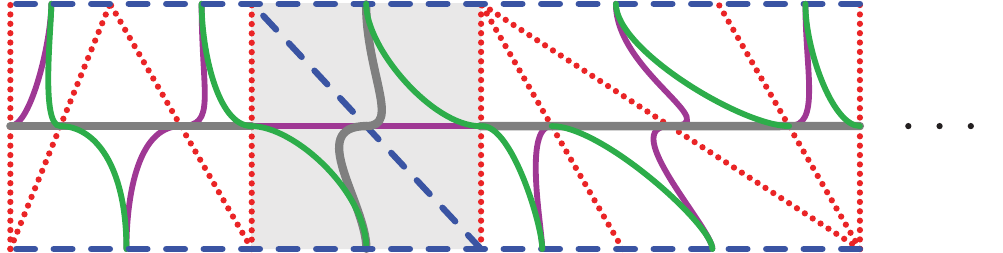}
\label{Fig:NoMixedTypeAsymptoticsExits}
}
\caption{Layering a tetrahedron on top of a strip of triangles. The altered triangles are shaded in light grey. Where the upper and lower tracks coincide we draw a thicker grey arc.}
\label{Fig:NoMixedTypeAsymptoticsFigs}
\end{figure}

\begin{claim*}
Any in-fill tetrahedron $t$ attached above $P$, to an interior edge of $P$, is a fan tetrahedron.
\end{claim*}

\begin{proof}
Let $K'$ be the landscape obtained from $K$ by removing the bottom two triangles of $t$ from $K$ and replacing it with the top two triangles of $t$.  Note that $K'$ is again a layer of some layering.  We define $P' \subset K'$ similarly.  Applying \reflem{Laminations}\refitm{LeavesAreCarried}, the leaf $\lambda$ determines a train line $\ell'$ in the track $\tau^{K'}$.  

Suppose that $t$ is a toggle.
Thus $P'$ has an internal edge which is blue.
Since $\ell'$ cannot cross this edge, we have reached a contradiction.
See \reffig{NoMixedTypeAsymptoticsExits}.
\end{proof}

Suppose that $t$ is an in-fill fan tetrahedron attached above $P$, to an interior edge of $P$.
The resulting strip $P'$ again has red interior edges.
Also, $\tau^{P'}$ still carries $\ell'$.  
See \reffig{NoMixedTypeAsymptoticsRemains}.  

By \reflem{CannotTurnLeftForever} the rays $\ell \cap P$ and $m \cap P$ both turn right and left infinitely often.  Thus there is an upper track-cusp $s$ meeting an interior edge of $P$ so that $s$ points in the same direction as $\ell$ and $m$.  By the claim above, the branch line $S$ passing through $s$ meets only fan tetrahedra.  This contradicts \refcor{BranchLinesToggle}.
\end{proof}

\subsection{Crowns}

\begin{definition}
\label{Def:Crown}
We define $\Lambda^c \subset \Lambda^\calV$, the \emph{upper crown} of the cusp $c$, to be the set of leaves $\lambda(S, T)$ where here $S$ and $T$ range over all adjacent upper branch lines in the upper cusp neighbourhood $N^c$.
We define the \emph{lower crown} $\Lambda_c$ similarly.
The endpoints of the boundary leaves of a crown are called its \emph{tips}.
\end{definition}

\begin{figure}[htbp]
\labellist
\small\hair 2pt
\pinlabel {$c$} [t] at 285 0
\pinlabel {$\bdy S'$} [tl] at 475 80
\pinlabel {$\bdy T$} [l] at 565 230
\pinlabel {$d$} [b] at 285 570
\pinlabel {$\bdy T'$} [br] at 97 492
\pinlabel {$\bdy S$} [r] at 5 341
\endlabellist
\includegraphics[width=0.4\textwidth]{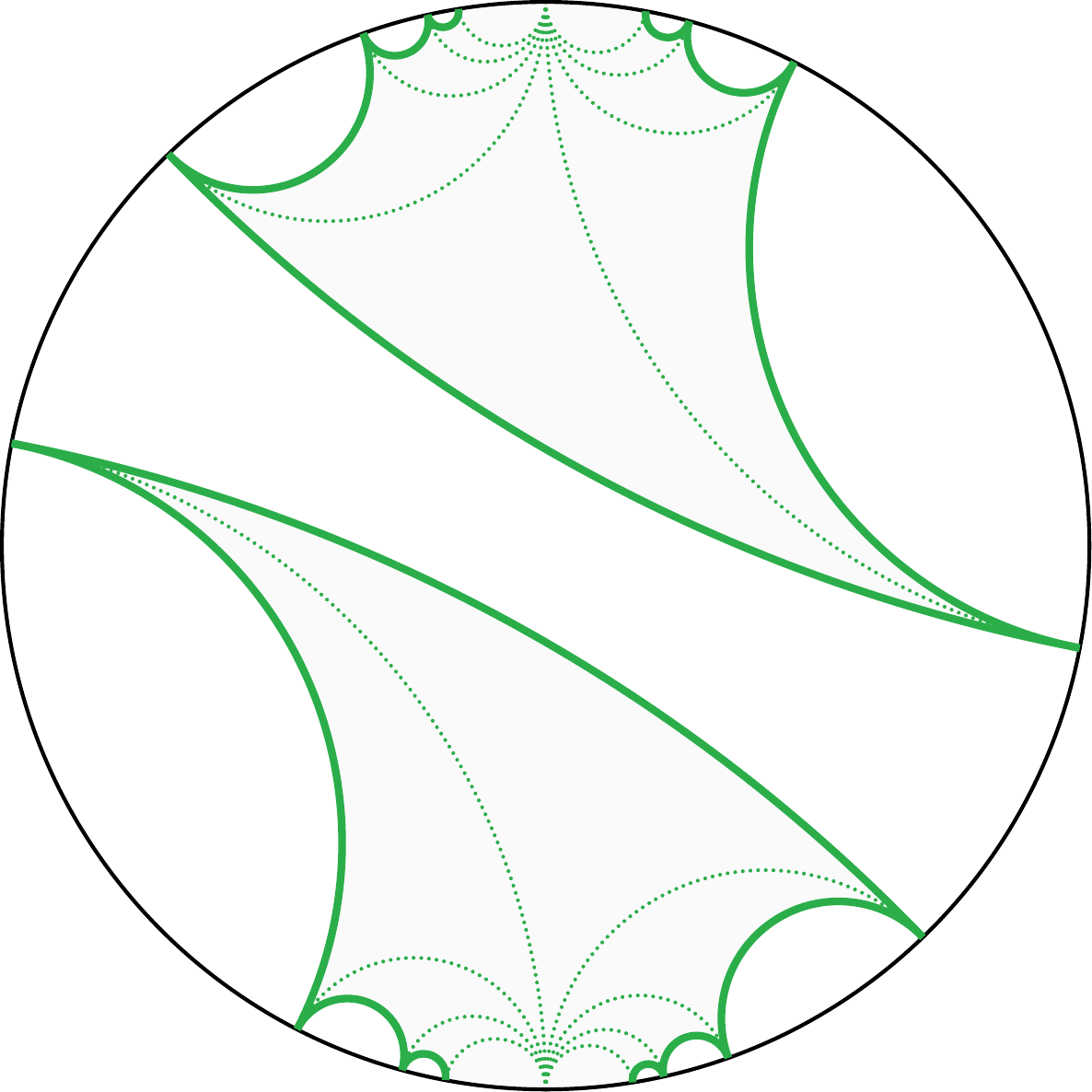}
\caption{Two upper crowns $\Lambda^c$ and $\Lambda^d$.
  The solid lines are boundary leaves and the dotted lines are cusp leaves.
  Compare with \reffig{TotalOrder}.
}
\label{Fig:UpperCrowns}
\end{figure}

By \reflem{BranchLines}\refitm{Cyclic} the boundary leaves of an upper crown $\Lambda^c$ are adjacent in pairs and ordered by $\ZZ$.  
By \reflem{CrossingParabolics}\refitm{TipsOfCrown} the tips of $\Lambda^c$ accumulate (only) on $c$.  
See \reffig{UpperCrowns}. 

\begin{definition}
\label{Def:LinkedCrowns}
Suppose that $\Lambda^c$ and $\Lambda_d$ are upper and lower crowns. 
\begin{itemize}
\item 
They are \emph{unlinked} if no leaf of $\Lambda^c$ links any leaf of $\Lambda_d$. 
\item 
They \emph{cross} if there are leaves $\lambda(R,S)$ and $\lambda(S, T)$ in $\Lambda^c$ and leaves $\lambda(U,V)$ and $\lambda(V,W)$ in $\Lambda_d$ so that these leaves of the crowns link and no others do.  
See \reffig{CrossingCrowns}.
\item 
They \emph{interleave} if $c = d$ and we may index the upper and lower branch lines $S_i$ and $V_i$ such that $\lambda(S_i,S_{i+1})$ links $\lambda(V_{i-1},V_i)$ and $\lambda(V_i,V_{i+1})$.
See \reffig{InterleavingCrowns}. \qedhere
\end{itemize}
\end{definition}

\begin{figure}[htb]
\centering
\subfloat[]{
\labellist
\small\hair 2pt
\pinlabel {$\bdy U$} at 590 200
\pinlabel {$\bdy V$} at 60 500
\pinlabel {$\bdy W$} at 370 -15
\pinlabel {$d$} at 500 70
\pinlabel {$\bdy R$} at -20 200
\pinlabel {$\bdy S$} at 520 500
\pinlabel {$\bdy T$} at 200 -15
\pinlabel {$c$} at 70 70
\endlabellist
\includegraphics[width=0.43\textwidth]{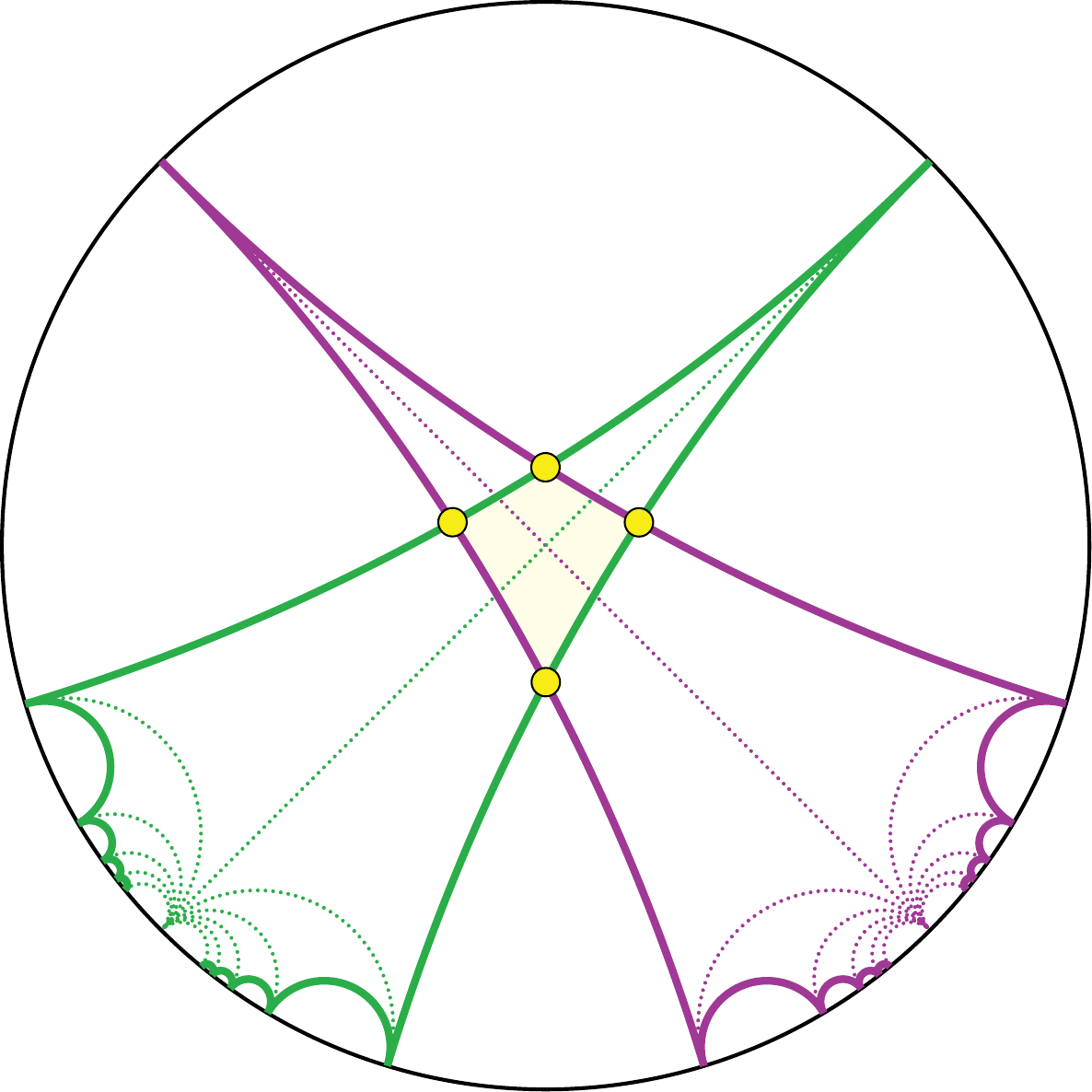}
\label{Fig:CrossingCrowns}
}
\qquad
\subfloat[]{
\includegraphics[width=0.43\textwidth]{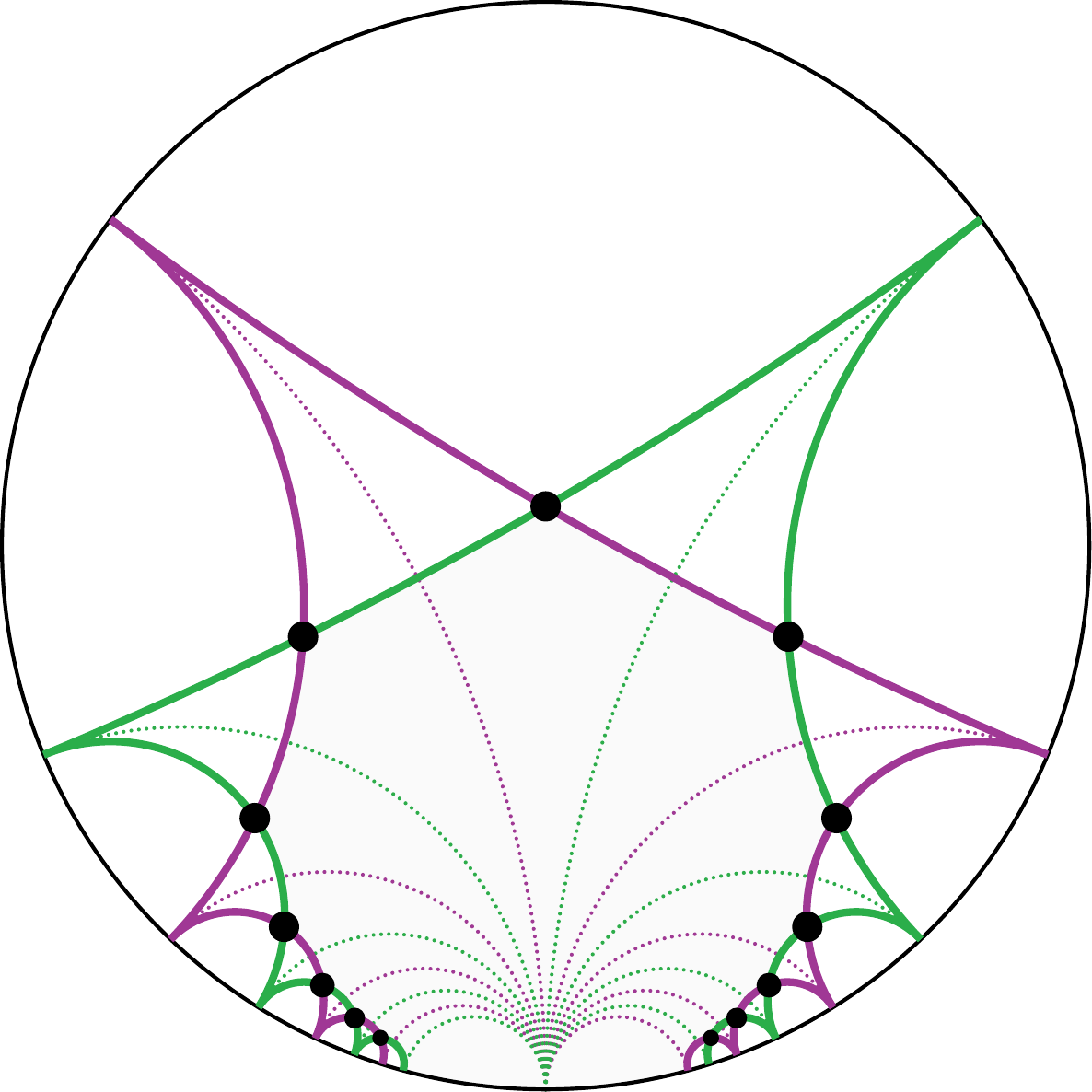}
\label{Fig:InterleavingCrowns}
}
\caption{Crossing and interleaving crowns.  We illustrate each leaf by drawing an arc through the unit disk joining the two points of the leaf.  We indicate linking pairs of leaves with black or yellow dots as the pair comes from interleaving crowns or otherwise.}
\label{Fig:LeafIdentifications}
\end{figure}

\begin{lemma}
\label{Lem:CrownsInterleave}
The upper and lower crowns $\Lambda^c$ and $\Lambda_c$ interleave.  
\end{lemma}

The analysis here is similar to, but more delicate than that in the proof of \reflem{BranchLines}\refitm{NeighbourhoodLayer}.

\begin{proof}[Proof of~\reflem{CrownsInterleave}]
Fix $K$ a layer of a layering.
We walk in $K$ anticlockwise about the cusp $c$.
As we do so, we pass a sequence $(e_i)$ of edges meeting $c$.
Let $f_i$ be the triangle meeting $e_i$ and $e_{i+1}$.
If $e_i$ is red and $e_{i+1}$ is blue then $f_i$ contains an upper track-cusp pointing away from $c$.
If the colours are interchanged then $f_i$ contains an lower track-cusp pointing away from $c$.
If $e_i$ and $e_{i+1}$ have the same colour then $f_i$ does not contain a track-cusp pointing away from $c$.
See \reffig{VeeringTriangles}.
Thus the connecting arcs emanating from $c$ alternate in the correct fashion.
By \reflem{Disk} the layer $K$ is a copy of the Farey tessellation.
Thus we deduce that the cusp lines emanating from $c$ also alternate.
Thus the upper and lower crowns interleave.
See \reffig{InterleavingCrowns}.
\end{proof}

Before dealing with crossing crowns, we require the following.

\begin{lemma}
\label{Lem:TriangleLeafCross}
Suppose that $c$ is a cusp and $R$ and $S$ are adjacent upper branch lines in $N^c$.
Suppose that $\mu$ is a lower cusp leaf, a lower boundary leaf or a lower interior leaf.
Then $\mu$ links at most one of the cusp leaves $\lambda(c,R)$ and $\lambda(c,S)$.
The same statement holds swapping upper and lower. 
\end{lemma}

\begin{proof}
Suppose for contradiction that $\mu$ links both $\lambda(c,R)$ and $\lambda(c,S)$. 
Thus $\mu$ is neither a cusp leaf based at $c$ nor a boundary leaf of the crown $\Lambda_c$. 
By \reflem{CrownsInterleave}, there is a lower branch line $U$ in $N_c$ such that the cusp leaf $\lambda(c,U)$ links the boundary leaf $\lambda(R,S)$.
We deduce that $\mu$ also links the cusp leaf $\lambda(c,U)$.
If $\mu$ is a cusp leaf then we have a contradiction to \reflem{Unlinked}.
Thus by \reflem{CrossingParabolics}\refitm{TipsOfCrown} the endpoints of $\mu$ lie in distinct components of $\Circle$ minus the tips of the crown $\Lambda_c$. 
Thus, $\mu$ links exactly two boundary leaves of $\Lambda_c$. 
This contradicts \reflem{Laminations}\refitm{NoLinking}.
\end{proof}

\begin{figure}[htb]
\centering
\subfloat[]{
\labellist
\tiny\hair 2pt
\pinlabel {$\bdy U$} at 594 200
\pinlabel {$\bdy V$} at 60 500
\pinlabel {$\bdy R$} at -24 200
\pinlabel {$\bdy S$} at 520 500
\endlabellist
\includegraphics[width=0.27\textwidth]{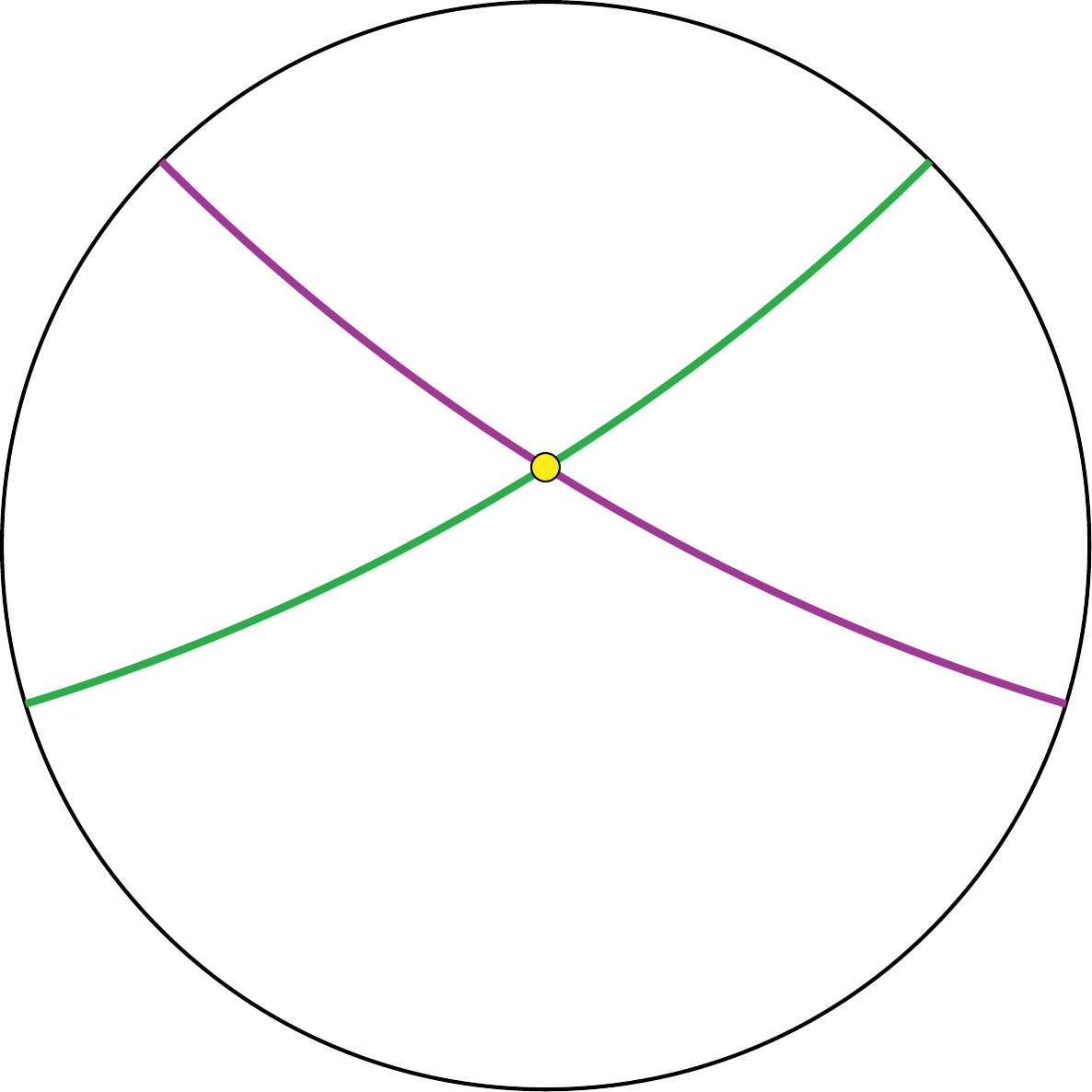}
\label{Fig:CrownsCross1}
}
\qquad
\subfloat[]{
\labellist
\tiny\hair 2pt
\pinlabel {$\bdy U$} at 594 200
\pinlabel {$\bdy V$} at 60 500
\pinlabel {$\bdy R$} at -24 200
\pinlabel {$\bdy S$} at 520 500
\pinlabel {$c$} at 70 70
\endlabellist
\includegraphics[width=0.27\textwidth]{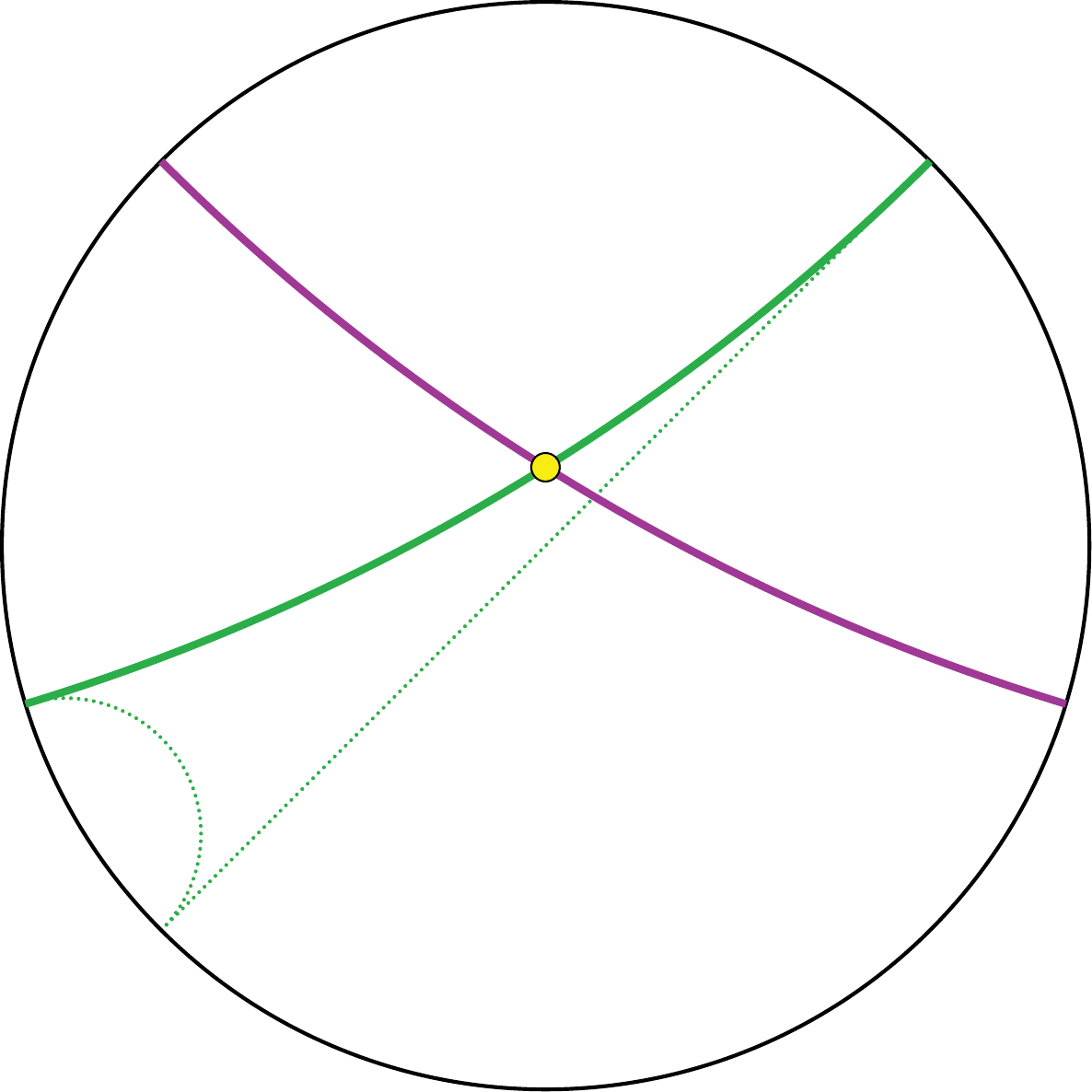}
\label{Fig:CrownsCross2}
}
\qquad
\subfloat[]{
\labellist
\tiny\hair 2pt
\pinlabel {$\bdy U$} at 594 200
\pinlabel {$\bdy V$} at 60 500
\pinlabel {$\bdy R$} at -24 200
\pinlabel {$\bdy S$} at 520 500
\pinlabel {$\bdy T$} at 200 -15
\pinlabel {$c$} at 70 70
\endlabellist
\includegraphics[width=0.27\textwidth]{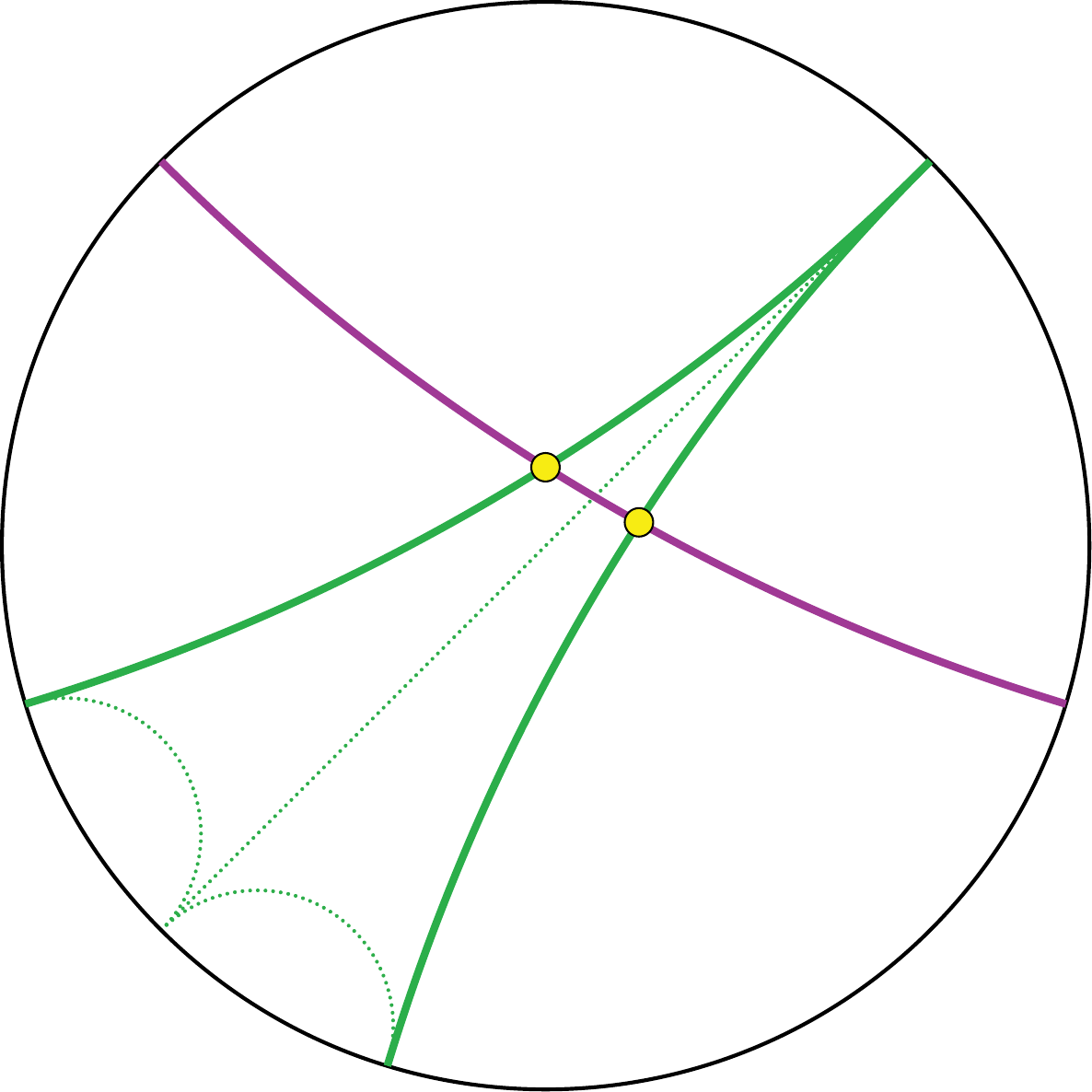}
\label{Fig:CrownsCross3}
}

\subfloat[]{
\labellist
\tiny\hair 2pt
\pinlabel {$\bdy U$} at 594 200
\pinlabel {$\bdy V$} at 60 500
\pinlabel {$d$} at 500 70
\pinlabel {$\bdy R$} at -24 200
\pinlabel {$\bdy S$} at 520 500
\pinlabel {$\bdy T$} at 200 -15
\pinlabel {$c$} at 70 70
\endlabellist
\includegraphics[width=0.27\textwidth]{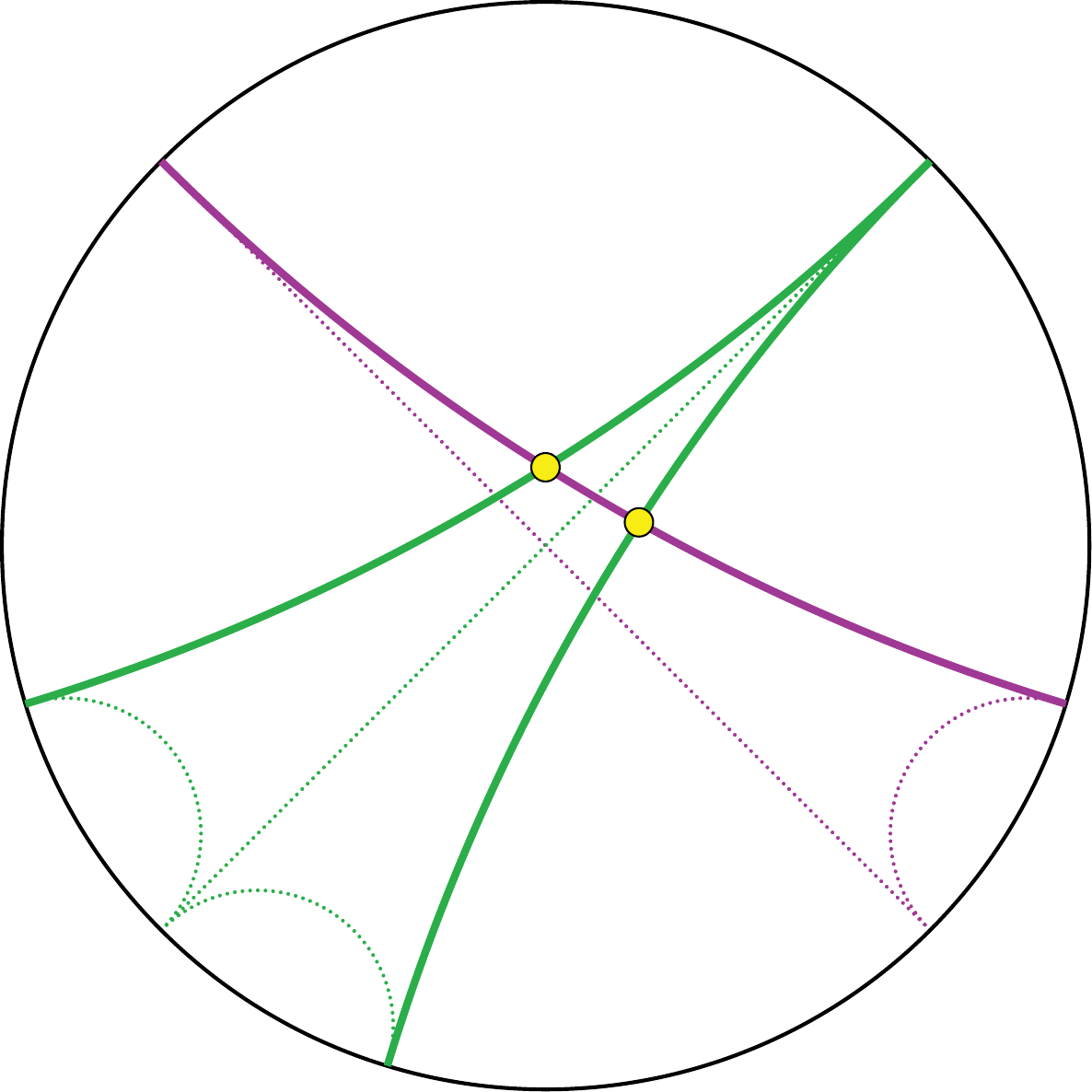}
\label{Fig:CrownsCross4}
}
\qquad
\subfloat[]{
\labellist
\tiny\hair 2pt
\pinlabel {$\bdy U$} at 594 200
\pinlabel {$\bdy V$} at 60 500
\pinlabel {$\bdy W$} at 370 -15
\pinlabel {$d$} at 500 70
\pinlabel {$\bdy R$} at -24 200
\pinlabel {$\bdy S$} at 520 500
\pinlabel {$\bdy T$} at 200 -15
\pinlabel {$c$} at 70 70
\endlabellist
\includegraphics[width=0.27\textwidth]{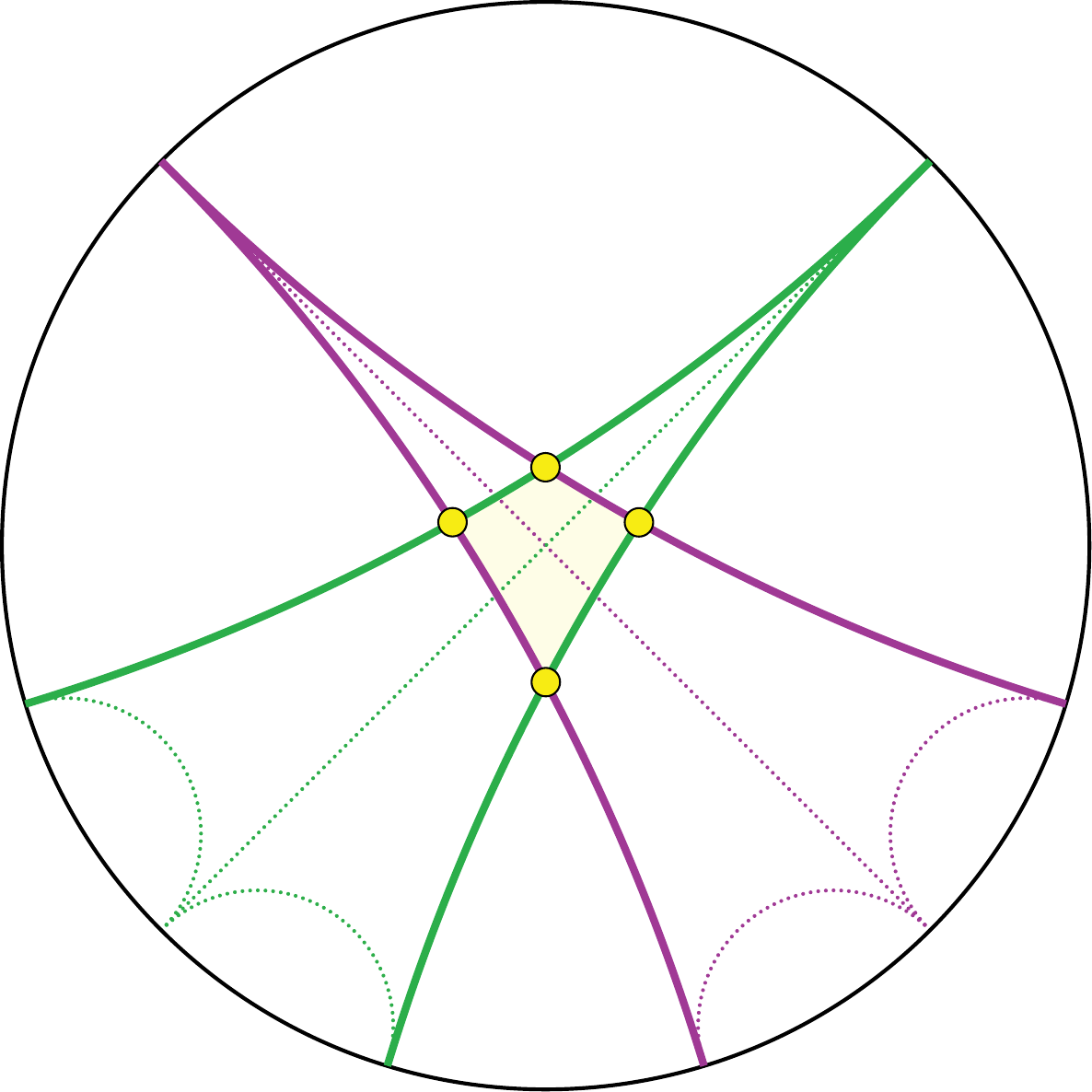}
\label{Fig:CrownsCross5}
}
\caption{Crowns with distinct cusps that have linking boundary leaves must cross. }
\label{Fig:CrownsCross}
\end{figure}

\begin{lemma}
\label{Lem:CrownsCross}
Suppose that $c$ and $d$ are distinct cusps. 
Then $\Lambda^c$ and $\Lambda_d$ are either unlinked or cross. 
\end{lemma}

Recall that $[x, y]^{\acw}$ is the closed arc in $\Circle$ between $x$ and $y$ and anticlockwise of $x$. 

\begin{proof}[Proof of \reflem{CrownsCross}]
By \refcor{Irrational}, the cusps $c$ and $d$ are distinct from the tips of both crowns.  
Also, by \reflem{NoMixedTypeAsymptotics}, the tips of $\Lambda^c$ are all distinct from the tips of $\Lambda_d$. 

Suppose that there is a boundary leaf $\lambda(R,S)$ of $\Lambda^c$ that links some boundary leaf $\lambda(U,V)$ of $\Lambda_d$.  
Breaking symmetry, we assume that the points $\bdy R, \bdy U, \bdy S, \bdy V$ lie on the circle in that anticlockwise order.  See \reffig{CrownsCross1}.  The cusp $c$ is distinct from these points.  Breaking symmetry we will assume that $c$ lies in $[\bdy R, \bdy U]^{\acw}$.  See \reffig{CrownsCross2}.  Let $T$ be the branch line of $N^c$ that is adjacent to $S$ and not equal to $R$.  Note that $\bdy T$ is on the opposite side of $\lambda(c, S)$ from $\bdy R$.  \reflem{TriangleLeafCross} implies that $\bdy T$ cannot lie in $[\bdy U, \bdy S]^{\acw}$.  So $\bdy T$ lies in $[c, \bdy U]^{\acw}$.  See \reffig{CrownsCross3}.

We now consider the location of the cusp $d$.  Again by \reflem{TriangleLeafCross} the cusp $d$ cannot lie in $[\bdy R, \bdy T]^{\acw}$.  Breaking symmetry, there are two cases.  Either $d$ lies in $[\bdy T, \bdy U]^{\acw}$ or $[\bdy U, \bdy S]^{\acw}$.

Suppose that $d$ lies in $[\bdy T, \bdy U]^{\acw}$.  See \reffig{CrownsCross4}.  Let $W$ be the branch line of $N_d$ that is adjacent to $V$ and not equal to $U$.  Note that $\bdy W$ is on the opposite side of $\lambda(d, V)$ from $\bdy U$.  \reflem{TriangleLeafCross} implies that $\bdy W$ cannot lie in $[\bdy V, \bdy T]^{\acw}$.  So $\bdy W$ lies in $[\bdy T, d]^{\acw}$.  
See \reffig{CrownsCross5}.  
Thus the crowns $\Lambda^c$ and $\Lambda_d$ cross, as desired. 

The case where $d$ instead lies in $[\bdy U, \bdy S]^{\acw}$ is easier; we omit it.
\end{proof}

\chapter{The link space}

We now give a direct construction of the \emph{link space} $\link(\calV)$.
With this done, and after obtaining several important properties, we prove \refthm{LinkIsLoom}. 
That is, we show that link spaces are \emph{loom spaces}, in the sense of~\cite[Definition~2.11]{SchleimerSegerman24}.

\begin{remark}
Here is one way to think of the link space construction, at least in the layered case. 
The upper and lower laminations $\Lambda^\calV$ and $\Lambda_\calV$ collapse to give a pair of dendrites. 
The fundamental group of the fibre acts on their product.  
Removing $\Delta_\calV^2$ and then taking the \emph{Guirardel core}~\cite{Guirardel05} gives (a quotient of) the link space. 

However, in the non-layered case the upper and lower laminations need not admit (projective) measures compatible with the action of $\pi_1(M)$ or any obvious subgroup.
Thus the geometric steps of the construction must be replaced.  
Also we do not have the same uniformity of the action.
Furthermore, we will need quite delicate topological control over $\link(\calV)$; in fact we will show it is a copy of the plane. 
Also, we will require an in-depth combinatorial understanding of the \emph{rectangles} in $\link(\calV)$ in order to complete our programme of classifying pseudo-Anosov flows (without perfect fits).
\end{remark}

\begin{definition}
\label{Def:LocallyVeering}
Suppose that $M$ is a three-manifold which is not necessarily orientable or compact.
Suppose that $\calT$ is a taut ideal triangulation of $M$.
Let $\cover{T}$ be the induced taut ideal triangulation on the universal cover $\cover{M}$.
Fix an orientation on $\cover{M}$.
We say that $\calT$ is \emph{locally veering} if $\cover{T}$ has a transverse veering structure.
\end{definition}

As an example, there is a locally veering triangulation of the Gieseking manifold with a single tetrahedron. 
The double cover is the veering triangulation shown in \reffig{VeerFigEight}.

\begin{remark}
\label{Rem:FixOrientations}
Suppose that $\calV$ is a locally veering triangulation of a three-manifold $M$.
Since all of our constructions occur in the universal cover $\cover{M}$, we may fix an orientation of $\cover{M}$ and a transverse veering structure on $\cover{\calV}$.
We define $\Lambda^\calV$ and $\Lambda_\calV$ with respect to these choices.
Making different choices may switch the laminations and may swap colours.
\end{remark}

\begin{definition}
\label{Def:PairSpace}
We define the \emph{pair space} to be
\[
\pair(\calV)  = \big\{ (\lambda, \mu) \in \Lambda^\calV \cross \Lambda_\calV \st \mbox{$\lambda$ and $\mu$ are linked} \big\}
\]
The topology on $\pair(\calV)$ is that of a subspace of a product.  
\end{definition} 

We define an equivalence relation on points of $\pair(\calV)$ as follows.  
If $\lambda$ and $\lambda'$ are asymptotic upper leaves and if $\mu$ links both then $(\lambda, \mu) \sim (\lambda', \mu)$.  
We do the same for asymptotic lower leaves.  
Finally we take the transitive closure.  
Let $[(\lambda,\mu)]$ be the equivalence class of $(\lambda,\mu)$.

\begin{lemma}
\label{Lem:EquivalenceInL}
An equivalence class $[(\lambda,\mu)]$ in $\pair(\calV)$ has either one, two, four or infinitely many representatives. 
These correspond exactly to the cases that:
\begin{itemize}
\item the leaves $\lambda$ and $\mu$ are both interior,
\item one is interior and the other is boundary,
\item the leaves are boundary about different cusps, or
\item the leaves are boundary about a single cusp.
\end{itemize}
\end{lemma}

\begin{proof}
\mbox{}
\begin{itemize}
\item Suppose that $\lambda$ and $\mu$ are interior leaves. By \reflem{Laminations}\refitm{Asymptotic}, the leaf $\lambda$ has no asymptotic partner in $\Lambda^\calV$. The same holds for $\mu$.

\item Breaking symmetry, suppose that $\mu$ is an interior leaf and $\lambda = \lambda(R,S)$ is a boundary leaf. Let $c$ be the cusp such that $N^c$ contains $R$ and $S$.  Breaking symmetry again, the leaf $\mu$ also links $\lambda(c, S)$. Let $T$ be the other branch line adjacent to $S$.  Applying \reflem{TriangleLeafCross} we have 
\[
[(\lambda, \mu)] = \{ (\lambda(R, S), \mu), (\lambda(S, T), \mu) \}
\] 

\item Suppose that both $\lambda$ and $\mu$ are boundary leaves, adjacent to distinct cusps $c$ and $d$.  Thus by \reflem{CrownsCross}, the crowns $\Lambda^c$ and $\Lambda_d$ cross, and $[(\lambda,\mu)]$ consists of four pairs.  See \reffig{CrossingCrowns}.

\item Suppose that both $\lambda$ and $\mu$ are boundary leaves, both adjacent to the cusp $c$. Thus by \reflem{CrownsInterleave},  the crowns $\Lambda^c$ and $\Lambda_c$ interleave, and $[(\lambda,\mu)]$ consists of a countable collection of pairs. See \reffig{InterleavingCrowns}. \qedhere
\end{itemize}
\end{proof}

We call the last type of equivalence class a \emph{cusp class}.  

\begin{definition}
\label{Def:LinkSpace}
The \emph{link space} $\link(\calV)$ is the quotient of $\pair(\calV)$, minus the cusp classes.  That is: 
\[
\link(\calV) \,\,=\,\, \elec \,\,-\,\, \{\mbox{cusp classes}\}
\]
We give $\link(\calV)$ a topology by realising it as a subspace of a quotient.  
\end{definition}

\begin{definition}
\label{Def:UpperFoliation}
We now define the \emph{upper foliation} $F^\calV$ of $\link(\calV)$.
It consists of two kinds of leaves.  Suppose that $\lambda \in \Lambda^\calV$ is an interior leaf.  Then we take 
\[
\ell^\lambda = \big\{ [(\lambda, \mu)] \in \link(\calV) \mbox{ for some $\mu \in \Lambda_\calV$} \big\}
\]
This is a \emph{non-cusp leaf} of $F^\calV$.  

Suppose instead that $c$ is a cusp and $S$ is an upper branch line in $N^c$. 
Let $\lambda(c,S)$ be the corresponding upper cusp leaf.  
By \reflem{BranchLines}\refitm{Cyclic}, there are exactly two upper branch lines $S'$ and $S''$ (in the boundary of $N^c$) that are adjacent to $S$. 
Take $\lambda$ to be one of $\lambda(S', S)$ or $\lambda(S, S'')$.  
Then we define 
\[
\ell^S = \big\{ [(\lambda, \mu)] \in \link(\calV) \st \mbox{$\mu \in \Lambda_\calV$ links $\lambda(c,S)$} \big\}
\]
This is a \emph{cusp leaf} of $F^\calV$.  
(Note that applying \reflem{TriangleLeafCross} twice shows that $\ell^S$ does not depend on whether we chose $\lambda(S', S)$ or $\lambda(S, S'')$ above.)  

We define the lower foliation $F_\calV$ similarly.
\end{definition}

\begin{theorem}
\label{Thm:LinkSpace}
Suppose that $M$ is a three-manifold equipped with a locally veering triangulation $\calV$. 
Then we have the following.
\begin{enumerate}
\item
\label{Itm:LinkSpacePlane}
The link space $\link(\calV)$ is homeomorphic to $\RR^2$.
\item
\label{Itm:LinkSpaceTransverse}
$F^\calV$ and $F_\calV$ are transverse foliations of $\link(\calV)$.  
\item 
\label{Itm:LinkSpaceDense}
Non-cusp leaves are dense in each of $F^\calV$ and $F_\calV$. 
\item
\label{Itm:LinkSpaceAction}
The induced action of $\pi_1(M)$ on $\link(\calV)$ is continuous and faithful.
\end{enumerate}
Suppose additionally that $M$ is orientable and $\calV$ is transverse veering. 
Then we have the following.
\begin{enumerate}[resume]
\item
\label{Itm:LinkSpaceOri}
There is a canonical orientation on $\link(\calV)$. 
\item
\label{Itm:LinkSpaceOriAction}
The induced action of $\pi_1(M)$ on $\link(\calV)$ is orientation preserving.
\end{enumerate}
\end{theorem}

Before proving the theorem, we develop some structure.   
If $M$ is orientable, and $\calV$ is transverse veering, then there are induced orientations in the universal cover. 
If not, then we appeal to \refrem{FixOrientations} to obtain such orientations in the universal cover. 

\begin{definition}
\label{Def:Rectangle}
A \emph{rectangle} $R$ in $\link(\calV)$ is an embedding of $(0,1)^2$ into $\link(\calV)$ that sends line segments parallel to the $y$--axis ($x$--axis) to arcs of the upper (lower) foliations.  
\end{definition}


\begin{definition}
\label{Def:Boundary}
Suppose that $R$ is a rectangle in $\link(\calV)$.  
Suppose that $(\ell_i)$ is a monotonic sequence of leaves of $R \cap F^\calV$ (or of $R \cap F_\calV$), exiting $R$.  
Then the set of accumulation points of $(\ell_i)$ is one of the four \emph{sides} of $R$.  
\end{definition}

Let $\closure{R}$ denote the closure of $R$ in $\link(\calV)$.
As usual, we set $\bdy R = \closure{R} - \interior(R) = \closure{R} - R$.
Note that the sides of a rectangle $R$ lie in $\bdy R$.  
Note also that a side need not be connected.  
Any point of $\bdy R$ lying in two sides is called a \emph{material corner} of $R$.  

\begin{definition}
\label{Def:Ideal}
Suppose that $R$ is a rectangle in $\link(\calV)$.  
Note that the closure of $R$, taken in $\elec$, consists of $\closure{R}$, together with some number of cusp classes.
We say that such a cusp class $p$ is an \emph{ideal point} of $R$.
We say that $p$ is in the \emph{interior} of a side of $R$ if it is accumulated by only that side.
We say that $p$ is an \emph{ideal corner} of $R$ if it is accumulated by exactly two, adjacent, sides of $R$.
\end{definition}

\subsection{Edge rectangles}

We take up the task of building rectangles in $\link(\calV)$. 

\begin{definition}
\label{Def:EdgeRect}
Fix an edge $e$ of $\cover{\calV}$.  Fix $(\lambda, \mu) \in \pair(\calV)$.  We say that a pair $(\lambda, \mu) \in \pair(\calV)$ \emph{links} $e$ if both $\lambda$ and $\mu$ link $e$.  The \emph{edge rectangle} $\rect(e)$ is defined to be
\[
\rect(e) = \{ p \in \link(\calV) \st \mbox{every representative $(\lambda,\mu)$ of $p$ links $e$} \} \qedhere
\]
\end{definition}

\begin{lemma}
\label{Lem:EdgeRect}
Suppose that $e$ is an edge of $\cover{\calV}$ with endpoints at cusps $c$ and $d$.
\begin{enumerate}
\item
\label{Itm:EdgeRectRect}
The edge rectangle $\rect(e)$ is a rectangle in the sense of \refdef{Rectangle}.
\item
\label{Itm:EdgeRectBdy}
The boundary of $\rect(e)$ is contained in the union of four cusp leaves that alternatingly lie in $F^\calV$ and $F_\calV$.
The closure of each side in $\elec$ is a closed interval connecting an ideal corner to a material corner.  
The former are the cusp classes coming from $c$ and $d$. 
\end{enumerate}
\end{lemma}

\begin{figure}[htb]
\centering
\labellist
\small\hair 2pt
\pinlabel {$d$} [r] at 0 21
\pinlabel {$c$} [l] at 141 56
\pinlabel {$e$} [b] at 69 40
\pinlabel {$\rect(e)$} at 289 36
\pinlabel {$d$} [r] at 198 8
\pinlabel {$c$} [l] at 378 69
\endlabellist
\includegraphics[width=0.6\textwidth]{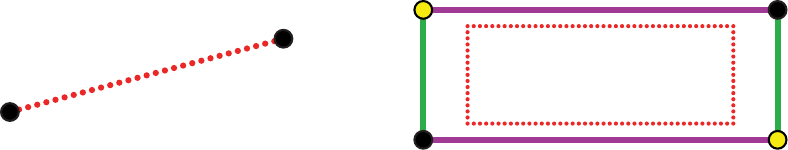}
\caption{Left: a red edge $e$.  Right: the edge rectangle $\rect(e)$.  The dotted, shrunken rectangle is intended to remind the reader that $\rect(e)$ does not include its boundary.  We have coloured the material corners yellow and the ideal corners black.}
\label{Fig:LinkSpaceEdge}
\end{figure}

\begin{proof}
Breaking symmetry, we suppose that $e$ is coloured red.  We orient $e$ from $c$ to $d$.   \reflem{CrownsInterleave} tells us that the crowns $\Lambda^c$ and $\Lambda_c$ interleave, as do the crowns $\Lambda^d$ and $\Lambda_d$.  

\begin{claim*}
There are boundary leaves 
\[
\lambda(S, S') \in \Lambda^c, \quad
\lambda(U, U') \in \Lambda_c, \quad
\lambda(T, T') \in \Lambda^d, \quad
\lambda(V, V') \in \Lambda_d
\]
so that each links the next, cyclically. 
\end{claim*}

\begin{proof}
Fix a layer $K$ of a layering, chosen so that $e$ lies in $K$.
Let $P$ be the maximal strip of majority red faces in $K$ to the right of $e$.
See \reffig{FiniteFellowTravel}.
When $P$ is finite, we add a single majority blue face $f'$ at the end of $P$.
In this case, let $b$ be the cusp of $f'$ not meeting a red edge of $P$.

By construction, the boundary of $P$ consists of one red edge, $e$, and the rest blue.
Recall that $\tau^P$ and $\tau_P$ are the upper and lower tracks in $P \subset K$.
Consider the blue edge $e' \subset \bdy P$ adjacent to $c$.
Since $e$ and all interior edges of $P$ are red, the face of $P$ meeting $e'$ contains a lower track-cusp $u$ that points away from $c$.
Again see \reffig{FiniteFellowTravel}.
Let $U$ be the branch line containing $u$.
Similarly, let $e'' \subset \bdy P$ be the blue edge adjacent to $d$.
The face of $P$ meeting $e''$ contains an upper track-cusp $t$ that points away from $d$.
Let $T$ be the branch line containing $t$. 

Partition the blue edges of $P$ into two sets: those connected either to $c$ or to $d$ by a sequence of blue edges (not passing through $b$ in the case that $P$ is finite).
Call these the $c$--edges and the $d$--edges.
Due to the colouring of the edges of $P$, the cusp line $\ell_K(c, U)$ cannot exit $P$ through a $c$--edge.
Similarly, $\ell^K(d, T)$ cannot exit $P$ through a $d$--edge.
By \reflem{NoMixedTypeAsymptotics}, at least one must exit $P$, and hence the cusp leaves $\lambda(c, U)$ and $\lambda(d, T)$ link.

\begin{figure}[htbp]
\labellist
\small\hair 2pt
\pinlabel {$c$} [tr] at 0 0
\pinlabel {$d$} [br] at 0 120
\pinlabel {$e$} [r] at 0 30
\pinlabel {$e'$} [tl] at 30 0
\pinlabel {$e''$} [bl] at 30 122
\pinlabel {$u$} [r] at 83 53
\pinlabel {$t$} [r] at 46 70
\pinlabel {$b$} [l] at 400 61
\endlabellist
\[
\begin{array}{c}
\includegraphics[width = 0.7\textwidth]{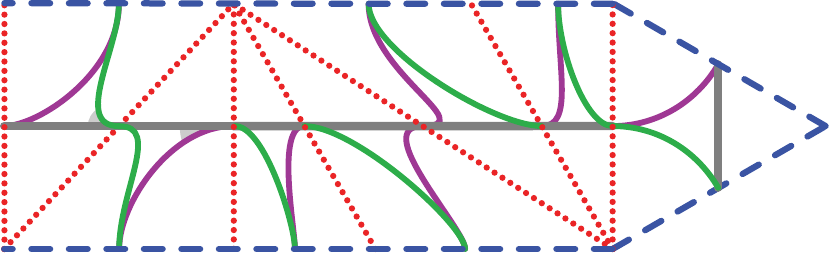}
\end{array}
\]
\caption{A possible (finite) strip of faces $P$ starting at the edge $e$.}
\label{Fig:FiniteFellowTravel}
\end{figure}

We make the same argument to the left of $e$.
This produces branch lines $S$ in $N^c$ and $V$ in $N_d$ so that $\lambda(d, V)$ links $\lambda(c, S)$.
Thus the points 
\[
c, \bdy T, \bdy U, d, \bdy S, \bdy V
\]
appear in $\Circle$ in that anticlockwise order.
Let $T'$ be the branch line in $N^d$, adjacent to $T$, chosen so that $c$ lies in $[\bdy T', \bdy T]^{\acw}$.
We define $S'$, $U'$, and $V'$ similarly.
Applying \reflem{CrownsInterleave} and \reflem{CrownsCross} we deduce that the points
\[
c, \bdy S', \bdy T, \bdy U, \bdy V', d, \bdy T', \bdy S, \bdy V, \bdy U'
\]
appear in $\Circle$ in that anticlockwise order.
See \reffig{LeafIdentificationsEdge}.
\end{proof}

\begin{figure}[htb]
\centering
\labellist
\small\hair 2pt
\pinlabel {$c$} [t] at 285 0
\pinlabel {$\bdy S'$} [tl] at 475 80
\pinlabel {$\bdy T$} [l] at 565 230
\pinlabel {$\bdy U$} [l] at 565 341
\pinlabel {$\bdy V'$} [bl] at 475 492
\pinlabel {$d$} [b] at 285 570
\pinlabel {$\bdy T'$} [br] at 97 492
\pinlabel {$\bdy S$} [r] at 5 341
\pinlabel {$\bdy V$} [r] at 5 230
\pinlabel {$\bdy U'$} [tr] at 93 80
\endlabellist
\includegraphics[width=0.6\textwidth]{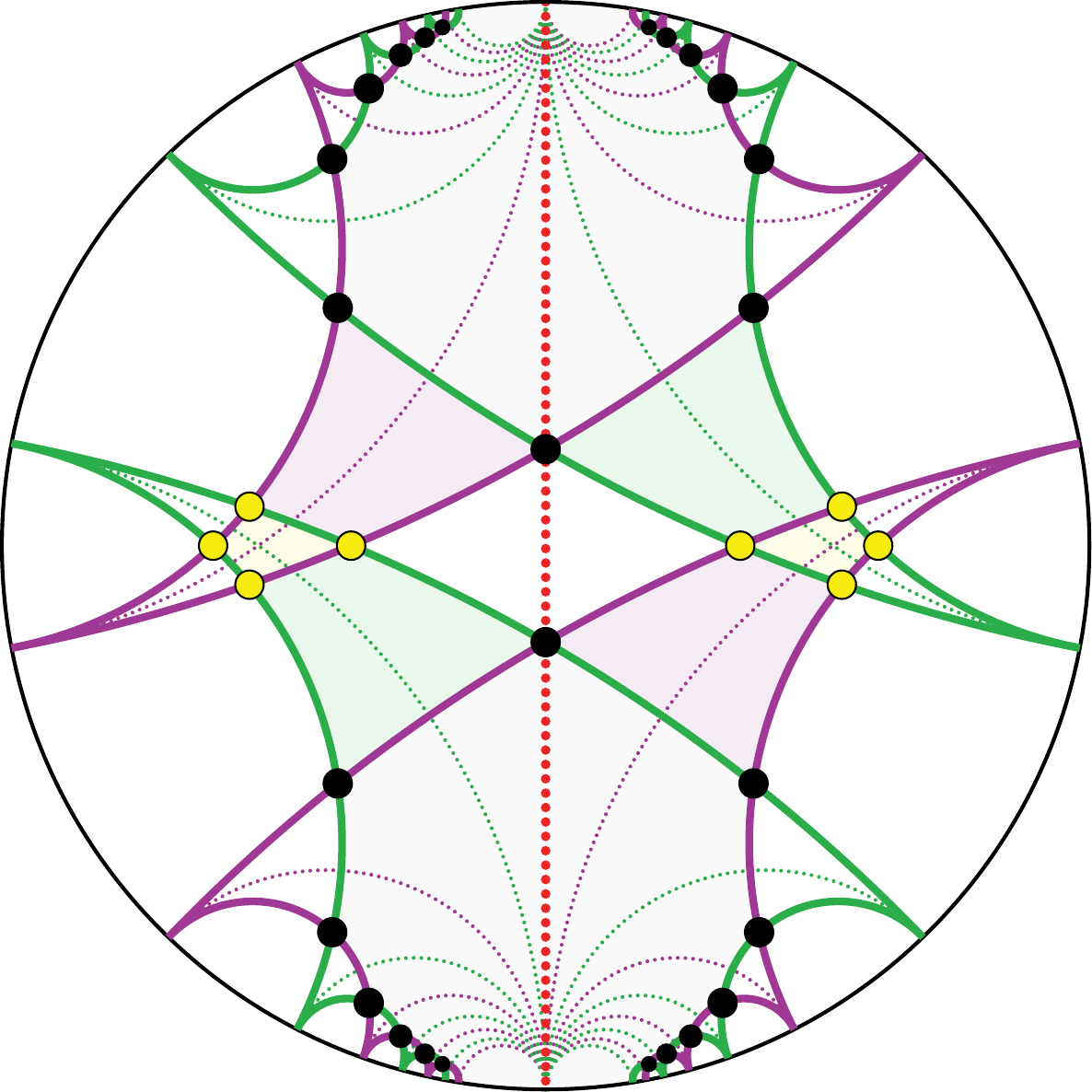}
\caption{The upper and lower crowns for cusps at the two ends of an edge.}
\label{Fig:LeafIdentificationsEdge}
\end{figure}

Recall that $e$ meets the cusps $c$ and $d$.
Set $\Lambda^e = \Lambda^{(c, d)}$ as in \refdef{LeavesThatLink}.
Note that the leaves of $\Lambda^e$ are exactly those of $\Lambda^\calV$ that link $e$.
We define $\Lambda_e$ in similar fashion.
By the claim immediately above, every $\lambda$ in $\Lambda^e$ links every $\mu$ in $\Lambda_e$.
Thus $\Lambda^e \cross \Lambda_e$ is a subset of $\pair(\calV)$.
By \reflem{Cantor} this is a product of two copies of the Cantor set $\calC$.
Mapping this product to its quotient in $\elec$ is realised by the applying the Cantor function in each coordinate.
That is, we replace ternary expansions by binary by replacing all twos by ones.
Now, since $0.\bar{1} = 1.\bar{0}$ and applying \reflem{EquivalenceInL}, pairs and four-tuples of points are identified in the desired fashion.
Thus the image in $\elec$ is homeomorphic to the closed square $[0,1]^2$.  

Note, however, that $\lambda(S, S')$ is smallest in $(\Lambda^e,<^e)$.
So, its adjacent boundary leaf in $\Lambda^c$, sharing the endpoint $\bdy S$, does not link $e$.
Thus any point $p \in \link(\calV)$, having a representative of the form $(\lambda(S, S'), \mu)$, is not in $\rect(e)$.
The same holds for $\lambda(U', U)$ as well as $\lambda(T', T)$ and $\lambda(V, V')$ in $\Lambda_e$.
This proves that $\rect(e)$ is homeomorphic to the open square $(0,1)^2$.
Fixing one coordinate and applying the Cantor function to the other produces the desired foliations, obtaining \refitm{EdgeRectRect}.

We deduce that $\bdy \rect(e)$ lies in the union of the cusp leaves $\ell^S$, $\ell_U$, $\ell^T$, and $\ell_V$.
Note that $\ell^S \cap \ell_V$ and $\ell^T \cap \ell_U$ give points of $\link(\calV)$.
These are the material corners of $\rect(e)$.
On the other hand $\ell^S$ and $\ell_U$ have a common cusp class in $\elec$ coming from $c$.
Similarly, $\ell^T$ and $\ell_V$ have a common cusp class coming from $d$.
These are the ideal corners of $\rect(e)$.
This gives \refitm{EdgeRectBdy}. 
\end{proof}

With \reflem{EdgeRect} in hand (and recalling \refrem{FixOrientations}), we make the following definition.

\begin{definition}
\label{Def:EdgeRectOri}
We now define the \emph{induced orientation} on edge rectangles.
Suppose that $e$ is red. 
Then we orient the sides of $\rect(e)$ as follows. 
The sides of $\rect(e)$ contained in $\ell^S$ and $\ell^T$ are oriented towards the cusps $c$ and $d$, respectively. 
The sides of $\rect(e)$ contained in $\ell_U$ and $\ell_V$ are oriented away from the cusps $c$ and $d$, respectively. 
This induces an orientation on $\rect(e)$.
Edge rectangles for blue edges are oriented using the opposite convention.
\end{definition}

In the following, we will abuse notation and refer to a ``cusp class coming from a cusp $c$'' simply as $c$.

\begin{remark}
\label{Rem:BoundaryEdgeRect}
We have the following useful characterisation of points of $\bdy \rect(e)$:
these are the classes where some (but not every) representative links $e$.  
\end{remark}


\begin{lemma}
\label{Lem:EdgesCover}
Suppose that $K$ is a layer of a layering.  Then
\[
\left\{ \closure{\rect(e)} \st \mbox{$e$ is an edge of $K$} \right\}
\]
is a closed cover of $\link(\calV)$. 
\end{lemma}

\begin{proof}
Fix any linking pair $(\lambda, \mu) \in \pair(\calV)$.
Let $\ell$ and $m$ be the train lines in $\tau^K$ and $\tau_K$ given by \reflem{Laminations}\refitm{LeavesAreCarried}.
These cross.
Consulting \reffig{VeeringTriangles} we deduce that $\ell$ and $m$ are tangent to each other along some train interval, perhaps of length zero.
Some edge or edges are transverse to this interval;
any such edge $e$ links both $\lambda$ and $\mu$.
Thus $[(\lambda, \mu)]$ lies in $\closure{\rect(e)}$.  
\end{proof}

\subsection{Face rectangles}

\begin{definition}
\label{Def:FaceRect}
Fix a face $f$ in $\cover{\calV}$.
We say that a pair $(\lambda, \mu) \in \pair(\calV)$ \emph{links} $f$ if the pair links at least one of edges of $f$. 
The \emph{face rectangle} $\rect(f)$ is defined to be
\[
\rect(f) = \{ p \in \link(\calV) \st \mbox{every representative $(\lambda,\mu)$ of $p$ links $f$} \} \qedhere
\]
\end{definition}

Suppose that $e$, $e'$, and $e''$ are the edges of a face $f$. We choose the labels so that $e'$ and $e''$ have the same colour, and the track-cusp of $\tau^f$ meets $e''$. Note that this implies that the track-cusp of $\tau_f$ meets $e'$. See \reffig{VeeringTriangles}.
From the definitions we deduce that the edge rectangles $\rect(e)$, $\rect(e')$, and $\rect(e'')$ are contained in $\rect(f)$.
However, $\rect(f)$ is not the union of the edge rectangles.
For example, there are points in $\link(\calV)$ with one representative linking only $e$, and another representative linking only $e'$, say.
See the right side of \reffig{LinkSpaceFace}.

\begin{figure}[htb]
\centering
\labellist
\tiny\hair 2pt
\pinlabel {$s$} at 77.5 91
\pinlabel {$u$} at 88 80
\small\hair 2pt
\pinlabel {$c$} [l] at 140 141
\pinlabel {$c'$} [r] at 3 107
\pinlabel {$c''$} [l] at 105 8
\pinlabel {$e$} [tr] at 51 52
\pinlabel {$e'$} [l] at 120 70
\pinlabel {$e''$} [b] at 66 123
\pinlabel {$c$} [l] at 378 141
\pinlabel {$c'$} [r] at 201 82
\pinlabel {$c''$} [t] at 304 3
\pinlabel {$\rect(e)$} at 250 42
\pinlabel {$\rect(e')$} at 337 50
\pinlabel {$\rect(e'')$} at 260 108
\endlabellist
\includegraphics[width=0.6\textwidth]{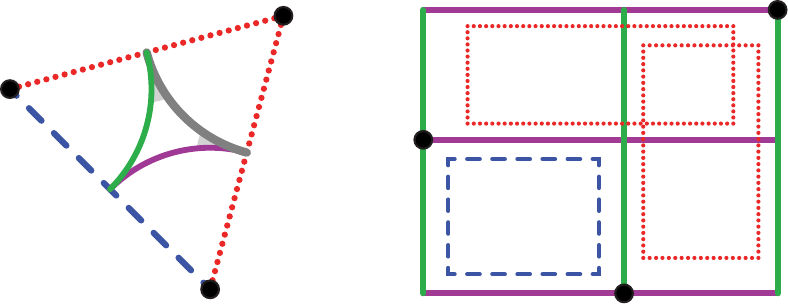}
\caption{Left: a majority red triangle with upper and lower tracks shown.
Right: the three edges $e$, $e'$, $e''$ of a face $f$ of $\cover{\calV}$ give three edge rectangles, all subrectangles of the face rectangle $\rect(f)$.}
\label{Fig:LinkSpaceFace}
\end{figure}

\begin{definition}
\label{Def:Spans}
Let $Q$ and $R$ be rectangles.
We say that $Q$ \emph{south-north spans $R$} 
if there is a leaf of $R \cap F^\calV$ contained in $Q$.
As a refinement of this, we say that $Q$ \emph{strictly south-north spans $R$}
if there is a leaf $\ell$ of $R \cap F^\calV$ so that $\closure{\ell}$ is contained in $Q$ and is compact.

We make similar definitions, replacing south-north with west-east and replacing $F^\calV$ with $F_\calV$.
\end{definition}


Note that in \refdef{Spans}, we do \emph{not} assume symmetry.
That is, $Q$ may south-north span $R$ while $R$ does not west-east span $Q$.

\begin{lemma}
\label{Lem:FaceRect}
Suppose that $f$ is a face of $\cover{\calV}$ with edges $e$, $e'$, and $e''$.
Suppose that $e'$ and $e''$ have the same colour, and that the track cusp of $\tau^f$ meets $e''$.
Suppose that $c$, $c'$, and $c''$ are the cusps opposite $e$, $e'$, and $e''$ respectively.
\begin{enumerate}
\item
\label{Itm:FaceRectRect}
The face rectangle $\rect(f)$ is a rectangle in the sense of \refdef{Rectangle}.
\item
\label{Itm:FaceRectBdy}
The boundary of $\rect(f)$ is contained in a union of six cusp leaves.
Two sides of $\rect(f)$ are segments connecting the ideal corner $c$ to distinct material corners.
The two remaining sides meet two material corners and contain $c'$ and $c''$ respectively. 
\item
\label{Itm:FaceSpan}
The edge rectangles $\rect(e')$ and $\rect(e'')$ respectively south-north and west-east span the face rectangle $\rect(f)$.
\end{enumerate}
\end{lemma}

\begin{proof}
Breaking symmetry, suppose that $e'$ and $e''$ are red and $e$ is blue.
Let $s$ be the track-cusp of $\tau^f$ on $e''$;
let $u$ be the track-cusp of $\tau_f$ on $e'$.
Let $S$ be the upper branch line through $s$;
let $U$ be the lower branch line through $u$. 

\begin{claim*}
The intersection $\closure{\rect(e)} \cap \closure{\rect(e')}$ is equal to $\ell^S \cap \bdy \rect(e)$. 
\end{claim*}

\begin{proof} 
Suppose that $p$ is a point of $\link(\calV)$.
Thus $p$ is not a cusp class.
We claim that the following are equivalent.  
\begin{enumerate}[label=(\alph*)]
\item 
The point $p$ lies in $\closure{\rect(e)} \cap \closure{\rect(e')}$.
\item 
The point $p$ has representatives $(\lambda,\mu)$ and $(\lambda',\mu')$, where $(\lambda,\mu)$ links $e$, where $(\lambda',\mu')$ links $e'$, where $\lambda$ and $\lambda'$ are distinct and asymptotic, and where $\mu$ and $\mu'$ are asymptotic.
\item 
The point $p$ has representatives $(\lambda,\mu)$ and $(\lambda',\mu')$, where $(\lambda,\mu)$ links $e$, where $(\lambda',\mu')$ links $e'$, where $\lambda$ and $\lambda'$ are the distinct boundary leaves sharing the endpoint $\bdy S$, and where $\mu$ and $\mu'$ are asymptotic.
\item 
The point $p$ lies in $\ell^S \cap \bdy \rect(e)$.
\end{enumerate}

Statement (a) implies the first half of statement (b) by definition.
Consulting the left-hand side of \reffig{LinkSpaceFace}, we see that $\lambda$ does not link $e'$ and that $\lambda'$ does not link $e$, so $\lambda$ and $\lambda'$ are distinct.
The various leaves are asymptotic by \reflem{EquivalenceInL}.
Statement (b) implies statement (a) by definition.

Statement (b), and \reflem{Laminations}\refitm{Asymptotic}, imply that $\lambda$ and $\lambda'$ are boundary leaves.  They share an endpoint; 
examining the track $\tau^f$ this endpoint must be $\bdy S$.
Thus (b) implies (c); 
the converse direction follows from the definition of asymptotic.

The first part of statement (c) implies that $p$ lies on $\ell^S$.
Since $(\lambda,\mu)$ links $e$, but $(\lambda',\mu')$ does not, \refrem{BoundaryEdgeRect} tells us that the point $p$ lies in $\bdy \rect(e)$. 
This gives (d).
To see the converse, since $p$ lies in $\ell^S$, we may choose two representatives, $(\lambda,\mu)$ and $(\lambda',\mu)$, so that $\lambda$ and $\lambda'$ are distinct asymptotic boundary leaves sharing the endpoint $\bdy S$ and $\lambda$ links $e$ while $\lambda'$ links $e'$.
If $\mu$ links $e$ then it also links $e'$ and we are done.
If not, then as $p$ lies in $\bdy \rect(e)$, we deduce that $\mu$ is asymptotic to some $\mu'$ which links $e$.
Replacing $\mu$ with $\mu'$ proves the claim.
\end{proof}

Similarly, the intersection $\closure{\rect(e)} \cap \closure{\rect(e'')}$ is equal to $\ell_U \cap \bdy \rect(e)$.

We adopt the following notation.
Set $R_{00} = \rect(e)$.
From $\rect(e')$ we remove the open interval $\ell_U \cap \rect(e')$ to obtain a pair of open rectangles $R_{10}$ and $R_{11}$.  We choose these so that $\closure{R}_{00}$ and $\closure{R}_{10}$ intersect in an arc; 
thus $\closure{R}_{00}\cap\closure{R}_{10}$ is exactly $\ell^S \cap \bdy \rect(e)$.
Appealing to \reflem{EdgeRect}\refitm{EdgeRectBdy} this interval is homeomorphically embedded in the boundaries of $R_{00}$ and $R_{10}$.
We cut $\rect(e'')$ using the interior interval $\ell^S \cap \rect(e'')$ into a pair of rectangles $R_{01}$ and $R'_{11}$.
Again $R_{00}$ and $R_{01}$ are nicely glued along $\ell_U \cap \bdy \rect(e)$.
We claim that $R'_{11} = R_{11}$.
To see this, note that if $\lambda$ and $\mu$ link each other, link $e'$, and neither link $e$, then necessarily both link $e''$. 

We deduce that the closures of the four rectangles $R_{00}$, $R_{10}$, $R_{11}$, and $R_{01}$ meet cyclically along subintervals of $\ell^S$ and $\ell_U$.
The only point contained in all four closures is their common material corner $\ell^S \cap \ell_U$.
This proves \refitm{FaceRectRect}.
Conclusion \refitm{FaceRectBdy} follows from \refitm{FaceRectRect} and \reflem{EdgeRect}\refitm{EdgeRectBdy}.
The two parts of Conclusion \refitm{FaceSpan} follow from the decomposition of $\rect(f)$ given above.
\end{proof}

\begin{lemma}
\label{Lem:EdgeSide}
Suppose that $K$ is a layer of a layering. 
Suppose that $e$ is an edge of $K$.
Suppose that $s$ is a side of $\rect(e)$.
Then there is a face of $K$ whose face rectangle contains the interior of $s$.
\end{lemma}


\begin{figure}[htbp]
\centering
\labellist
\small\hair 2pt
\pinlabel {(a)} at -20 742
\pinlabel {(b)}  at 522 777
\pinlabel {(c)}  at 336 500
\pinlabel {(d)} at 1070 597
\pinlabel {(e)} at 783 185
\pinlabel {(f)} at -20 295
\scriptsize
\pinlabel {$c$} [bl] at 174 760
\pinlabel {$s$} [l] at 175 646
\pinlabel {$e$} [b] at 90 726
\pinlabel {$e$} [b] at 390 990
\pinlabel {$e$} [b] at 737 990
\pinlabel {$e$} [b] at 455 626
\pinlabel {$e'$} [tr] at 427 560
\pinlabel {$e''$} [l] at 505 584
\pinlabel {$e'''$} [l] at 1170 640
\pinlabel {$e'''$} [bl] at 704 279
\pinlabel {$e'''$} [bl] at 1023 266
\pinlabel {$e'''$} [bl] at 205 395
\endlabellist
\includegraphics[width=\textwidth]{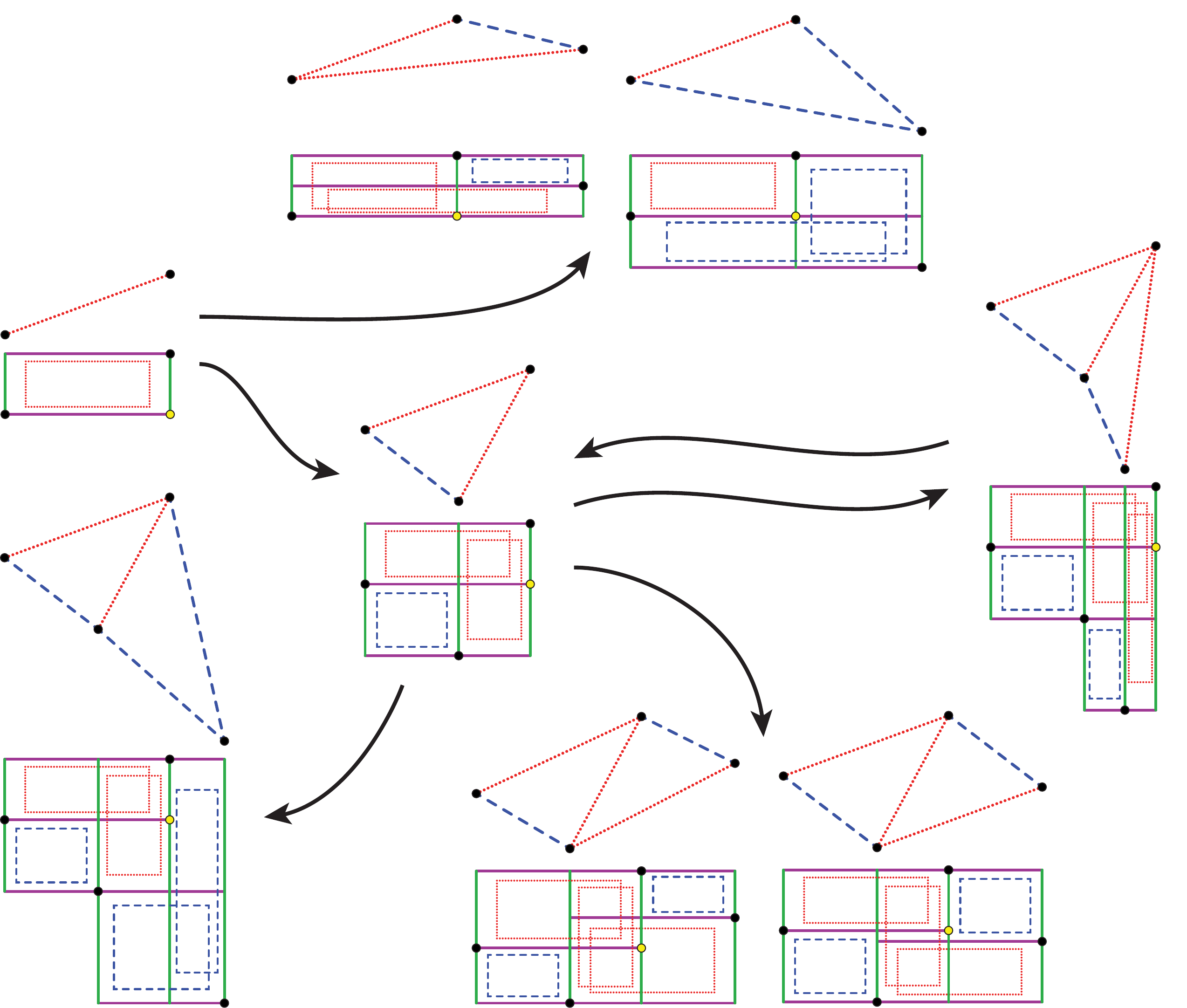}
\caption{
The logic of the proof of \reflem{EdgeSide}. The material point at the end of the side $s$ is marked with a yellow dot.
} 
\label{Fig:LinkSpaceFaceFlowChart}
\end{figure}

\begin{proof}
Breaking symmetry, we suppose that $e$ is red and $s$ is the east side of $\rect(e)$.
Let $c$ be the cusp corresponding to the northeastern ideal corner of $\rect(e)$.
See \reffig{LinkSpaceFaceFlowChart}(a).
Let $f$ be the face of $K$ incident to $e$ so that $\rect(f)$ extends to the south or east of $\rect(e)$.

Label the edges of $f$ as $e$, $e'$, and $e''$, ordered anticlockwise. 
Thus $e''$ is the other edge of $f$ incident to $c$.
There are now two cases as $e''$ is blue or red.
If $e''$ is blue then the interior of $s$ lies in $\rect(f)$ and we are done.
The two subcases in which $e''$ is blue are shown in \reffig{LinkSpaceFaceFlowChart}(b). 

Suppose instead that $e''$ is red.
See \reffig{LinkSpaceFaceFlowChart}(c).
Suppose that $f'$ is the face that meets $f$ along $e''$.
Let $e'''$ be the edge of $f'$ that meets $c$.
There are two cases as $e'''$ is red or blue.
\begin{itemize}
\item
If $e'''$ is red then $s$ again lies in $\bdy \rect(f')$.
So, in this case, we replace $f$ by $f'$ and continue rotating about $c$ in the anticlockwise direction.
See \reffig{LinkSpaceFaceFlowChart}(d).
By \reflem{CannotTurnLeftForever} we do not revisit this case forever and must eventually find that $e'''$ is blue.
\item
If $e'''$ is blue then there are two cases, as the remaining edge of $f'$ is red (\reffig{LinkSpaceFaceFlowChart}(e), two possible combinatorial configurations) or blue (\reffig{LinkSpaceFaceFlowChart}(f)).
In all cases, the interior of $s$ lies in $\rect(f')$.  \qedhere
\end{itemize}
\end{proof}

\begin{lemma}
\label{Lem:FacesCover}
Suppose that $K$ is a layer of a layering.  Then
\[
\{ \rect(f) \st \mbox{$f$ is a face of $K$} \}
\]
is an open cover of $\link(\calV)$.  
This cover has no finite subcover.
\end{lemma}

\begin{proof}
Fix a point $p$ in $\link(\calV)$.
By \reflem{EdgesCover} there is some edge $e$ in $K$ so that $p$ lies in $\closure{\rect(e)}$.
Suppose that $p$ lies in $\rect(e)$.
Then let $f$ be either of the two faces in $K$ incident to $e$.
Thus $\rect(e) \subset \rect(f)$, proving the first statement.
Suppose instead that $p$ lies in $\bdy \rect(e)$.
If $p$ lies in the interior of a side of $\rect(e)$ then we are done by \reflem{EdgeSide}.
The final possibility is that $p$ is a material corner of $\rect(e)$. 
Let $s$ be either of the sides of $\rect(e)$ meeting $p$.
By \reflem{EdgeSide}, we obtain a face $f$ so that the interior of $s$ lies in $\rect(f)$.
We deduce that $p$ lies in the interior of a side $s'$ of some edge rectangle in $\rect(f)$.
Applying \reflem{EdgeSide} again proves the first statement.

Finally, note that any finite union of face rectangles has only finitely many cusp classes as ideal points.
Thus the open cover by face rectangles has no finite subcover. 
\end{proof}

\begin{proof}[Proof of \refthm{LinkSpace}]
Fix $K$, a layer of a layering $\calK$.
We fix one edge $e_0$ of $K$;
we order the faces $(f_i)$ of $K$ by their combinatorial distance from $e_0$ and break ties arbitrarily.
Set $\calL_k = \bigcup_{i = 0}^k \rect(f_i)$.

We claim that $\calL_k$ is homeomorphic to a closed disk with $k + 3$ points removed from its boundary.
The base case of $\calL_0 = \rect(f_0)$ follows from \reflem{FaceRect}\refitm{FaceRectBdy}.
To pass from $\calL_k$ to $\calL_{k+1}$ we note that $\calL_k \cap \rect(f_{k+1})$ is an edge rectangle.  

We now take $D_k \subset \calL_k$ to be a compact disk, obtained from $\calL_k$ by removing small neighbourhoods of the $k + 3$ cusp classes.
By appropriately shrinking these neighbourhoods we arrange that $D_k \subset D_{k+1}$.
By \reflem{FacesCover}, the link space $\link(\calV)$ is the increasing union of the closed disks $D_k$.
Thus $\link(\calV)$ is a non-compact connected surface without boundary so that any compact subsurface is planar.
By Ker\'ekj\'art\'o's \textit{Hauptsatz der Fl\"achentopologie f\"ur offen Fl\"achen}~\cite[page~170]{Kerekjarto23}, we deduce that $\link(\calV)$ is homeomorphic to the plane, as desired.
(See also \cite[Theorem~1]{Richards63}.)
This gives \refitm{LinkSpacePlane}. 



By \reflem{FaceRect}\refitm{FaceRectRect}, we have that $F^\calV$ and $F_\calV$ are foliations and are transverse inside of every face rectangle.
By \reflem{FacesCover} these cover $\link(\calV)$, and we obtain \refitm{LinkSpaceTransverse}.
From \reflem{Cantor}, we obtain \refitm{LinkSpaceDense}.

The action of $\pi_1(M)$ on $\Circle$ preserves the laminations (perhaps swapping them), hence the induced action preserves the foliations (perhaps swapping them).
The action of $\pi_1(M)$ on $\cover{\calV}$ sends faces to faces so the action on $\link(\calV)$ sends face rectangles to face rectangles. 
Thus the action is continuous.

Note that $\pi_1(M)$ acts faithfully on the edges of $\cover{\calV}$.
So fix $\gamma \in \pi_1(M)$ as well as distinct edges $e$ and $e'$ with $\gamma(e) = e'$.
If $\rect(e) = \rect(e')$ then the cusps at the ends of $e$ and $e'$ must agree, by \reflem{EdgeRect}\refitm{EdgeRectBdy}.  
However, this contradicts \reflem{NoParallelEdges}.
We deduce that the action on $\link(\calV)$ is faithful. 
Thus we obtain~\refitm{LinkSpaceAction}.

Suppose now that $M$ is oriented and that $\calV$ is transverse veering.
Suppose that $f$ is a face of $K$, the given layer of $\calK$.
Let $e, e'$, and $e''$ be the three edges of $f$, labelled as in \reffig{LinkSpaceFace}.
The combinatorics of $\rect(f)$ (see \reflem{FaceRect}\refitm{FaceRectBdy}) implies that the induced orientation on $\rect(e')$ and $\rect(e'')$ agrees on their intersection.
Furthermore, the induced orientations on $\bdy \rect(e)$ and $\bdy \rect(e')$ disagree on their intersection.
The same holds on the intersection of $\bdy \rect(e)$ and $\bdy \rect(e'')$.
This induces an orientation on $\rect(f)$.
The same holds for any face of $K$.
By \reflem{FacesCover}, and since face rectangles (for faces in $K$) are either disjoint or intersect in an edge rectangle, this induces an orientation on $K$.
Note that any other layer $K'$ of $\calK$ shares infinitely many faces with $K$ and thus has the same orientation.
Finally, any other layering is made up of the same faces as $\calK$.
Thus we obtain~\refitm{LinkSpaceOri}.

The action of $\pi_1(M)$ on $\cover{\calV}$ preserves the red and blue edges and the transverse orientation.
Also the action of $\pi_1(M)$ does not swap the laminations and thus does not swap the foliations.
Thus it preserves the given orientation of face rectangles.
Thus we obtain~\refitm{LinkSpaceOriAction}. 
\end{proof}

\subsection{Cardinal directions}
\label{Sec:CardinalDirections}

Applying \refthm{LinkSpace}\refitm{LinkSpacePlane} we deduce that $F^\calV$ and $F_\calV$ are orientable.
Thus, in addition to the choices made in \refrem{FixOrientations},
we fix an arbitrary orientation on $F_\calV$; we call this orientation \emph{east}.
This together with \refthm{LinkSpace}\refitm{LinkSpaceOri} determines an orientation on $F^\calV$, which we will call \emph{north}. 
As usual, the opposite of east is \emph{west}, and the opposite of north is \emph{south}.

Suppose that $p$ and $q$ are points of $\link(\calV)$. 
Suppose that $\rho = (\rho_i)$ is a \emph{polygonal path} from $p$ to $q$.
That is, each $\rho_i$ is an interval in some leaf, and for all $i$, the arcs $\rho_i$ and $\rho_{i+1}$ are in opposite foliations. 
If every $\rho_i$ lying in $F_\calV$ is oriented to the east, then we say that \emph{$q$ is to the east of $p$}.
We make similar definitions for the other cardinal directions.

\begin{lemma}
\label{Lem:PartialOrders} 
The relation \emph{to the east of} gives a partial order on $\link(\calV)$.
The same holds for the other three cardinal directions.
\end{lemma}

\begin{proof}
Transitivity follows by concatenating.  
Antisymmetry follows because $F^\calV$ is a foliation of $\RR^2$ (\refthm{LinkSpace}\refitm{LinkSpacePlane} and \refitm{LinkSpaceTransverse}) so the leaves of $F^\calV$ are properly embedded lines (as in \cite[page~168]{Calegari07}), which thus separate $\link(\calV)$.
\end{proof}

For the remainder of the paper we require all rectangles $R \from (0, 1)^2 \to \link(\calV)$ to send oriented segments in the positive $y$--direction ($x$--direction) to oriented segments of $F^\calV$ ($F_\calV$).  

\begin{lemma}
\label{Lem:EdgeRectSlope}
With the choice of orientations of $F^\calV$ and $F_\calV$ as above we have the following.
An edge $e \in \cover{\calV}$ is red if and only if the northeast and southwest corners of $\rect(e)$ are ideal. 
\end{lemma}

\begin{proof}
We refer to the proof of \reflem{EdgeRect}, where $e$ was assumed to be red. 
With notation as in that lemma, we have that 
\[
c, \bdy U, \bdy V', d, \bdy T', \bdy S, 
\]
appear in $\Circle$ in that anticlockwise order.
See \reffig{LeafIdentificationsEdge}.
Fix $[(\lambda,\mu)] \in \rect(e)$ so that both $\lambda$ and $\mu$ are interior leaves.
Breaking symmetry, we may assume that $\bdy_+ \mu$ lies in $[\bdy U, \bdy V']^{\acw}$.
Since $\link(\calV)$ is connected, we deduce that $\bdy_+ \lambda$ lies in $[\bdy T', \bdy S]^{\acw}$.
Thus $\ell_U$ and $\ell^S$ both point away from $c$.
Thus $c$ lies in the southwest corner of the rectangle $\rect(e)$.
\end{proof}

In this paper, all figures showing parts of $\link(\calV)$ are drawn using this choice of orientation on $F^\calV$ and on $F_\calV$.

\subsection{Tetrahedron rectangles}
\label{Sec:TetrahedronRectangles}

\begin{definition}
\label{Def:TetRect}
Fix a tetrahedron $t$ in $\cover{\calV}$.
We say that a pair $(\lambda, \mu) \in \pair(\calV)$ \emph{links} $t$ if the pair links at least one of the six edges of $t$. 
The \emph{tetrahedron rectangle} $\rect(t)$ is defined to be
\[
\rect(t) = \{ p \in \link(\calV) \st \mbox{every representative $(\lambda,\mu)$ of $p$ links $t$} \} \qedhere
\]
\end{definition}

Recall from \refsec{TautIdealTriangulation} that two of the edges of $t$ are the upper and lower  edges (respectively) of $t$ and the remaining four are the equatorial edges of $t$.
Similarly we have two upper and two lower faces of $t$.

\begin{lemma}
\label{Lem:TetRect}
Suppose that $t$ is a tetrahedron of $\cover{\calV}$.
Let $f$ and $g$ be the upper faces of $t$ and also $f'$ and $g'$ be the lower faces. 
\begin{enumerate}
\item
\label{Itm:TetRectRect}
The tetrahedron rectangle $\rect(t)$ is a rectangle in the sense of \refdef{Rectangle}.
\item
\label{Itm:TetRectBdy}
The boundary of $\rect(t)$ is contained in a union of eight cusp leaves.
There are four material corners.
Each of the four sides contains exactly one ideal point, corresponding to one of the four cusps of $t$.
\item
\label{Itm:TetRectFace}
The tetrahedron rectangle $\rect(t)$ is equal to both $\rect(f) \cup \rect(g)$ and $\rect(f') \cup \rect(g')$.
Moreover, $\rect(f)$ and $\rect(g)$ both south-north span $\rect(t)$, while $\rect(f')$ and $\rect(g')$ both west-east span $\rect(t)$.
\item
\label{Itm:TetKeane}
No pair of ideal points in the boundary of $\rect(t)$ share a common leaf.
\end{enumerate}
\end{lemma}

\begin{proof}
It follows from the definitions that $\rect(f) \cup \rect(g) \subset \rect(t)$ and, similarly, $\rect(f') \cup \rect(g') \subset \rect(t)$.  
\begin{claim*}
$\rect(t) \subset \rect(f) \cup \rect(g)$ and $\rect(t) \subset \rect(f') \cup \rect(g')$.  
\end{claim*}

\begin{proof}
Suppose that $p$ is a point of $\rect(t)$.
We prove that $p$ is contained in $R(f) \cup R(g)$; the other case is similar.
If $p$ is contained in an edge rectangle for an equatorial edge of $t$, or the upper edge of $t$, then we are done by the remarks following \refdef{FaceRect}.

In the remaining cases, $p$ lies on the common boundary of two edge rectangles, both contained in one of $R(f)$ or $R(g)$; see \reffig{LinkSpaceFaceFlowChart}(e).
By \reflem{FaceRect}\refitm{FaceRectRect}, the face rectangles $R(f)$ and $R(g)$ are rectangles in the sense of \refdef{Rectangle}.
Thus the common boundary of the two edge rectangles, and so $p$, lies in the interior of $R(f)$ or of $R(g)$.
\end{proof}

The claim proves the first statement of \refitm{TetRectFace}.
Let $e = f \cap g$ and $e' = f' \cap g'$ be the upper and lower (respectively) edges of $t$.
By \reflem{FaceRect}\refitm{FaceSpan}, the rectangle $\rect(e)$ south-north spans both $\rect(f)$ and $\rect(g)$.
Thus $\rect(e)$ south-north spans $\rect(t)$.
Since $\rect(f)$ contains $\rect(e)$, we deduce that $\rect(f)$ south-north spans $\rect(t)$.
The other three statements of \refitm{TetRectFace} are proved similarly.

Note, by \reflem{FaceRect}\refitm{FaceSpan},  that $\rect(e)$ is a subrectangle of, and south-north spans, both $\rect(f)$ and $\rect(g)$.
\reffig{LinkSpaceFaceFlowChart}(e) shows the cases where the upper edge $e$ is red.   
Thus $\rect(t)$ is a rectangle, and we obtain \refitm{TetRectRect}.

Note that $\rect(f)$ and $\rect(g)$ share precisely two cusp classes, namely those on the boundary of $\rect(e)$.
Thus there are four cusp classes in $\bdy \rect(t)$.
No two can be in a single side as that would give a cusp leaf meeting two cusp classes, contradicting \refcor{Irrational}.  
This gives \refitm{TetRectBdy}.  

Finally, suppose that $a$ and $c$ are cusps in the southern and northern sides of $\rect(t)$ respectively.
Again appealing to \refcor{Irrational}, we find that $a$ and $c$ cannot lie on a common leaf.
This gives \refitm{TetKeane}.
\end{proof}

The next result is needed to relate the combinatorics of $\cover{\calV}$ to the combinatorics of rectangles in $\link(\calV)$. 

\begin{lemma}
\label{Lem:FaceInTet}
Suppose that $f$ and $t$ are, respectively, a face and a tetrahedron of $\cover{\calV}$.
Then $f \subset t$ if and only if $\rect(f) \subset \rect(t)$. 
\end{lemma}

\begin{proof}
The forwards direction follows from \reflem{TetRect}\refitm{TetRectFace}.

We prove the contrapositive of the backwards direction.
Suppose that $f$ is not a face of $t$.
Fix a layering $\calK$ of $\cover{\calV}$.
We break symmetry and suppose that $f$ is contained in a layer $K$, of $\calK$, that lies above $t$.
Furthermore, we take $K$ to be the lowest such layer in $\calK$.  

Suppose that the upper faces of $t$ lie in $K$.
Thus there is some edge, say $e$, of $f$ that separates $f$ (in $K$) from the upper faces of $t$.
Thus the cusp of $f$ not meeting $e$ cannot lie in a side of $\rect(t)$ and we are done.
See Figures~\ref{Fig:LinkSpaceFace} and~\ref{Fig:LinkSpaceFaceFlowChart}(e).

Suppose instead that the upper faces of $t$ do not lie in $K$.
In this case let $e$ be the edge that is removed from $K$ in order to obtain the immediately lower layer.
We deduce that $e$ is an edge of $f$.
Let $c$ and $d$ be the endpoints of $e$.
The faces of $t$ are all strictly below $K$, hence $e$ is not an edge of $t$.
Therefore by \reflem{NoParallelEdges}, no edge of $t$ connects $c$ and $d$.
Therefore at least one of $c$ or $d$ does not lie in any side of $\rect(t)$.
Therefore $\rect(f)$ is not a subset of $\rect(t)$ and we are done.
\end{proof}

\subsection{All rectangles}

We now prove that the tetrahedron rectangles control all rectangles of $\cover{\calV}$.

\begin{theorem}
\label{Thm:RectInTetRect}
Suppose that $R$ is a rectangle in $\link(\calV)$.
Then there is a tetrahedron $t$ of $\cover{\calV}$ so that $R \subset \rect(t)$.
\end{theorem}

Before giving the proof we need several lemmas. 

\begin{lemma}
\label{Lem:FinitelySpanned}
Suppose that $K$ is a layer of a layering and $R$ is a rectangle.
Then $R$ is south-north spanned by at most finitely many edge rectangles $\rect(e)$ for $e$ in $K$.
The same holds west-east.
\end{lemma}

\begin{proof}
We consider only the south-north case.
Applying the second statement of \reflem{Cantor}, we choose distinct non-cusp leaves $m, m' \in F_\calV$ which intersect $R$.  
For any edge $e$, if $\rect(e)$ south-north spans $R$, then $m$ and $m'$ both cross $\rect(e)$.

Since $m$ is non-cusp, by \refdef{UpperFoliation}, there is a unique leaf $\mu$ of $\Lambda_\calV$ so that
\[
m = \big\{ [(\lambda, \mu)] \in \link(\calV) \mbox{ for some $\lambda \in \Lambda^\calV$} \big\}
\]
Let $\mu'$ be the corresponding leaf for $m'$.
Thus $\mu$ and $\mu'$ are distinct interior leaves.
\reflem{Laminations}\refitm{LeavesAreCarried} provides train lines $m_K$ and $m'_K$ carried by the train track $\tau_K$ with the same endpoints as $\mu$ and $\mu'$ respectively.
Note that for any edge $e$ in $K$, if the leaf $m$ crosses $\rect(e)$ then the train line $m_K$ links $e$.

By \reflem{Laminations}\refitm{Asymptotic}, the train lines $m_K$ and $m'_K$ share no endpoints;
thus they fellow travel for at most a finite collection of edges $E$ of $K$.
Thus, at most $|E|$ edge rectangles south-north span the subrectangle of $R$ between $m$ and $m'$.
Therefore, at most $|E|$ edge rectangles south-north span $R$.
\end{proof}

\begin{definition}
\label{Def:FaceRectMedian}
Let $f$ be a face with edges $e$, $e'$, and $e''$.
The \emph{median} of $\rect(f)$ is the intersection $\closure{\rect(e)} \cap \closure{\rect(e')} \cap \closure{\rect(e'')}$.
\end{definition}

In the right-hand side of \reffig{LinkSpaceFace}, the median is the central yellow dot. 

\begin{lemma}
\label{Lem:MedianSpan}
Suppose that $f$ is a face with edges $e$, $e'$, and $e''$, where $e$ has the minority colour in $f$.
If $R$ is a rectangle that contains the median of $\rect(f)$ then $R$ is spanned by at least one of $\rect(e')$ or $\rect(e'')$.
\end{lemma}

\begin{proof}
Breaking symmetry, suppose that $e$ is blue while $e'$ and $e''$ are red.
For example, see \reffig{LinkSpaceFace}.
Since $R$ contains the median, it cannot extend out of $\rect(f)$ to the west or south;
the cusp leaves running through the median end at ideal points in the sides of $\rect(f)$.
The rectangle $R$ may extend out of $\rect(f)$ to either the east or north, but not both.
This is because the northeastern ideal corner of $\rect(f)$ is not contained in $R$. 

We apply \reflem{FaceRect}\refitm{FaceSpan}.
If $R$ extends out of $\rect(f)$ to the east then $R$ is south-north spanned by $\rect(e')$.
If it extends out the north then it is west-east spanned by $\rect(e'')$.
If $R$ is contained in $\rect(f)$ then it is spanned by both edge rectangles.
\end{proof}

\begin{lemma}
\label{Lem:SpannedSNOrWE}
Suppose that $K$ is a layer of a layering and $R$ is a rectangle.
Then there is an edge $e$ in $K$ so that $\rect(e)$ either south-north or west-east spans $R$.
\end{lemma}

\begin{proof}
By \reflem{EdgesCover}, there is an edge $e_0$ in $K$ so that $R\cap \closure{\rect(e_0)}$ is non-empty.
Since $R$ is open, we also have that $R \cap \rect(e_0)$ is non-empty.
Breaking symmetry, suppose that $e_0$ is red.
If $\rect(e_0)$ does not span $R$ in either direction, then $R$ contains precisely one of the two material corners of $\rect(e_0)$.
Breaking symmetry again and applying \reflem{EdgeRectSlope}, we assume that $R$ contains $p_0$, the southeastern corner of $\rect(e_0)$.  

Consider now the maximal strip $P \subset K$ of edges and triangles $(f_i, e_i)_{i \geq 1}$ so that 
\begin{enumerate}
\item
all edges $e_i$ are red, 
\item
all triangles $f_i$ are majority red,
\item 
$e_0$ is contained in $f_1$, 
\item
$e_i$ is contained in $f_i$ (and in $f_{i+1}$ when $f_i$ is not the last triangle in $P$), 
\item
$R$ contains $p_i$, the southeastern material corner of $\rect(e_i)$. 
\end{enumerate}

Let $c_0$ and $d_0$ be the endpoints of $e_0$.
We label the other vertices of $P$ as follows.
Suppose that we have labelled the vertices of the edges $e_0, \ldots, e_{k-1}$.
Two of the vertices of $f_k$ are already labelled.  
The remaining vertex of $f_k$ receives the label $c_k$ if we turn left through $f_k$ when travelling from $e_{k-1}$ to $e_k$, 
or the label $d_k$ if we turn right.  
For an example, see the left side of \reffig{RedStripRectangles}.

\begin{figure}[htbp]
\subfloat[The rectangle $R$ is not spanned by any of the edge rectangles $\rect(e_i)$.]{
\labellist
\small\hair 2pt
\pinlabel {$e_0$} [br] at 107 237
\pinlabel {$c_0$} [r] at 3 170
\pinlabel {$c_1$} [tr] at 67 129
\pinlabel {$c_2$} [tr] at 117 81
\pinlabel {$c_5$} [tr] at 154 29
\pinlabel {$d_0$} [bl] at 195 297
\pinlabel {$d_3$} [bl] at 272 251
\pinlabel {$d_4$} [bl] at 328 215
\pinlabel {$R$} [tl] at 660 209
\pinlabel {$c_0$} [r] at 430 148
\pinlabel {$c_1$} [tr] at 526 92
\pinlabel {$c_2$} [tr] at 580 60
\pinlabel {$c_5$} [t] at 620 23
\pinlabel {$p_0$} [l] at 730 151
\pinlabel {$d_0$} [b] at 728 348
\pinlabel {$d_3$} [bl] at 801 301
\pinlabel {$d_4$} [bl] at 854 263
\endlabellist
\includegraphics[width = 0.9\textwidth]{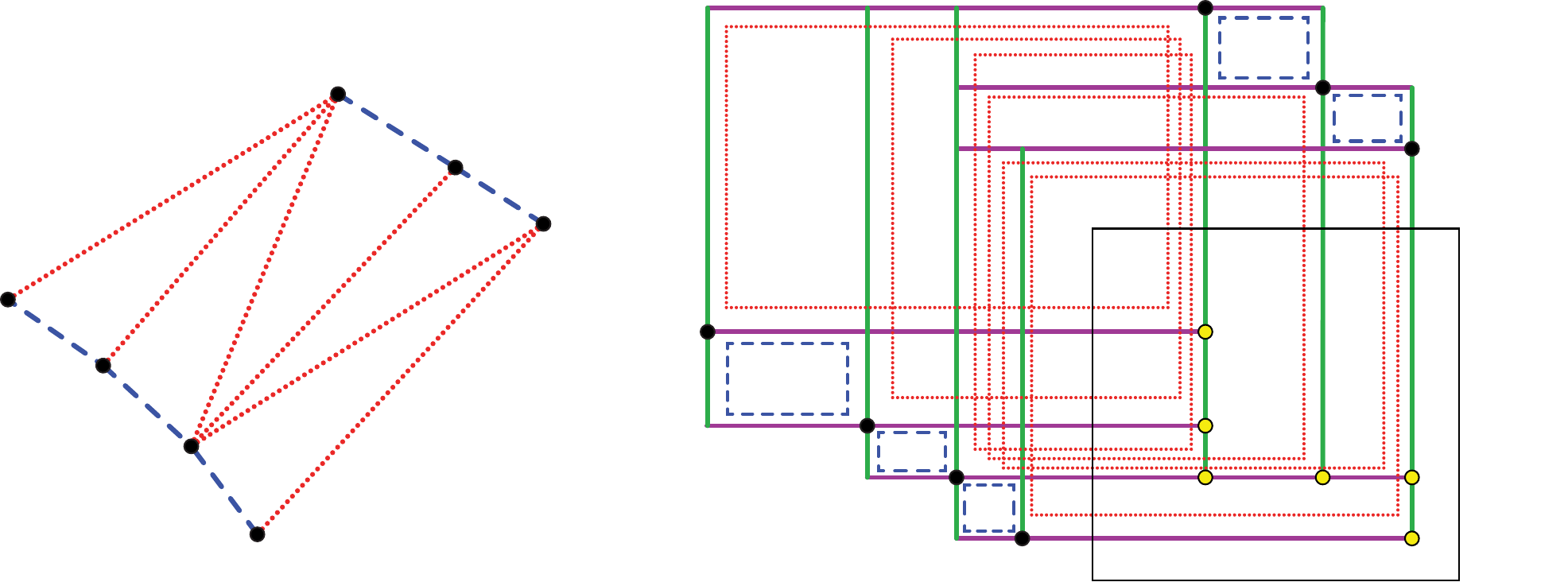}
\label{Fig:RedStripRectangles1}
}

\subfloat[The rectangle $R$ is spanned by the edge rectangle $\rect(e')$.]{
\labellist
\small\hair 2pt
\pinlabel {$e_0$} [br] at 107 237
\pinlabel {$c_0$} [r] at 3 170
\pinlabel {$c_1$} [tr] at 67 129
\pinlabel {$c_2$} [tr] at 117 81
\pinlabel {$c_5$} [tr] at 154 29
\pinlabel {$d_0$} [bl] at 195 297
\pinlabel {$d_3$} [bl] at 272 251
\pinlabel {$d_4$} [bl] at 328 215
\pinlabel {$d$} [bl] at 367 159
\pinlabel {$e'$} [tl] at 263 100
\pinlabel {$R$} [tl] at 660 209
\pinlabel {$c_0$} [r] at 430 148
\pinlabel {$c_1$} [tr] at 526 92
\pinlabel {$c_2$} [tr] at 580 60
\pinlabel {$c_5$} [t] at 620 23
\pinlabel {$p_0$} [l] at 730 151
\pinlabel {$d_0$} [b] at 728 348
\pinlabel {$d_3$} [bl] at 801 301
\pinlabel {$d_4$} [bl] at 854 263
\pinlabel {$d$} [bl] at 945 123
\endlabellist
\includegraphics[width = 0.9\textwidth]{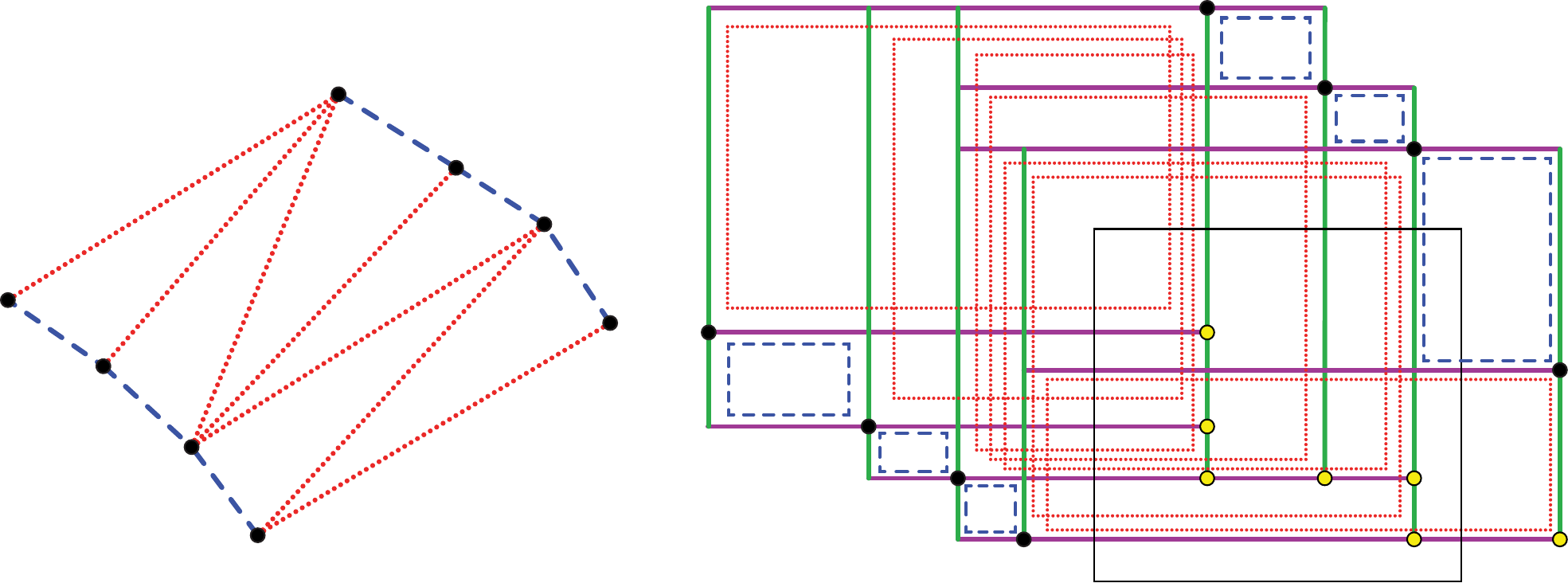}
\label{Fig:RedStripRectangles2}
}
\caption{A possible strip $P$ and its corresponding edge rectangles.
A possible rectangle $R$ is drawn in black.}
\label{Fig:RedStripRectangles}
\end{figure}

\begin{figure}[htbp]
\subfloat[Northwest.]{
\includegraphics[width = 0.4\textwidth]{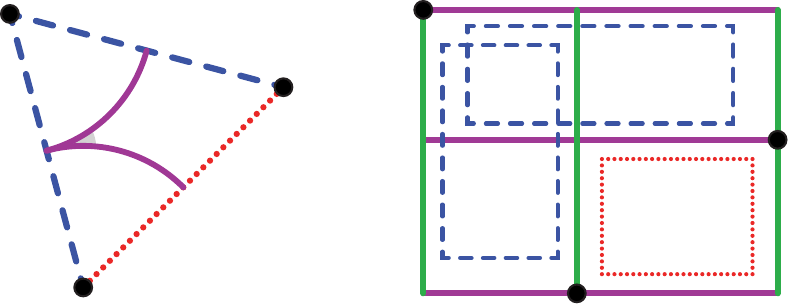}
\label{Fig:Northwest}
}
\qquad
\subfloat[Northeast.]{
\includegraphics[width = 0.4\textwidth]{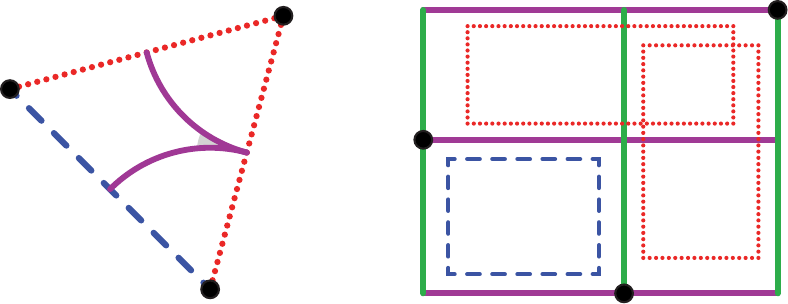}
\label{Fig:Northeast}
}

\subfloat[Southwest.]{
\includegraphics[width = 0.4\textwidth]{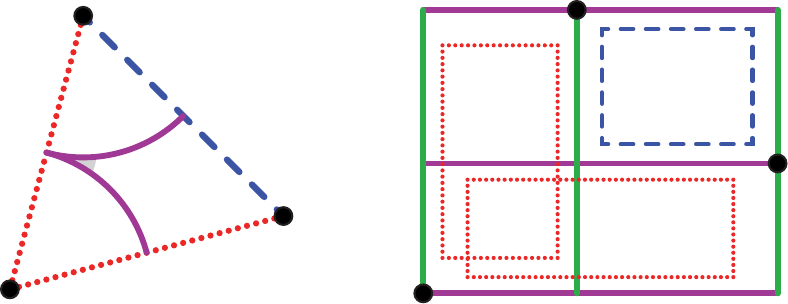}
\label{Fig:Southwest}
}
\qquad
\subfloat[Southeast.]{
\includegraphics[width = 0.4\textwidth]{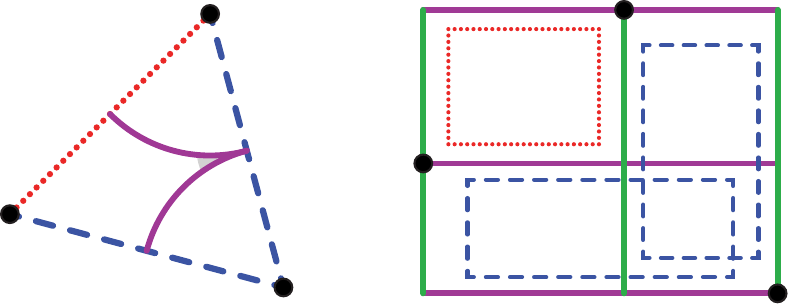}
\label{Fig:Southeast}
}

\caption{The four possible face rectangles, labelled by the intercardinal direction of the ideal corner.}
\label{Fig:FourFaces}
\end{figure}

\begin{claim*}
The closure of $\rect(e_k)$ contains $p_0$.
\end{claim*}

\begin{proof}
The base case of the induction is trivial.
Suppose that the claim holds for $e_k$.
Breaking symmetry, suppose that the southwest corner of $\rect(f_{k+1})$ is ideal.
See \reffig{Southwest}.
By the inductive hypothesis, $p_0$ is in the closure of $\rect(e_k)$: the western half of $\rect(f_{k+1})$.
Note that $\rect(e_{k+1})$ is the southern half; 
thus $\rect(e_{k})$ and $\rect(e_{k+1})$ intersect non-trivially, in a rectangle lying in the southwest of $\rect(f_{k+1})$.
Again see \reffig{Southwest}.
Since $R$ contains $p_{k+1}$, the north side of $R$ is (perhaps non-strictly) south of $d_{k+1}$.
Since $R$ contains $p_0$, this and \reflem{PartialOrders} means that $p_0$ is strictly south of $d_{k+1}$.
Thus $p_0$ is contained in the closure of $\rect(e_{k+1})$.
\end{proof}

\begin{claim*}
Every $d_{k'}$ is northeast of every $c_k$. 
\end{claim*}

\begin{proof}
The previous claim implies that the northeast (southwest) corner of $\rect(e_k)$ is northeast (southwest) of $p_0$.
Thus all $d_{k'}$ are northeast of $p_0$, and all $c_k$ are southwest of $p_0$.
Now \reflem{PartialOrders} gives the claim.
\end{proof}

\begin{claim*}
The strip $P$ is finite.
\end{claim*}

\begin{proof}
Suppose that $P$ is infinite.
We produce a sequence of strips $(P_i)$ below $P = P_0$, where $P_{i+1}$ is obtained from $P_i$ by flipping down through the first pair of triangles in $P_i$ that form the top of a tetrahedron.
Such a pair must exist by \reflem{CannotTurnLeftForever}.
Note that the vertices and boundary edges of $P_{i+1}$ are the same as those of $P_i$, and hence the same as those of $P_0 = P$.
Let $P_n$ be the first strip that contains an interior blue edge, $e_\triangledown$ say.
Note that $K$ exists by \refcor{BranchLinesToggle} (for lower branch lines).
An example sequence of strips is shown in \reffig{RedStripTrainTracks}. 

The edge $e_\triangledown$ connects vertices on opposite sides of the strip $P_n$.
Suppose that the endpoints of $e_\triangledown$ are $c_k$ and $d_{k'}$.
Applying \reflem{EdgeRectSlope}, we find that one of $c_k$ and $d_{k'}$ is southeast of the other.  
By \reflem{PartialOrders} this contradicts the previous claim.
\end{proof}

We finish the proof of \reflem{SpannedSNOrWE} as follows. 
Let $e_n$ be the last edge of $P$.
Let $f'$ be the face adjacent to $e_n$ that is not in $P$.  
Either $f'$ is majority blue (see \reffig{Southeast}) or majority red (see \reffig{Northeast} or \reffig{Southwest}).
In the majority blue case, $R$ contains the median of $\rect(f')$, so by \reflem{MedianSpan} we are done.
In the majority red case, let $e'$ be the other red edge of $f'$.
Since $P$ is a maximal strip, $R$ does not contain the southeastern corner of $\rect(e')$.
Thus $R$ is spanned by $\rect(e')$ (south-north in the case of \reffig{Northeast}, and west-east in the case of \reffig{Southwest}) and we are done.
\end{proof}

\begin{figure}[htbp]
\subfloat[]{
\includegraphics[width = 0.45\textwidth]{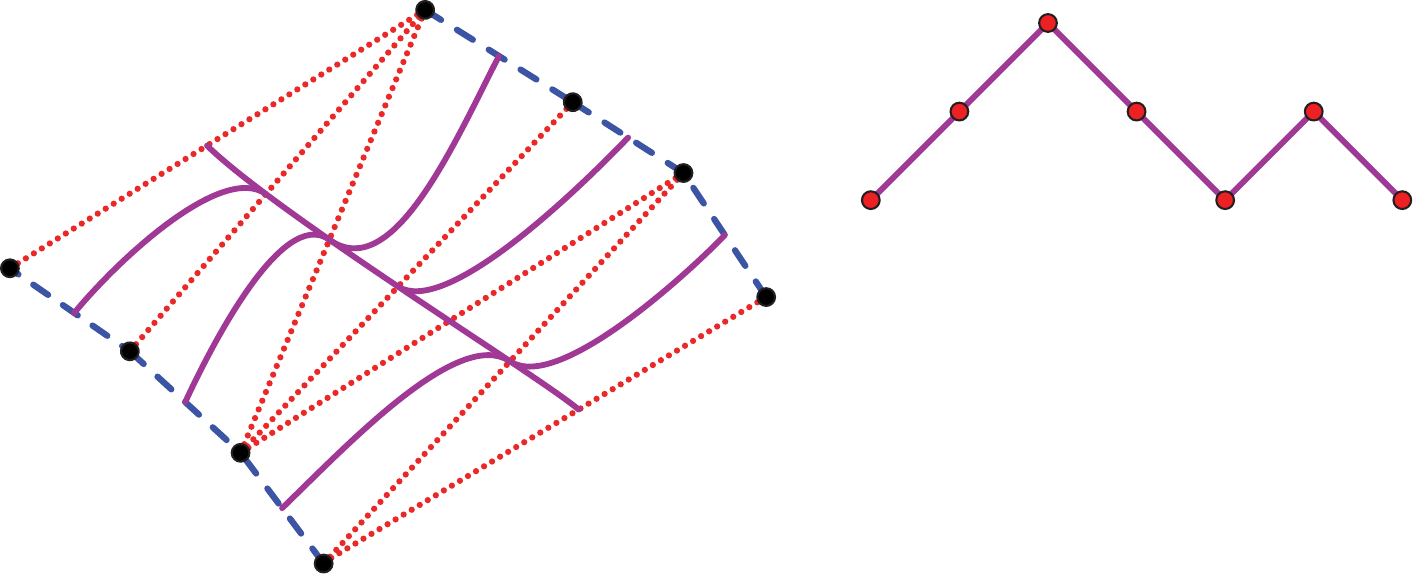}
\label{Fig:RedStripTrainTrack1}
}
\quad
\subfloat[]{
\includegraphics[width = 0.45\textwidth]{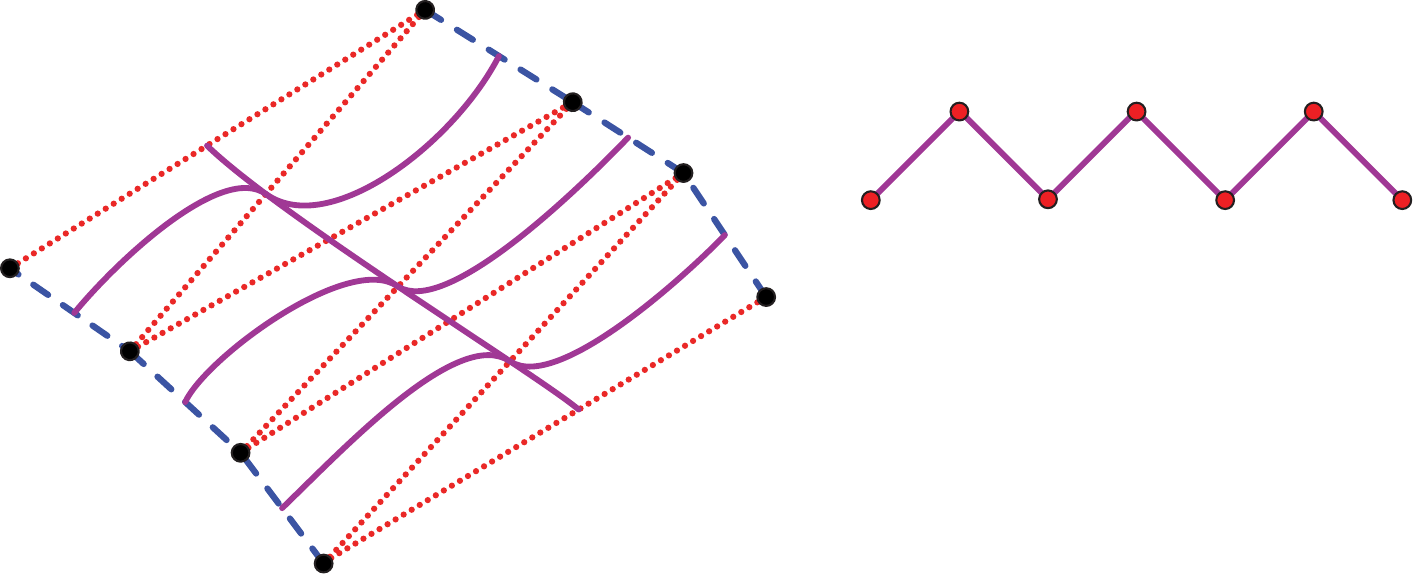}
\label{Fig:RedStripTrainTrack1p5}
}

\subfloat[]{
\includegraphics[width = 0.45\textwidth]{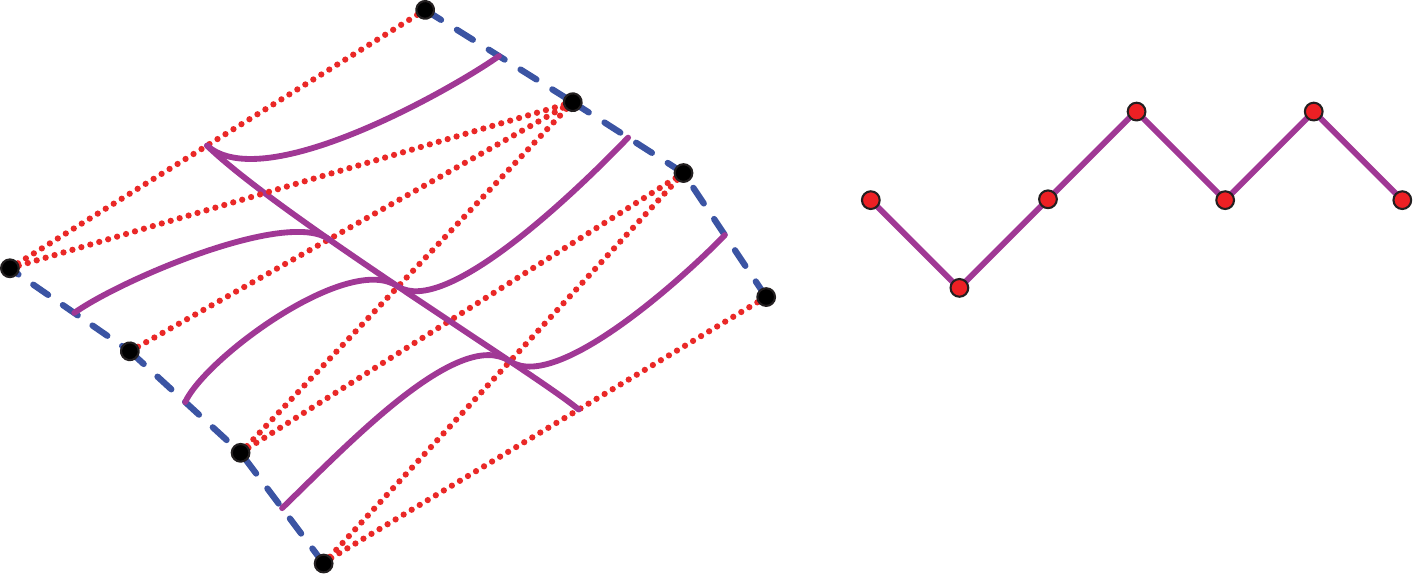}
\label{Fig:RedStripTrainTrack2}
}
\quad
\subfloat[]{
\includegraphics[width = 0.45\textwidth]{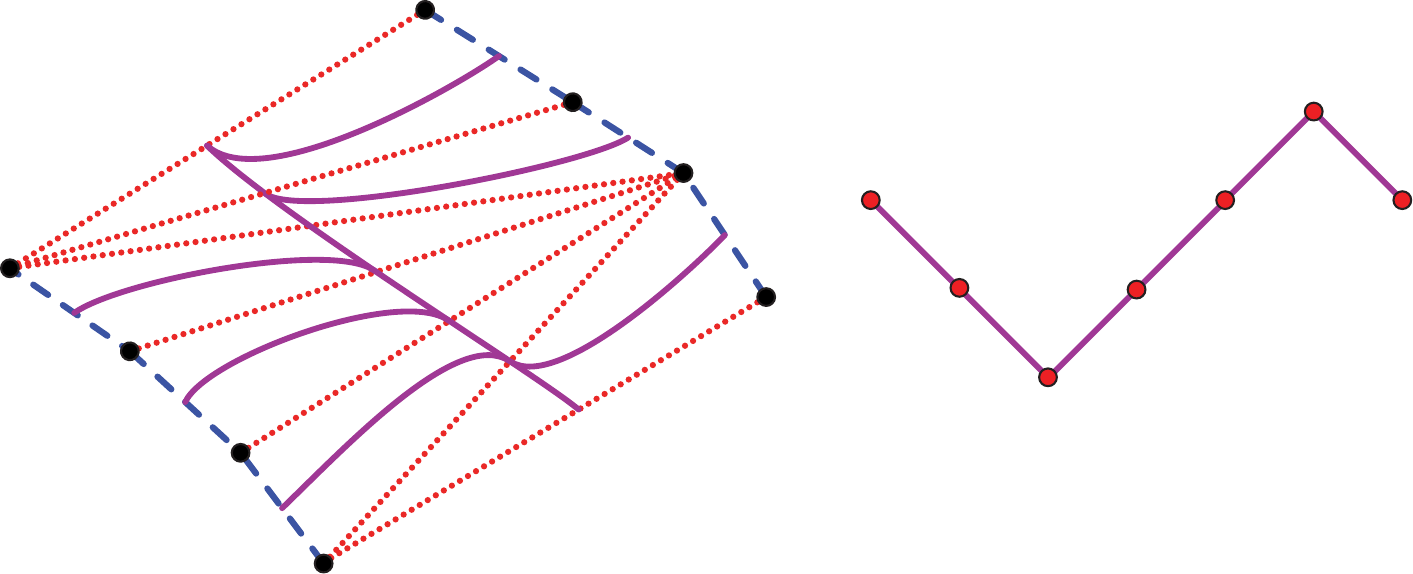}
\label{Fig:RedStripTrainTrack3}
}

\subfloat[]{
\includegraphics[width = 0.45\textwidth]{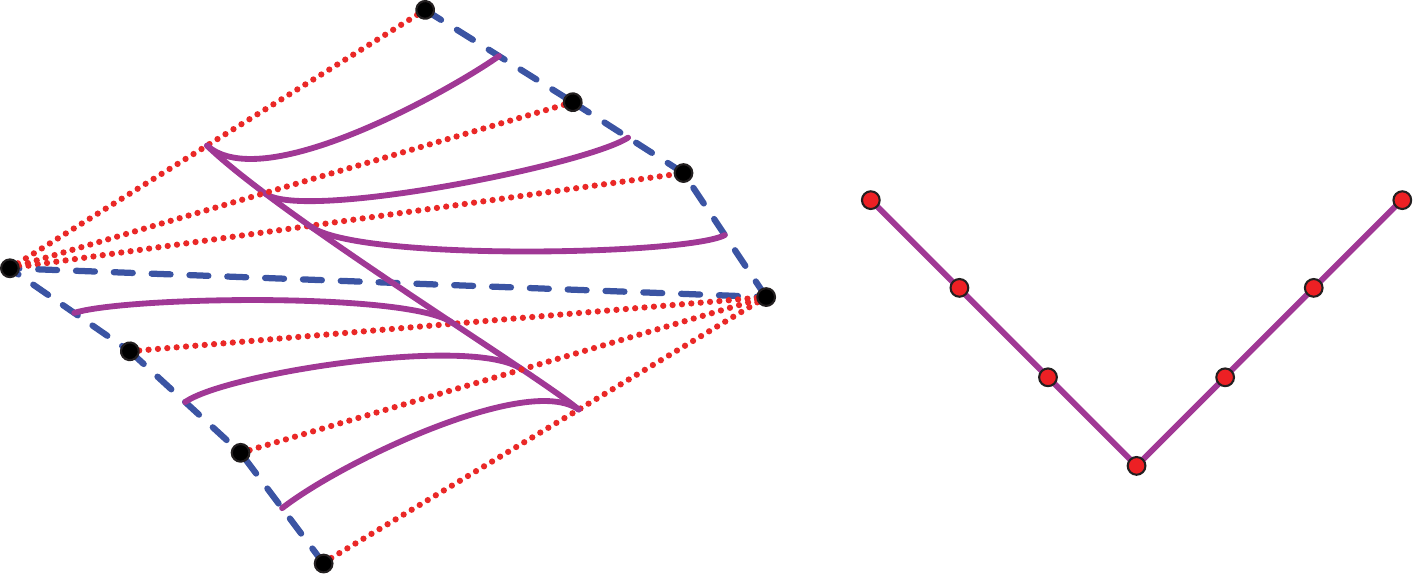}
\label{Fig:RedStripTrainTrack4}
}
\caption{On the left in each subfigure we see a strip of triangles from above.
On the right, we see a side view, showing the height of the mid-curve through the strip.
In particular, the top edge of a tetrahedron appears as a local maximum of the curve.}
\label{Fig:RedStripTrainTracks}
\end{figure}

\begin{proof}[Proof of \refthm{RectInTetRect}]
It suffices to consider the case when $R$ is maximal.
Fix $K$, a layer of a layering.
By \reflem{SpannedSNOrWE}, there is an edge of $K$ whose edge rectangle spans $R$.
Breaking symmetry, we suppose that $R$ is south-north spanned.
By \reflem{FinitelySpanned}, the set $E$ of edges in $K$, whose rectangles south-north span $R$, is finite.
For $e, e' \in E$ we say that $\rect(e)$ is \emph{to the west} of $\rect(e')$ if 
\begin{itemize}
\item the west side of $\rect(e)$ is to the west of the west side of $\rect(e')$ or, in the case of a tie, 
\item the east side of $\rect(e)$ is to the west of the east side of $\rect(e')$.
\end{itemize}

\noindent 
Let $e_\triangleleft$ be any edge whose rectangle $\rect(e_\triangleleft)$ is the furthest to the west. 
(The subscript indicates a cone pointing to the west.)

Breaking symmetry, suppose that $e_\triangleleft$ is red.
Let $f_\triangleleft$ be the face containing $e_\triangleleft$ and so that the (material) southeastern corners of $\rect(e_\triangleleft)$ and $\rect(f_\triangleleft)$ coincide. 

There are four possible face rectangles, depending on which corner is ideal.
See \reffig{FourFaces}. 

\begin{claim*}
The rectangle $\rect(f_\triangleleft)$ has its ideal corner to its northeast. 
\end{claim*}

\begin{proof}
We rule out three of the four cases.
\begin{itemize}
\item 
By assumption, the southeastern corner of $\rect(f_\triangleleft)$ is material.
\item 
Suppose that the ideal corner of $\rect(f_\triangleleft)$ is in the southwest, as shown in \reffig{Southwest}.
Let $e$ be the other red edge of $f_\triangleleft$.
Note that $\rect(e)$ extends further to the north than $\rect(e_\triangleleft)$.
Since $R$ is maximal, it meets the west side of $\rect(e_\triangleleft)$.
Thus $\rect(e)$ also spans $R$; we deduce that $e \in E$.
This contradicts our assumption that $e_\triangleleft$ is the westmost edge in $E$. 
\item 
A similar argument rules out the ideal corner being in the northwest. See \reffig{Northwest}.  \qedhere
\end{itemize}
\end{proof}

Since $R$ is maximal, its west side contains the cusp class at the west side of $\rect(f_\triangleleft)$.
We have a similar situation on the east of $R$, with an eastmost edge $e_\triangleright$ and face $f_\triangleright$.
See \reffig{MaxRectangleCaps}.

\begin{figure}[htbp]
\centering
\includegraphics[width=0.6\textwidth]{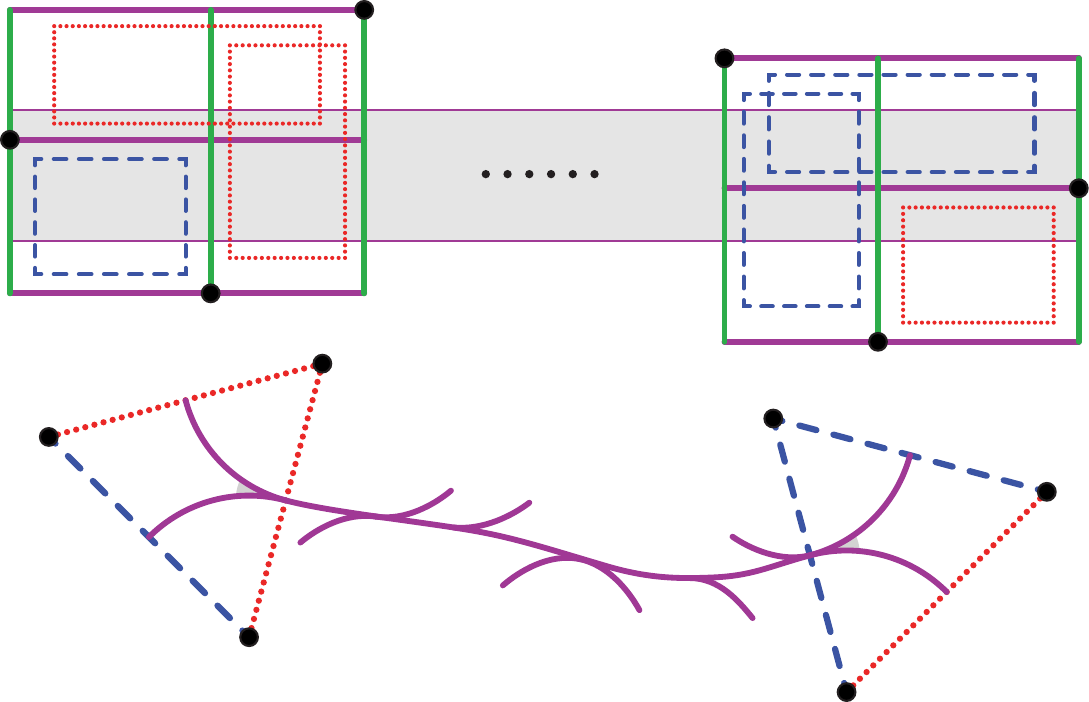}
\caption{Above: An example of a maximal rectangle $R$ (shaded).
It has a westmost south-north spanning edge rectangle that is in the east of a face rectangle $\rect(f_\triangleleft)$ as in \reffig{Northeast}.
It also has an eastmost south-north spanning edge rectangle that is in the west of a face rectangle $\rect(f_\triangleright)$ as in \reffig{Northwest}.
Below: a sketch of the corresponding strip of triangles $P$ in $K$, connecting $f_\triangleleft$ with $f_\triangleright$.} 
\label{Fig:MaxRectangleCaps}
\end{figure}

Since $K$ is a copy of the Farey triangulation (\reflem{Disk}), there is a unique strip of triangles $P$ connecting $f_\triangleleft$ with $f_\triangleright$.
By \refthm{LinkSpace}\refitm{LinkSpaceDense}, we may take $m$ to be a non-cusp leaf of $F_\calV$ meeting $R$.
By \reflem{Laminations}\refitm{LeavesAreCarried}, there is a corresponding train line $m_K \subset \tau_K$.
Note that $m_K \cap \tau_P$ contains a train interval which runs from the track-cusp in $f_\triangleleft$ to the track-cusp in $f_\triangleright$.
These track-cusps point towards each other, so there is at least one large switch of $\tau_P$ at an edge in the interior of $P$.
The two faces incident to such an edge are the upper two faces of a tetrahedron.
Moving $P$ down through such a tetrahedron has the effect of performing a split on $\tau_P$.
After finitely many such moves, $P$ consists of two faces only.
These are thus the top two faces of a tetrahedron $t$.
Thus $R$ is contained in (and in fact equal to) $\rect(t)$.
\end{proof}

\subsection{Loom spaces}

We now recall the definition of a loom space from our previous paper~\cite[Definition~2.11]{SchleimerSegerman24}.

Suppose that $\calL$ is a \emph{bifoliated plane}:
a copy of $\RR^2$ equipped with a pair of transverse foliations $F^\calL$ and $F_\calL$.
We define rectangles in $\calL$ just as in \refdef{Rectangle}.
We define the sides of rectangles just as in \refdef{Boundary}.
In both cases, we replace $\link(\calV)$, $F^\calV$, and $F_\calV$ with $\calL$, $F^\calL$, and $F_\calL$.

\begin{definition}
\label{Def:CuspRectLoom}
Suppose that $R$ is a rectangle in the bifoliated plane $\calL$.
Suppose that $f_R\from (0,1)^2 \to R$ is an associated embedding.
We say that $R$ is a \emph{south-west cusp rectangle} if there is a continuous extension of $f_R$ to a homeomorphism
\[
\closure{f}_R\from [0,1]^2 - \{(0,0)\} \to \closure{R} 
\]
We call the southern and western sides of $R$ its \emph{cusp sides}. 
We orient these to the east and north respectively.
We define the other three intercardinal cusp rectangles, and their cusp sides, similarly.
\end{definition}

As in \cite[Definitions~3.1 and 3.3]{SchleimerSegerman24}, a \emph{loom-cusp} of $\calL$ is an equivalence class of cusp rectangles.
The relation is generated by shrinking or expanding cusp rectangles, and walking from one cusp rectangle to another through a shared cusp side.

\begin{definition}
\label{Def:TetRectLoom}
Suppose that $R$ is a rectangle in the bifoliated plane $\calL$.
Suppose that $f_R\from (0,1)^2 \to R$ is an associated embedding.
We say that $R$ is a \emph{loom-tetrahedron rectangle} if there are $a,b,c,d \in (0,1)$ and a continuous extension of $f_R$ to a homeomorphism
\[
\closure{f}_R\from [0,1]^2 - \{(a,0), (1,b), (c,1), (0,d)\} \to \closure{R} \qedhere
\]
\end{definition}

\begin{definition}
\label{Def:Loom}
Suppose that $\calL$ is a bifoliated plane.
We say that $\calL$ is a \emph{loom space} if it satisfies the following two axioms.
\begin{enumerate}
\item
\label{Itm:Cusp}
For every cusp side $s$ of every cusp rectangle $R$, some rectangle $Q$ contains an initial open interval of $s$.
\item
\label{Itm:Tet}
For every rectangle $R$, some loom-tetrahedron rectangle $Q$ contains $R$.
\qedhere
\end{enumerate} 
\end{definition}

\begin{definition}
\label{Def:LoomIso}
Suppose that $\calL$ and $\calM$ are loom spaces.
Suppose that $f\from \calL \to \calM$ is a homeomorphism.
We say that $f$ is a \emph{loom isomorphism} if it sends leaves to leaves.
\end{definition}

We now show that link spaces are loom spaces.

\begin{theorem}
\label{Thm:LinkIsLoom}
Suppose that $M$ is a three-manifold equipped with a locally veering triangulation $\calV$. 
\begin{enumerate}
\item
\label{Itm:LinkIsLoom}
The link space $\link(\calV)$, equipped with the foliations $F^\calL$ and $F_\calL$ is a loom space. 
\item
\label{Itm:TetIsLoomTet}
Every tetrahedron rectangle is a loom-tetrahedron rectangle; and conversely.
\item
\label{Itm:CuspIsLoomCusp}
The set of cusp rectangles in $\link(\calV)$, with ideal point a given cusp class of $\elec$, is a loom-cusp; and conversely.
\item
\label{Itm:LoomAction}
The action of $\pi_1(M)$ is faithful and by loom automorphisms.
\end{enumerate}
\end{theorem}



\begin{proof}
By \refthm{LinkSpace}\refitm{LinkSpacePlane} the link space $\link(\calV)$ is a copy of $\RR^2$.
By \refthm{LinkSpace}\refitm{LinkSpaceTransverse}, we have that $F^\calV$ and $F_\calV$ are foliations and are transverse.

Suppose that $R$ is a cusp rectangle in $\link(\calV)$. 
Suppose that $s$ is a cusp side of $R$.
By \refthm{RectInTetRect}, there is a tetrahedron $t$ of $\cover{\calV}$ so that $R \subset \rect(t)$.
We deduce that the cusp sides of $R$ accumulate, in $\pair(\calV)$, at one of the four ideal points in $\bdy \rect(t)$.
Therefore we are justified in saying that the cusp rectangle $R$ has exactly one ideal corner.
Consulting \reffig{LinkSpaceFaceFlowChart}(e), we deduce that an initial segment of $s$ is contained in a side of some edge rectangle.
Applying \reflem{EdgeSide}, we obtain \refdef{Loom}\refitm{Cusp}.
\refdef{Loom}\refitm{Tet} follows from \refthm{RectInTetRect}.
Thus $\link(\calV)$ satisfies all axioms of a loom space and we have obtained~\refitm{LinkIsLoom}.

By \refthm{RectInTetRect}, each loom-tetrahedron rectangle is contained in some tetrahedron rectangle.
By \reflem{TetRect} (parts~\refitm{TetRectRect} and~\refitm{TetRectBdy}) each tetrahedron rectangle is a loom-tetrahedron rectangle.
Thus we have obtained~\refitm{TetIsLoomTet}.

Fix any cusp $c$ of $\Delta_\calV$.
As in \reflem{CrownsInterleave}, this yields a pair of interleaving crowns $\Lambda^c$ and $\Lambda_c$.
As in \refdef{UpperFoliation}, the boundary leaves of $\Lambda^c$ collapse to give cusp-leaves in $F^\calV$, and similarly for $\Lambda_c$ and $F_\calV$.
By \reflem{FacesCover}, some initial segment of each cusp leaf is contained in some face rectangle.
Cutting down gives a cusp rectangle $R$ in $\link(\calV)$.
The cusp rectangles generated in this way form a loom-cusp by another application of \reflem{CrownsInterleave}.

In the other direction, suppose that $R$ is a cusp rectangle.
By parts \refitm{LinkIsLoom} and \refitm{TetIsLoomTet}, the rectangle $R$ is contained in a tetrahedron rectangle.
Thus it arises as in the previous paragraph, and we have obtained \refitm{CuspIsLoomCusp}.

By construction, the action of $\pi_1(M)$ on $\link(\calV)$ sends tetrahedron rectangles to tetrahedron rectangles, thus is continuous.
It also sends leaves to leaves.
Since elements of $\pi_1(M)$ have inverses, we have obtained~\refitm{LoomAction}.
\end{proof}

\chapter{...and back again}
\label{Cha:BackAgain}

Let $\Loom(\RR^2)$ be the category in which objects are loom spaces and arrows are loom isomorphisms.
Let $\Veer(\RR^3)$ be the category in which objects are locally veering triangulations of $\RR^3$ and arrows are \emph{taut isomorphisms}: that is, isomorphisms of ideal triangulations that preserve tautness. 

In our previous work \cite[Proposition~5.19 and Theorem~6.50]{SchleimerSegerman24}, following Gu\'eritaud, we gave a functor 
\[
\veer \from \Loom(\RR^2) \to \Veer(\RR^3)
\]
defined as follows.
In a loom space $\calL$, there are three kinds of \emph{skeletal rectangles}.
These are the loom-tetrahedron rectangles (\refdef{TetRectLoom}), the loom-face rectangles, and the loom-edge rectangles.
These meet four, three, or two cusps respectively, with exactly one cusp in the closure of each side.

\begin{definition}
\label{Def:InducedVeering}
We define the \emph{induced} locally veering triangulation $\calV = \veer(\calL)$ by taking one model cell $\cell(R)$ 
for every skeletal rectangle $R\subset \calL$.
We glue two model cells $\cell(R)$ and $\cell(R')$ along another model cell $\cell(S)$ if $R \cap R' = S$.
\end{definition}

Our goal in \refcha{BackAgain} is to show that the link space construction 
\[
\link \from \Veer(\RR^3) \to \Loom(\RR^2)
\]
is a functor (\refprop{LinkIsFunctor}) and is an ``inverse'' for $\veer$ (\refthm{Equivalence}).

\subsection{Partial orders}

Throughout this section we make parallel definitions for loom spaces and locally veering triangulations. 

Suppose that $\calL$ is a loom space.
We choose an orientation on $\calL$.
We also choose one of the foliations to be upper, that is $F^\calL$, and the other to be lower, that is $F_\calL$.
The loom automorphisms of $\calL$ may not respect these choices. 
However, there is a subgroup of finite index which does.

\begin{definition}
\label{Def:OrderLoom}
Suppose that $Q$ and $R$ are loom-tetrahedron rectangles of $\calL$ so that $Q$ west-east spans $R$. 
In this case, we write $Q \prec R$.
\end{definition}
\noindent
Note that $Q \prec R$ if and only if $R$ south-north spans $Q$.
Also, \refdef{OrderLoom} gives a partial order on loom-tetrahedron rectangles.

Suppose that $\calV$ is a locally veering triangulation of $\RR^3$.
We choose an orientation of $\RR^3$ and a transverse veering structure for $\calV$.
The taut automorphisms of $\calV$ may not respect these choices.
However, there is a subgroup of finite index which does.

\begin{definition}
\label{Def:OrderVeer}
Suppose that $s$ and $t$ are tetrahedra of $\calV$ sharing a face $f$.
If the transverse orientation on $f$ points from $s$ to $t$ 
then we write $s \prec t$, and take the transitive closure.  
\end{definition}
 
The orders given by Definitions~\ref{Def:OrderLoom} and~\ref{Def:OrderVeer} are related as follows.
 
\begin{lemma}
\label{Lem:Order}
\mbox{}
\begin{itemize}
\item Suppose that $Q$ and $R$ are loom-tetrahedron rectangles of $\calL$.
Then $Q \prec R$ if and only if $\cell(Q) \prec \cell(R)$.
\item Suppose that $s$ and $t$ are tetrahedra of $\calV$. 
Then $s \prec t$ if and only if $\rect(s) \prec \rect(t)$.
\end{itemize}
\end{lemma}

\begin{proof}
Suppose that $Q \prec R$.
By \cite[Lemma~4.22]{SchleimerSegerman24} and by the construction given in \cite[Definition~5.9]{SchleimerSegerman24}, we deduce that $\cell(Q) \prec \cell(R)$.

Now suppose that $\cell(Q) \prec \cell(R)$. 
Let 
\[
(\cell(Q) = \cell(R_0), \cell(R_1), \ldots, \cell(R_n) = \cell(R))
\]
be an alternating sequence of tetrahedra and faces with the following property. 
\begin{itemize}
\item For all $i$, we have $\cell(R_{2i+1})$ is an upper face of $\cell(R_{2i})$ and is a lower face of $\cell(R_{2i+2})$.
\end{itemize}
By the construction given in \cite[Definition~5.9]{SchleimerSegerman24}, we have that $R_{2i} \prec R_{2i+2}$.
By transitivity, $Q \prec R$.

Suppose $s \prec t$.
As above, there is an alternating sequence 
\[
(s = c_0, c_1, \ldots, c_n = t)
\]
of tetrahedra and faces with the following property.
\begin{itemize}
\item For all $i$, we have $c_{2i+1}$ is an upper face of $c_{2i}$ and is a lower face of $c_{2i+2}$.
\end{itemize}
By \reflem{TetRect}\refitm{TetRectFace}, we have that $\rect(c_i)$ west-east spans $\rect(c_{i+1})$.
We deduce that $\rect(s)$ west-east spans $\rect(t)$.
Thus $\rect(s) \prec \rect(t)$.

Finally, suppose that $\rect(s) \prec \rect(t)$.
By \refthm{LinkIsLoom}, we may apply \cite[Lemma~4.22]{SchleimerSegerman24}.  
From \reflem{TetRect}\refitm{TetRectFace} we deduce that $s \prec t$.
\end{proof}

\begin{remark}
\label{Rem:LoomBetter}
\reflem{Order} implies that the relation given in \refdef{OrderVeer}, on the tetrahedra of $\calV$, is isomorphic to the relation given in \refdef{OrderLoom}, on tetrahedron rectangles of $\link(\calV)$.
Since the latter is a partial order, so is the former.
Thus it suffices to carry out constructions in the setting of loom spaces.
\end{remark}

\begin{lemma}
\label{Lem:Bookends}
Suppose that $R \prec S \prec T$ are loom-tetrahedron rectangles of $\calL$. 
Then $R \cap T$ is nonempty and is contained in $S$.
\end{lemma}

\begin{proof}
Since $\prec$ is transitive, $R \prec T$, and so $R \cap T$ is nonempty.
Suppose that $p$ is a point of $R \cap T$.
Let $\ell$ and $m$ be the leaves of $F^\calV$ and $F_\calV$, respectively, through $p$.
Note that $p \in \ell \cap R \subset \ell \cap S$.
Similarly, $p \in m \cap T \subset m \cap S$.
Therefore $p \in S$ and we are done.
\end{proof}

Recalling \refdef{Spans}, we have the following.

\begin{lemma}
\label{Lem:StrictlySpansLoom}
Suppose that $R \prec T$ are loom-tetrahedron rectangles in $\calL$.  
Then $T$ strictly south-north spans $R$ if and only if
the south and north vertices of $T$ are disjoint from the south and north vertices of $R$.
In similar fashion $R$ strictly west-east spans $T$ if and only if
the west and east vertices of $R$ are disjoint from the west and east vertices of $T$.  \qed
\end{lemma}

Note that neither hypothesis in \reflem{StrictlySpansLoom} implies the other. 
See \reffig{StrictBound}.
We now give a slightly weakened version of the \emph{astroid lemma}~\cite[Lemma~4.10]{SchleimerSegerman24}.

\begin{lemma}
\label{Lem:WeakAstroid}
Suppose that $p$ is a point of a loom space $\calL$.
Suppose that $S$ is a rectangle of $\calL$ containing $p$.
Then there is a loom-tetrahedron rectangle $R$ with the following properties.
\begin{itemize}
\item $p \in R$.
\item $R$ west-east spans $S$.
\item $S$ strictly south-north spans $R$.
\end{itemize}
Similarly, there is a loom-tetrahedron rectangle $T$ with the corresponding properties, swapping west-east and south-north. 
\end{lemma}

\begin{figure}[htbp]
\centering
\subfloat[]{
\labellist
\small\hair 2pt
\pinlabel {$R$} at 330 187
\pinlabel {$T$} at 140 65
\endlabellist
\includegraphics[width=0.29\textwidth]{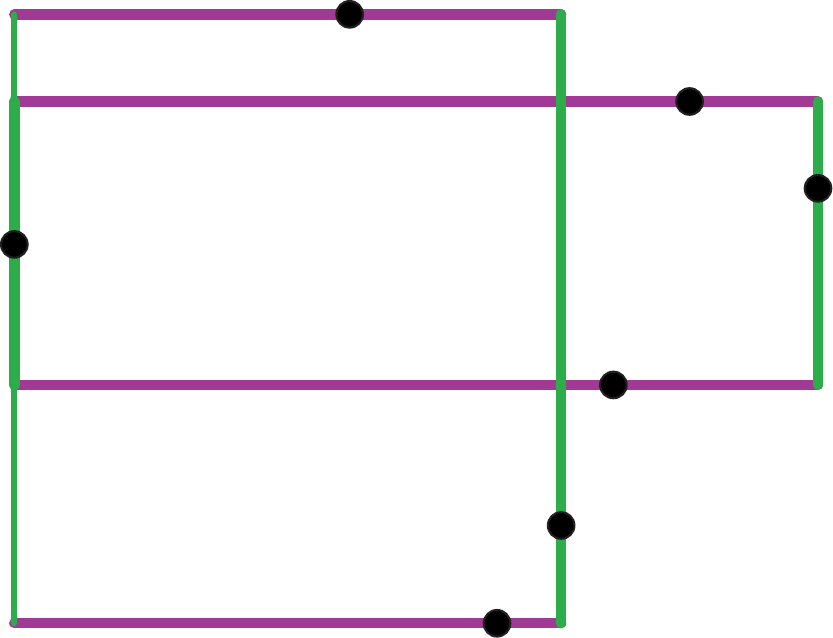}
\label{Fig:NotStrictWE}
}
\quad
\subfloat[]{
\labellist
\small\hair 2pt
\pinlabel {$R$} at 330 210
\pinlabel {$T$} at 190 65
\endlabellist
\includegraphics[width=0.29\textwidth]{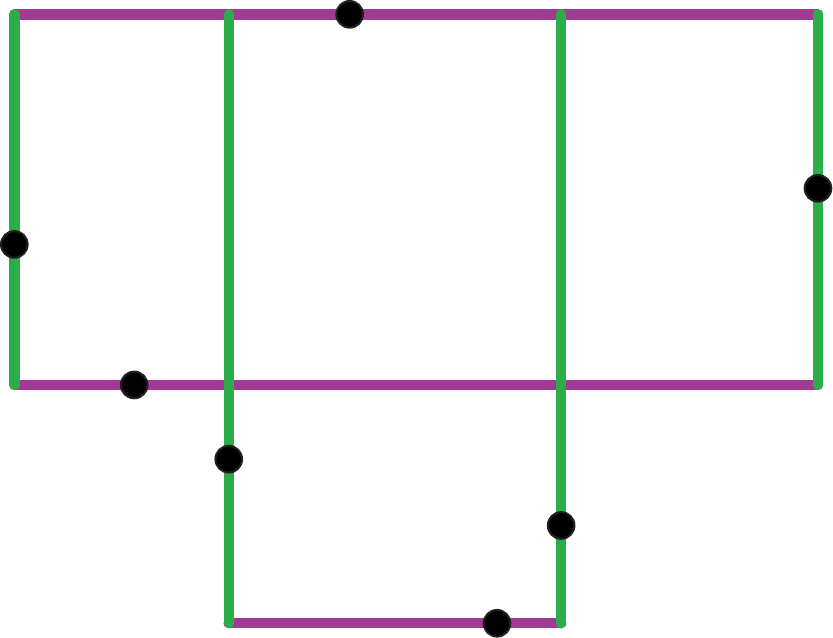}
\label{Fig:NotStrictSN}
}
\quad
\subfloat[]{
\labellist
\small\hair 2pt
\pinlabel {$R$} at 330 187
\pinlabel {$T$} at 190 65
\endlabellist
\includegraphics[width=0.29\textwidth]{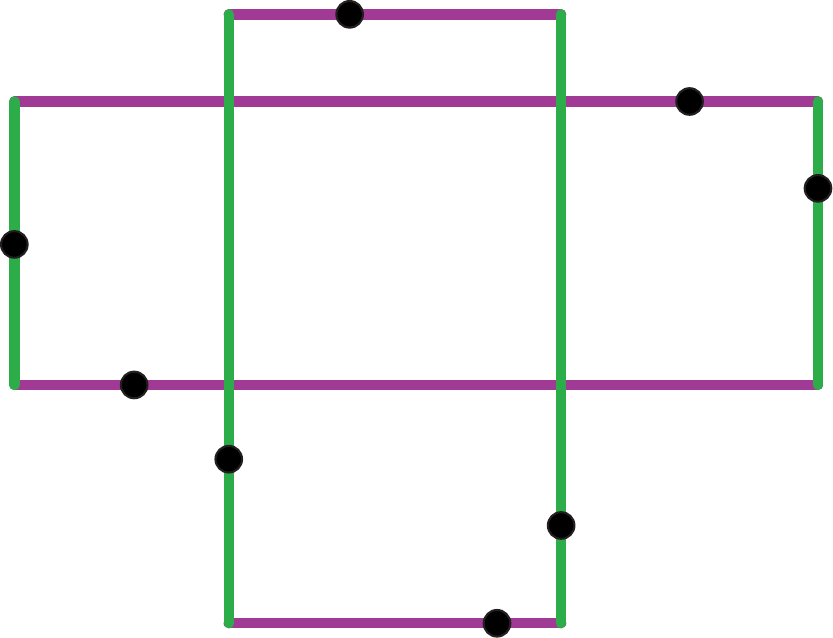}
\label{Fig:BothStrict}
}
\caption{Loom-tetrahedron rectangles $R \prec T$.  
In \reffig{NotStrictWE} we have that $T$ strictly south-north spans $R$, but $R$ does not strictly west-east span $T$. 
In \reffig{NotStrictSN} we have that $R$ strictly west-east spans $T$, but $T$ does not strictly south-north span $R$.  
In \reffig{BothStrict} we have that each strictly spans the other.
}
\label{Fig:StrictBound}
\end{figure}

\begin{proof}[Sketch proof of \reflem{WeakAstroid}]
Let $m$ be the west-east leaf through $p$.
By \cite[Lemma~3.10]{SchleimerSegerman24}, the leaf $m$ meets at most one cusp.
Breaking symmetry, we may assume that this cusp, if it exists, is not to the west of $p$.
Take a copy of $S$ and push its northern side slightly to the south and its southern side slightly to the north to obtain a new rectangle, still containing $p$.
Now push its western side to the west.
By the astroid lemma~\cite[Lemma~4.10]{SchleimerSegerman24}, the western side eventually meets a cusp, say $c_0$, and we stop.

By our assumption, $c_0$ does not lie on $m$.
Breaking symmetry we assume that $c_0$ is to the north of $p$.
Push the northern side of our rectangle south until $c_0$ lies at its north-west corner.
We repeat this process, replacing $c_0$ with $c_i$, until we find $c_n$, a cusp on the western side of the rectangle but which is to the south of $p$.
Again, we terminate after a finite number of steps by the astroid lemma.

Now push the southern side of our rectangle north until $c_n$ lies at its south-west corner.
At this point, push the western side west and the eastern side east until they meet cusps, $c_W$ and $c_E$ say.
Again, the astroid lemma ensures that $c_W$ and $c_E$ exist.

We now have a loom-tetrahedron rectangle $R$ with $c_{n-1}$ as its northern cusp, with $c_n$ as its southern cusp, and with $c_W$ and $c_E$ as its western and eastern cusps.
Moreover, since we never increased the south-north extent, and in the first step, decreased it by a definite amount, $S$ strictly south-north spans $R$.
Since we never decreased the west-east extent, $R$ west-east spans $S$.
\end{proof}

\begin{corollary}
\label{Cor:SubBasis}
The loom-tetrahedron rectangles give a subbasis for the topology of the loom space. \qed
\end{corollary}

We also have the following finiteness property~\cite[Lemma~4.16]{SchleimerSegerman24}.

\begin{lemma}
\label{Lem:Finiteness}
Suppose that $S$ is a rectangle of $\calL$.
Then there are only finitely many loom-tetrahedron rectangles containing $S$. \qed
\end{lemma}



\begin{lemma}
\label{Lem:CommonBounds}
Suppose that $S$ and $S'$ are loom-tetrahedron rectangles of $\calL$.
Suppose that $p$ is a point of $S \cap S'$.
Then there are loom-tetrahedron rectangles $R$ and $T$ so that $p \in R \cap T$ and so that $S$ and $S'$ strictly south-north span $R$ and strictly west-east span $T$.
\end{lemma}

\begin{proof}
We apply \reflem{WeakAstroid} to the (not necessarily loom-tetrahedron) rectangle $S \cap S'$. This gives loom-tetrahedron rectangles $R$ and $T$ with the desired properties.
\end{proof}

\subsection{Astroids and axes}

\begin{definition}
\label{Def:Astroid}
Suppose that $\calL$ is a loom space.
Suppose that $p$ is a point of $\calL$.
The \emph{astroid} $\astro(p)$ is the set of loom-tetrahedron rectangles containing $p$.
\end{definition}

\begin{figure}[htbp]
\includegraphics[width=\textwidth]{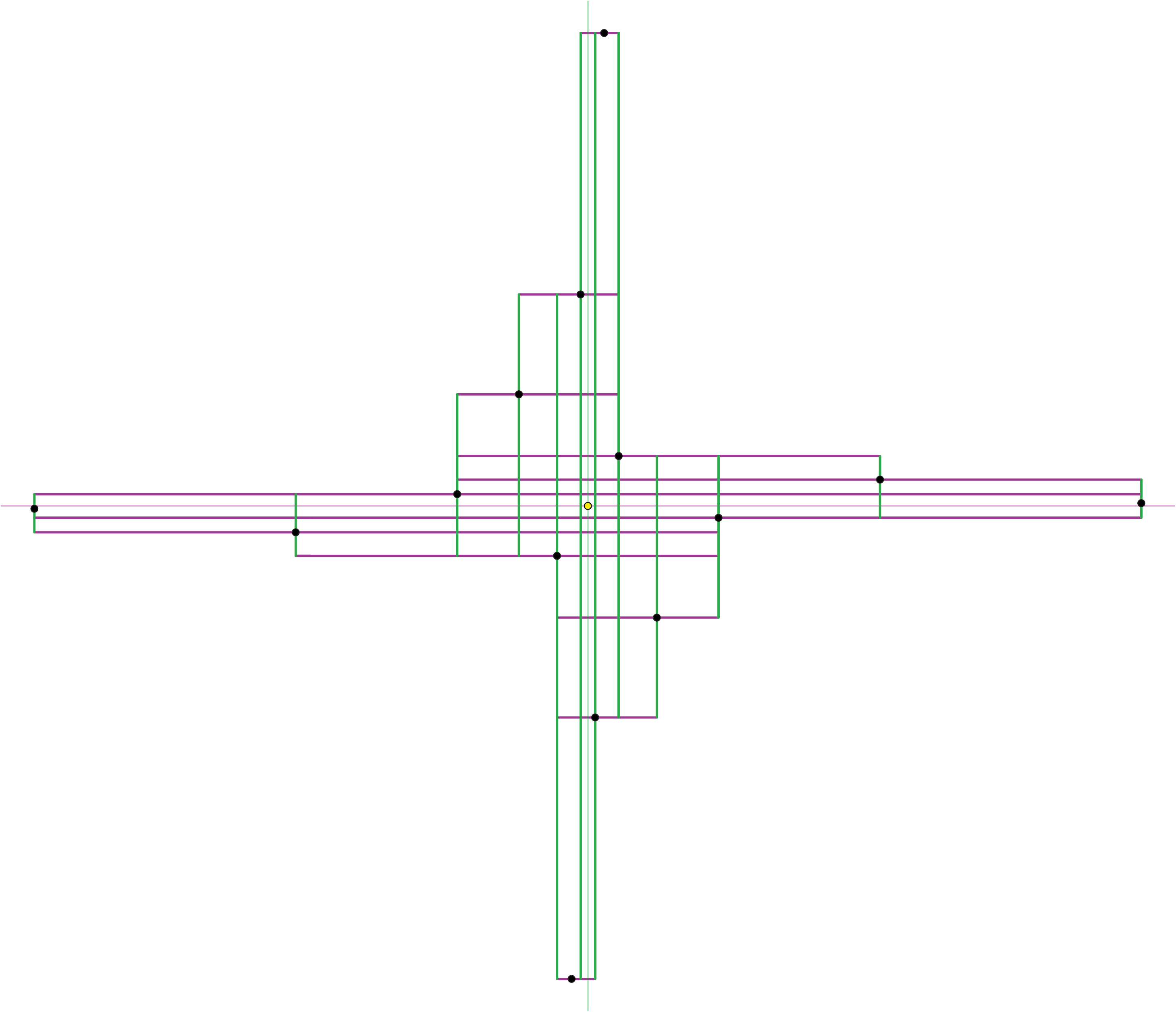}
\caption{Some of the astroid $\astro(p)$ when $p$ is a Weierstrass point (centre) for the fibre of the monodromy of the figure eight knot complement.
If $p$ lies on one or two cusp leaves then the corresponding ``arms'' of the astroid are truncated.}
\label{Fig:Astroid}
\end{figure}

\begin{lemma}
\label{Lem:AstroidIntersection}
The intersection of all rectangles in $\astro(p)$ equals $\{p\}$.
\end{lemma}

\begin{proof}
Let $P$ be the intersection of all rectangles in $\astro(p)$.
Suppose that $S$ is any (not necessarily loom-tetrahedron) rectangle containing $p$.
By \reflem{WeakAstroid}, there is a rectangle $R \in \astro(p)$ so that $S$ strictly south-north spans $R$.
Thus no point of $P$ lies to the south or to the north of $S$.
By a similar argument, no point of $P$ lies to the west or to the east of $S$.
Since $S$ was arbitrary, $P = \{p\}$.
\end{proof}

\begin{definition}
\label{Def:LoomAxis}
Suppose that $\calL$ is a loom space.
We say that a non-empty collection $A$ of loom-tetrahedron rectangles is a \emph{loom-axis} if it satisfies the following properties.
\begin{enumerate}
\item
\label{Itm:LoomBounds}
For all $S, S' \in A$ there are $R, T \in A$ so that $S$ and $S'$ strictly south-north span $R$ and strictly west-east span $T$.
\item
\label{Itm:LoomMaximal}
$A$ is maximal with respect to the previous property. \qedhere
\end{enumerate}
\end{definition}

\begin{theorem}
\label{Thm:AxisIFFAstroid}
Suppose that $\calL$ is a loom space.
Suppose that $A$ is a set of loom-tetrahedron rectangles in $\calL$.
Then $A$ is a loom-axis if and only if it is an astroid.
\end{theorem}

\begin{proof}
Suppose that $A$ is a loom-axis.
By definition, $A$ is nonempty.
From \cite[Corollary~4.25]{SchleimerSegerman24} we deduce that $A$ is countable.
We order the elements of $A = (S_i)$. 
By \refdef{LoomAxis}\refitm{LoomBounds} there is a sequence $(R_k) \subset A$ so that 
each element of the collection 
\[
\{ S_0, S_1, \ldots, S_{k}\} \cup \{ R_{k} \}  
\]
strictly south-north spans $R_{k+1}$.
There is a similar sequence $(T_k) \subset A$ replacing south-north by west-east.
We deduce that 
\[
R_{k+1} \prec R_{k} \prec T_{k} \prec T_{k+1}
\quad \mbox{and} \quad 
R_{k+1} \prec S_{k} \prec T_{k+1}
\]

By \reflem{Bookends}, the intersection $R_{k+1} \cap T_{k+1}$ is nonempty and lies in $R_{k} \cap S_k \cap T_{k}$.
Since $R_{k}$ and $T_{k}$ strictly span $R_{k+1}$ and $T_{k+1}$, respectively, the closure of $R_{k+1} \cap T_{k+1}$ is strictly contained in $R_k \cap T_k$. 
\reflem{Finiteness} implies that the nested intersection is a single point $p$ of $\calL$.
\reflem{Bookends} now implies that $p$ lies in $S_k$ for all $k$.
Thus $A \subset \astro(p)$.

On the other hand, by \reflem{CommonBounds}, the astroid $\astro(p)$ satisfies \refdef{LoomAxis}\refitm{LoomBounds}.
This, together with $A \subset \astro(p)$ and \refdef{LoomAxis}\refitm{LoomMaximal} for $A$ proves that $A = \astro(p)$.

For the opposite direction, fix $p$ and consider the astroid $\astro(p)$.
Again, by \reflem{CommonBounds}, we have that $\astro(p)$ satisfies \refdef{LoomAxis}\refitm{LoomBounds}.
Suppose that $A$ is a set of loom-tetrahedron rectangles that also satisfies \refdef{LoomAxis}\refitm{LoomBounds} and contains $\astro(p)$.
Suppose, for a contradiction, that there is some loom-tetrahedron rectangle $S$ so that $S \in A - \astro(p)$.
There are two cases as $p$ misses or lies in $\bdy S$. 

Suppose that $p$ is not in $\bdy S$. 
Breaking symmetry, suppose that $m$ is the leaf of $F_\calL$ which contains a side of $S$ and separates $S$ from $p$.
By \reflem{WeakAstroid} there is a loom-tetrahedron rectangle $S'$ that contains $p$ and is disjoint from $m$, and thus is disjoint from $S$.  
Since $S'$ contains $p$, we have that $S' \in \astro(p) \subset A$.
Now applying \refdef{LoomAxis}\refitm{LoomBounds} and \reflem{Bookends}, we contradict the fact that $S$ and $S'$ are disjoint.

Suppose instead that $p$ lies in $\bdy S$. 
Breaking symmetry, suppose that $p$ lies in the north side of $S$.
By \refdef{LoomAxis}\refitm{LoomBounds} (applied to the loom-tetrahedron rectangle $S$ twice)
there is a loom-tetrahedron rectangle $R$ in $A$ so that $S$ strictly south-north spans $R$.
Thus $R$ and $S$ have disjoint north sides.
Thus $p$ does not lie in the closure of $R$.
Taking $R$ in place of $S$, we are now in the first case:
since $R \in A$ and $p$ misses $\bdy{R}$ we reach a contradiction.

Thus $\astro(p)$ is maximal and so satisfies \refdef{LoomAxis}\refitm{LoomMaximal}. 
Therefore $\astro(p)$ is a loom-axis, as desired.
\end{proof}

\subsection{Bijections}

We now translate \reflem{StrictlySpansLoom} into the language of veering triangulations.

\begin{definition}
\label{Def:UpperLower}
As in \refsec{TautIdealTriangulation}, for every tetrahedron $t$ in $\calV$, the transverse orientation chooses one of the two non-equatorial edges to be the upper edge of $t$. 
The other non-equatorial edge is the lower edge of $t$.
We refer to the two vertices of $t$ incident to the upper edge as its \emph{upper vertices}.
The remaining two vertices of $t$ are its \emph{lower vertices}.
\end{definition}

\begin{definition}
\label{Def:StrictlySpansVeer}
Suppose that $r \prec t$ are tetrahedra in $\calV$.  
We say that $t$ \emph{strictly south-north spans} $r$ if the upper vertices of $t$ are disjoint from the upper vertices of $r$.
In similar fashion, we say that $r$ \emph{strictly west-east spans} $t$ if 
the lower vertices of $r$ are disjoint from the lower vertices of $t$.  
\end{definition}

We transport \refdef{LoomAxis} to veering triangulations, as follows.

\begin{definition}
\label{Def:FlowAxis}
Suppose that $\calV$ is a transverse veering triangulation of $\RR^3$.
We say that a non-empty collection $A \subset \calV$ of tetrahedra is a \emph{(coarse flow) axis} if
\begin{enumerate}
\item
\label{Itm:TautStrictBounds}
For all $s, s' \in A$ there are $r, t \in A$ so that $s$ and $s'$ strictly south-north span $r$ and strictly west-east span $t$.
\item
\label{Itm:TautMaximal}
$A$ is maximal with respect to the previous property. \qedhere
\end{enumerate}
\end{definition}

\begin{remark}
\label{Rem:Swap}
Choosing the other transverse veering structure on $\calV$ swaps upper with lower but does not change the set of flow-axes.
\end{remark}

Note that, following \refrem{LoomBetter}, we do not need to define ``astroids'' for veering triangulations. 
We now define two functions.

\begin{definition}
\label{Def:AxisForP}
Suppose that $p$ is a point of a loom space $\calL$.
Recall that if $R$ is a loom-tetrahedron rectangle then $\cell(R)$ is the corresponding tetrahedron of $\veer(\calL)$.
We define $\axis(p)$, the \emph{flow-axis for} $p$, as follows.
\[
\axis(p) = \{ \cell(R) \in \veer(\calL) \st  R \in \astro(p) \} \qedhere
\] 
\end{definition}

By \refthm{AxisIFFAstroid} and then \reflem{Order}, we have that $\axis(p)$ is a flow-axis in $\veer(\calL)$.
Thus $\axis$ is a function from $\calL$ to the set of flow-axes in $\veer(\calL)$.

\begin{definition}
\label{Def:PointForA}
Suppose that $A$ is a flow-axis of a veering triangulation $\calV$. 
Recall that if $t$ is a tetrahedron then $\rect(t)$ is the corresponding rectangle of $\link(\calV)$.
We define $\point(A)$, the \emph{point for $A$}, as follows.
\[
\{\point(A)\} = \cap \{ \rect(t) \st t \in A \} \qedhere
\]
\end{definition}

By \reflem{Order}, 
\refthm{AxisIFFAstroid}, 
and \reflem{AstroidIntersection}
we have that the intersection is a singleton set in $\link(\calV)$.
Therefore, $\point(A)$ is a well-defined point of $\link(\calV)$.
That is, $\point$ is a function from the set of flow-axes in $\calV$ to $\link(\calV)$.
Applying \reflem{Order} and \refthm{AxisIFFAstroid} again, we deduce the following.

\begin{corollary}
\label{Cor:Bijections}
The functions $\axis$ and $\point$ are bijections. \qed
\end{corollary}

Restricting to tetrahedra and loom-tetrahedron rectangles, and applying \refdef{InducedVeering} and \refthm{RectInTetRect}, we also have the following.

\begin{corollary}
\label{Cor:MoreBijections}
The functions $\cell$ and $\rect$ are bijections. \qed
\end{corollary}


\begin{lemma}
\label{Lem:AxisCentre}
Suppose that $\calV$ is a veering triangulation.
Suppose that $A$ is a flow-axis of $\calV$.
Then we have the following.
\begin{enumerate}
\item
\label{Itm:AxisCentre}
$\rect(A) = \astro(\point(A))$.
\item
\label{Itm:TetInAxis}
Suppose that $t$ is a tetrahedron of $\calV$.
Then $t \in A$ if and only if $\point(A) \in \rect(t)$.
\end{enumerate}
\end{lemma}

\begin{proof}
By \reflem{Order} and \refcor{MoreBijections}, we have that $\rect(A)$ is a loom-axis.
By \refthm{AxisIFFAstroid}, we have that $\rect(A)$ is an astroid.
By \refdef{Astroid}, there is some $q \in \link(\calV)$ so that  $\rect(A) = \astro(q)$.
By \reflem{AstroidIntersection}, we have that $\{q\} = \cap\astro(q) = \cap\rect(A)$.
However, by \refdef{PointForA}, we have that $\cap\rect(A) = \point(A)$.
Thus $\{q\} = \point(A)$.
Substituting, we find that $\rect(A) = \astro(\point(A))$, giving \refitm{AxisCentre}.

Now suppose that $t$ lies in $A$.
So $\rect(t) \in \rect(A)$ and thus by \refdef{PointForA}, we have that $\point(A) \in \rect(t)$.

Suppose instead that $\point(A) \in \rect(t)$. 
By \refdef{Astroid}, we have that $\rect(t) \in \astro(\point(A))$.
Since $\astro(\point(A)) = \rect(A)$, we have that $\rect(t) \in \rect(A)$.
Applying \refcor{MoreBijections}, we obtain \refitm{TetInAxis}.
\end{proof}

\begin{definition}
\label{Def:Quiver}
Suppose that $\calV$ is a veering triangulation.
Suppose that $t$ is a tetrahedron of $\calV$.
The \emph{quiver} $\qui(t)$ is the set of flow-axes in $\calV$ containing $t$.
\end{definition}

\begin{lemma}
\label{Lem:PointInverse}
$\point(\qui(t)) = \rect(t)$.
\end{lemma}

\begin{proof}
Suppose that $p \in \point(\qui(t))$.
That is, there is a flow-axis $A \in \qui(t)$ with $p = \point(A) = \cap\rect(A)$.
Since $t \in A$, by the forward direction of \reflem{AxisCentre}\refitm{TetInAxis} we have that $p \in \rect(t)$.

Suppose that $p \in \rect(t)$. 
Set $A = \rect^{-1}(\astro(p))$.
By \refthm{AxisIFFAstroid}, \reflem{Order}, and \refcor{MoreBijections}, the set $A$ is a flow-axis.
Since $\rect(t) \in \astro(p)$, applying \refcor{MoreBijections} again we have that $t \in A$.
By \refdef{Quiver}, we have that $A \in \qui(t)$.
On the other hand, $\rect(A) = \astro(p)$, so by \reflem{AxisCentre}\refitm{AxisCentre}, we have that $p = \point(A)$.
Since $A \in \qui(t)$, we may apply $\point$ and obtain $p = \point(A) \in \point(\qui(t))$.
\end{proof}

\subsection{Functoriality}

For the convenience of the reader, we recall the definition of the taut isomorphism induced by a loom isomorphism~\cite[Definition~5.10]{SchleimerSegerman24}.

\begin{definition}
\label{Def:InducedTautIso}
Suppose that $\calL$ and $\calM$ are loom spaces. 
Suppose that $\phi \from \calL \to \calM$ is a loom isomorphism.
We define $\veer(\phi) = \veer_\phi$, the \emph{induced map} from $\veer(\calL)$ to $\veer(\calM)$, as follows.
\[
\veer_\phi(\cell(R)) = \cell(\phi(R)) \qedhere 
\] 
\end{definition}

\begin{lemma}
\label{Lem:ImageOfAxis}
With hypotheses as in \refdef{InducedTautIso}, suppose that $p$ is a point of $\calL$. 
Then we have the following.
\[
\veer_\phi(\axis(p)) = \axis(\phi(p))
\]
\end{lemma}

\begin{proof}
Recall from \refdef{Astroid} that $\astro(p)$ is the set of loom-tetrahedron rectangles containing $p$.
Since $\phi$ is a loom isomorphism, we have that $R \in \astro(p)$ if and only if $\phi(R) \in \astro(\phi(p))$. 
This proves the following.
\begin{claim}
\label{Clm:AstroCommute}
$\phi(\astro(p)) = \astro(\phi(p))$ \qed
\end{claim}
\noindent
We now compute.
\begin{align*}
\veer_\phi(\axis(p)) &= \veer_\phi(\cell(\astro(p))) && \mbox{\refdef{AxisForP}} \\
            &= \cell(\phi(\astro(p))) && \mbox{\refdef{InducedTautIso}} \\
            &= \cell(\astro(\phi(p))) && \mbox{\refclm{AstroCommute}} \\ 
            &= \axis(\phi(p)) && \mbox{\refdef{AxisForP}} \qedhere
\end{align*}
\end{proof}

We now set out the parallel definition and lemma in the context of veering triangulations.

\begin{definition}
\label{Def:InducedLoomIso}
Suppose that $\calV$ and $\calW$ are locally veering triangulations of $\RR^3$. 
Suppose that $f \from \calV \to \calW$ is a taut isomorphism.
We define $\link(f) = \link_f$, the \emph{induced map} from $\link(\calV)$ to $\link(\calW)$, as follows.
\[
\link_f(\point(A)) = \point(f(A)) \qedhere 
\] 
\end{definition}

Note that in making this definition, by \refcor{Bijections} it suffices to consider points of $\link(\calV)$ of the form $\point(A)$ where $A$ is a flow-axis in $\calV$.

\begin{lemma}
\label{Lem:ImageOfRect}
With hypotheses as in \refdef{InducedLoomIso}, suppose that $t$ is a tetrahedron of $\calV$. 
Then we have the following.
\[
\link_f(\rect(t)) = \rect(f(t))
\]
\end{lemma}

\begin{proof}
Recall from \refdef{Quiver} that $\qui(t)$ is the set of flow-axes containing $t$.
Since $f$ is a taut isomorphism, $t$ lies in a flow-axis $A$ if and only if $f(t) \in f(A)$. 
This proves the following.
\begin{claim}
\label{Clm:QuiCommutes}
$f(\qui(t)) = \qui(f(t))$ \qed
\end{claim}
\noindent
We now compute.
\begin{align*}
\link_f(\rect(t)) &= \link_f(\point(\qui(t))) && \mbox{\reflem{PointInverse}} \\
            &=  \point(f(\qui(t)))&& \mbox{\refdef{InducedLoomIso}} \\
            &=  \point(\qui(f(t)))&& \mbox{\refclm{QuiCommutes}} \\ 
            &= \rect(f(t))&& \mbox{\reflem{PointInverse}} \qedhere
\end{align*}
\end{proof}

We are now equipped to prove the following.

\begin{proposition}
\label{Prop:LinkIsFunctor}
The link space construction $\link$ is a functor.
\end{proposition}

\begin{proof}
Suppose that $\calV$ is a locally veering triangulation of $\RR^3$. 
By \refthm{LinkIsLoom}, the link space $\link(\calV)$ is a loom space.

Applying \refdef{InducedLoomIso} twice, we have 
\[
\link_{\id_\calV}(\point(A)) = \point(\id_\calV(A)) = \point(A) = \id_{\link(\calV)}(\point(A))
\]
Thus $\link(\id_\calV) = \id_{\link(\calV)}$. 

Suppose that $\calU$ and $\calW$ are also locally veering triangulations of $\RR^3$. 
Suppose that $f \from \calU \to \calV$ and $g \from \calV \to \calW$ are taut isomorphisms.
Suppose that $A$ is a flow-axis in $\calU$.
Applying \refdef{InducedLoomIso} three times we have the following.
\[
\link_{g\circ f}(\point(A)) = \point((g\circ f)(A)) = \point(g(f(A))) = \link_g(\point(f(A))) =
 \link_g \circ \link_f(\point(A))
\]

Since $f$ is a taut isomorphism, so is $f^{-1} \from \calV \to \calU$.
We now compute.
\[
\link(f) \circ \link(f^{-1}) = \link(f \circ f^{-1}) = \link(\id_\calV) = \id_{\link(\calV)}
\]
Similarly, $\link(f^{-1}) \circ \link(f) = \id_{\link(\calU)}$.

By \reflem{ImageOfRect}, we have that $\link(f)$ sends tetrahedron rectangles to tetrahedron rectangles.
The same holds for  $\link(f^{-1})$.
By \refcor{SubBasis}, we deduce that $\link(f)$ and $\link(f^{-1})$ are both continuous.
Thus both are homeomorphisms.

\begin{claim}
\label{Clm:Wall}
Suppose that $p$ and $p'$ are points of a loom space $\calL$.
Then $p$ and $p'$ are on a common leaf of $F^\calL$ if and only if for any $S \in \astro(p)$ and $S' \in \astro(p')$ there is a loom-tetrahedron rectangle which south-north spans both $S$ and $S'$.
The same holds, replacing $F^\calL$ with $F_\calL$ and south-north with west-east.
\end{claim}

\begin{proof}
Suppose that $p$ and $p'$ lie on $\ell$, a leaf of $F^\calL$. 
Fix $S\in \astro(p)$ and $S' \in \astro(p')$.
Let $\ell'$ be an interval in $\ell$, not containing any non-cusp end of $\ell$ but meeting both the north and south sides of both $S$ and $S'$.

Since rectangles provide a basis for the topology of $\calL$, and applying \refdef{Loom}\refitm{Cusp} if needed, we may cover the interval $\ell'$ by a finite collection of rectangles. 
Thus $\ell'$ is contained in a single rectangle $S''$.
We deduce that $S''$ south-north spans both $S$ and $S'$.
By \reflem{WeakAstroid}, there is a loom-tetrahedron rectangle which contains $p$ and which south-north spans $S''$. 

We now prove the converse.
Suppose that for every $S\in \astro(p)$ and $S' \in \astro(p')$ there is a loom-tetrahedron rectangle that south-north spans both.
By \refthm{AxisIFFAstroid}, we have a sequence $(S_n) \subset \astro(p)$ so that $S_n$ strictly west-east spans $S_{n+1}$.
Let $\ell_p$ be the leaf of $F^\calL$ containing $p$. 
From \reflem{Finiteness} we deduce that $\limsup S_n$ is contained in $\ell_p$.
Suppose for a contradiction that $\limsup S_n$ is not all of $\ell_p$.
Breaking symmetry, suppose that $\limsup S_n$ has a material boundary point, say $q$, to its south.
Let $Q$ be a small rectangle about $q$.
Thus $Q$ contains the southern sides of infinitely many of the $S_n$.
We deduce that these loom-tetrahedron rectangles $S_n$ do not have cusps on their southern sides, a contradiction.

Similarly, there is a sequence $(S'_n) \subset \astro(p')$ so that $\limsup S'_n = \ell_{p'}$.
Let $T_n$ be the given loom-tetrahedron rectangle that south-north spans both $S_n$ and $S'_n$.
We deduce that $\limsup T_n$ is equal to both $\ell_p$ and $\ell_{p'}$.
Thus $\ell_p = \ell_{p'}$, as desired.
\end{proof}

\refclm{Wall} gives a combinatorial characterisation of leaves in terms of loom-axes.
From \reflem{Order} we deduce that $\link(f)$ sends leaves to leaves.
Thus $\link(f)$ is a loom isomorphism.
\end{proof}

\subsection{Equivalence of categories}
\label{Sec:EquivalenceOfCategories}

The goal of this section is to prove the following.

\begin{theorem}
\label{Thm:Equivalence}
The functors 
\[
\link \from \Veer(\RR^3) \to \Loom(\RR^2) \quad \mbox{and} \quad \veer \from  \Loom(\RR^2) \to \Veer(\RR^3)
\]
give an equivalence of categories.
\end{theorem}

To prove the theorem, we produce two natural isomorphisms:
\[
\eta \from \Id_{\Loom} \to \link \circ \veer \quad \mbox{and} \quad \zeta \from \Id_{\Veer}  \to \veer \circ \link
\]

\begin{definition}
\label{Def:Eta}
Suppose that $\calL$ is a loom space. 
We define a map 
\[
\eta_\calL \from \calL \to \link \circ \veer (\calL) \quad \mbox{by setting} \quad
\eta_\calL(p) = \point(\axis(p)) \qedhere
\]
\end{definition}

\begin{proposition}
\label{Prop:Eta}
$\eta$ is a natural isomorphism.
\end{proposition}

\begin{proof}
We verify the defining properties as follows.

\begin{claim}
Suppose that $\calL$ is a loom space.
Then the map $\eta_\calL$ is a loom isomorphism.
\end{claim}

\begin{proof}
By \refcor{Bijections}, the map $\eta_\calL$ is a bijection.

Let $R$ be a loom-tetrahedron rectangle in $\calL$.
We now have the following.
\begin{align*}
p \in R \,\, &\mbox{if and only if $R \in \astro(p)$} && \mbox{\refdef{Astroid}} \\
            &\mbox{if and only if $\cell(R) \in \axis(p)$} && \mbox{\refdef{AxisForP}} \\
            &\mbox{if and only if $\point(\axis(p)) \in \rect(\cell(R))$} && \mbox{\reflem{AxisCentre}\refitm{TetInAxis}} 
\end{align*}
We deduce that $\eta_\calL(R) = \rect(\cell(R))$ is a loom-tetrahedron rectangle in $\link \circ \veer (\calL)$.
By \refcor{SubBasis}, the map $\eta_\calL$ is a homeomorphism.

By \refclm{Wall}, \reflem{Order}, and \refthm{AxisIFFAstroid} the map $\eta_\calL$ sends leaves to leaves.
\end{proof}

\begin{claim}
Suppose that $\phi \from \calL \to \calM$ is a loom isomorphism.
Then the following diagram commutes.
\[
\begin{tikzcd}[row sep = large, column sep=large]
\calL \arrow[r, "\phi"] \arrow[d, "\eta_\calL"] & \calM \arrow[d, "\eta_\calM"] \\
\link\circ\veer(\calL) \arrow[r, "\link\circ\veer(\phi)"] & \link\circ\veer(\calM)
\end{tikzcd}
\]
\end{claim}

\begin{proof}
Suppose that $p \in \calL$.
We now compute.
\begin{align*}
(\link\circ\veer)_\phi(\eta_\calL(p))
  &= (\link\circ\veer)_\phi(\point(\axis(p))) && \mbox{\refdef{Eta}}  \\ 
  &= \link_{\veer(\phi)}(\point(\axis(p))) && \mbox{$(\link\circ\veer)_\phi = \link(\veer(\phi)) = \link_{\veer(\phi)}$}\\ 
  &= \point(\veer_\phi(\axis(p))) && \mbox{\reflem{ImageOfAxis}}\\ 
  &= \point(\axis(\phi(p))) && \mbox{\refdef{InducedTautIso}}\\ 
  &= \eta_\calM(\phi(p)) && \mbox{\refdef{Eta}}  \qedhere
\end{align*}
\end{proof}
This completes the proof that $\eta$ is a natural isomorphism.
\end{proof}

\begin{definition}
\label{Def:Zeta}
Suppose that $\calV$ is a veering triangulation of $\RR^3$. 
We define a map 
\[
\zeta_\calV \from \calV \to \veer \circ \link (\calV) \quad \mbox{by setting} \quad
\zeta_\calV(t) = \cell(\rect(t)) \qedhere
\]
\end{definition}

\begin{proposition}
\label{Prop:Zeta}
$\zeta$ is a natural isomorphism.
\end{proposition}

\begin{proof}
We verify the defining properties as follows.

\begin{claim}
\label{Clm:ZetaIsTaut}
Suppose that $\calV$ is a veering triangulation of $\RR^3$.
Then the map $\zeta_\calV$ is a taut isomorphism.
\end{claim}

\begin{proof}
By \refcor{Bijections}, the map $\zeta_\calV$ is a bijection.

Suppose that $f$ is a face of $\calV$.
Let $t$ and $t'$ be tetrahedra of $\calV$.
By \reflem{FaceInTet}, we have that $t$ meets $t'$ along $f$ if and only if $\rect(t) \cap \rect(t') = \rect(f)$.
By \cite[Lemma~2.31 and Definition~5.9]{SchleimerSegerman24}, we have that $\rect(t) \cap \rect(t') = \rect(f)$ if and only if $\cell(\rect(t))$ meets $\cell(\rect(t'))$ along the face $\cell(\rect(f))$.
Thus $\zeta_\calV$ is an isomorphism of triangulations.

By \reflem{TetRect}, we have that $e$ is an upper or lower edge of $t$ if and only if $\rect(e)$ is contained in and spans $\rect(t)$ (either south-north or west-east).
By \cite[Definition~5.9]{SchleimerSegerman24} we have that $\rect(e)$ is contained in and spans $\rect(t)$ if and only if $\cell(\rect(e))$ is an upper or lower edge of $\cell(\rect(t))$.
Thus $\zeta_\calV$ sends the taut structure of $\calV$ to that of $\veer(\link(\calV))$.
\end{proof}

\begin{claim}
Suppose that $f \from \calV \to \calW$ is a taut isomorphism.
Then the following diagram commutes.
\[
\begin{tikzcd}[row sep = large, column sep=large]
\calV \arrow[r, "f"] \arrow[d, "\zeta_\calV"] & \calW \arrow[d, "\zeta_\calW"] \\
\veer\circ\link(\calV) \arrow[r, "\veer\circ\link(f)"] & \veer\circ\link(\calW)
\end{tikzcd}
\]
\end{claim}

\begin{proof}
Suppose that $t \in \calV$ is a tetrahedron.
We now compute.
\begin{align*}
(\veer\circ\link)_f(\zeta_\calV(t))
  &= (\veer\circ\link)_f(\cell(\rect(t))) && \mbox{\refdef{Zeta}}  \\ 
  &= \veer_{\link(f)}(\cell(\rect(t))) && \mbox{$(\veer\circ\link)_f = \veer(\link(f)) = \veer_{\link(f)}$}\\ 
  &= \cell(\link_f(\rect(t))) && \mbox{\refdef{InducedTautIso}}\\ 
  &= \cell(\rect(f(t))) && \mbox{\reflem{ImageOfRect}}\\ 
  &= \zeta_\calW(f(t)) && \mbox{\refdef{Zeta}}  \qedhere
\end{align*}
\end{proof}
This completes the proof that $\zeta$ is a natural isomorphism.
\end{proof}

\begin{proof}[Proof of \refthm{Equivalence}]
Propositions~\ref{Prop:Eta} and~\ref{Prop:Zeta} prove that $\eta$ and $\zeta$ are natural isomorphisms. 
Thus the categories $\Veer(\RR^3)$ and $\Loom(\RR^2)$ are equivalent.
\end{proof}

With \refthm{Equivalence} in hand, we prove the following. 

\begin{corollary}
\label{Cor:Recover}
Suppose that $M$ is a connected three-manifold.  
Suppose that $\calV$ is a locally veering triangulation on $M$.
Suppose that $\Gamma = \veer \circ \link (\pi_1(M))$ is the induced subgroup of $\Aut(\veer \circ \link (\cover{\calV}))$.  
Then $\zeta_\cover{\calV}$ descends to give an isomorphism from $\calV$ to $\veer \circ \link(\cover{\calV}) / \Gamma$. 
\end{corollary}

\begin{proof}
By \refclm{ZetaIsTaut}, the map $\zeta_\calV$ is a taut isomorphism from $\cover{\calV}$ to $\veer \circ \link(\cover{\calV})$.
By \refthm{Equivalence}, we have that $\zeta$ carries the $\pi_1(M)$--action on $\cover{\calV}$ to the $\Gamma$--action on $\veer \circ \link(\cover{\calV})$. 
In particular, $\zeta_\cover{\calV}$ takes orbits of cells to orbits of cells.
Thus, the quotient of $\zeta_\cover{\calV}$ is the desired isomorphism between $\calV$ and $\veer \circ \link(\cover{\calV}) / \Gamma$. 
\end{proof}

\appendix
\chapter{Towards Cannon--Thurston maps}
\label{App:CT}

Here we give results required for our work with Manning~\cite{ManningSchleimerSegerman}.

\subsection{Parabolics revisited}
\label{App:ParabolicsRevisited}

Suppose that $M$ is an oriented three-manifold equipped with a transverse veering triangulation $\calV$.
Here we suppose that $\calV$ contains finitely many tetrahedra.
Thus $M$ is a compact three-manifold with non-empty toroidal boundary.
As in \refsec{TautIdealTriangulation}, we take $\Delta_\calV$ to be the collection of vertices of $\cover{\calV}$.
Fixing a cusp $c \in \Delta_\calV$, we call $\Stab(c) < \pi_1(M)$ a \emph{(maximal) parabolic} subgroup.

We continue the analysis of the action of $\Stab(c)$ on the veering circle $\Circle$, as begun in \refsec{CrossingParabolics}.
We begin with a review of the \emph{ladder decomposition} of the \emph{cusp torus} due to Futer--Gu\'eritaud~\cite{FuterGueritaud13}.

\subsection{Ladders}
\label{Sec:Ladders}

Let $L(c)$ be a \emph{vertex link} about $c$ in $\cover{M}$.
That is, if $t$ is a tetrahedron incident to $c$ then $t \cap L(c)$ is a small triangle close to $c$. 
We choose $L(c)$ equivariantly, so it is stabilised by $\Stab(c) \isom \ZZ^2$.
We also take $L(c)$ sufficiently close to $c$ to ensure that $L(c) \subset N^c \cap N_c$, the intersection of the upper and lower cusp neighbourhoods, as defined in \refsec{BranchLines}.

As discussed in \refsec{HorizontalBranchedSurface}, the taut structure makes the two-skeleton of $\cover{\calV}$ into a branched surface, denoted $\cover{B}$. 
Thus the intersection $\tau(c) = L(c) \cap \cover{B}$ is a \emph{bigon train track}.
See \reffig{CuspTorus} for an example of a fundamental domain of the $\Stab(c)$ action on $\tau(c) \subset L(c)$. 
The veering census~\cite{GSS19} contains many more examples. 
Also see 
Figure~3 of~\cite{FuterGueritaud13}, 
Figure~6 of~\cite{Landry18}, and 
Figure~6 of~\cite{Landry22}. 

\begin{figure}[htbp]
\centering
\labellist
\small\hair 2pt
\pinlabel {$Q$} at 10 15
\pinlabel {$S$} at 152 65
\pinlabel {$P$} at 290 15
\pinlabel {$V$} at 440 595
\endlabellist
\includegraphics[width=0.6\textwidth]{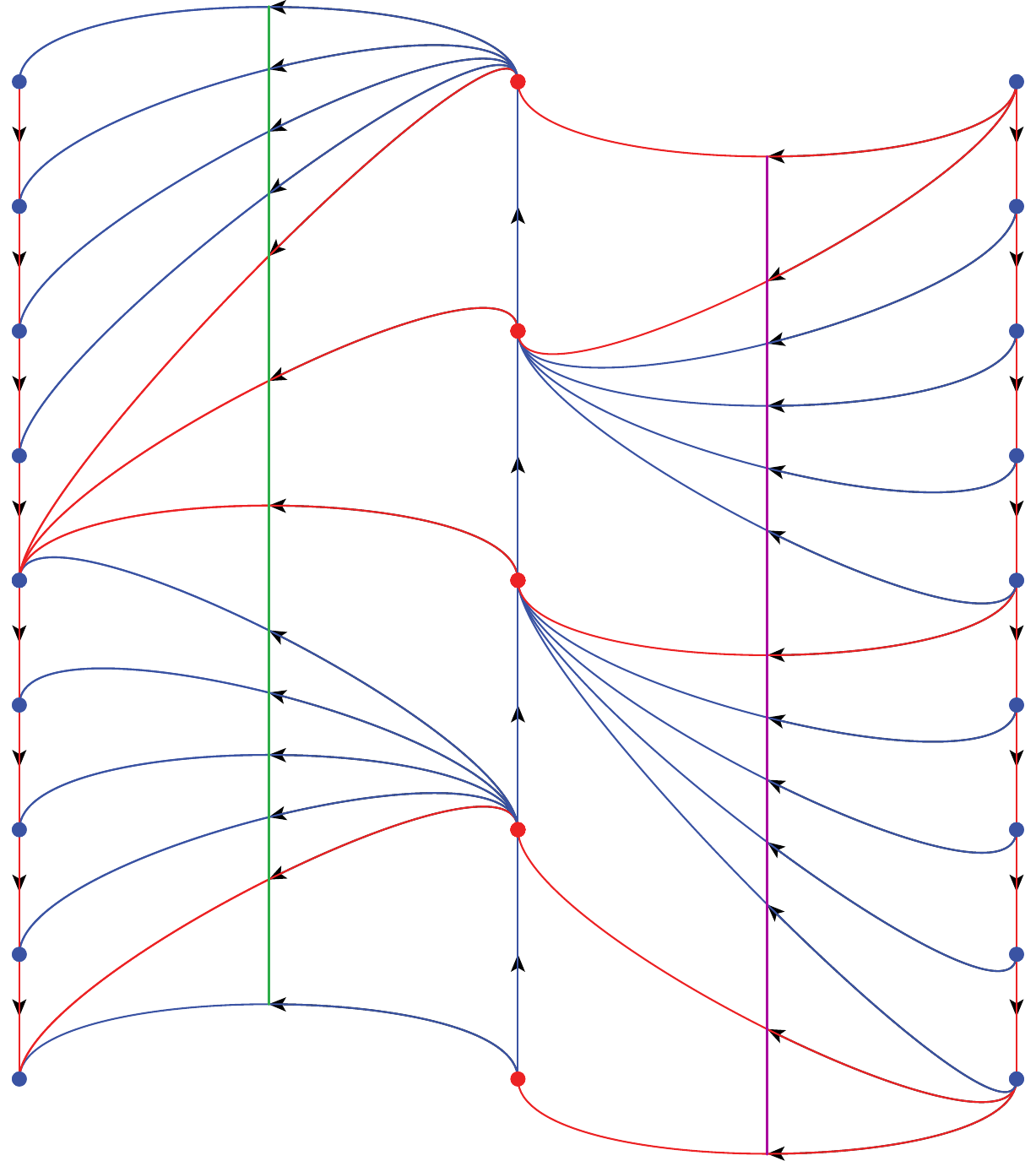}
\caption{
The bigon track $\tau(c)$ in the vertex link $L(c) \subset N(c)$, for the SnapPy manifold \texttt{s227} as determined by the veering triangulation \texttt{gLLAQbecdfffhhnkqnc\_120012} 
from the veering census~\cite{GSS19}.  
Note that we only give a fundamental domain for the action of $\Stab(c)$.  
Here the action of $\Stab(c)$ is generated by two translations parallel to the $x$-- and $y$--axes, respectively.  
The colours and orientations are discussed in the main text.
}
\label{Fig:CuspTorus}
\end{figure}

Every edge $e$ in $\cover{\calV}$ meeting $c$ gives a switch $v$ to the bigon track $\tau(c)$; 
we colour $v$ in $\tau(c)$ with the veering colour of $e$.  
Every face $f$ in $\cover{\calV}$ meeting $c$ gives a branch $b$ to the bigon track $\tau(c)$; 
we colour $b$ with the veering colour of the third edge of $f$ (that is, by \refthm{TetEmbeds}, the edge of $f$ which does not meet $c$). 
Suppose that $v$ and $v'$ are the switches at the ends of $b$.  
If these are \emph{different} colours we call $b$ a \emph{rung branch}.  
If these are the \emph{same} colour we call $b$ a \emph{ladderpole branch}.
Again, see \reffig{CuspTorus} for an example.  

\begin{definition}
\label{Def:Ladders}
By Observation~2.4 of~\cite{FuterGueritaud13}, the union of the ladderpole branches give a collection of parallel loops in the torus $L(c)/\Stab(c)$;
this is the \emph{ladderpole slope}.
Similarly, the union of the ladderpole branches in $L(c)$ cut it into strips called \emph{ladders}.  
See also \refrem{BranchLines}.

Following the above observation, we call the two sides of a ladder its \emph{ladderpoles}.  
We give a ladderpole the colour of its switches.  
\end{definition}

The colouring convention for ladderpoles, in \refdef{Ladders}, is forced by the convention set out in \reffig{VeeringTriangles}.  
Note that, as we scan along $L(c)$ from right to left, we encounter, cyclically, a red ladderpole, then an upper (green) branch line, then a blue ladderpole, and then a lower (purple) branch line.
This convention is the opposite of the one suggested by Lemma~6.5 and Figure~3 of~\cite{FuterGueritaud13}.

\subsection{Orienting the bigon track}

Suppose that $f$ is a face of $\cover{\calV}$ meeting $c$;
let $b$ be the resulting branch of $\tau(c)$.
Let $e$ and $e'$ be the edges of $f$ meeting $c$;
let $v$ and $v'$ be the resulting switches of $\tau(c)$.
We also take $d$ and $d'$ to be the other cusps (not $c$) on $e$ and $e'$ respectively.
Suppose that $e'$ is immediately \emph{before} $e$ in the boundary of $f$ (with orientation induced by the transverse taut structure).
By \refthm{VeerImpliesUnique} the circular order $\calO_\calV$ is compatible;
so $\calO_\calV(c, d, d') = 1$.
We record this in $\tau(c)$ by orienting $b$ from $v$ towards $v'$.
Again, see \reffig{CuspTorus} for an example.
Note that, in the example, all rungs are oriented from right to left, 
all red ladderpoles are oriented upwards, and 
all blue ladderpoles are oriented downwards.
This is generally true; see \reffig{UpperGluingAutomaton}.

With notation as in the previous paragraph, we additionally assume that $e$ is red and $e'$ is blue. 
Let $e''$ be the third edge of the face $f$. 
In this case, the branch $b''$ running from $v$ to $v'$ is a rung. 
Consulting \reffig{VeeringTrianglesUpperTrack}, we deduce that the upper track $\tau^f$ has a switch $s$, associated to $c$, at the midpoint of $e''$.   
The switch $s$ lies on an upper branch line $S \subset \bdy N^c$.  
If we reverse the colours of $e$ and $e'$ then instead a lower branch line $T$ meets the midpoint of $b''$.  
In any case, if $e$ and $e'$ have the same colour then no branch line associated to $c$ meets $b''$. 
From this and \reffig{NormalUpperBranchedSurface} we deduce the following;
see also~\cite[Lemma~2.7]{Landry23}.

\begin{lemma}
\label{Lem:BranchLinesCuspTorus}
There is exactly one branch line running along the centre of every ladder;
the branch line is upper (green) if and only if the right ladderpole is red. \qed
\end{lemma}

It follows that the ladderpole slope and the branch slope $\beta$ are equal (up to sign) in $\Stab(c)$. 

\subsection{The action of the branch loop}

Recall that $\Circle$ is the veering circle.  
Suppose that $c \in \Delta_\calV$ is a cusp and $\beta \in \Stab(c)$ is the branch slope, as defined in \refdef{BranchSlope}.
Recall that $\Lambda^c$ and $\Lambda_c$ are the crowns for $c$.  
We will denote their tips in $\Circle$ by $\bdy \Lambda^c$ and $\bdy \Lambda_c$. 

\begin{lemma}
\label{Lem:BranchLoop}
Suppose that $\beta \in \Stab(c)$ is the branch slope.
Its set of fixed points, as it acts on $\Circle$, is 
\[
\{ c \} \cup \bdy \Lambda^c \cup \bdy \Lambda_c 
\]
The points of $\bdy \Lambda^c$ are attracting while the points of $\bdy \Lambda_c$ are repelling.
\end{lemma}

\begin{proof}
Suppose that $R$ and $S$ are adjacent upper branch lines.
Suppose that $\calO_\calV(c, \bdy R, \bdy S) = 1$. 
By \reflem{CrossingParabolics}\refitm{TipsOfCrown} it suffices to consider the action of $\beta$ on the interval $[\bdy R, \bdy S]^{\acw}$.
By \reflem{CrownsInterleave} there is exactly one lower branch line $V$ with $\bdy V$ in $[\bdy R, \bdy S]^{\acw}$. 

By \reflem{BranchLinesCuspTorus} there are ladders $L(S)$ and $L(V)$ in $L(c)$ containing $S$ and $V$, respectively.  
By our choice of $V$, the ladders $L(S)$ and $L(V)$ share a ladderpole $P$.
Breaking symmetry, we will assume that $P$ is red.
See \reffig{CuspTorus} for such an example; there $S$ is the green branch line, $V$ is the purple branch line, and $P$ is the ladderpole inbetween. 
Let $(d_i)_{i \in \ZZ} \subset \Delta_\calV$ be the cusps corresponding to the switches of $P$.  
We arrange matters so that $d_i$ and $d_{i+1}$ share a face with $c$ and so that $\calO_\calV(c, d_i, d_{i+1}) = 1$. 
Again, see \reffig{CuspTorus}.

Let $e_i$ be any rung meeting, and to the left of, $d_i$.
Thus $e_i$ crosses $S$.
Note that the orientations on the ladderpoles show that the endpoints of $A(e_{i+1})$, and thus $A(e_{i+1})$ itself, are contained in $A(e_i)$.  
From \reflem{BranchLinesShrink} and \refdef{VeeringCircle} we deduce that $\bigcap A(e_i) = \{ \bdy S \}$.  
Thus the cusps $d_i$ converge to $\bdy S$ (in $\Circle$) as $i$ tends to infinity.  
(Also, the cusps $d_i$ converge to $\bdy V$ as $i$ tends to negative infinity.)
We can repeat this argument for the blue ladderpole $Q$ on the other side of $S$.

Let $k$ be the number of switches on the ladderloop $P / \beta$.  
We deduce that $\beta(d_i) = d_{i+k}$, for all $i$.  
A similar statement holds for $Q$.
Thus $\bdy S$ is an attracting fixed point of $\beta$ as it acts on $\Circle$. 
A similar argument shows that $\bdy T$ is a repelling fixed point. 
Finally, $\beta$ sends $[d_{i - k}, d_i]^{\acw}$ to $[d_i, d_{i + k}]^{\acw}$; thus $\beta$ has no fixed points in the interval $(\bdy T, \bdy S)^{\acw}$.
\end{proof}

\subsection{Conical limit points}
\label{App:Conical}

We will also need delicate control of neighbourhoods of \emph{singletons} in $\Circle$ and leaves in $\Mobius$.

\begin{definition}
A point $p \in \Circle$ is a \emph{singleton} if
\begin{itemize}
\item
it is not a cusp and
\item
it is not an endpoint of a leaf of $\Lambda^\calV$ or of $\Lambda_\calV$. \qedhere
\end{itemize}
\end{definition}

Suppose that $e$ is a co-oriented edge of $\cover{\calV}$.
Recall that $A(e)$ is the closed arc of $\Circle$ between the endpoints of $e$, which is pointed at by the co-orientation. See \reffig{CoorientedEdge}. 

\begin{lemma}
\label{Lem:Singleton}
Let $x \in \Circle$ be a singleton.  
Then there are two co-oriented edges $e$ and $e'$ in $\cover{\calV}$ and a sequence 
$(\gamma_i) \subset \pi_1(M)$ with the following properties.
\begin{enumerate}
\item
$A(e')$ is contained in the interior of $A(e)$,
\item
$A(\gamma_{i+1}(e)) \subset A(\gamma_i(e'))$, for all $i$, and
\item
$\bigcap A(\gamma_i(e)) = \{ x \}$.
\end{enumerate}
\end{lemma}


\begin{proof}
Let $K$ be a layer of a layering.  
We co-orient every edge $e$ of $K$ so that the co-orientation points at $x$. 
Fix any edge $e_0$ of $K$.  
Let $P$ be the strip of triangles in $K$ so that every interior edge of $P$ separates $e_0$ from $x$. 

By \reflem{AlternatingPathsShrink},
since $x$ is not a cusp the strip $P$ turns right and left infinitely often as it travels away from $e_0$. 
As in \refrem{NeighbourhoodBasis}, the edges of $P$ give a neighbourhood basis for $x$ in $\Circle$. 

\begin{figure}[htbp]
\centering
\labellist
\small\hair 2pt
\pinlabel {$e_0$} [r] at 6 64
\pinlabel {$e_n$} [l] at 263.5 70
\pinlabel {$e''_n$} [tl] at 317 70
\pinlabel {$e'_n$} [l] at 374 50
\pinlabel {$x$} [l] at 620 64
\endlabellist
\includegraphics[width=0.8\textwidth]{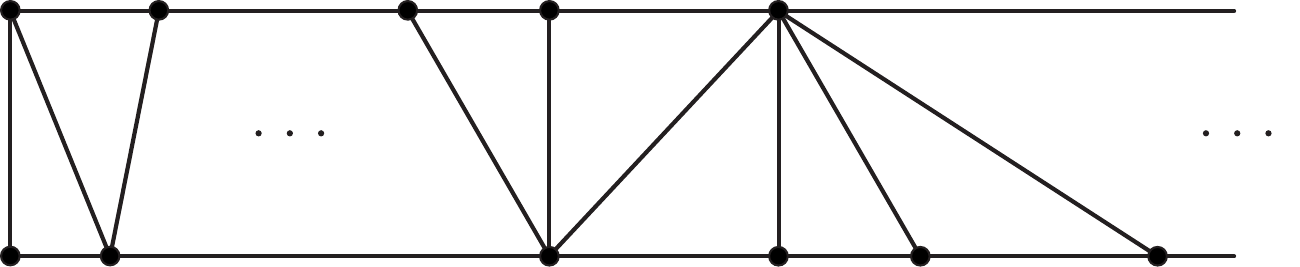}
\caption{The strip $P$ turns right and then left immediately before and after, respectively, the edge $e''_n$.}
\label{Fig:Turning}
\end{figure}

Consider the (infinitely many) edges $e''_n$ of $P$ where $P$ turns right in the triangle before $e''_n$ and left in the triangle after $e''_n$.  
Let $e_n$ and $e'_n$ be the edges immediately before and after $e''_n$ along $P$.  
Note that $e_n$ and $e'_n$ share no cusps.  
See \reffig{Turning}.
Note that, up to the action of $\pi_1(M)$, there are only finitely many orbits of such pairs of edges.  
Thus there is at least one such orbit which is represented infinitely often in $P$.
Any such orbit gives us a subsequence of pairs of edges with the desired properties. 
\end{proof}

We now turn our attention to leaves.
Here we require neighbourhood bases in the M\"obius band past infinity, $\Mobius$.
See \refdef{UpperLamination}.

\begin{definition}
Suppose that $e$ is a red edge of $\cover{\calV}$. 
Suppose that $c$ and $d$ are the cusps meeting $e$, with $c$ being the southwest corner of $\rect(e)$. 
Following the notation of \reffig{LeafIdentificationsEdge} we choose $\lambda(S, S')$ and $\lambda(T, T')$ to be the boundary leaves, 
in \refdef{UpperFoliation}, giving rise to the upper cusp leaves $\ell^S$ and $\ell^T$. 
These, in turn, contain the west and east sides of $\rect(e)$.    

We define the \emph{upper product neighbourhood} $P^\calV(e)$ to be the image of $[\bdy S', \bdy T]^{\acw} \cross [\bdy T', \bdy S]^{\acw}$ inside of $\Mobius$. 
We note that this is a closed neighbourhood for any leaf of $F^\calV$ meeting the interior of $\rect(e)$. 
We define the \emph{lower product neighbourhood} $P_\calV(e)$ similarly, using $F_\calV$.

We make similar definitions in the case that $e$ is blue. 
\end{definition}

\begin{lemma}
\label{Lem:Leaf}
Suppose that $\lambda$ is an interior leaf of $\Lambda^\calV$.  
Let $\ell = \ell^\lambda$ be the corresponding non-cusp leaf of $F^\calV$.
Then there are two edges $e$ and $e'$ in $\cover{\calV}$ and a sequence $( \gamma_i ) \subset \pi_1(M)$ as follows:
\begin{enumerate}
\item
\label{Itm:NoCusp}
$e'$ and $e$ have no cusps in common,
\item
\label{Itm:Spans}
$\rect(e')$ south-north spans $\rect(e)$, and
\item 
\label{Itm:Basis}
the products $P^\calV(\gamma_i(e'))$ give a neighbourhood basis for $\lambda$ in $\Mobius$, as do the products $P^\calV(\gamma_i(e))$.
\end{enumerate}
A similar statement holds for an interior leaf $\mu$ of $\Lambda_\calV$ replacing $\lambda$.
\end{lemma}

\begin{proof}
In order to prove the lemma, we will need the following.

\begin{claim}
\label{Clm:Leaf}
Suppose that $e''$ is an edge of $\cover{\calV}$ so that $\ell$ meets $\rect(e'')$.
Then there are edges $e$ and $e'$ such that
\begin{enumerate}[label=(\roman*)]
\item
\label{Itm:ClmNoCusp}
$e'$ and $e$ have no cusps in common,
\item
\label{Itm:ClmSpans}
$\ell$ meets $\rect(e')$, which south-north spans $\rect(e)$, which south-north spans $\rect(e'')$, and
\item 
\label{Itm:ClmFaces}
There are faces $f$ and $f'$, sharing an edge, such that $e$ lies in $f$ and $e'$ lies in $f'$. 
\end{enumerate}
\end{claim}

This implies \reflem{Leaf} as follows.
Suppose that $K$ is a layer of a layering. 
By \reflem{EdgeRect}\refitm{EdgeRectBdy} the boundary of any edge rectangle is contained in a union of cusp leaves.  
Thus, by \reflem{EdgesCover} there is some edge $e''_0$ in $K$ so that $\rect(e''_0)$ meets $\ell$.
\refclm{Leaf} now gives us $e_0$ and $e'_0$.
Given $e_n$ and $e'_n$, we set $e''_{n+1} = e'_n$ and appeal to \refclm{Leaf} to obtain $e_{n+1}$ and $e'_{n+1}$. 

With the sequence of pairs of edges $(e_n, e'_n)$ now in hand, we proceed as follows.
Conclusions \ref{Itm:ClmNoCusp} and \ref{Itm:ClmSpans} of the claim imply that $e_n$ and $e'_n$ satisfy conclusions \refitm{NoCusp} and \refitm{Spans} of the lemma.
Also, each product $P^\calV(e'_n)$ is a neighbourhood of $\ell$ in $\Mobius$, 
and if $m<n$ then $P^\calV(e'_n)$ is contained in the interior of $P^\calV(e'_m)$. 

Suppose, for a contradiction, that the sequence $(P^\calV(e'_n))$ is not a neighbourhood basis of $\ell$.
We deduce that $\limsup \rect(e'_n)$ contains $\ell$ as a proper subset.
Let $t_n$ be any tetrahedron that contains $e'_n$.
Thus $\limsup \rect(t_n)$ contains $\ell$ as a proper subset. 
Appealing to \refthm{LinkIsLoom}, we obtain a contradiction to \reflem{Finiteness}.
Thus the sequence $(P^\calV(e'_n))$ is a neighbourhood basis of $\ell$.

By conclusion \ref{Itm:ClmFaces} of the claim, there are only finitely many orbits (under the action of $\pi_1(M)$) of such pairs of edges.
We pass to an infinite subsequence consisting of a single orbit and reindex.
This gives the sequence $(\gamma_i) \subset \pi_1(M)$. 
We set $\gamma_0=1$.
We set $e=e_0$ and $e'=e'_0$.
By the previous paragraph, the sequence $(P^\calV(\gamma_i(e')))$ is a neighbourhood basis of $\ell$.
The same argument applies to $(P^\calV(\gamma_i(e)))$.
This gives conclusion \refitm{Basis} of the lemma. 

\begin{figure}[htbp]
\centering
\subfloat[Fan tetrahedron.]{
\labellist
\small\hair 2pt
\pinlabel {$e$} [b] at 57 83
\pinlabel {$\eu$} [l] at 124 180
\pinlabel {$\en$} [br] at 86 164
\pinlabel {$\es$} [tl] at 135 100
\pinlabel {$\ell$} [r] at 411 269
\pinlabel {$c$} [b] at 206 163
\endlabellist
\includegraphics[width=0.45\textwidth]{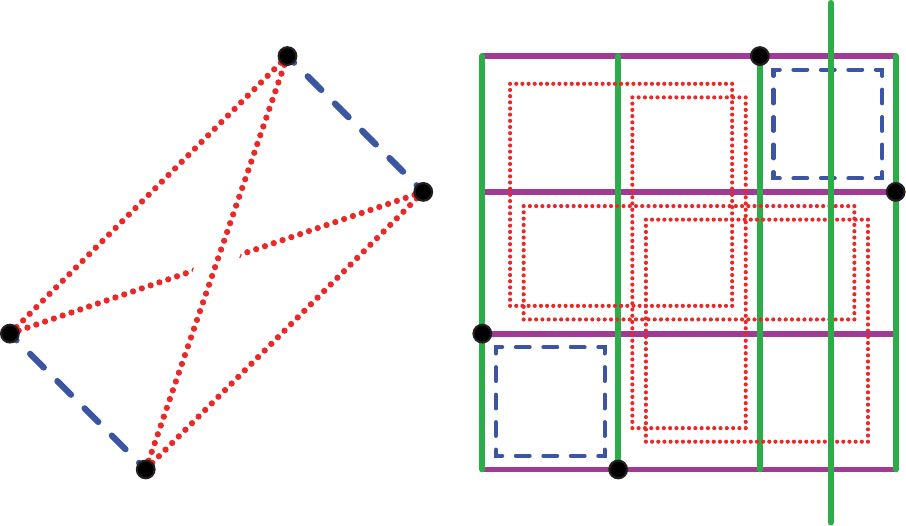}
\label{Fig:WhereIsEllFan}
}
\quad
\subfloat[Toggle.]{
\labellist
\small\hair 2pt
\pinlabel {$e$} [b] at 58 85
\pinlabel {$\eu$} [l] at 92 172
\pinlabel {$\en$} [br] at 42 155
\pinlabel {$\es$} [tl] at 172 104
\pinlabel {$\ell$} [r] at 411 269
\pinlabel {$c$} [b] at 206 163
\endlabellist
\includegraphics[width=0.45\textwidth]{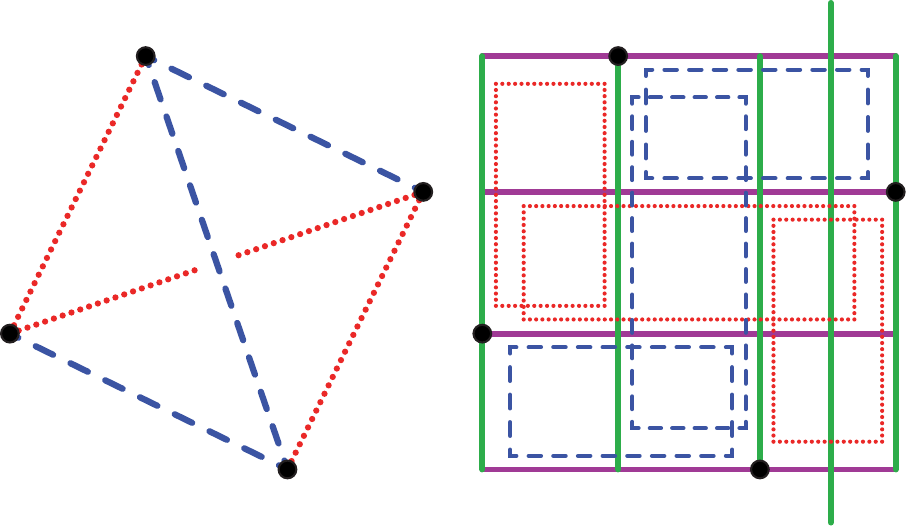}
\label{Fig:WhereIsEllToggle}
}
\caption{}
\label{Fig:WhereIsEll}
\end{figure}

\begin{proof}[Proof of \refclm{Leaf}]
As in \refsec{CardinalDirections}, we choose orientations on the foliations $F^\calV$ and $F_\calV$.
We use the following notation.  
Suppose that $e$ is any red edge of $\cover{\calV}$.  
Let $t$ be the tetrahedron immediately above $e$.  
See \reffig{WhereIsEll}.
We use $\eu$ to denote the upper edge of $t$. (The superscript indicates a cone pointing out of the page.)
We use $\es$ and $\en$ to denote the red equatorial edges of $t$ which meet the southern and northern cusps of $e$, respectively. 
(The superscripts indicate a cone pointing south or north.)

Breaking symmetry, we suppose that the given edge $e''$ is red.
Set $e_0 = e''$.
If $\ell$ meets $\rect(\eu_0)$, then we take $e = e_0$, we take $e'=\eu_0$, and we are done. 

Otherwise, let $t_0$ be the tetrahedron immediately above $e_0$.  
Let $\es_0$ and $\en_0$ be the red equatorial edges of $t_0$. 
By \reflem{FaceRect}\refitm{FaceSpan}, both $\rect(\es_0)$ and $\rect(\en_0)$ south-north span $\rect(e_0)$.
See \reffig{WhereIsEll}. 
Since $\ell$ does not meet $\rect(\eu_0)$, it must meet exactly one of $\rect(\es_0)$ or $\rect(\en_0)$.
Breaking symmetry again, assume that $\ell$ meets $\rect(\es_0)$.
Let $c$ be the cusp shared by $e_0$ and $\es_0$.
Again, see \reffig{WhereIsEll}. 

By induction, at stage $k$ we now have the following.
\begin{enumerate}
\item[($\textrm{A}_k$)]
$e_k$ is red, 
\item[($\textrm{B}_k$)]
$e_k$ and $\es_k$ meet $c$, 
\item[($\textrm{C}_k$)]
$\ell$ meets $\rect(e_k)$,
\item[($\textrm{D}_k$)]
$\ell$ does not meet $\rect(\eu_k)$, and
\item[($\textrm{E}_k$)]
$\ell$ meets $\rect(\es_k)$.
\end{enumerate}
We must now either produce $e$ and $e'$ fulfilling the conclusion of \refclm{Leaf} or extend the induction to stage $k+1$.

Let $t_k$ be the tetrahedron immediately above $e_k$.
By definition, $\es_k$ is red. 
Set $e_{k+1} = \es_k$, and note that we have ($\textrm{A}_{k+1}$). 
By definition, the edges $e_{k+1}$ and $\es_{k+1}$ have the same north-east cusp. 
By ($\textrm{B}_k$), this is $c$, verifying ($\textrm{B}_{k+1}$).
By ($\textrm{E}_k$), the leaf $\ell$ meets $\rect(e_{k+1})$, giving ($\textrm{C}_{k+1}$).

Let $t_{k+1}$ be the tetrahedron immediately above $e_{k+1}$.
If $\ell$ meets $\rect(\eu_{k+1})$ then we take $e = e_{k+1}$, we take $e' = \eu_{k+1}$, and \refclm{Leaf} holds.
If not, then ($\textrm{D}_{k+1}$) holds. 

\begin{figure}[htbp]
\centering
\labellist
\small\hair 2pt
\pinlabel {$\en_k$} [r] at 45 246
\pinlabel {$e_k$} [b] at 82 172
\pinlabel {$\eu_k$} [l] at 126 250
\pinlabel {$\en_{k+1}$} [br] at 308 370
\pinlabel {$\es_k = e_{k+1}$} [tl] at 332 155
\pinlabel {$\ell$} [r] at 807 540
\pinlabel {$c$} [l] at 437 243
\endlabellist
\includegraphics[width=0.85\textwidth]{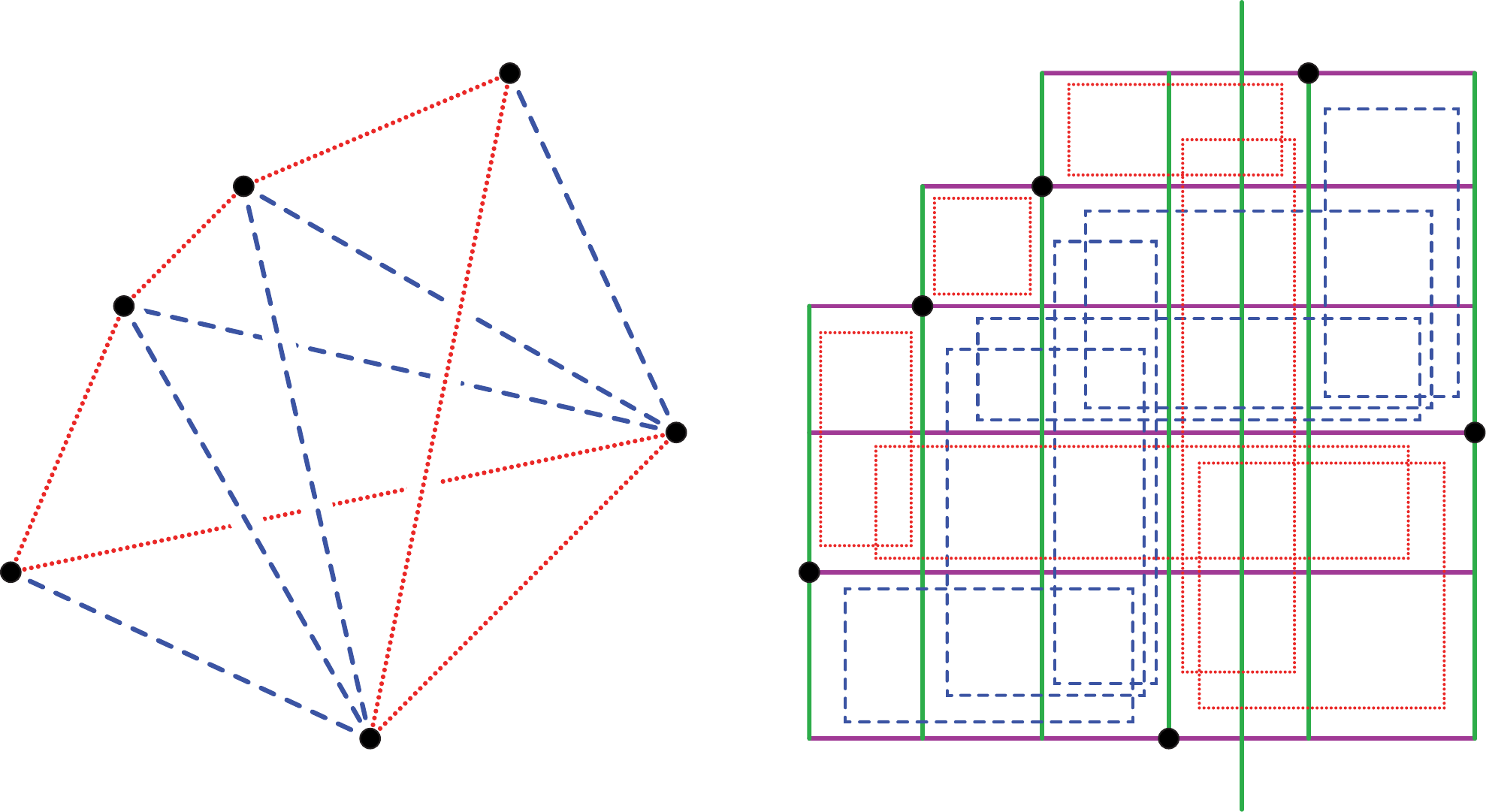}
\caption{}
\label{Fig:WhereIsEllToggleNext}
\end{figure}

There are two cases, as $t_k$ (not $t_{k+1}$) is a fan tetrahedron or a toggle.

\begin{itemize}
\item
Suppose that $t_k$ is a fan tetrahedron.
See \reffig{WhereIsEllFan}.
Then the common face of $t_k$ and $t_{k+1}$ contains the edges $\es_k = e_{k+1}$ and $\eu_k = \en_{k+1}$.
By ($\textrm{D}_{k}$) we have that $\ell$ does not meet $\rect(\eu_{k}) = \rect(\en_{k+1})$. 
Since $\ell$ does not meet $\rect(\eu_{k+1})$ (by ($\textrm{D}_{k+1}$) above) and it does not meet $\rect(\en_{k+1})$, it must meet  $\rect(\es_{k+1})$, giving ($\textrm{E}_{k+1}$).
 
 \item
Suppose that $t_k$ is a toggle.
See \reffig{WhereIsEllToggleNext}.
If $\ell$ meets $\rect(\en_{k+1})$ then we take $e = e_k$, we take $e' = \en_{k+1}$, and \refclm{Leaf} holds. 
Suppose instead that $\ell$ does not meet $\rect(\en_{k+1})$.
By ($\textrm{D}_{k+1}$), the leaf $\ell$ does not meet $\rect(\eu_{k+1})$ either. 
Thus it must meet $\rect(\es_{k+1})$, giving ($\textrm{E}_{k+1}$).
\end{itemize}

Suppose that the induction never terminates. 
Then we have found an infinite sequence of distinct edges $e_k$, all meeting $c$, 
with $\ell$ meeting each $\rect(e_k)$, which south-north spans $\rect(e_0)$. 
Let $q$ be the intersection of $\ell$ with the south side of $\rect(e_0)$.
Let $R$ be the rectangle with south-west corner at $q$ and north-east corner at $c$.
Note that $R$ is contained in $\rect(e_k)$ for all $k$. 
Since every edge rectangle is contained in a tetrahedron rectangle, and each tetrahedron rectangle contains only six edge rectangles, this contradicts \reflem{Finiteness}.

Thus the induction does terminate, giving \refclm{Leaf}.  
\end{proof}
With \refclm{Leaf} in hand, \reflem{Leaf} follows.
\end{proof}

\chapter{Normal paths are essential}
\label{App:CarriedArcs}

Here we sketch a proof of \refthm{NormalPathsEssential}, a version of~\cite[Theorem~5.1]{SchleimerSegerman20}
that applies to proper paths instead of loops.  
Many of the definitions, statements, and proofs are unchanged;
we indicate only the necessary changes. 
Let $M$ be a connected three-manifold equipped with a taut ideal triangulation $\calT$.
Since we are dealing with properly immersed paths we use \emph{truncated} models -- 
that is, we have removed a small open neighbourhood of all model vertices.
For a picture of a truncated tetrahedron, see \reffig{RaisedArcsTrunc}.
Note that truncated faces are hexagons with three sides in $\bdy M$. 

\begin{figure}[htbp]
\labellist
\small\hair 2pt
 \pinlabel {$\textsc{d}$} [br] at 153 198
 \pinlabel {$\textsc{e}_1$} [br] at 220 187
 \pinlabel {$\textsc{f}$} [b] at 304 131
 \pinlabel {$\textsc{e}_2$} [br] at 246 52
\endlabellist
\includegraphics[height = 4.5 cm]{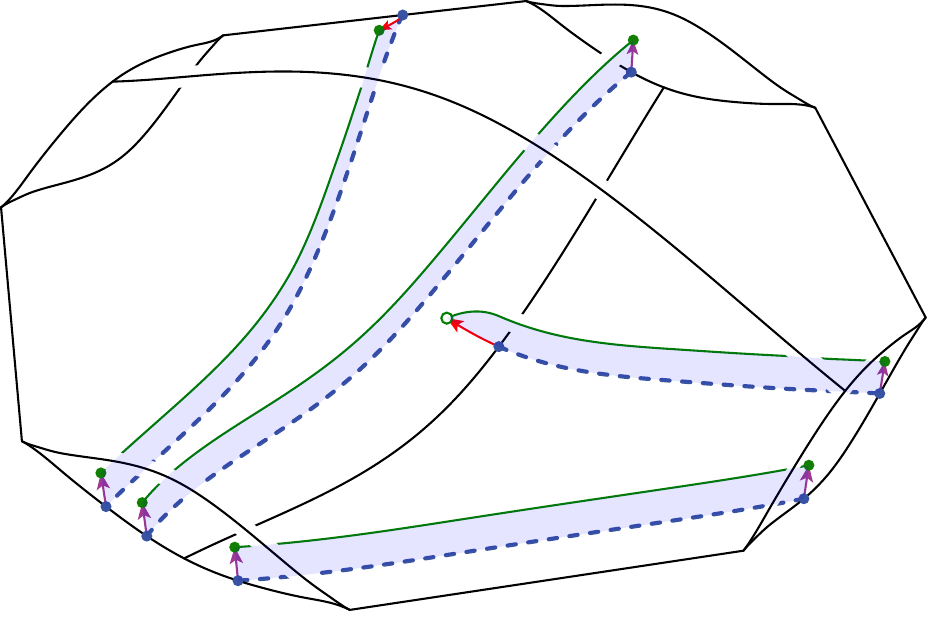}
\caption{A truncated taut tetrahedron containing the four additional types of raised arcs of $\delta$ coming from normal arcs.  Compare with~\cite[Figure~3]{SchleimerSegerman20}.}
\label{Fig:RaisedArcsTrunc}
\end{figure}

\subsection{Normal paths}

We adapt~\cite[Definition~2.2]{SchleimerSegerman20}
to properly immersed paths. 

\begin{definition}
\label{Def:NormalPath}
Suppose that $\gamma$ is a smooth path properly immersed in the horizontal branched surface $B = B(\calT)$ and transverse to the edges of $B$. 
Suppose that for every truncated model face $f$ of $B$ and for every component $J$ of $\gamma^{-1}(f)$ the arc $\gamma|J$ connects distinct, non-adjacent sides of (the hexagon) $f$.  
Then we say that the path $\gamma$ is \emph{normal} with respect to $B$. 
\end{definition}

For examples, see the dashed arcs of \reffig{RaisedArcsTrunc}.

\subsection{Peripheral homotopies}

Let $D = D^2 \subset \CC$ be the closed unit disk. 
We partition $\bdy D = b \cup c$ with $b$ lying in the left half-plane and $c$ lying in the right. 
Suppose that $\gamma \from c \to M$ is a proper path, transverse to $B$.   
We say that $\gamma$ is \emph{peripherally homotopic} if there is a map $H \from D \to M$ so that $H|c = \gamma$ and $H(b) \subset \bdy M$.
If there is no such $H$ we call $\gamma$ \emph{essential}.

Given a peripheral homotopy $H$, we may homotope it, keeping the image of $b$ in $\bdy M$ and not altering $H|c$, to ensure that $H$ is transverse to $\gamma$ and to $B$.
This done we set $\tau = H^{-1}(B)$. 
In \reffig{DiskIndicesTrunc} we give pictures of, and names to, the disk regions of $D - \tau$ with positive index which do not already appear in~\cite[Table~1]{SchleimerSegerman20}.
 
\def\m{0.16}
\begin{figure}[htbp]
\centering
\subfloat[Index $1/2$.]{
\includegraphics[width=\m\textwidth]{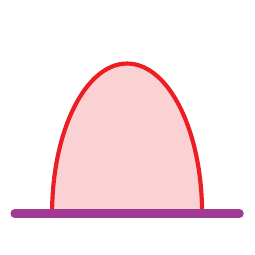}
}
\qquad
\subfloat[Index $1/2$.]{
\includegraphics[width=\m\textwidth]{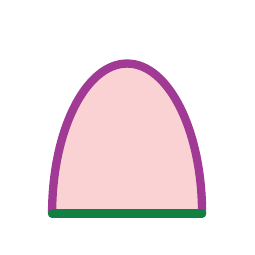} 
}
\qquad
\subfloat[Index $1/4$.]{
\includegraphics[width=\m\textwidth]{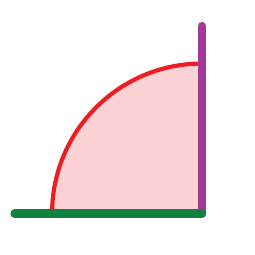}
}
\caption{Disk regions (which are not already listed in~\cite[Table~1]{SchleimerSegerman20})
with positive index arising from proper paths.
From left to right these are 
a \emph{peripheral bigon}, 
a \emph{complete bigon}, and 
a \emph{peripheral trigon}. 
We draw arcs of $b$ in purple and arcs of $c$ in green.}
\label{Fig:DiskIndicesTrunc}
\end{figure}
 
Suppose that $H \from D \to M$ is a peripheral homotopy, as above.  The sum of the indices of the regions of $D - \tau$ is $1/2$; to see this, note that there are two corners at the intersection of $b$ and $c$. 

The proof of~\cite[Lemma~3.1]{SchleimerSegerman20}
goes through without changes; 
now the possible regions of $D - \tau$ with positive index are a single complete bigon, or a non-empty collection of boundary bigons and peripheral trigons (shown in~\cite[Table~1]{SchleimerSegerman20}
and \reffig{DiskIndicesTrunc}).

\subsection{New results}

Here is the promised version of~\cite[Theorem~5.1]{SchleimerSegerman20}
for normal paths.

\begin{theorem}
\label{Thm:NormalPathsEssential}
Let $M$ be a three-manifold equipped with a taut ideal triangulation $\calT$.  
Let $B = B(\calT)$ be the resulting horizontal branched surface in $M$.  
Any path $\gamma$ in $M$, which is normal with respect to $B$, is essential.
\end{theorem}

With $M$, $\calT$, and $B$ as above, let $\cover{B}$ be the lift of $B$ to $\cover{M}$. 
We then have the following.

\begin{corollary}
\label{Cor:VerticesDistinct}
Suppose that $F$ is a connected surface (perhaps with boundary) carried by $\cover{B}$ and realised as a union of faces of $\cover{B}$.  
Then distinct vertices of $F$ are distinct vertices of $\cover{B}$. \qed
\end{corollary}


\subsection{Minimal homotopies}

We define minimal peripheral homotopies as before: 
that is, the number $r(H)$ of regions of $D - \tau$ is minimised across all peripheral homotopies.
We alter the statement of~\cite[Lemma~3.1]{SchleimerSegerman20}, 
taking $\delta$ to be a proper path, and taking $H$ to be a minimal peripheral homotopy.
The conclusions are unchanged except for the final sentence.
Now the allowed regions of positive index are boundary bigons, complete bigons, and
peripheral trigons.
The proof that peripheral bigons do not appear in minimal peripheral homotopies is very similar to that of~\cite[Lemma~3.1(2)]{SchleimerSegerman20}.
That is, peripheral bigons give non-normal branches of $H(\tau)$ in some truncated face and these can be used to reduce $r(H)$. 

We now indicate how to extend the proof of~\cite[Theorem~5.1]{SchleimerSegerman20} to include normal paths. 

\begin{proof}[Sketch proof of \refthm{NormalPathsEssential}]
The model annulus $A$ in the proof of~\cite[Theorem~5.1]{SchleimerSegerman20}
is replaced with a model rectangle.
The homotopy $G$ is proper, so sends the vertical boundary of $A$ into $\bdy M$.
There are now an additional four types of raised arc; see \reffig{RaisedArcsTrunc}.

Suppose that $H \from D \to M$ is a minimal peripheral homotopy of $\delta$.
Recall that $D$ has index $1/2$.
Applying the modification of~\cite[Lemma~3.1]{SchleimerSegerman20}
gives us that $D - \tau$ has either
at least one boundary bigon, or
at least two peripheral trigons.
Using the transverse orientation, every region with positive index is either a max-bigon, max-trigon, min-bigon, or min-trigon.
The discussion of~\cite[Section~5.2]{SchleimerSegerman20}
applies to min-trigons just as it does to min-bigons; see \reffig{PushOverMinTrigon}. 

\begin{figure}[htbp]
\includegraphics[width=0.7\textwidth]{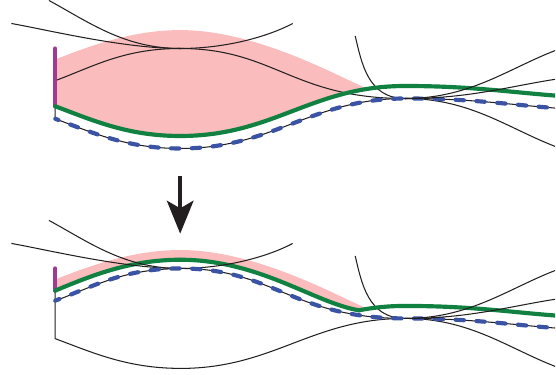}
\caption{Pushing over a min-peripheral trigon, necessarily of type $\textsc{d}$ (see \reffig{RaisedArcsTrunc}).
Compare with~\cite[Figure 4]{SchleimerSegerman20}.}
\label{Fig:PushOverMinTrigon}
\end{figure}

Now~\cite[Definition~5.6]{SchleimerSegerman20}
becomes a definition of \emph{properly transverse} homotopies: that is, these homotopies must also respect the cellulation of $\bdy M$.

Furthermore~\cite[Section~5.3]{SchleimerSegerman20}
applies to max-bigons of the peripheral homotopy as before, except that the region $R_N$ may meet $\bdy M$.
Each max-trigon $R_0$ is either \emph{left} or \emph{right} depending on whether it contains the final, or initial, point of $\delta$ as equipped with the (tangential) orientation it receives from $D$.
Let $d_0 = \bdy R_0 - (s \cup \bdy M)$. 
In the first conclusion of~\cite[Claim~5.4]{SchleimerSegerman20},
the raised arc $H(d_0)$ has type $\textsc{c}$ or $\textsc{f}$, as $R_0$ is a bigon or trigon respectively.

Given a left max-trigon $R_0$, the argument of~\cite[Section~5.3]{SchleimerSegerman20}
extends it to the right to obtain $B_\triangleright$.
Similarly, if $R_0$ is a right max-trigon, the argument extends it to the left to obtain $B_\triangleleft$.
Finally, $S(R_0)$ is either $B_\triangleright \cup Q_\triangleright$ or $B_\triangleleft \cup Q_\triangleleft$. 
Now~\cite[Claim~5.8]{SchleimerSegerman20}
applies to any pair of regions of positive index.

The remainder of the argument is unchanged.
\end{proof}

\chapter{Notation}
\label{App:Notation}

We list some of the notation used throughout the paper. 

\newcolumntype{L}[1]{>{\raggedright\let\newline\\\arraybackslash\hspace{0pt}}p{#1}}
\newcolumntype{C}[1]{>{\centering\let\newline\\\arraybackslash\hspace{0pt}}p{#1}}
\newcolumntype{R}[1]{>{\raggedleft\let\newline\\\arraybackslash\hspace{0pt}}p{#1}}

\begin{longtable}{L{0.17\textwidth}L{0.82\textwidth}}
\renewcommand{\arraystretch}{1.3}$M$ 
& tame connected three-manifold \\ 

$\calT$ & \refsec{TautIdealTriangulation} -- taut ideal triangulation of $M$ \\

$\cover{M}$ & universal cover \\

$\cover{\calT}$ & taut ideal triangulation of $\cover{M}$ \\

$e, f, t$ & edges, faces, tetrahedra of $\cover{\calT}$ \\

$\Delta_\calT$ &  \refsec{TautIdealTriangulation} -- cusps of $\cover{\calT}$ \\

$a, b, c, d$ & cusps of $\cover{\calT}$ \\

$B(\calT)$ & \refsec{HorizontalBranchedSurface} -- horizontal branched surface \\

$\cover{B}$ & preimage of $B(\calT)$ in $\cover{M}$ \\

$B^{(k)}$ & $k$--skeleton \\

$\calO$ & \refdef{CircularOrder} -- circular order \\

$\calK = ( K_i )$ & \refdef{Layered} -- layering of $\cover{\calT}$ \\

$\rho, \sigma$ & \refsec{TrainTracks} -- train routes \\

$s, t, u, v$ & \refdef{TrackCusp} -- upper and lower track-cusps \\

$\tau^L, \tau_L$ & \refdef{UpperTrack} -- upper and lower tracks of a landscape $L$ \\

$( C_i )$ & \refdef{ContinentalExhaustion} -- continental exhaustion of $\cover{\calT}$ \\
 
$\calV$ & \refdef{Veering} -- veering triangulation \\
 
$B^\calV, B_\calV$ & \refsec{UpperLowerSurfaces} -- upper and lower branched surfaces \\

$\cover{B}^\calV$ & preimage of $B^\calV$ in $\cover{M}$ \\

$S, T, U, V$ & \refsec{BranchLines} -- upper or lower branch lines of $\cover{B}^\calV$ and $\cover{B}_\calV$ \\

$\Delta(e)$ & \refdef{Arcs} -- arc in $\Delta_\calV$ \\

$A(e)$ & \refsec{BuildingVeeringCircle} -- arc in $\Circle$ \\

$\bdy \rho, \bdy S$ & \refsec{BuildingVeeringCircle} -- endpoint at infinity \\

$[a, b]_\Delta^{\acw}$ & \refdef{Arcs} -- cusps between $a$ and $b$, anticlockwise of $a$ \\

$\Circle$ & \refdef{VeeringCircle} -- veering circle \\

$[x, y]^{\acw}$ & \refdef{CircleArc} -- arc between $x$ and $y$, anticlockwise of $x$ \\

$\lambda = \lambda(a, b)$ & leaf in $\Circle$ (that is, a point of $\Mobius$) \\

$\Lambda^\calV, \Lambda_\calV$ & \refcha{LaminationsAlone} -- upper and lower laminations in $\Circle$ \\

$\lambda, \mu$ & upper and lower leaves of the laminations \\

$\Mobius$ & \refsec{Mobius} -- M\"obius band for $\Circle$ \\

$\lambda(c, S)$ & \refsec{ConnectingArc} -- cusp leaf\\

$\ell^K(c,S)$ & \refdef{CuspLine} -- cusp (train) line \\

$\lambda(S, T)$ & \refdef{BoundaryLeaf} -- boundary leaf \\

$\ell^K(S,T)$ & \refdef{BoundaryLeaf} -- boundary (train) line \\ 

$\Lambda^{(c, d)}, \Lambda_{(c,d)}$ & \refdef{LeavesThatLink} -- upper or lower leaves separating $c$ and $d$ \\

$<^{(c, d)}, <_{(c, d)}$ & \refdef{LeavesThatLink} -- linear orders \\

$\Sigma^\calV, \Sigma_\calV$ & \refsec{SuspendingDescending} -- upper and lower laminations in $M$ \\

$\Lambda^c, \Lambda_c$ & \refdef{Crown} -- upper and lower crowns \\

$\pair(\calV)$ & \refdef{PairSpace} -- pair space \\

$(\lambda, \mu)$ & element of the pair space \\

$\link(\calV)$ & \refdef{LinkSpace} -- link space \\

$[(\lambda, \mu)]$ & element of the link space \\

$F^\calV, F_\calV$ & \refdef{UpperFoliation} -- upper and lower foliations \\

$\ell, m$ & upper and lower leaves of the foliations \\

$\ell^S, \ell_U$ & upper and lower cusp leaves of the foliations \\

$\rect(e)$ & \refdef{EdgeRect} -- edge rectangle \\

$\rect(f)$ & \refdef{FaceRect} -- face rectangle \\

$\rect(t)$ & \refdef{TetRect} -- tetrahedron rectangle \\

$\calL$ & \refdef{Loom} -- loom space \\

$\veer$ & \refdef{InducedVeering} -- veering triangulation functor \\

$\link$ & \refcha{BackAgain} -- link space functor \\

$\cell(R)$ & \refdef{InducedVeering} -- model cell for the skeletal rectangle $R$ \\

$Q \prec R$ & \refdef{OrderLoom} -- partial order on rectangles \\

$s \prec t$ & \refdef{OrderVeer} -- partial order on tetrahedra \\

$\astro(p)$ & \refdef{Astroid} -- astroid for $p$ \\

$\axis(p)$ & \refdef{AxisForP} -- flow-axis for $p$ \\

$\point(A)$ & \refdef{PointForA} -- point for $A$ \\

$\qui(t)$ & \refdef{Quiver} -- quiver for $t$ \\

\end{longtable}

\backmatter

\renewcommand{\UrlFont}{\tiny\ttfamily}
\renewcommand\hrefdefaultfont{\tiny\ttfamily}

\bibliographystyle{plainurl}
\bibliography{bibfile}
\printindex

\end{document}